\documentclass[9pt]{amsart}
\usepackage{amsmath}
\usepackage{amsxtra}
\usepackage{amscd}
\usepackage{amsthm}
\usepackage{amsfonts}
\usepackage{amssymb}
\usepackage{eucal}
\usepackage[all]{xy}
\usepackage{graphicx}
\usepackage[usenames]{color}

\newtheorem{cor}[subsubsection]{Corollary}

\newtheorem{lem}[subsubsection]{Lemma}
\newtheorem{prop}[subsubsection]{Proposition}

\newtheorem{thm}[subsubsection]{Theorem}

\newtheorem{defn}[subsubsection]{Definition}
\newtheorem{quest}[subsubsection]{Question}

\theoremstyle{remark}
\newtheorem{rem}[subsubsection]{Remark}

\theoremstyle{definition}

\theoremstyle{remark}

\newcommand{\thmref}[1]{Theorem~\ref{#1}}

\newcommand{\secref}[1]{Sect.~\ref{#1}}
\newcommand{\lemref}[1]{Lemma~\ref{#1}}
\newcommand{\propref}[1]{Proposition~\ref{#1}}
\newcommand{\corref}[1]{Corollary~\ref{#1}}

\numberwithin{equation}{section}

\newcommand{\nc}{\newcommand}
\nc{\renc}{\renewcommand}
\nc{\ssec}{\subsection}
\nc{\sssec}{\subsubsection}
\nc{\on}{\operatorname}

\nc{\ips}{{\iota_P^{(S)}}}
\nc{\ipms}{{\iota_{P^-}^{(S)}}}
\nc{\sfpps}{{\sfp_P^{(S)}}}
\nc{\sfppms}{{\sfp_{P^-}^{(S)}}}

\nc\ol{\overline}
\nc\ul{\underline}
\nc\wt{\widetilde}
\nc\tboxtimes{\wt{\boxtimes}}
\nc\tstar{\wt{\star}}
\nc{\alp}{\alpha}

\nc{\ZZ}{{\mathbb Z}}
\nc{\NN}{{\mathbb N}}
\nc{\OO}{{\mathbb O}}
\renc{\SS}{{\mathbb S}}
\nc{\DD}{{\mathbb D}}
\nc{\GG}{{\mathbb G}}

\nc{\Fq}{{\mathbb F}_q}
\nc{\Fqb}{\ol{\mathbb F}_q}
\nc{\Ql}{{\mathbb Q}_\ell}
\nc{\Qlb}{{\ol{\mathbb Q}_\ell}}
\nc{\id}{\text{id}}
\nc\X{\mathcal X}

\nc{\red}{\on{red}}
\nc{\Ho}{\on{Ho}}
\nc{\Hom}{\on{Hom}}
\nc{\coHom}{\ul{\on{coHom}}}
\nc{\coMaps}{{\bf{coMaps}}}
\nc{\coef}{\on{coef}}
\nc{\Lie}{\on{Lie}}
\nc{\Loc}{\on{Loc}}
\nc{\Pic}{\on{Pic}}
\nc{\Bun}{\on{Bun}}
\nc{\IC}{\on{IC}}
\nc{\Aut}{\on{Aut}}
\nc{\rk}{\on{rk}}
\nc{\Sh}{\on{Sh}}
\nc{\Perv}{\on{Perv}}
\nc{\pos}{{\on{pos}}}
\nc{\Conv}{\on{Conv}}
\nc{\Sph}{\on{Sph}}
\nc{\Sym}{\on{Sym}}
\nc{\BunBb}{\overline{\Bun}_B}
\nc{\BunNb}{\overline{\Bun}_N}
\nc{\BunTb}{\overline{\Bun}_T}
\nc{\BunBbm}{\overline{\Bun}_{B^-}}
\nc{\BunBbel}{\overline{\Bun}_{B,el}}
\nc{\BunBbmel}{\overline{\Bun}_{B^-,el}}
\nc{\Buno}{\overset{o}{\Bun}}
\nc{\BunPb}{{\overline{\Bun}_P}}
\nc{\BunBM}{\Bun_{B(M)}}
\nc{\BunBMb}{\overline{\Bun}_{B(M)}}
\nc{\BunPbw}{{\widetilde{\Bun}_P}}
\nc{\BunBP}{\widetilde{\Bun}_{B,P}}
\nc{\GUb}{\overline{G/U}}
\nc{\GUPb}{\overline{G/U(P)}}

\nc{\Hhom}{\underline{\on{Hom}}}
\nc\syminfty{\on{Sym}^{\infty}}
\nc\lal{\ol{\lambda}}
\nc\xl{\ol{x}}
\nc\thl{\ol{\theta}}
\nc\nul{\ol{\nu}}
\nc\mul{\ol{\mu}}
\nc\Sum\Sigma
\nc{\oX}{\overset{o}{X}{}}
\nc{\hl}{\overset{\leftarrow}h{}}
\nc{\hr}{\overset{\rightarrow}h{}}
\nc{\M}{{\mathcal M}}
\nc{\N}{{\mathcal N}}
\nc{\F}{{\mathcal F}}
\nc{\D}{{\mathcal D}}
\nc{\Q}{{\mathcal Q}}
\nc{\Y}{{\mathcal Y}}
\nc{\G}{{\mathcal G}}
\nc{\E}{{\mathcal E}}
\nc{\CalC}{{\mathcal C}}
\nc\Dh{\widehat{\D}}

\nc{\C}{{\mathcal C}}
\nc{\K}{{\mathcal K}}
\renewcommand{\H}{{\mathcal H}}

\nc{\T}{{\mathcal T}}
\nc{\V}{{\mathcal V}}
\renc{\P}{{\mathcal P}}
\nc{\A}{{\mathcal A}}
\nc{\B}{{\mathcal B}}
\nc{\U}{{\mathcal U}}

\nc{\Gr}{{\on{Gr}}}

\nc{\frn}{{\check{\mathfrak u}(P)}}

\nc{\fC}{\mathfrak C}
\nc{\fT}{\mathfrak T}
\nc{\p}{\mathfrak p}
\nc{\q}{\mathfrak q}
\nc\f{{\mathfrak f}}

\nc{\qo}{{\mathfrak q}}
\nc{\po}{{\mathfrak p}}
\nc{\s}{{\mathfrak s}}
\nc\w{\text{w}}

\nc\mathi\iota
\nc\Spec{\on{Spec}}
\nc\Proj{\on{Proj}}
\nc\Mod{\on{Mod}}
\nc{\tw}{\widetilde{\mathfrak t}}
\nc{\pw}{\widetilde{\mathfrak p}}
\nc{\qw}{\widetilde{\mathfrak q}}
\nc{\jw}{\widetilde j}

\nc{\grb}{\overline{\Gr}}
\nc{\I}{\mathcal I}

\nc{\lambdach}{{\check\lambda}}
\nc{\Lambdach}{{\check\Lambda}{}}
\nc{\much}{{\check\mu}}
\nc{\omegach}{{\check\omega}}
\nc{\nuch}{{\check\nu}}
\nc{\etach}{{\check\eta}}
\nc{\alphach}{{\check\alpha}}
\nc{\oblvtach}{{\check\oblvta}}
\nc{\rhoch}{{\check\rho}}
\nc{\ch}{{\check h}}

\nc{\Hb}{\overline{\H}}

\nc{\BA}{{\mathbb{A}}}
\nc{\BC}{{\mathbb{C}}}
\nc{\BE}{{\mathbb{E}}}
\nc{\BF}{{\mathbb{F}}}
\nc{\BG}{{\mathbb{G}}}
\nc{\BL}{{\mathbb{L}}}
\nc{\BM}{{\mathbb{M}}}
\nc{\BO}{{\mathbb{O}}}
\nc{\BD}{{\mathbb{D}}}
\nc{\BN}{{\mathbb{N}}}
\nc{\BP}{{\mathbb{P}}}
\nc{\BQ}{{\mathbb{Q}}}
\nc{\BR}{{\mathbb{R}}}
\nc{\BZ}{{\mathbb{Z}}}
\nc{\BS}{{\mathbb{S}}}
\nc{\Deep}{{\bf{deep}}}
\nc{\deep}{deep}

\nc{\CA}{{\mathcal{A}}}
\nc{\CB}{{\mathcal{B}}}

\nc{\CE}{{\mathcal{E}}}
\nc{\CF}{{\mathcal{F}}}
\nc{\CH}{{\mathcal{H}}}

\nc{\CL}{{\mathcal{L}}}
\nc{\CC}{{\mathcal{C}}}
\nc{\CG}{{\mathcal{G}}}
\nc{\CalD}{{\mathcal{D}}}
\nc{\CM}{{\mathcal{M}}}
\nc{\CN}{{\mathcal{N}}}
\nc{\CK}{{\mathcal{K}}}
\nc{\CO}{{\mathcal{O}}}
\nc{\CP}{{\mathcal{P}}}
\nc{\CQ}{{\mathcal{Q}}}
\nc{\CR}{{\mathcal{R}}}
\nc{\CS}{{\mathcal{S}}}
\nc{\CT}{{\mathcal{T}}}
\nc{\CU}{{\mathcal{U}}}
\nc{\CV}{{\mathcal{V}}}
\nc{\CW}{{\mathcal{W}}}
\nc{\CX}{{\mathcal{X}}}
\nc{\CY}{{\mathcal{Y}}}
\nc{\CZ}{{\mathcal{Z}}}
\nc{\CI}{{\mathcal{I}}}

\nc{\csM}{{\check{\mathcal A}}{}}
\nc{\oM}{{\overset{\circ}{\mathcal M}}{}}
\nc{\obM}{{\overset{\circ}{\mathbf M}}{}}
\nc{\oCA}{{\overset{\circ}{\mathcal A}}{}}
\nc{\obA}{{\overset{\circ}{\mathbf A}}{}}
\nc{\ooM}{{\overset{\circ}{M}}{}}
\nc{\osM}{{\overset{\circ}{\mathsf M}}{}}
\nc{\vM}{{\overset{\bullet}{\mathcal M}}{}}
\nc{\nM}{{\underset{\bullet}{\mathcal M}}{}}
\nc{\oD}{{\overset{\circ}{\mathcal D}}{}}
\nc{\obD}{{\overset{\circ}{\mathbf D}}{}}
\nc{\oA}{{\overset{\circ}{A}}{}}
\nc{\op}{{\overset{\bullet}{\mathbf p}}{}}
\nc{\cp}{{\overset{\circ}{\mathbf p}}{}}
\nc{\oU}{{\overset{\bullet}{\mathcal U}}{}}
\nc{\oZ}{{\overset{\circ}{\mathcal Z}}{}}
\nc{\ofZ}{{\overset{\circ}{\mathfrak Z}}{}}
\nc{\oF}{{\overset{\circ}{\fF}}}

\nc{\fa}{{\mathfrak{a}}}
\nc{\ofa}{\overset{\circ}{\mathfrak{a}}}
\nc{\fb}{{\mathfrak{b}}}
\nc{\fd}{{\mathfrak{d}}}
\nc{\ff}{{\mathfrak{f}}}
\nc{\fg}{{\mathfrak{g}}}
\nc{\fgl}{{\mathfrak{gl}}}
\nc{\fh}{{\mathfrak{h}}}
\nc{\fj}{{\mathfrak{j}}}
\nc{\fl}{{\mathfrak{l}}}
\nc{\fm}{{\mathfrak{m}}}
\nc{\ofm}{\overset{\circ}{\mathfrak{m}}}
\nc{\fn}{{\mathfrak{n}}}
\nc{\fu}{{\mathfrak{u}}}
\nc{\fp}{{\mathfrak{p}}}
\nc{\fr}{{\mathfrak{r}}}
\nc{\fs}{{\mathfrak{s}}}
\nc{\ft}{{\mathfrak{t}}}
\nc{\oft}{\overset{\circ}{\mathfrak{t}}}
\nc{\fz}{{\mathfrak{z}}}
\nc{\fsl}{{\mathfrak{sl}}}
\nc{\hsl}{{\widehat{\mathfrak{sl}}}}
\nc{\hgl}{{\widehat{\mathfrak{gl}}}}
\nc{\hg}{{\widehat{\mathfrak{g}}}}
\nc{\chg}{{\widehat{\mathfrak{g}}}{}^\vee}
\nc{\hn}{{\widehat{\mathfrak{n}}}}
\nc{\chn}{{\widehat{\mathfrak{n}}}{}^\vee}

\nc{\fA}{{\mathfrak{A}}}
\nc{\fB}{{\mathfrak{B}}}
\nc{\fD}{{\mathfrak{D}}}
\nc{\fE}{{\mathfrak{E}}}
\nc{\fF}{{\mathfrak{F}}}
\nc{\fG}{{\mathfrak{G}}}
\nc{\fK}{{\mathfrak{K}}}
\nc{\fL}{{\mathfrak{L}}}
\nc{\fM}{{\mathfrak{M}}}
\nc{\fN}{{\mathfrak{N}}}
\nc{\fP}{{\mathfrak{P}}}
\nc{\fU}{{\mathfrak{U}}}
\nc{\fV}{{\mathfrak{V}}}
\nc{\fZ}{{\mathfrak{Z}}}

\nc{\ba}{{\mathbf{a}}}
\nc{\bb}{{\mathbf{b}}}
\nc{\bc}{{\mathbf{c}}}
\nc{\bd}{{\mathbf{d}}}
\nc{\bbf}{{\mathbf{f}}}
\nc{\be}{{\mathbf{e}}}
\nc{\bi}{{\mathbf{i}}}
\nc{\bj}{{\mathbf{j}}}
\nc{\bh}{{\mathbf{h}}}
\nc{\bm}{{\mathbf{m}}}
\nc{\bn}{{\mathbf{n}}}
\nc{\bo}{{\mathbf{o}}}
\nc{\bp}{{\mathbf{p}}}
\nc{\bq}{{\mathbf{q}}}
\nc{\bu}{{\mathbf{u}}}
\nc{\bv}{{\mathbf{v}}}
\nc{\bx}{{\mathbf{x}}}
\nc{\bs}{{\mathbf{s}}}
\nc{\by}{{\mathbf{y}}}
\nc{\bw}{{\mathbf{w}}}
\nc{\bA}{{\mathbf{A}}}
\nc{\bK}{{\mathbf{K}}}
\nc{\bB}{{\mathbf{B}}}
\nc{\bC}{{\mathbf{C}}}
\nc{\bG}{{\mathbf{G}}}
\nc{\bD}{{\mathbf{D}}}
\nc{\bE}{{\mathbf{E}}}
\nc{\bH}{{{\mathbf{H}}}}
\nc{\bM}{{\mathbf{M}}}
\nc{\bN}{{\mathbf{N}}}
\nc{\bO}{{\mathbf{O}}}
\nc{\bQ}{{\mathbf{Q}}}
\nc{\bV}{{\mathbf{V}}}
\nc{\bW}{{\mathbf{W}}}
\nc{\bX}{{\mathbf{X}}}
\nc{\bZ}{{\mathbf{Z}}}
\nc{\bS}{{\mathbf{S}}}

\nc{\sA}{{\mathsf{A}}}
\nc{\sB}{{\mathsf{B}}}
\nc{\sC}{{\mathsf{C}}}
\nc{\sD}{{\mathsf{D}}}
\nc{\sF}{{\mathsf{F}}}
\nc{\sG}{{\mathsf{G}}}
\nc{\sH}{{\mathsf{H}}}
\nc{\sK}{{\mathsf{K}}}
\nc{\sM}{{\mathsf{M}}}
\nc{\sN}{{\mathsf{N}}}
\nc{\sO}{{\mathsf{O}}}
\nc{\sW}{{\mathsf{W}}}
\nc{\sQ}{{\mathsf{Q}}}
\nc{\sP}{{\mathsf{P}}}
\nc{\sR}{{\mathsf{R}}}
\nc{\sS}{{\mathsf{S}}}
\nc{\sT}{{\mathsf{T}}}
\nc{\sZ}{{\mathsf{Z}}}
\nc{\sfp}{{\mathsf{p}}}
\nc{\sfq}{{\mathsf{q}}}
\nc{\sft}{{\mathsf{t}}}
\nc{\sr}{{\mathsf{r}}}
\nc{\bk}{{\mathsf{k}}}
\nc{\sa}{{\mathsf{s}}}
\nc{\sg}{{\mathsf{g}}}
\nc{\sn}{{\mathsf{n}}}
\nc{\sh}{{\mathsf{h}}}
\nc{\sff}{{\mathsf{f}}}
\nc{\sfb}{{\mathsf{b}}}
\nc{\sfc}{{\mathsf{c}}}
\nc{\sfe}{{\mathsf{e}}}
\nc{\sd}{{\mathsf{d}}}

\nc{\BK}{{\bar{K}}}

\nc{\tA}{{\widetilde{\mathbf{A}}}}
\nc{\tB}{{\widetilde{\mathcal{B}}}}
\nc{\tg}{{\widetilde{\mathfrak{g}}}}
\nc{\tG}{{\widetilde{G}}}
\nc{\TM}{{\widetilde{\mathbb{M}}}{}}
\nc{\tO}{{\widetilde{\mathsf{O}}}{}}
\nc{\tU}{{\widetilde{\mathfrak{U}}}{}}
\nc{\TZ}{{\tilde{Z}}}
\nc{\tx}{{\tilde{x}}}
\nc{\tbv}{{\tilde{\bv}}}
\nc{\tfP}{{\widetilde{\mathfrak{P}}}{}}
\nc{\tz}{{\tilde{\zeta}}}
\nc{\tmu}{{\tilde{\mu}}}

\nc{\urho}{\underline{\rho}}
\nc{\uB}{\underline{B}}
\nc{\uC}{{\underline{\mathbb{C}}}}
\nc{\ui}{\underline{i}}
\nc{\uj}{\underline{j}}
\nc{\ofP}{{\overline{\mathfrak{P}}}}
\nc{\oB}{{\overline{\mathcal{B}}}}
\nc{\og}{{\overline{\mathfrak{g}}}}
\nc{\oI}{{\overline{I}}}

\nc{\eps}{\varepsilon}
\nc{\hrho}{{\hat{\rho}}}

\nc{\one}{{\mathbf{1}}}
\nc{\two}{{\mathbf{t}}}

\nc{\Rep}{{\mathop{\operatorname{\rm Rep}}}}
\nc{\Tot}{{\mathop{\operatorname{\rm Tot}}}}
\nc{\Ker}{{\mathop{\operatorname{\rm Ker}}}}
\nc{\im}{{\mathop{\operatorname{\rm Im}}}}
\nc{\Hilb}{{\mathop{\operatorname{\rm Hilb}}}}
\nc{\End}{{\mathop{\operatorname{\rm End}}}}
\nc{\Ext}{{\mathop{\operatorname{\rm Ext}}}}
\nc{\CHom}{{\mathop{\operatorname{{\mathcal{H}}\it om}}}}
\nc{\CEnd}{{\mathop{\operatorname{{\mathcal{E}}\it nd}}}}
\nc{\GL}{{\mathop{\operatorname{\rm GL}}}}
\nc{\gr}{{\mathop{\operatorname{\rm gr}}}}
\nc{\HN}{{\mathop{\operatorname{\rm HN}}}}
\nc{\Id}{{\mathop{\operatorname{\rm Id}}}}
\nc{\de}{{\mathop{\operatorname{\rm def}}}}
\nc{\length}{{\mathop{\operatorname{\rm length}}}}
\nc{\supp}{{\mathop{\operatorname{\rm supp}}}}

\nc{\Cliff}{{\mathsf{Cliff}}}
\nc{\Fl}{\on{Fl}}
\nc{\Fib}{{\mathsf{Fib}}}
\nc{\Coh}{{\on{Coh}}}
\nc{\QCoh}{{\on{QCoh}}}
\nc{\IndCoh}{{\on{IndCoh}}}
\nc{\FCoh}{{\mathsf{FCoh}}}

\nc{\reg}{{\text{\rm reg}}}

\nc{\cplus}{{\mathbf{C}_+}}
\nc{\cminus}{{\mathbf{C}_-}}
\nc{\cthree}{{\mathbf{C}_\bullet}}
\nc{\Qbar}{{\bar{Q}}}
\nc\Eis{\on{Eis}}
\nc\Eisb{\ol\Eis{}}
\nc\Eisr{\on{Eis}^{rat}{}}
\nc\wh{\widehat}
\nc{\Def}{\on{Def_{\check{\fb}}(E)}}
\nc{\barZ}{\overline{Z}{}}
\nc{\barbarZ}{\overline{\barZ}{}}
\nc{\barpi}{\overline\pi}
\nc{\barbarpi}{\overline\barpi}
\nc{\barpip}{\overline\pi{}^+}
\nc{\barpim}{\overline\pi{}^-}

\nc{\fq}{\mathfrak q}

\nc{\fqb}{\ol{\sfq}{}}
\nc{\fpb}{\ol{\sfp}{}}
\nc{\fpr}{{\sfp^{rat}}{}}
\nc{\fqr}{{\sfq^{rat}}{}}

\nc{\hattimes}{\wh\otimes}

\nc{\bOmega}{{\overline{\Omega(\check \fn)}}}

\nc{\seq}[1]{\stackrel{#1}{\sim}}

\nc{\cT}{{\check{T}}}
\nc{\cG}{{\check{G}}}
\nc{\cM}{{\check{M}}}
\nc{\cB}{{\check{B}}}

\nc{\ct}{{\check{\mathfrak t}}}
\nc{\cg}{{\check{\fg}}}
\nc{\cb}{{\check{\fb}}}
\nc{\cn}{{\check{\fn}}}

\nc{\cLambda}{{\check\Lambda}}

\nc{\cla}{{\check\lambda}}
\nc{\cmu}{{\check\mu}}
\nc{\cnu}{{\check\nu}}
\nc{\ceta}{{\check\eta}}

\nc{\DefbE}{{\on{Def}_{\cB}(E_\cT)}}

\nc{\imathb}{{\ol{\imath}}}
\nc{\rlr}{\overset{\longrightarrow}{\underset{\longrightarrow}\longleftarrow}}

\nc{\oBun}{\overset{\circ}\Bun}
\nc{\LocSys}{\on{LocSys}}
\nc{\BunBbb}{\ol{\ol{Bun}}_B}
\nc{\BunBr}{\Bun_B^{rat}}
\nc{\BunBrsg}{\Bun_B^{rat,\on{s.g.}}}
\nc{\BunBrp}{\Bun_B^{rat,polar}}
\nc{\BunBrpbg}{\Bun_B^{rat,polar,\on{b.g.}}}
\nc{\BunBrpsg}{\Bun_B^{rat,polar,\on{s.g.}}}
\nc{\BunTrp}{\Bun_T^{rat,polar}}
\nc{\BunTrpbg}{\Bun_T^{rat,polar,\on{b.g.}}}
\nc{\BunTrpsg}{\Bun_T^{rat,polar,\on{s.g.}}}
\nc{\BunNr}{\Bun_N^{rat}}
\nc{\BunNre}{\Bun_N^{enh,rat}}
\nc{\BunTr}{\Bun_T^{rat}}
\nc{\Vect}{\on{Vect}}
\nc{\Whit}{\on{Whit}}
\nc{\CTb}{\ol{\on{CT}}}
\nc{\Ran}{{\on{Ran}}}
\nc{\CTr}{\on{CT}^{rat}{}}
\nc\jmathr{\jmath^{rat}{}}
\nc{\ux}{\underline{x}}
\nc{\clambda}{{\check\lambda}}
\nc{\calpha}{{\check\alpha}}
\nc{\ind}{{\mathbf{ind}}}
\nc{\coinv}{{\mathbf{coinv}}}
\nc{\oblv}{{\mathbf{oblv}}}
\nc{\free}{{\mathbf{free}}}
\nc{\ox}{{\overline{x}}}
\nc{\cLa}{\check{\Lambda}}
\nc{\StinftyCat}{\on{DGCat}}
\nc{\inftyCat}{\infty\on{-Cat}}
\nc{\inftygroup}{\infty\on{-Grpd}}
\nc{\Dmod}{\on{D-mod}}
\nc{\CMaps}{{\mathcal Maps}}
\nc{\Maps}{\on{Maps}}
\nc{\affSch}{\on{Sch}^{\on{aff}}}
\nc{\dr}{{\on{dR}}}
\nc{\oCF}{\overset{\circ}\CF}
\nc{\oCY}{\overset{\circ}\CY}
\nc{\opi}{\overset{\circ}\pi}
\nc{\leqG}{\underset{G}\leq}
\nc{\leqM}{\underset{M}\leq}
\nc{\leqGad}{\underset{G_{ad}}\leq}
\nc{\leqMad}{\underset{M_{ad}}\leq}
\nc{\Tr}{\on{Tr}}
\nc{\Frob}{{\on{Frob}}}
\nc{\DGCat}{\on{DGCat}}
\nc{\tDGCat}{2\on{-DGCat}_{\on{u.g.}}}
\nc{\ev}{\on{ev}}
\nc{\mmod}{\on{-}\mathbf{mod}}
\nc{\sotimes}{\overset{!}\otimes}
\nc{\lstar}{\overset{l}\star}
\nc{\Shv}{\on{Shv}}
\nc{\Spc}{\on{Spc}}
\nc{\LS}{\Lisse}
\nc{\Res}{\on{Res}}
\nc{\bDelta}{{\mathbf{\Delta}}}
\nc{\bMaps}{{\mathbf{Maps}}}
\nc{\cD}{\mathcal D}
\nc{\ocD}{\overset{\circ}\cD}
\nc{\ppart}{(\!(t)\!)}
\nc{\qqart}{[\![t]\!]}
\nc{\oCU}{\overset{\circ}{\CU}}
\nc{\Exc}{{\mathcal{E}xc}}
\nc{\Sht}{\on{Sht}}
\nc{\Nilp}{{\on{Nilp}}}
\nc{\Drinf}{\on{Drinf}}
\nc{\Sing}{\on{Sing}}
\nc{\IndLisse}{\Lisse}
\nc{\Shvl}{\on{Shv}_{\on{lisse}}} 
\nc{\Lisse}{\on{Lisse}}
\nc{\qLisse}{\on{QLisse}}
\nc{\Mir}{\on{Mir}}
\nc{\Se}{\on{Se}}

\begin{document}

\title[Duality for automorphic sheaves with nilpotent singular support]
{Duality for automorphic sheaves with nilpotent \\ singular support}

\author{D.~Arinkin, D.~Gaitsgory, D.~Kazhdan, 
S.~Raskin, N.~Rozenblyum, Y.~Varshavsky}

\date{\today}

\begin{abstract}
We identify the category $\Shv_\Nilp(\Bun_G)$ of automorphic sheaves with nilpotent singular support
with its own dual, and relate this structure to the Serre functor on $\Shv_\Nilp(\Bun_G)$ and miraculous duality.
\end{abstract} 

\maketitle

\tableofcontents

\section*{Introduction}

\ssec{What is this paper about?}

This paper is a sequel to the paper \cite{AGKRRV} by the same set of authors. In {\it loc. cit.}
a conjecture was suggested that the trace of the Frobenius endofunctor on the category of
automorphic sheaves with nilpotent singular support is isomorphic to the space of (non-ramified)
automorphic functions. This paper carries out the first step towards the proof of this conjecture. 

\sssec{}

Let $X$ be a smooth projective curve and $G$ a reductive group (over a ground field $k$).
Let $\Bun_G$ be the moduli stack of $G$-bundles on $X$. We consider the category of automorphic
sheaves $\Shv(\Bun_G)$ (see \secref{sss:sheaves} for our conventions regarding sheaf theories) 
and its full subcategory
$$\Shv_\Nilp(\Bun_G)\subset \Shv(\Bun_G).$$

The category $\Shv_\Nilp(\Bun_G)$ is compactly generated, and hence dualizable. So, given an endofunctor
$\Phi$ of $\Shv_\Nilp(\Bun_G)$, it makes sense to consider its trace
$$\Tr(\Phi,\Shv_\Nilp(\Bun_G))\in \Vect,$$
where $\Vect$ is the DG category of chain complexes of vector space over the field $\sfe$ of coefficients of 
our sheaf theory.

\sssec{}

By definition, the trace of an endofunctor $\Phi$ of a dualizable category $\bC$ is the endofunctor of $\Vect$ equal to the composition
$$\Vect\overset{\on{unit}}\longrightarrow \bC\otimes \bC^\vee
\overset{\Phi\otimes \on{Id}}\longrightarrow \bC\otimes \bC^\vee \overset{\on{counit}}\to \Vect,$$
where $\bC^\vee$ is the dual category, and $\on{unit}$ and $\on{counit}$ are the unit and the counit of the duality, respectively.

\medskip

Hence, in order to say something explicit about the trace $\Tr(\Phi,\bC)$, one must first describe $\bC^\vee$ and the functors
$\on{unit}$ and $\on{counit}$. 

\medskip

This is what we do for the category $\Shv_\Nilp(\Bun_G)$.

\sssec{} \label{sss:intro new}

From the above perspective, the main result of this note (\thmref{t:non-standard duality}) is easy to explain. It says that
there exists a canonical equivalence
\begin{equation} \label{e:new intro Nilp}
\Shv_\Nilp(\Bun_G)^\vee \simeq \Shv_\Nilp(\Bun_G),
\end{equation}
for which the counit of the duality is the functor
$$\Shv_\Nilp(\Bun_G)\otimes \Shv_\Nilp(\Bun_G)\to \Vect,$$
denoted $\on{ev}^l_{\Bun_G}$, given by
$$\CF_1,\CF_2\mapsto \on{C}^\cdot_c(\Bun_G,\CF_1\overset{*}\otimes \CF_2).$$

\sssec{}

We should emphasize that the above pairing is really different from the usual Verdier duality pairing,
which makes sense for a quasi-compact algebraic stack $\CY$:
$$\on{ev} _\CY(\CF_1,\CF_2):=\on{C}^\cdot_\blacktriangle(\Bun_G,\CF_1\sotimes \CF_2)$$
(here $\on{C}^\cdot_\blacktriangle(-)$ is the renormalized version of the functor
of sheaf cochains, see \secref{sss:renorm cochains}). 

\medskip

However, $\Bun_G$ is \emph{not} quasi-compact, so the pairing $\on{ev} _{\Bun_G}$ would
not make sense ``as-is". Rather, $\on{ev} _{\Bun_G}$ defines a perfect pairing between
$\Shv(\Bun_G)$ (resp., $\Shv_\Nilp(\Bun_G)$) and its ``co" version, denoted 
$\Shv(\Bun_G)_{\on{co}}$ (resp., $\Shv_\Nilp(\Bun_G)_{\on{co}}$), see \secref{ss:co BunG}.

\medskip

The relationship between $\on{ev}^l_{\Bun_G}$ and $\on{ev} _{\Bun_G}$ is the second main
point of this paper, to be discussed below.

\sssec{}

The assertion that $\on{ev}^l_{\Bun_G}$ defines a perfect pairing may seem innocuous enough,
but it is actually very non-trivial. In fact, it uses the full strength of some of the key results of the
paper \cite{AGKRRV}.

\medskip

Namely, consider the embedding
$$\Shv_{\Nilp\times \Nilp}(\Bun_G\times \Bun_G) \hookrightarrow \Shv(\Bun_G\times \Bun_G);$$
let 
$$\on{ps-u}_{\Bun_G,\Nilp}\in \Shv_{\Nilp\times \Nilp}(\Bun_G\times \Bun_G)$$
be the object equal to the value of the spectral projector $\sP_{\Bun_G,\Nilp}\boxtimes \on{Id}_{\Bun_G}$
(see \secref{sss:the projector}) on the object
$$(\Delta_{\Bun_G})_!(\ul\sfe_{\Bun_G})\in \Shv(\Bun_G\times \Bun_G),$$
where $\ul\sfe_{\Bun_G}\in \Shv(\Bun_G)$ is the constant sheaf.

\medskip

The assertion of \thmref{t:non-standard duality} is easily deduced from the fact that the above object $\on{ps-u}_{\Bun_G,\Nilp}$
belongs to the essential image of (what is a priori just a fully faithful functor):
\begin{equation} \label{e:Kunneth intro}
\Shv_\Nilp(\Bun_G)\otimes \Shv_\Nilp(\Bun_G) \to \Shv_{\Nilp\times \Nilp}(\Bun_G\times \Bun_G),
\end{equation}
thereby providing the unit for the sought-for self-duality of $\Shv_\Nilp(\Bun_G)$.

\medskip

Now, it turns out that the functor \eqref{e:Kunneth intro} is actually an equivalence; this is the
highly non-trivial \cite[Theorem 16.3.3]{AGKRRV}. We give a slightly different proof of this
theorem in the present paper (see \thmref{t:Kunneth BunG}), but one which still uses the key results of \cite{AGKRRV}.

\ssec{Relation to miraculous duality and the Serre functor}

\sssec{}

Recall that there \emph{is} a way to define a self-duality on the entire category $\Shv(\Bun_G)$,
as well as its subcategory $\Shv_\Nilp(\Bun_G)$.

\medskip

Namely, we consider the category $\Shv(\Bun_G)_{\on{co}}$ (see \secref{ss:co BunG} for the definition),
and Verdier duality defines an identification
$$\Shv(\Bun_G)^\vee \simeq \Shv(\Bun_G)_{\on{co}},$$
with the counit defined by the pairing $\on{ev} _{\Bun_G}$ mentioned above.

\medskip

Now, in the paper \cite{Ga1}, it is shown that a certain canonically defined functor, 
denoted in this note by 
$$\Mir_{\Bun_G}: \Shv(\Bun_G)_{\on{co}}\to \Shv(\Bun_G),$$
is an equivalence. The functor $\Mir_{\Bun_G}$, which makes sense for any algebraic stack, was proposed by Drinfeld;
we refer to it as the \emph{miraculous functor}. 

\medskip

Combining Verdier duality with the miraculous functor, we obtain an identification
\begin{equation} \label{e:Mir intro}
\Shv(\Bun_G)^\vee \simeq \Shv(\Bun_G).
\end{equation}

Following Drinfeld, we refer to \eqref{e:Mir intro} as the \emph{miraculous self-duality} of $\Shv(\Bun_G)$.

\sssec{}

One shows (see \corref{c:Mir duality Nilp}) that the identification \eqref{e:Mir intro}
induces an identification
\begin{equation} \label{e:Mir intro Nilp}
\Shv_\Nilp(\Bun_G)^\vee \simeq \Shv_\Nilp(\Bun_G). 
\end{equation}

Now, the second main result of this note, \corref{c:non-st vs mir}, says that the identification
\eqref{e:Mir intro Nilp} equals \eqref{e:new intro Nilp} introduced in \secref{sss:intro new}. 

\sssec{}

In the process of identifying \eqref{e:Mir intro Nilp} with \eqref{e:new intro Nilp} we relate 
$\Mir_{\Bun_G}$ to the \emph{pseudo-identity} and \emph{Serre} functors
(see Sects. \ref{sss:ps-id} and \ref{sss:Serre} for what we mean by the latter functors).

\medskip

Namely, let $\CU$ be a universally $\Nilp$-\emph{contruncative} quasi-compact open substack of $\Bun_G$
(see \secref{sss:N-cotrunc} for what this means; such substacks form a cofinal subset among all quasi-compact open substacks of $\Bun_G$).
Consider the endofunctor 
$$\Mir_\CU:\Shv(\CU)\to \Shv(\CU),$$
and let us restrict it to $\Shv_\Nilp(\CU)\subset \Shv(\CU)$.

\medskip

We show that this restriction sends $\Shv_\Nilp(\CU)$ to $\Shv_\Nilp(\CU)$. We show that 
the resulting endofunctor of $\Shv_\Nilp(\CU)$ is an equivalence, which 
identifies with $\on{Ps-Id}_{\Shv_\Nilp(\CU)}$, 
and also with the \emph{inverse} of the Serre functor $\on{Se}_{\Shv_\Nilp(\CU)}$. In particular, the categories 
$\Shv_\Nilp(\CU)$ are Serre (see \secref{sss:Serre defn} for what this means).

\medskip

We should emphasize that the idea that for a Serre category, the Serre and pseudo-identity endofunctors
are mutually inverse goes back to A.~Yom Din (it was recorded in the paper \cite{GaYo}).

\sssec{}

Returning to the entire $\Bun_G$, we show that the Serre functor on $\Shv_\Nilp(\Bun_G)$ canonically
factors as
$$\Shv_\Nilp(\Bun_G) \overset{\on{Se}_{\Shv_\Nilp(\Bun_G),\on{co}}}\longrightarrow \Shv_\Nilp(\Bun_G)_{\on{co}}
\overset{\on{Id}^{\on{naive}}_{\Bun_G}}\longrightarrow \Shv_\Nilp(\Bun_G),$$
where the first arrow is an equivalence inverse to
$$\Mir_{\Bun_G}|_{\Shv_\Nilp(\Bun_G)_{\on{co}}}:\Shv_\Nilp(\Bun_G)_{\on{co}}\to \Shv_\Nilp(\Bun_G),$$
and the second arrow is the tautological functor, denoted $\on{Id}^{\on{naive}}_{\Bun_G}$, defined for any
non-quasi-compact stack (see \secref{sss:Id naive}).

\medskip

Thus, we obtain that $\Shv_\Nilp(\Bun_G)$ is ``Serre, up to the issue of non-quasi-compactness", which is compensated
by replacing the target $\Shv_\Nilp(\Bun_G)$ with $\Shv_\Nilp(\Bun_G)_{\on{co}}$. 

\begin{rem}

The functor 
$$\Mir_{\Bun_G}^{-1}:\Shv(\Bun_G)\to \Shv(\Bun_G)_{\on{co}}$$
has been recently studied by L.~Chen in \cite{Ch}. Namely, in {\it loc. cit.} 
it was shown that it is canonically isomorphic to the \emph{Deligne-Lusztig} functor,
the latter being an explicit complex whose terms are composites of Eisenstein and
Constant Term functors.

\medskip

A spectral counterpart of the Deligne-Lusztig functor, which is an endofunctor of 
the category $\IndCoh_\Nilp(\LocSys_\cG(X))$ has been studied by D.~Beraldo in \cite{Be}.
In {\it loc. cit.} it is shown that the spectral Deligne-Lusztig functor is the composition
$$\IndCoh_\Nilp(\LocSys_\cG(X))\twoheadrightarrow \QCoh(\LocSys_\cG(X))\to \QCoh(\LocSys_\cG(X))\hookrightarrow
\IndCoh_\Nilp(\LocSys_\cG(X)),$$
where the middle arrow $\QCoh(\LocSys_\cG(X))\to \QCoh(\LocSys_\cG(X))$ is given by tensor product with
an explicit object of $\QCoh(\LocSys_\cG(X))$, called the Steinberg object, which is in some sense
``the structure sheaf of the locus of semi-simple local systems".

\end{rem}

\ssec{An intrinsic characterization of $\Shv_\Nilp(\Bun_G)$}

\sssec{}

The isomorphism between the functors \eqref{e:Mir intro Nilp} with \eqref{e:new intro Nilp} has another,
rather unexpected consequence:

\medskip

It turns out that the subcategory
%\footnote{For simplicity, we are assuming here the constraccessiblity conjecture
%of \cite{AGKRRV}; otherwise, in what follows, instead of $\Shv_\Nilp(\Bun_G)^c$, one should consider $\Shv(\Bun_G)^c\subset \Shv(\Bun_G)$.} 
$\Shv_\Nilp(\Bun_G)^{\on{constr}}$ can be characterized intrinsically, without mentioning the singular support condition. 

\sssec{}

In \secref{s:ker} we recall the construction of a 2-category, whose objects are algebraic stacks,
and whose category of morphisms for a given pair of stacks $\CY_1$ and $\CY_2$ is
$$\Shv(\CY_1\times \CY_2).$$

An object $\CQ\in \Shv(\CY_1\times \CY_2)$ gives rise to a functor 
$$\sQ:\Shv(\CY_1)\to \Shv(\CY_2),$$
but the data of $\CQ$ carries more information. Namely, $\CQ$ gives rise to functors denoted
$$\on{Id}_\CZ\boxtimes \sQ:\Shv(\CZ\times \CY_1)\to \Shv(\CZ\times \CY_2)$$
for any algebraic stack $\CZ$ that commute appropriately with functors that relate different $\CZ$'s
(see \secref{sss:compat kernels new}).

\sssec{}

Thus, given $\CQ\in \Shv(\CY_1\times \CY_2)$ one can ask whether it admits a right or a left adjoint
within the above 2-category. 

\medskip 

For example, $\CQ$ admits a right adjoint if and only if the above functors $\on{Id}_\CZ\boxtimes \sQ$ admit right
adjoints for every $\CZ$ (this is not very restrictive), \emph{and} if these right adjoints, denoted 
$(\on{Id}\otimes \sQ)^R$, themselves commute appropriately with functors that relate different $\CZ$'s
(this is a really non-trivial condition).

\sssec{}

Let us consider the case when $\CY_1=Y$ is a connected separated scheme and $\CY_2=\on{pt}$. Then an
object 
$$\CF\in \Shv(Y)=\Shv(\CY_1\times \CY_2)$$
admits a right adjoint if and only if the following conditions are satisfied: (i) $Y$ is smooth and proper;
(ii) $\CF\in \Lisse(Y)$ (i.e., $\CF$ is bounded, and each of its cohomologies is a locally system of finite rank). 

\sssec{}

The third main result of this paper (\thmref{t:char of Nilp adj}) says that an object
$\CF\in \Shv(\Bun_G)^{\on{constr}}$ admits a right adjoint if and only if it belongs to $\Shv_\Nilp(\Bun_G)$.

\medskip

This result suggests that in a certain sense, the stack $\Bun_G$ (or any $\Nilp$-cotruncative open substack
thereof) together with the subset $\Nilp$ of its cotangent bundle, behaves similarly to a proper smooth
scheme $Y$, with $\{0\}\subset T^*(Y)$, see also Sects. \ref{sss:ex smooth prel} and \ref{sss:ex smooth}.

\ssec{Organization of the paper} Let us briefly review the contents of this work.

\medskip

In \secref{s:Nilp} we extend the results of \cite[Part III]{AGKRRV} concerning $\Shv_\Nilp(\Bun_G)$
to the case when instead of $\Bun_G$, we consider the relative situation, i.e., $\CZ\times\Bun_G$,
where $\CZ$ is an arbitrary algebraic stack.

\medskip

In \secref{s:Verdier} we show that the category $\Shv_\Nilp(\Bun_G)$ is self-dual via a combination 
of Verdier duality and the miraculous functor. 

\medskip

In \secref{s:pairing} we prove the main result relevant to the trace calculation, namely, that the 
pairing $\on{ev}^l_{\Bun_G}$ defines the counit of a self-duality on $\Shv_\Nilp(\Bun_G)$.

\medskip

In \secref{s:adj} we establish the intrinsic characterization of the subcategory
$\Shv_\Nilp(\Bun_G)^{\on{constr}}\subset \Shv(\Bun_G)$ as consisting of objects that admit
right adjoints when viewed as kernels.

\medskip

In \secref{s:Serre} we relate the miraculous functor on $\Shv_\Nilp(\Bun_G)$ to the Serre functor, 
and establish the Serre property of categories $\Shv_\Nilp(\CU)$, where $\CU$
is a universally $\Nilp$-contruncative quasi-compact open substack of $\Bun_G$. 

\medskip

In \secref{s:sheaves} we collect some basic facts pertaining to the category of sheaves
on algebraic stacks; in particular to Verdier duality in this situation and the functor of
renormalized direct image.

\medskip

In \secref{s:ker} we review the theory of functors defined by kernels,
their behavior with respect to adjunctions and relation to the miraculous functor. 

\medskip

In \secref{s:non qc} we review the theory of sheaves on non quasi-compact stacks,
the ``co"- categories, and functors defined by kernels in this situation.

\ssec{Conventions}

\sssec{Algebraic geometry} \label{sss:stacks}

Throughout the note we will work over a ground field $k$, assumed algebraically closed. 
The algebraic geometry over $k$ will be \emph{classical}, i.e., non-derived.

\medskip

Our algebro-geometric objects will be either schemes of finite type or algebraic stacks locally
of finite type over $k$. In this note we will not need more general prestacks. 

\medskip

All \emph{quasi-compact} algebraic stacks that appear in this paper will be of the form 
$Z/H$, where $Z$ is a scheme (of finite type) and $H$ a linear algebraic group. All 
non quasi-compact algebraic stacks that appear in this paper will be unions of 
quasi-compact stacks of the above form. 

\sssec{Higher algebra}

Let $\sfe$ be a field of coefficients, assumed algebraically closed and of characteristic $0$.
Our main objects of study are DG categories over $\sfe$. In our treatment of DG categories
we follow the conventions of \cite{AGKRRV}.

\medskip

In particular, we will denote by $\DGCat$ the $(\infty,1)$-category, whose objects are 
cocomplete DG categories, and whose 1-morphisms are colimit-preserving functors.

\medskip 

The category $\DGCat$ carries a symmetric monoidal structure, given by the Lurie tensor
product; we denote it by
$$\bC_1,\bC_2\mapsto \bC_1\otimes \bC_2.$$

The unit object for this symmetric monoidal structure is the DG category of chain complexes
of $\sfe$-vector spaces, denoted $\Vect$. 

\medskip

In particular, we can talk about dualizable objects in $\DGCat$. These are 
\emph{dualizable} DG categories. For a dualizable DG category $\bC$ we will denote
by $\bC^\vee$ its dual. 

\sssec{Compact generation} \label{sss:comp}

For a given DG category $\bC$ and objects $\bc_1,\bc_2\in \bC$ we will denote by
$$\CHom_\bC(\bc_1,\bc_2)\in \Vect$$
the object corresponding to the canonical enrichment of $\bC$ over $\Vect$.

\medskip

For a given $\bC$ we will denote by $\bC^c$ the subcategory consisting of compact objects, i.e., 
those objects for which the functor
$$\CHom_\bC(\bc,-):\bC\to \Vect$$
preserves colimits. 

\medskip

A category $\bC$ is said to be compactly generated if $\CHom_\bC(\bc,\bc')=0$ for all $\bc\in \bC^c$ implies
$\bc'=0$. In this case, $\bC$ can be recovered as the \emph{ind-completion} of its subcategory $\bC^c$.

\medskip

Furthermore, if $\bC$ is compactly generated, it is dualizable, and its dual $\bC^\vee$ can be described
explicitly as follows: $\bC^\vee$ is also compactly generated and its subcategory of compact objects
$(\bC^\vee)^c$ identifies with $(\bC^c)^{\on{op}}$. 

\ssec{Acknowledgements}

D.G. would like to thank A.~Yom Din for suggesting the initial idea, worked out in \cite{GaYo},
about the relationship between the Serre and pseudo-identity functors. 

\medskip

The entire project was supported by David Kazhdan's ERC grant No 669655
and BFS grant 2020189. The work of D.K. and Y.V. was supported BSF grant 2016363. 

\medskip

The work of D.A. was supported by NSF grant DMS-1903391. 
The work of D.G. was supported by NSF grant DMS-2005475. 
The work of D.K. was partially supported by ISF grant 1650/15. 
The work of S.R. was supported by NSF grant DMS-2101984.
The work of Y.V. was partially supported by ISF grants 822/17 and 2019/21. 

\section{The category $\Shv_\Nilp(\Bun_G)$ and is variants} \label{s:Nilp}

In this section we extend the results of \cite[Part III]{AGKRRV} concerning $\Shv_\Nilp(\Bun_G)$
to the case when instead of $\Bun_G$, we consider the relative situation, i.e., $\CZ\times\Bun_G$,
where $\CZ$ is an arbitrary algebraic stack.

\ssec{Hecke action in the relative situation}

\sssec{}

For the duration of the paper we fix a smooth projective curve $X$, 
and a reductive group $G$ (both over $k$). 

\medskip

We let $\Bun_G$ denote the moduli stack of principal $G$-bundles on $X$. 

\sssec{}

For a finite set $I$, we consider the $I$-legged Hecke stack $\CH_{G,X^I}$ 
$$\Bun_G \overset{\hl}\leftarrow \CH_{G,X^I} \overset{\hr \times s}\longrightarrow \Bun_G\times X^I.$$

Let 
$$\on{Sat}_I:\Rep(\cG)^{\otimes I}\to \Shv(\CH_{G,X^I})$$
denote the geometric Satake functor. 

\medskip

We normalize it so that the value of $\on{Sat}_I$ on the trivial representation 
$(\one_{\Rep(\cG)})^{\otimes I}$ is (the direct image of) $\omega_{\Bun_G\times X^I}$,
where 
$$\Bun_G\times X^I\hookrightarrow \CH_{G,X^I}$$
is the unit section. 

\begin{rem}
The entire $\CH_{G,X^I}$ is an ind-algebraic stack, i.e., a filtered union of algebraic stacks
that map to each other by means of closed embeddings. 

\medskip

For each compact object $V\in \Rep(\cG)^{\otimes I}$, the object $\on{Sat}_I(V)\in \Shv(\CH_{G,X^I})$
is supported on one of the closed algebraic substacks of $\CH_{G,X^I}$. 

\medskip

Hence, in the discussion below,
one can start by working with compact objects of $\Rep(\cG)^{\otimes I}$, and thus deal with usual
algebraic stacks, and then ind-extend to all of $\Rep(\cG)^{\otimes I}$. 

\end{rem}

\sssec{}

Let $\CZ$ be an arbitrary algebraic stack. We can consider Hecke functors, denoted
\begin{equation} \label{e:Hecke action}
\on{Id}_\CZ\boxtimes \sH(-,-):\Rep(\cG)^{\otimes I}\otimes \Shv(\CZ\times \Bun_G)\to\Shv(\CZ\times \Bun_G\times X^I), \quad I\in \on{fSet}. 
\end{equation}

\medskip

Namely, for $V\in \Rep(\cG)^{\otimes I}$, the Hecke functor $\on{Id}_\CZ\boxtimes \sH(V,-)$ is given by
\begin{equation} \label{e:Hecke action formula}
\on{Id}_\CZ\boxtimes \sH(V,\CF):=
(\on{id}_\CZ \times (\hr \times s))_* \left((\on{id}_\CZ \times \hl)^!(\CF)\overset{!}\otimes (\omega_\CZ \boxtimes \on{Sat}_I(V))\right), 
\quad \CF\in \Shv(\CZ\times \Bun_G).
\end{equation}

\medskip

Note that in the above formula, the functor $(\on{id}_\CZ \times (\hr \times s))_*$ is canonically isomorphic
to the functor $(\on{id}_\CZ \times (\hr \times s))_\blacktriangle$, since the morphism $\hr \times s$ is (ind)-schematic
(it is actually (ind)-proper). 

\medskip

The functors $\on{Id}_\CZ\boxtimes \sH(V,-)$ introduced above fall into the paradigm of functors \emph{defined} by a kernel,
see \secref{ss:ker} for what this means (see also \secref{ss:functors by ker non qc} for 
the formalism of functors defined by kernels on non quasi-compact stacks). 

\medskip

Thus, the notation $\on{Id}_\CZ\boxtimes \sH(V,-)$ is consistent with one in \secref{sss:ker Z}. 

%\begin{rem}
%This notation $\on{Id}_\CZ\boxtimes \sH(-,-)$
%is consistent with one in \secref{sss:functors by kernels}, since the above functors
%fall into the paradigm of functors given by kernels, see \secref{sss:Hecke ULA} below; 
%we refer the reader to \secref{sss:???}, where the formalism of functors given by kernels is 
%developed for non quasi-compact stacks. 
%\end{rem}

\sssec{} \label{sss:Hecke ULA}

A key property of the objects of the form $\on{Sat}_I(V)$, for $V\in (\Rep(\cG)^{\otimes I})^c$, 
 is that they are ULA with respect to the 
map $\hl\times s$ (and, by symmetry, also with respect to the map $\hr\times s$). 

\medskip

In particular, since $X^I$ is smooth, the objects $\on{Sat}_I(V)$, for $V\in (\Rep(\cG)^{\otimes I})^c$, 
are ULA with respect to the projection $\hl$. This implies that there are canonical isomorphisms
$$(\on{Id}_\CZ \times \hl)^!(\CF)\sotimes (\omega_\CZ \boxtimes \on{Sat}_I(V))
\simeq 
(\on{Id}_\CZ \times \hl)^*(\CF)\overset{*}\otimes (\ul\sfe_\CZ \boxtimes (\on{Sat}_I(V)\sotimes  \hl^!(\ul\sfe_{\Bun_G}))), \quad \CF\in \Shv(\CZ\times \Bun_G),$$
where we note that the operation $-\sotimes \hl^!(\ul\sfe_{\Bun_G})$ amounts to 
the cohomological shift $[-2\dim(\Bun_G)]$.

\medskip

Furthermore, the map $\hr \times s$ is ind-proper. This implies that the functor $\on{Id}_\CZ\boxtimes \sH(V,-)$ 
can be rewritten as
$$(\on{Id}_\CZ \times (\hr \times s))_! \left((\on{id}_\CZ \times \hl)^*(\CF)\overset{*}\otimes 
(\ul\sfe_\CZ \boxtimes \on{Sat}_I(V))\right)[-2\dim(\Bun_G)].$$

This implies that the functors $\on{Id}_\CZ\boxtimes \sH(V,-)$ are \emph{both defined and codefined by kernels}, see \secref{sss:defined and codefined}
for what this means. In particular, the functors $\on{Id}_\CZ\boxtimes \sH(V,-)$ commute with $*$-pullbacks and 
$!$-pushforwards along the $\CZ$ variable. 

\sssec{} \label{sss:Ran action}

Recall the monoidal category $\Rep(\cG)_{\Ran}$, see \cite[Sect. 11.1]{AGKRRV}. As in 
\cite[Sect. 15.1]{AGKRRV}, the functors \eqref{e:Hecke action} gives rise to a monoidal action of 
$\Rep(\cG)_{\Ran}$ on $\Shv(\CZ\times \Bun_G)$. 

\medskip

We will denote this action by
$$\CV\in \Rep(\cG)_{\Ran} \, \rightsquigarrow \CV\star -.$$

For a fixed $\CV$, we will also denote the above functor by 
\begin{equation} \label{e:int Hecke funct}
\on{Id}_\CZ\boxtimes \sH_\CV.
\end{equation}

This is again consistent with the notation of \secref{ss:ker}, since these functors
are defined by kernels, by construction. 

\medskip

We will refer to functors \eqref{e:int Hecke funct} as \emph{integral Hecke functors}. 

\sssec{} \label{sss:Hecke codefined}

The next observation will play an important role in the sequel:

\begin{prop} \label{p:Hecke codefined}
The functors $\sH_\CV$ are \emph{defined and codefined} by kernels.
\end{prop}

The proof is given in the next subsection. 

\ssec{Proof of \propref{p:Hecke codefined}}

\sssec{}

We need to show that the functors $\on{Id}_\CZ\boxtimes \sH_\CV$ of
\eqref{e:int Hecke funct} commute with $*$-pullbacks and $!$-pushforwards along the $\CZ$ variable. 
It is enough to check this assertion on the generators of the category $\Rep(\cG)_{\Ran}$. 

\medskip

Thus,
we fix a finite set $I$, an object $V\in (\Rep(\cG)^{\otimes I})^c$ and $\CM\in \Shv(X^I)^c$. The corresponding
functor $\on{Id}_\CZ\boxtimes \sH_\CV$ is given by
\begin{equation} \label{e:int Hecke expl}
\CF\mapsto  (p_{\CZ\times \Bun_G})_*((\on{Id}_\CZ\boxtimes \sH(V,-))(\CF)\sotimes p_{X^I}^!(\CM)),
\end{equation}
where 
$$\CZ\times \Bun_G \overset{p_{\CZ\times \Bun_G}}\longleftarrow \CZ\times \Bun_G\times X^I\overset{p_{X^I}}\longrightarrow X^I$$
are the two projections. 

\medskip

We need to show that these functors commute with $*$-pullbacks and $!$-pushforwards along the $\CZ$ variable. We will
prove this by showing that the functor \eqref{e:int Hecke expl} is canonically isomorphic to the functor 
\begin{equation} \label{e:int Hecke expl !}
\CF\mapsto  (p_{\CZ\times \Bun_G})_!((\on{Id}_\CZ\boxtimes \sH(V,-))(\CF)\overset{*}\otimes p_{X^I}^*(\CM))[-2|I|].
\end{equation}

From here, the commutation with $*$-pullbacks and $!$-pushforwards would then follow from the fact that the functor 
$\sH(V,-)$ is defined and codefined by a kernel, see \secref{sss:Hecke ULA} above. 

\medskip

We can assume that $\CF$ is compact. Since the map $p_{\CZ\times \Bun_G}$ is proper, the isomorphism between
\eqref{e:int Hecke expl} and \eqref{e:int Hecke expl !} follows from the next assertion:

\begin{prop} \label{p:ULA}
An object of the form
$$(\on{Id}_\CZ\boxtimes \sH(V,-))(\CF) \in \Shv(\CZ\times \Bun_G\times X^I), \quad \CF\in \Shv(\CZ\times \Bun_G)^c$$
is ULA with respect to the projection $p_{X^I}:\CZ\times \Bun_G\times X^I\to X^I$.
\end{prop}

\qed[\propref{p:Hecke codefined}]

\sssec{Proof of \propref{p:ULA}}

Since the map $$\on{id}_\CZ\times (\hr \times s): \CZ\times \CH_{G,X^I}\to \CZ\times \Bun_G\times X^I$$ is ind-proper, it suffices to show that the object
\begin{equation} \label{e:ULA upstairs}
(\on{Id}_\CZ \times \hl)^!(\CF)\sotimes (\omega_\CZ \boxtimes \on{Sat}_I(V))\in \Shv(\CZ\times \CH_{G,X^I})
\end{equation}
is ULA with respect to the map 
$$p_{X_I}\circ (\on{id}_\CZ\times (\hr \times s)),$$
where we note that 
$$p_{X_I}\circ (\on{id}_\CZ\times (\hr \times s))=s=p_{X_I}\circ (\on{id}_\CZ\times (\hl \times s)).$$ 

\medskip

Thus, it suffices to show that \eqref{e:ULA upstairs} is ULA with respect to the projection 
$p_{X_I}\circ (\on{id}_\CZ\times (\hl \times s))$. However, this follows from the fact that
$\on{Sat}_I(V)$ is ULA with respect to $\hl \times s$ using the following general lemma:

\begin{lem}
Let 
$$\CY_1 \overset{p_1}\leftarrow \CY \overset{p_2}\to \CY_2$$
be a diagram of stacks. Let 
$\CG\in \Shv(\CY)^c$ be ULA with respect to the map
$$(p_1,p_2):\CY\to \CY_1\times \CY_2.$$
Then for any $\CG_1\in \Shv(\CY_1)^{\on{constr}}$, the object
$$\CG\sotimes p_1^!(\CG_1)\in  \Shv(\CY)^{\on{constr}}$$
is ULA with respect to $p_2$.
\end{lem}

\qed[\propref{p:ULA}]

%\sssec{} \label{sss:conj Hecke}
%
%The category $\Rep(\cG)$ carries a canonical involution, denoted $\tau$, given by the Chevalley involution on $\cG$. 
%
%For an integral
%Hecke functor $\sH$ we will denote by $\sH^\tau$ the integral Hecke functor obtained by applying the induced involution on $\Rep(\cG)_{\Ran}$.
%
%\medskip
%
%We will refer to $\sH^\tau$ as the \emph{conjugate} Hecke functor of $\sH$. 

\ssec{Spectral decomposition in a relative situation}

Let
$$\Nilp\subset T^*(\Bun_G)$$
be the nilpotent cone. It is known to be half-dimensional, under very mild assumptions on $\on{char}(k)$,
see \cite[Sect. D]{AGKRRV}, which we impose from now on. 

\medskip

In \cite[Theorem 14.4.3]{AGKRRV}, it was shown that the subcategory $\Shv_\Nilp(\Bun_G)\subset \Shv(\Bun_G)$
can be characterized as that of ``Hecke-lisse" objects. In this subsection we will establish an analog of this result
in the relative situation. 

\sssec{}  \label{sss:Hecke lisse}

For $V\in \Rep(\cG)$, consider the corresponding Hecke functor
$$\on{Id}_\CZ\boxtimes \sH(V,-):\Shv(\CZ\times \Bun_G)\to \Shv(\CZ\times \Bun_G\times X).$$

%The following is an extension of \cite[Theorem 14.4.4]{AGKRRV} (proved in the same way as {\it loc.cit.}):
%
%\begin{thm} \label{t:Hecke action Nilp 1}
%For any stack $\CZ$ and $\CF\in \Shv(\CZ\times \Bun_G)$, the following conditions are equivalent:
%
%\smallskip
%
%\noindent{\em(i)} $\CF\in \Shv_{T^*(\CZ)\times \Nilp}(\CZ\times \Bun_G)$;
%
%\smallskip
%
%\noindent{\em(ii)} For every $V\in \Rep(\cG)$, 
%$$\sH(V,\CF)\in \Shv_{T^*(\CZ)\times T^*(\Bun_G)\times \{0\}}(\Bun_G\times X);$$
%
%\end{thm} 
%
%\medskip
%
%In the relative situation (i.e., with the stack $\CZ$ present), \thmref{t:Hecke action Nilp 1}
%admits a variant, see \thmref{t:Hecke action Nilp 2} below that generalizes \cite[Theorem 14.4.3]{AGKRRV}.

%\sssec{} \label{sss:times 1/2}
%
%Let 
%$$\Shv_{\frac{1}{2}\on{-dim}\times \Nilp}(\CZ\times \Bun_G)$$
%denote the full subcategory of $\Shv(\CZ\times \Bun_G)$ that consists of objects that satisfy the
%following condition: for every $m$ and every constructible sub-object $\CF'$ of $H^m(\CF)$, the singular 
%support of $\CF'$ is contained in a subset of the form 
%$$\CN\times \Nilp\subset T^*(\CZ\times \Bun_G),$$
%where $\CN\subset T^*(\CZ)$ is half-dimensional.
%
%\medskip
%
%Note that
%$$\Shv_{\frac{1}{2}\on{-dim}\times \Nilp}(\CZ\times \Bun_G) \subset \Shv_{T^*(\CZ)\times \Nilp}(\CZ\times \Bun_G),$$
%but if $\on{char}(p)\neq 0$, this is a proper containment. 

Let 
$$\Shv(\CZ\times \Bun_G)^{\on{Hecke-lisse}}$$
denote the full subcategory of $\Shv(\CZ\times \Bun_G)$ that consists of objects $\CF$ such that for 
every $V\in \Rep(\cG)$, 
$$\sH(V,\CF)\in \Shv(\CZ\times \Bun_G)\otimes \qLisse(X) \subset \Shv(\CZ\times \Bun_G\times X).$$

\medskip

As in \cite[Proposition C.2.5]{GKRV}, we obtain that the Hecke functors \eqref{e:Hecke action}
give rise to a compatible family of functors
\begin{equation} \label{e:Hecke functors Nilp}
\Rep(\cG)^{\otimes I} \to \End(\Shv(\CZ\times \Bun_G)^{\on{Hecke-lisse}})\otimes \qLisse(X)^{\otimes I}.
\end{equation}

Thus, in the terminology of \cite[Sect. 8.4.2]{AGKRRV}, we obtain that $\Shv(\CZ\times \Bun_G)^{\on{Hecke-lisse}}$
acquires an action of the (symmetric) monoidal category $\Rep(\cG)^{\otimes X\on{-lisse}}$. 

\sssec{}

Let $\LocSys_\cG^{\on{restr}}(X)$ be the prestack introduced in \cite[Sect. 1.4.2]{AGKRRV},
and recall now (see \cite[Equation (8.10)]{AGKRRV}) that we have a canonically defined functor 
\begin{equation} \label{e:pre-loc}
\Rep(\cG)^{\otimes X\on{-lisse}}\to  \QCoh(\LocSys_\cG^{\on{restr}}(X)).
\end{equation}

A key result \cite[Theorem 8.3.7]{AGKRRV} says that the functor \eqref{e:pre-loc} is an equivalence. 

\medskip

Thus, we obtain that the above $\Rep(\cG)^{\otimes X\on{-lisse}}$-action on $\Shv(\CZ\times \Bun_G)^{\on{Hecke-lisse}}$
is obtained from a uniquely defined monoidal action of the category 
$$\QCoh(\LocSys_\cG^{\on{restr}}(X))$$
on $\Shv(\CZ\times \Bun_G)^{\on{Hecke-lisse}}$. 

\medskip

We will refer to this phenomenon as the \emph{spectral decomposition} of $\Shv(\CZ\times \Bun_G)^{\on{Hecke-lisse}}$
along $\QCoh(\LocSys_\cG^{\on{restr}}(X))$.

\sssec{}

Recall the symmetric monoidal functor
\begin{equation} \label{e:Ran to lisse}
\Rep(\cG)_\Ran\to \Rep(\cG)^{\otimes X\on{-lisse}},
\end{equation}
see \cite[Sect. 11.2.3]{AGKRRV}. Its composition with \eqref{e:pre-loc} is the functor
$$\Loc:\Rep(\cG)_\Ran\to \QCoh(\LocSys_\cG^{\on{restr}}(X))$$
of \cite[Sect. 12.7.1]{AGKRRV}.
 
\medskip 

Following \cite[Sect. 13.3.1]{AGKRRV}, set
$$\Shv(\CZ\times \Bun_G)^{\on{spec}}:=
\on{Funct}_{\Rep(\cG)_\Ran}(\QCoh(\LocSys_\cG^{\on{restr}}(X)), \Shv(\CZ\times \Bun_G)).$$

Pre-composition with $\Loc$ defines a functor
$$\Shv(\CZ\times \Bun_G)^{\on{spec}}\to 
\on{Funct}_{\Rep(\cG)_\Ran}(\Rep(\cG)_\Ran, \Shv(\CZ\times \Bun_G))=\Shv(\CZ\times \Bun_G),$$
which is fully faithful, according to \cite[Proposition 13.3.4]{AGKRRV}. 

\medskip

Thus, we can view $\Shv(\CZ\times \Bun_G)^{\on{spec}}$ as a full subcategory of $\Shv(\CZ\times \Bun_G)$. 

\sssec{} \label{sss:lisse to spec}

By construction, the action of $\Rep(\cG)_\Ran$ on $\Shv(\CZ\times \Bun_G)^{\on{Hecke-lisse}}$ factors
via \eqref{e:Ran to lisse}. 

\medskip

Hence, we obtain
\begin{equation} \label{e:lisse to spec}
\Shv(\CZ\times \Bun_G)^{\on{Hecke-lisse}} \subset \Shv(\CZ\times \Bun_G)^{\on{spec}}.
\end{equation}
%
%We claim:
%
%\begin{prop} \label{p:lisse to spec}
%The inclusion \eqref{e:lisse to spec} is an equality.
%\end{prop}
%
%\begin{proof}
%
%Repeats the argument of \cite[Proposition 4.4.4]{AGKRRV}.
%
%\end{proof}

\sssec{}

Consider also the subcategory 
$$\Shv_{\frac{1}{2}\on{-dim}\times \Nilp}(\CZ\times \Bun_G)\subset \Shv(\CZ\times \Bun_G),$$
see \secref{sss:times 1/2-dim} for the notation. 

\sssec{}

%Note \thmref{t:Hecke action Nilp 1} implies that
%$$\Shv(\CZ\times \Bun_G)^{\on{Hecke-lisse}} \subset \Shv_{T^*(\CZ)\times \Nilp}(\CZ\times \Bun_G).$$
%
%\medskip

We will prove:

\begin{thm} \label{t:Hecke action Nilp 2}
For any stack $\CZ$ the following four full subcategories of $\Shv(\CZ\times \Bun_G)$ coincide:

\smallskip

\noindent{\em(i)} The essential image of
$\Shv(\CZ)\otimes \Shv_{\Nilp}(\Bun_G)\to \Shv(\CZ\times \Bun_G)$.

\smallskip

\noindent{\em(ii)} $\Shv_{\frac{1}{2}\on{-dim}\times \Nilp}(\CZ\times \Bun_G)$;

\smallskip

\noindent{\em(iii)} $\Shv(\CZ\times \Bun_G)^{\on{Hecke-lisse}}$;

\smallskip

\noindent{\em(iv)} $\Shv(\CZ\times \Bun_G)^{\on{spec}}$;

\end{thm}

\begin{rem}
We will eventually prove also that for a conical half-dimensional 
closed subset $\CN\subset T^*(\CZ)$, the fully
faithful functor
$$\Shv_\CN(\CZ)\otimes \Shv_{\Nilp}(\Bun_G)\to \Shv_{\CN\times \Nilp}(\CZ\times \Bun_G)$$
is an equivalence, see \thmref{t:Kunneth fixed Nilp}. 
\end{rem}

\ssec{Proof of \thmref{t:Hecke action Nilp 2}} \label{ss:proof Hecke nilp}

\sssec{}

The inclusion (i) $\subset$ (ii) is evident. Let us establish (ii) $\subset$ (iii). 

\medskip

Let $\CF$ be an object in $\Shv_{\frac{1}{2}\on{-dim}\times \Nilp}(\CZ\times \Bun_G)$, 
and let $V$ be an object of $\Rep(\cG)$. It follows 
by the argument of \cite[Theorem B.5.2]{GKRV} that
$$\sH(V,\CF) \in \Shv_{\frac{1}{2}\on{-dim}\times \Nilp\times \{0\}}(\CZ\times \Bun_G\times X).$$

\medskip

However, the subcategory $\Shv_{\frac{1}{2}\on{-dim}\times \Nilp\times \{0\}}(\CZ\times \Bun_G\times X)$
is contained in the essential image of
$$\Shv(\CZ\times \Bun_G)\otimes \qLisse(X) \subset \Shv(\CZ\times \Bun_G\times X),$$
by \thmref{t:Kunneth} (applied to $Y_1=X$). 

\medskip

The inclusion (iii) $\subset$ (iv) has been noted in \secref{sss:lisse to spec}. 
Thus, it remains to prove (iv) $\subset$ (i). 

\sssec{}

Let 
$$Z_n\overset{f_n}\to  \LocSys_\cG^{\on{restr}}(X)$$ be as in \cite[Sect. 16.1.2]{AGKRRV}. Consider the category 
$$\Shv(\CZ\times \Bun_G)^{\on{spec}}\underset{\QCoh(\LocSys_\cG^{\on{restr}}(X))}\otimes \QCoh(Z_n),$$
equipped with the forgetful functor, denoted, $\oblv_{\on{Hecke}}$: 
\begin{multline*}
\Shv(\CZ\times \Bun_G)^{\on{spec}}\underset{\QCoh(\LocSys_\cG^{\on{restr}}(X))}\otimes \QCoh(Z_n)
\overset{\on{Id}\otimes (f_n)_*}\longrightarrow \\
\longrightarrow \Shv(\CZ\times \Bun_G)^{\on{spec}}\underset{\QCoh(\LocSys_\cG^{\on{restr}}(X))}\otimes \QCoh(\LocSys_\cG^{\on{restr}}(X))\simeq \\
\simeq \Shv(\CZ\times \Bun_G)^{\on{spec}}\hookrightarrow \Shv(\CZ\times \Bun_G).
\end{multline*}

The functor $\oblv_{\on{Hecke}}$ admits a left adjoint, to be denoted 
\begin{equation} \label{e:projector Zn}
(\on{Id}_\CZ\boxtimes \sP_{Z_n})^{\on{enh}},
\end{equation}
see \cite[Corollary 13.5.4]{AGKRRV}.

\sssec{}

We will prove:

\begin{prop} \label{p:generation}
Objects of the form 
$$(\on{Id}_\CZ\boxtimes \sP_{Z_n})^{\on{enh}}(\CF_\CZ\boxtimes \delta_y), \quad
\CF_\CZ\in \Shv(\CZ), \quad y\in \Bun_G(k)$$
generate the category $\Shv(\CZ\times \Bun_G)^{\on{spec}}\underset{\QCoh(\LocSys_\cG^{\on{restr}}(X))}\otimes \QCoh(Z_n)$.
\end{prop}

Let us assume this proposition temporarily and finish the proof of the containment (iv) $\subset$ (i) in \thmref{t:Hecke action Nilp 2}. 

\sssec{}

Consider the functor
\begin{multline*} 
\Shv(\CZ\times \Bun_G)^{\on{spec}}\underset{\QCoh(\LocSys_\cG^{\on{restr}}(X))}\otimes \QCoh(Z_n)\overset{\on{Id}\otimes (f_n)_*}\longrightarrow \\
\longrightarrow  \Shv(\CZ\times \Bun_G)^{\on{spec}}\underset{\QCoh(\LocSys_\cG^{\on{restr}}(X))}\otimes 
\QCoh(\LocSys_\cG^{\on{restr}}(X))\simeq \Shv(\CZ\times \Bun_G)^{\on{spec}}.
\end{multline*} 

As in \cite[Sect. 16.1.4]{AGKRRV}, one shows that the union (over the index $n$) of the essential images of these 
functors generates $\Shv(\CZ\times \Bun_G)^{\on{spec}}$. 

\medskip

Hence, applying \propref{p:generation}, we obtain that in order to prove the containment (iv) $\subset$ (i), it suffices to show that
for a fixed index $n$, the functor 
\begin{multline} \label{e:monad Zn}
\Shv(\CZ\times \Bun_G) \overset{(\on{Id}_\CZ\boxtimes \sP_{Z_n})^{\on{enh}}}\longrightarrow \\
\longrightarrow \Shv(\CZ\times \Bun_G)^{\on{spec}}\underset{\QCoh(\LocSys_\cG^{\on{restr}}(X))}\otimes \QCoh(Z_n) 
\overset{\oblv_{\on{Hecke}}}\longrightarrow \Shv(\CZ\times \Bun_G)
\end{multline}
applied to objects of the form 
$$\CF_\CZ\boxtimes \delta_y, \quad \CF_\CZ\in \Shv(\CZ), \quad y\in \Bun_G(k)$$
maps to the essential image of
$$\Shv(\CZ)\otimes \Shv_{\Nilp}(\Bun_G)\to \Shv(\CZ\times \Bun_G).$$

\sssec{}

Note that, as in \cite[Equations (15.15) and (15.16)]{AGKRRV}, the composite functor \eqref{e:monad Zn} is given by
$$\on{Id}_\CZ \boxtimes \sH_\CV$$
(see \secref{sss:Ran action} for the notation), for 
\begin{equation} \label{e:V for Zn}
\CV:=\left(\on{Id}_{\Rep(\cG)_\Ran}\otimes \Gamma(Z_n,-)\right)(\sR_{Z_n})\in \Rep(\cG)_\Ran,
\end{equation}
where $$\sR_{Z_n}\in \Rep(\cG)_\Ran \otimes \QCoh(Z_n)$$
is the object of \cite[Sect. 15.3.1]{AGKRRV}.

\medskip 

However, for any $\CV\in \Rep(\cG)_\Ran$, 
$$(\on{Id}_\CZ \boxtimes \sH_\CV)(\CF_\CZ\boxtimes \delta_y)\simeq
\CF_\CZ \boxtimes  \sH_\CV(\delta_y).$$

Now, for $\CV$ given by \eqref{e:V for Zn}, we have
$$\sH_\CV(\delta_y)\in \Shv_{\Nilp}(\Bun_G),$$
(e.g., by \cite[Equations (15.15) and (15.16) and Corollary 15.5.4]{AGKRRV}), which implies the desired containment. 

\qed[\thmref{t:Hecke action Nilp 2}]

\ssec{Proof of \propref{p:generation}}

\sssec{}

We will prove a slightly stronger assertion. Namely, let $y_i\in \Bun_G(k)$ be points
chosen as in \cite[Sect. 16.2.1-16.2.2]{AGKRRV}. We will show that the objects
$$(\on{Id}_\CZ\boxtimes \sP_{Z_n})^{\on{enh}}(\CF_\CZ\boxtimes \delta_{y_i})$$
generate 
$$\Shv(\CZ\times \Bun_G)^{\on{spec}}\underset{\QCoh(\LocSys_\cG^{\on{restr}}(X))}\otimes \QCoh(Z_n).$$

\sssec{}

By adjunction, it suffices to show that for every $0\neq \CF\in \Shv(\CZ\times \Bun_G)^{\on{spec}}$, we can find
$\CF_\CZ\in \Shv(\CZ)$ and an index $i$ such that
\begin{equation} \label{e:non-orth}
\CHom(\CF_\CZ\boxtimes \delta_{y_i},\CF)\neq 0.
\end{equation}

Let $\bi_{y_i}$ be the map $\on{pt}\to \Bun_G$ corresponding to the point $y_i$. Our assertion is equivalent to saying that
there exists an index $i$ such that
$$(\on{id}_\CZ \times \bi_{y_i})^!(\CF)\neq 0.$$

\sssec{} \label{sss:pullback Hecke lisse}

Let $g:\CZ_1\to \CZ_2$ be a map between algebraic stacks. The pullback functor
$$\Shv(\CZ_2\times \Bun_G)\overset{(g\times \on{id})^!}\longrightarrow \Shv(\CZ_1\times \Bun_G)$$
is compatible with the Hecke action. In particular, it sends 
$$\Shv(\CZ_2\times \Bun_G)^{\on{spec}}\to \Shv(\CZ_1\times \Bun_G)^{\on{spec}}.$$

\sssec{}

Let $\Spec(k')\overset{\bi_z}\to \CZ$ be a geometric point such that the !-pullback of $\CF$ to
$$\Bun'_G:=\Spec(k')\times \Bun_G\to \CZ \times \Bun_G$$
is non-zero (such a point exists by a Cousin argument). 

\medskip

We can view the above !-pullback as the !-pullback along
$$\Bun'_G \overset{\bi'_z}\to \CZ'\underset{\Spec(k')}\times \Bun'_G, \quad \CZ':=\Spec(k')\times \CZ$$
of the base change $\CF'$ of $\CF$ to $\CZ'\underset{\Spec(k')}\times \Bun'_G$. 

Tautologically,  
$$\CF'\in \Shv(\CZ'\underset{\Spec(k')}\times \Bun'_G)^{\on{spec}}.$$

By assumption, $(\bi'_z)^!(\CF')\neq 0$, and by \secref{sss:pullback Hecke lisse}, 
$$(\bi'_z)^!(\CF')\in \Shv(\Bun'_G)^{\on{spec}}.$$

\sssec{}

Let 
$$\bi'_{y_i}:\Spec(k')\to \Bun'_G$$ denote the base change of the map $\bi_{y_i}$. We have
$$(\bi_z)^! \circ (\on{id}_\CZ \times \bi_{y_i})^!(\CF) \simeq
(\bi'_{y_i})^!\circ (\bi'_z)^! (\CF').$$

Hence, it suffices to show that there exists an index $i$, such that 
$$(\bi'_{y_i})^!\circ (\bi'_z)^! (\CF')\neq 0.$$

However, by \cite[Proposition 15.4.4]{AGKRRV},
$$\Shv(\Bun'_G)^{\on{spec}}=\Shv_\Nilp(\Bun'_G),$$
so 
$$0\neq (\bi'_z)^! (\CF')\in \Shv_\Nilp(\Bun'_G),$$
and the assertion follows from \cite[Sect. 16.2.2-16.2.3]{AGKRRV}.

\qed[\propref{p:generation}]

\ssec{The projector onto the category with nilpotent singular support}

In \cite[Sect. 15.4.5]{AGKRRV}, a particular object of $\Rep(\cG)_\Ran$ was introduced whose action on
$\Shv(\Bun_G)$ effects a projection onto the full subcategory $\Shv_\Nilp(\Bun_G)\subset \Shv(\Bun_G)$.

\medskip

In this subsection, we will establish an analog of this result in the relative situation. 

\sssec{}  \label{sss:the projector}

Let 
$$\sR \in \Rep(\cG)_\Ran$$
be the object introduced in \cite[Sect. 13.4.1]{AGKRRV}. 

\medskip

Denote by $\on{Id}_\CZ\boxtimes \sP$ the resulting endofunctor of $\Shv(\CZ\times \Bun_G)$, i.e., 
$$\sP:=\sH_\sR,$$
as functors defined by kernels.

\medskip

In what follows we will also use the notation
$$\sP_{\Bun_G,\Nilp}:=\sP.$$

\begin{rem}
For $\CZ=\on{pt}$, the resulting functor endofunctor $\sP$ of $\Shv(\Bun_G)$
is the one considered in \cite[Sect. 15.4.7]{AGKRRV}.
\end{rem}

\sssec{}

We claim:

\begin{thm} \label{t:projector}
The endofunctor $\on{Id}_\CZ \boxtimes \sP$ of $\Shv(\CZ\times \Bun_G)$ is a projector
onto the full subcategory
$$\Shv(\CZ)\otimes \Shv_\Nilp(\Bun_G)\simeq
\Shv_{\frac{1}{2}\on{-dim}\times \Nilp}(\CZ\times \Bun_G)\subset \Shv(\CZ\times \Bun_G).$$
\end{thm}

\begin{proof}

Interpreting $\Shv_{\frac{1}{2}\on{-dim}\times \Nilp}(\CZ\times \Bun_G)$ as $\Shv(\CZ\times \Bun_G)^{\on{spec}}$
(see \thmref{t:Hecke action Nilp 2}), the assertion follows by \cite[Remark 13.4.8]{AGKRRV}. 

\end{proof}

\begin{cor} \label{c:double projector}
The endofunctor $\sP\boxtimes \sP$ of $\Shv(\Bun_G\times \Bun_G)$ 
is a projector onto the full subcategory
$$\Shv_\Nilp(\Bun_G)\otimes \Shv_\Nilp(\Bun_G) \subset \Shv(\Bun_G\times \Bun_G).$$
\end{cor}

\begin{proof}

Since $\sP$ acts as identity on $\Shv_\Nilp(\Bun_G)$, it follows from \eqref{e:boxtimes vs otimes} that $\sP\boxtimes \sP$
acts as identity on $\Shv_\Nilp(\Bun_G)\otimes \Shv_\Nilp(\Bun_G)$.

\medskip

We need to show that the essential image of $\sP\boxtimes \sP$ is contained in $\Shv_\Nilp(\Bun_G)\otimes \Shv_\Nilp(\Bun_G)$.
We write
\begin{equation} \label{e:P left and right}
\sP\boxtimes \sP=(\sP\boxtimes \on{Id}_{\Bun_G})\circ (\on{Id}_{\Bun_G}\boxtimes \sP).
\end{equation}

By Theorems \ref{t:projector} and \ref{t:Hecke action Nilp 2}, the essential image of $\on{Id}_{\Bun_G}\boxtimes \sP$ is contained in 
$$\Shv(\Bun_G)\otimes \Shv_\Nilp(\Bun_G) \subset \Shv(\Bun_G)\otimes \Shv(\Bun_G) \subset 
\Shv(\Bun_G\times \Bun_G).$$

By \eqref{e:boxtimes vs otimes}, the functor $\sP\boxtimes \on{Id}_{\Bun_G}$, restricted to $\Shv(\Bun_G)\otimes \Shv(\Bun_G)$
acts as $\sP\otimes \on{Id}$. Hence, the essential image of its further restriction to $\Shv(\Bun_G)\otimes \Shv_\Nilp(\Bun_G)$
is contained in the subcategory $\Shv_\Nilp(\Bun_G)\otimes \Shv_\Nilp(\Bun_G)$. 

\end{proof}

\sssec{}

Finally, we observe that \thmref{t:projector} gives us a (somewhat) alternative proof of 
the following result (\cite[Theorem 16.3.3]{AGKRRV}):

\begin{thm} \label{t:Kunneth BunG}
The functor
$$\Shv_\Nilp(\Bun_G)\otimes \Shv_\Nilp(\Bun_G)\to\Shv_{\Nilp\times \Nilp}(\Bun_G\times \Bun_G)$$
is an equivalence.
\end{thm}

\begin{proof}

It is enough to show that the functor $\sP\boxtimes \sP$ acts as identity on $\Shv_{\Nilp\times \Nilp}(\Bun_G\times \Bun_G)$.
Writing $\sP\boxtimes \sP$ as in \eqref{e:P left and right}, it suffices to show that both $\on{Id}_{\Bun_G}\boxtimes \sP$ and 
$\sP\boxtimes \on{Id}_{\Bun_G}$ act as identity on $\Shv_{\Nilp\times \Nilp}(\Bun_G\times \Bun_G)$.

\medskip

We have 
$$\Shv_{\Nilp\times \Nilp}(\Bun_G\times \Bun_G)\subset \Shv_{\frac{1}{2}\on{-dim}\times \Nilp}(\Bun_G\times \Bun_G)\cap 
\Shv_{\Nilp\times \frac{1}{2}\on{-dim}}(\Bun_G\times \Bun_G).$$

Now, by \thmref{t:projector}, $\on{Id}_{\Bun_G}\boxtimes \sP$ acts as identity on $\Shv_{\frac{1}{2}\on{-dim}\times \Nilp}(\Bun_G\times \Bun_G)$
and $\sP\boxtimes \on{Id}_{\Bun_G}$ acts as identity on $\Shv_{\Nilp\times \frac{1}{2}\on{-dim}}(\Bun_G\times \Bun_G)$.

\end{proof}

\ssec{(Universally) Nilp-cotruncative substacks}

In \cite[Theorem 14.1.5]{AGKRRV}, there was exhibited a collection of quasi-compact open substacks $\CU\subset \Bun_G$
with particularly favorable properties vis-a-vis the subcategory $\Shv_\Nilp(-)$. 

\medskip

In this subsection, we will show that these
properties carry over to the relative situation. 

\sssec{}
 
We shall say that an open substack $\CU\overset{j}\hookrightarrow \Bun_G$ is $\Nilp$-cotruncative if:

\begin{itemize}

\item It is \emph{cotruncative}, i.e., if the functor $j_!$ is defined by a kernel
(see \secref{sss:truncative});

\item The functor $j_!$ (equivalently, $j_*$) sends $\Shv_\Nilp(\CU)$ to $\Shv_\Nilp(\Bun_G)$.

\end{itemize}

We shall say that an open substack $\CU\overset{j}\hookrightarrow \Bun_G$ is \emph{universally} $\Nilp$-cotruncative if,
moreover:

\begin{itemize}

\item For any algebraic stack $\CZ$, the functor 
$$(\on{id}\times j)_!:\Shv(\CZ\times \CU)\to \Shv(\CZ\times \Bun_G)$$
(equivalently, $(\on{id}\times j)_*$) sends 
$$\Shv_{\CN\times \Nilp}(\CZ\times \CU)\to \Shv_{\CN\times \Nilp}(\CZ\times \Bun_G)$$
for any closed conical subset $\CN\subset T^*(\CZ)$.

\end{itemize}

We refer the reader to \secref{sss:N-cotrunc} where the general notion of (universal) cotruncativeness relative
to a subset in the cotangent bundle is discussed.

\sssec{}

The following is an extension of the combination of \cite[Theorem 9.1.2]{DrGa1}
and \cite[Theorem 14.1.5]{AGKRRV} (in fact, the arguments in {\it loc. cit.} 
prove this strengthened statement):

\begin{thm} \label{t:trunc}
The stack $\Bun_G$ can be written as a filtered union of its quasi-compact 
universally $\Nilp$-cotruncative open substacks. 
\end{thm} 

In the terminology of \secref{sss:N-trunc}, the assertion of \thmref{t:trunc} is that $\Bun_G$ is universally $\Nilp$-truncatable.

\sssec{}

For the remainder of this subsection, we fix a universally Nilp-cotruncative quasi-compact open substack
$$\CU\overset{j}\hookrightarrow \Bun_G.$$

Let
$$\Shv_{\frac{1}{2}\on{-dim}\times \Nilp}(\CZ\times \CU)\subset \Shv(\CZ\times \CU)$$
be as in \secref{sss:times 1/2-dim}.

\medskip

The assumption that $\CU$ is universally Nilp-cotruncative implies that the functors 
$(\on{id}\times j)_!$ and  $(\on{id}\times j)_*$ map 
$$\Shv_{\frac{1}{2}\on{-dim}\times \Nilp}(\CZ\times \CU)\to
\Shv_{\frac{1}{2}\on{-dim}\times \Nilp}(\CZ\times \Bun_G).$$

\sssec{}

We claim: 

\begin{cor} \label{c:Kunneth for cotrunc}
Let $\CU\overset{j}\hookrightarrow \Bun_G$ be a universally $\Nilp$-cotruncative open substack. Then for 
any stack $\CZ$, the functor
$$\Shv(\CZ)\otimes \Shv_{\Nilp}(\CU)\to \Shv_{\frac{1}{2}\on{-dim}\times \Nilp}(\CZ\times \CU)$$
is an equivalence.
\end{cor}

\begin{proof}

We only have to show that the functor in question is essentially surjective. This
follows from the equivalence (i) $\Leftrightarrow$ (ii) in \thmref{t:Hecke action Nilp 2} 
and the fact that in the commutative diagram
$$
\CD
\Shv(\CZ)\otimes \Shv_{\Nilp}(\Bun_G) @>>> \Shv_{\frac{1}{2}\on{-dim}\times \Nilp}(\CZ\times \Bun_G) \\
@V{\on{Id}\otimes j^*}VV @VV{(\on{id}\times j)^*}V \\
\Shv(\CZ)\otimes \Shv_{\Nilp}(\CU) @>>> \Shv_{\frac{1}{2}\on{-dim}\times \Nilp}(\CZ\times \CU)
\endCD
$$
the vertical arrows are essentially surjective (which is in turn implied by the fact that they admit fully
faithful left adjoints, due to the universal $\Nilp$-cotruncativeness assumption).

\end{proof}

\ssec{The projector on universally Nilp-cotruncative substacks}

In this subsection we fix a universally Nilp-cotruncative quasi-compact open substack
$$\CU\overset{j}\hookrightarrow \Bun_G.$$

We will show how the results pertaining to the projector $\sP$ can be applied to $\CU$. 

\sssec{} \label{sss:proj U}

Denote by $\sP_{\CU,\Nilp}$ the endofunctor of $\Shv(\CU)$ defined by the kernel equal to
$$j^*\circ \sP_{\Bun_G,\Nilp} \circ j_*,$$
where we remind that $\sP_{\Bun_G,\Nilp}$ is the same as the functor defined by a kernel, denoted $\sP$,
in \secref{sss:the projector}. 

\sssec{}

We claim: 

\begin{cor} \label{c:projector U}
For a stack $\CZ$, the endofunctor $\on{Id}_\CZ\boxtimes \sP_{\CU,\Nilp}$ of 
$\Shv(\CZ\times \CU)$ is a projector onto the full subcategory
$$\Shv(\CZ) \otimes \Shv_\Nilp(\Bun_G)\simeq \Shv_{\frac{1}{2}\on{-dim}\times \Nilp}(\CZ\times \Bun_G) \subset \Shv(\CZ\times \CU).$$
\end{cor}

\begin{proof}
The fact that the endofunctor functor $\on{Id}_\CZ\boxtimes \sP_{\CU,\Nilp}$ maps 
$\Shv(\CZ\times \CU)$ to $\Shv_{\frac{1}{2}\on{-dim}\times \Nilp}(\CZ\times \Bun_G)$ follows
from the fact that $\sP_{\Bun_G,\Nilp}$ maps 
$\Shv(\CZ\times \Bun_G)$ to $\Shv_{\frac{1}{2}\on{-dim}\times \Nilp}(\CZ\times \Bun_G)$
(see \thmref{t:projector}). 

\medskip

The fact that $\on{Id}_\CZ\boxtimes \sP_{\CU,\Nilp}$ acts as identity on $\Shv_{\frac{1}{2}\on{-dim}\times \Nilp}(\CZ\times \CU)$
follows from the fact that $(\on{id}\times j)_*$ maps $\Shv_{\frac{1}{2}\on{-dim}\times \Nilp}(\CZ\times \CU)$ 
to $\Shv_{\frac{1}{2}\on{-dim}\times \Nilp}(\CZ\times \Bun_G)$ and the fact that $\sP_{\Bun_G,\Nilp}$ acts as identity
on $\Shv_{\frac{1}{2}\on{-dim}\times \Nilp}(\CZ\times \Bun_G)$. 

\end{proof}

Furthermore, from \corref{c:double projector} we obtain:

\begin{cor} \label{c:double projector U}
The endofunctor $\sP_{\CU,\Nilp}\boxtimes \sP_{\CU,\Nilp}$ of $\Shv(\CU\times \CU)$
is a projector onto the full subcategory 
$$\Shv_\Nilp(\CU)\otimes \Shv_\Nilp(\CU)\subset \Shv_{\Nilp\times \Nilp}(\CU\times \CU)$$
\end{cor}

\begin{proof}

Indeed, the functor $\sP_{\CU,\Nilp}\boxtimes \sP_{\CU,\Nilp}$ is the composition
\begin{multline*}
\Shv(\CU\times \CU) \overset{(j\times j)_*}\longrightarrow 
\Shv(\Bun_G\times \Bun_G) \overset{\sP\boxtimes \sP}\longrightarrow \\
\to \Shv_\Nilp(\Bun_G)\otimes \Shv_\Nilp(\Bun_G) \overset{j^*\otimes j^*}\longrightarrow 
\Shv_\Nilp(\CU)\otimes \Shv_\Nilp(\CU).
\end{multline*}

\end{proof}

\sssec{}

Finally, from \thmref{t:Kunneth BunG} we obtain:

\begin{cor} \label{c:Kunneth U}
The functor
$$\Shv_\Nilp(\CU)\otimes \Shv_\Nilp(\CU)\to\Shv_{\Nilp\times \Nilp}(\CU\times \CU)$$
is an equivalence.
\end{cor}

\begin{proof}
Follows from \thmref{t:Kunneth BunG} in the same way as \corref{c:Kunneth for cotrunc} follows
from \thmref{t:Hecke action Nilp 2}.
\end{proof}

\section{Verdier duality for $\Shv_\Nilp(\Bun_G)$} \label{s:Verdier}

This section is devoted to the study of the pattern of Verdier (i.e., usual) self-duality
on $\Bun_G$ (and its quasi-compact open substacks) and how this self-duality interacts
with the subcategory $\Shv_\Nilp(\Bun_G)\subset \Shv(\Bun_G)$. 

\medskip

Since $\Bun_G$ is not quasi-compact, the category that is naturally 
dual to $\Shv(\Bun_G)$ is $\Shv(\Bun_G)_{co}$ (see \secref{ss:co BunG}). That said,
the self-duality on $\Shv(\Bun_G)$ by identifying $\Shv(\Bun_G)_{co}$ using
the miraculous functor $\Mir_{\Bun_G}$, see \secref{ss:BunG Mir}. 

\medskip

This leads us to study the corresponding subcategory 
$\Shv_\Nilp(\Bun_G)_{\on{co}}\subset \Shv(\Bun_G)_{co}$. 

\ssec{The diagonal object}

We begin this section by establishing the behavior of the diagonal object on $\Bun_G$
with respect to the Hecke action. The result is quite obvious (\propref{p:tau and sigma}),
but it serves as a basis for our further study. 

\sssec{}

Let $\CY$ be a quasi-compact algebraic stack. Throughout this paper, we use the following notation: 
$$\on{u}_\CY:=(\Delta_\CY)_*(\omega_\CY)\in \Shv(\CY\times \CY).$$

\sssec{} \label{sss:u shv}

This object should not be confused with the unit of the Verdier self-duality on $\Shv(\CY)$,
$$\on{u}_{\Shv(\CY)}\in  \Shv(\CY)\otimes \Shv(\CY),$$
see \secref{sss:unit Verdier}.

\medskip

When we view $\Shv(\CY)\otimes \Shv(\CY)$ as a full subcategory of $\Shv(\CY\times \CY)$
via \eqref{e:Kunneth 0}, we have a tautological map

\begin{equation} \label{e:unit to unit}
\on{u}_{\Shv(\CY)}\to \on{u}_\CY.
\end{equation}

In fact, according to \cite[Sect. 22.2.4]{AGKRRV},
$\on{u}_{\Shv(\CY)}$ isomorphic to the value on $\on{u}_\CY$ of the right adjoint functor to \eqref{e:Kunneth 0},
and \eqref{e:unit to unit} is given by the counit of the adjunction. 

\sssec{} \label{sss:naive BunG}

Let now $\CY$ be not necessarily quasi-compact, but \emph{truncatable} (see \secref{sss:truncatable} for what this means). 
In this case, one can still consider the object 
$$(\Delta_\CY)_*(\omega_\CY)\in \Shv(\CY\times \CY),$$
but we will denote it by
$$\on{u}^{\on{naive}}_\CY.$$

It has more refined variants, namely 
\begin{equation} \label{e:unit 1 and 2}
\on{u}_{\CY,\on{co}_1} \text{ and } \on{u}_{\CY,\on{co}_2},
\end{equation}
which are objects of the categories 
\begin{equation} \label{e:co 1 and 2}
\Shv(\CY\times \CY)_{\on{co}_1} \text{ and } \Shv(\CY\times \CY)_{\on{co}_2},
\end{equation}
respectively, see Sects. \ref{sss:co1} and \ref{sss:u co 1 2}, where the above categories
and objects are defined, respectively. 

\medskip

By \secref{sss:kernels non qc}, the categories in \eqref{e:co 1 and 2} correspond to functors defined by kernels
$$\Shv(\CY)\to \Shv(\CY) \text{ and } \Shv(\CY)_{\on{co}}\to \Shv(\CY)_{\on{co}},$$
respectively.  The objects \eqref{e:unit 1 and 2} correspond to the identity functor in both cases.

\medskip

By contrast, objects of $\Shv(\CY\times \CY)$ correspond to functors defined by kernels
$$\Shv(\CY)_{\on{co}}\to \Shv(\CY),$$
and the object $\on{u}^{\on{naive}}_\CY$ corresponds to the functor
$$\on{Id}^{\on{naive}}_\CY:\Shv(\CY)_{\on{co}}\to \Shv(\CY),$$
see \secref{sss:Id naive}. 

\sssec{}

The key player in this section is the object 
$$\on{u}^{\on{naive}}_{\Bun_G}\in \Shv(\Bun_G\times \Bun_G).$$

\sssec{} \label{sss:conj Hecke}

In what follows, we will need the following property of $\on{u}^{\on{naive}}_{\Bun_G}$ vis-a-vis the
Hecke action:

\medskip

The category $\Rep(\cG)$ carries a canonical involution, denoted $\tau$, given by the Chevalley involution on $\cG$. 
We will denote it by $V\mapsto V^\tau$. Let 
$$\CV\mapsto \CV^\tau$$
denote the induced involution on $\Rep(\cG)_\Ran$.

\medskip

We claim:  

\begin{prop} \label{p:tau and sigma}
For $\CV\in \Rep(\cG)_\Ran$, 
%\begin{equation} \label{e:sigma and tau}
$$\sigma((\on{Id}_{\Bun_G}\boxtimes \sH_\CV)(\on{u}^{\on{naive}}_{\Bun_G}))
\simeq (\on{Id}_{\Bun_G}\boxtimes \sH_{\CV^\tau})(\on{u}^{\on{naive}}_{\Bun_G}),$$
%\end{equation}
where $\sigma$ is the transposition acting on $\Bun_G\times \Bun_G$.
%, and $\tau$ is as in \secref{sss:conj Hecke}. 
\end{prop}

\begin{proof}

Unwinding the construction of the functors $\sH_\CV$, we obtain that we need 
to show that for every $V\in \Rep(\cG)^{\otimes I}$, we have
$$\sigma((\on{Id}_{\Bun_G}\boxtimes \sH(V,-))(\on{u}^{\on{naive}}_{\Bun_G}))
\simeq (\on{Id}_{\Bun_G}\boxtimes \sH(V^\tau,-))(\on{u}^{\on{naive}}_{\Bun_G})$$
as objects of 
$$\Shv(\Bun_G\times \Bun_G\times X^I),$$
where $\sigma$ is the transposition of the $\Bun_G$ factors. 

\medskip 

By the definition of the Hecke action, the object 
$$(\on{Id}_{\Bun_G}\boxtimes \sH(V,-))(\on{u}^{\on{naive}}_{\Bun_G})\in \Shv(\Bun_G\times \Bun_G\times X^I)$$
identifies with 
$$(\hl\times \hr\times s)_*(\on{Sat}_I(V)).$$

The assertion now follows from the fact that there is a canonical isomorphism 
\begin{equation} \label{e:sigma and tau}
\sigma(\on{Sat}_I(V))\simeq \on{Sat}_I(V^\tau), \quad V\in \Rep(\cG)^{\otimes I}.
\end{equation}

\end{proof}

\ssec{The projector and the diagonal} \label{ss:proj and diag}

In this subsection we establish the behavior of the diagonal object 
$\on{u}^{\on{naive}}_{\Bun_G}$ vis-a-vis the projector $\sP$ onto
the subcategory with nilpotent singular support. 

\sssec{}

We claim:

\begin{prop} \label{p:proj identity}
We have canonical isomorphisms 
$$(\sP_{\Bun_G,\Nilp}\boxtimes \on{Id}_{\Bun_G})(\on{u}^{\on{naive}}_{\Bun_G}) 
\simeq (\sP_{\Bun_G,\Nilp}\boxtimes \sP_{\Bun_G,\Nilp})(\on{u}^{\on{naive}}_{\Bun_G}) \simeq
(\on{Id}_{\Bun_G}\boxtimes \sP_{\Bun_G,\Nilp})(\on{u}^{\on{naive}}_{\Bun_G}).$$
\end{prop}

\begin{proof}

We will prove the second isomorphism; the first one would follow by symmetry.

\medskip

In the notation of \secref{sss:the projector}, we have
\begin{equation} \label{e:P sq new}
(\sP_{\Bun_G,\Nilp}\boxtimes \sP_{\Bun_G,\Nilp})(\on{u}^{\on{naive}}_{\Bun_G}) =
(\sH_\sR \boxtimes \sH_\sR)(\on{u}^{\on{naive}}_{\Bun_G}) 
\simeq (\on{Id}_{\Bun_G} \boxtimes \sH_\sR) \circ 
(\sH_\sR \boxtimes \on{Id}_{\Bun_G})(\on{u}^{\on{naive}}_{\Bun_G}).
\end{equation}

Note that  
$$(\sH_\sR \boxtimes \on{Id}_{\Bun_G})(\on{u}^{\on{naive}}_{\Bun_G})\simeq 
\sigma((\on{Id}_{\Bun_G}\boxtimes \sH_\sR)(\on{u}^{\on{naive}}_{\Bun_G})),$$
and by \propref{p:tau and sigma} we can rewrite the latter as 
$$(\on{Id}_{\Bun_G}\boxtimes \sH_{\sR^\tau})(\on{u}^{\on{naive}}_{\Bun_G}).$$

Hence, the expression in \eqref{e:P sq new} can be rewritten as 
\begin{equation} \label{e:P sq bis}
(\on{Id}_{\Bun_G}\boxtimes \sH_\sR) \circ (\on{Id}_{\Bun_G} \boxtimes \sH_{\sR^\tau}) (\on{u}^{\on{naive}}_{\Bun_G}).
\end{equation}

Now, the canonicity of the construction of the object $\sR$ implies that we have a canonical
identification
$$\sR^\tau\simeq \sR.$$

Hence, we can further rewrite the expression in \eqref{e:P sq bis} as
$$(\on{Id}_{\Bun_G}\boxtimes \sH_\sR) \circ (\on{Id}_{\Bun_G}\boxtimes \sH_\sR)(\on{u}^{\on{naive}}_{\Bun_G})=
(\on{Id}_{\Bun_G} \boxtimes \sP_{\Bun_G,\Nilp}) \circ (\on{Id}_{\Bun_G}\boxtimes \sP_{\Bun_G,\Nilp})(\on{u}^{\on{naive}}_{\Bun_G}).$$ 

Now, the isomorphism with $(\on{Id}_{\Bun_G}\boxtimes \sP_{\Bun_G,\Nilp})(\on{u}^{\on{naive}}_{\Bun_G})$
follows from \thmref{t:projector}.

\end{proof}

\sssec{}

In what follows we will denote the object that appears in \propref{p:proj identity} by
$$\on{u}^{\on{naive}}_{\Bun_G,\Nilp}.$$

By \corref{c:double projector}, it belongs to the full subcategory 
$$\Shv_\Nilp(\Bun_G)\otimes \Shv_\Nilp(\Bun_G)\subset \Shv(\Bun_G\times \Bun_G).$$

\sssec{} \label{sss:u U Nilp}

Let $\CU \overset{j}\hookrightarrow \Bun_G$ be a universally Nilp-cotruncative quasi-compact open prestack.
Denote by $\on{u}_{\CU,\Nilp}$ the object
$$(j\times j)^*(\on{u}^{\on{naive}}_{\Bun_G,\Nilp})\in \Shv(\CU\times \CU).$$

Since $\on{u}^{\on{naive}}_{\Bun_G,\Nilp}\in \Shv_\Nilp(\Bun_G)\otimes \Shv_\Nilp(\Bun_G)$, we obtain that 
$\on{u}_{\CU,\Nilp}$ belongs to the full subcategory
$$\Shv_\Nilp(\CU)\otimes \Shv_\Nilp(\CU)\subset \Shv(\CU\times \CU).$$

\begin{prop} \label{p:u naive projectors}
The object $\on{u}_{\CU,\Nilp}$ is isomorphic to
$$(\sP_{\CU,\Nilp}\boxtimes \on{Id}_{\CU})(\on{u}_{\CU}),\,\, (\sP_{\CU,\Nilp}\boxtimes \sP_{\CU,\Nilp})(\on{u}_{\CU}) \text{ and }
(\on{Id}_{\CU}\boxtimes \sP_{\CU,\Nilp})(\on{u}_{\CU}).$$
\end{prop}

\begin{proof}

We have
\begin{multline*} 
(\sP_{\CU,\Nilp}\boxtimes \on{Id}_{\CU})(\on{u}_{\CU})
\simeq ((j^*\circ \sP_{\Bun_G,\Nilp} \circ j_*) \boxtimes \on{Id}_{\CU})(\on{u}_{\CU})
\simeq 
((j^*\circ \sP_{\Bun_G,\Nilp}) \boxtimes \on{Id}_{\CU})\circ (j_*\boxtimes \on{Id}_{\CU})(\on{u}_{\CU}) \simeq \\
\simeq 
((j^*\circ \sP_{\Bun_G,\Nilp}) \boxtimes \on{Id}_{\CU})\circ (\on{Id}_{\CU}\boxtimes j^*)(\on{u}^{\on{naive}}_{\Bun_G}) \simeq
(j\times j)^* \circ (\sP_{\Bun_G,\Nilp}\boxtimes \on{Id}_{\Bun_G})(\on{u}^{\on{naive}}_{\Bun_G}) \simeq \\
\simeq (j\times j)^*(\on{u}^{\on{naive}}_{\Bun_G,\Nilp})=\on{u}_{\CU,\Nilp}.
\end{multline*}

This proves the first isomorphism. The third isomorphism follows by symmetry. To prove the middle isomorphism,
we note that
\begin{multline*} 
(\sP_{\CU,\Nilp}\boxtimes \sP_{\CU,\Nilp})(\on{u}_{\CU})  \simeq
(\on{Id}_\CU \boxtimes \sP_{\CU,\Nilp})\circ (\sP_{\CU,\Nilp}\boxtimes \on{Id}_\CU)(\on{u}_{\CU})  \overset{\text{1st isomorphism}}\simeq \\
\simeq 
(\on{Id}_\CU \boxtimes \sP_{\CU,\Nilp}) (\on{u}_{\CU,\Nilp}) \overset{\text{\corref{c:projector U}}}\simeq \on{u}_{\CU,\Nilp}.
\end{multline*}

\end{proof}

\ssec{Verdier self-duality for a universally Nilp-cotruncative substack}

Let $\CU \overset{j}\hookrightarrow \Bun_G$ be a universally Nilp-cotruncative quasi-compact open substack.

\medskip

In this subsection we will show that the Verdier self-duality on $\Shv(\CU)$
restricts to a self-duality of the full subcategory $\Shv_\Nilp(\CU)\subset \Shv(\CU)$. 

\sssec{}

Throughout this paper, for a quasi-compact algebraic stack $\CY$ we denote by $\on{ev} _\CY$ the functor 
\begin{equation} \label{e:pairing sheaves Y init}
\Shv(\CY)\otimes \Shv(\CY) \overset{\sotimes}\to \Shv(\CY) \overset{\on{C}^\cdot_\blacktriangle(\CY,-)}\longrightarrow \Vect.
\end{equation}

\medskip

Recall that according to \secref{ss:Verdier}, the functor \eqref{e:pairing sheaves Y init} defines the counit
of a self-duality on $\Shv(\CY)$, which we refer to as \emph{Verdier self-duality}.

\sssec{}

We claim:

\begin{prop} \label{p:usual duality U}
The functors
\begin{equation} \label{e:counit usual U}
\Shv_\Nilp(\CU)\otimes \Shv_\Nilp(\CU) \to \Shv(\CU)\otimes \Shv(\CU) 
\overset{\on{ev} _\CU}\longrightarrow \Vect
\end{equation}
and
$$\on{u}_{\CU,\Nilp}\in \Shv_\Nilp(\CU)\otimes \Shv_\Nilp(\CU)$$
define a datum of self-duality on $\Shv_\Nilp(\CU)$.
\end{prop}

\begin{proof}

Let $\CF$ be an object of $\Shv_\Nilp(\CU)$. We have to construct a canonical isomorphism
$$(p_2)_\blacktriangle(p_1^!(\CF)\sotimes \on{u}_{\CU,\Nilp})\simeq \CF.$$

Using \propref{p:u naive projectors} and interpreting $\on{u}_{\CU,\Nilp}$ as $(\on{Id}_{\CU}\boxtimes \sP_{\CU,\Nilp})(\on{u}_{\CU})$, we rewrite
$$(p_2)_\blacktriangle(p_1^!(\CF)\sotimes \on{u}_{\CU,\Nilp}) \simeq
\sP_{\CU,\Nilp} ((p_2)_\blacktriangle(p_1^!(\CF)\sotimes \on{u}_{\CU})) \simeq \sP_{\CU,\Nilp}(\CF),$$
while the latter is isomorphic to $\CF$ by \corref{c:projector U}.

\end{proof}

\begin{cor} \label{c:usual duality U}
The pair $(\CU,\Nilp)$ is duality-adapted\footnote{See \secref{sss:duality adapted} for what this means.}.
\end{cor}

\sssec{}

We now claim:

\begin{prop} \label{p:dual of projector U}
With respect to the self-duality
$$\Shv(\CU)^\vee\simeq \Shv(\CU)$$
of \eqref{e:Verdier} and the self-duality 
$$\Shv_\Nilp(\CU)^\vee \simeq \Shv_\Nilp(\CU)$$
of \propref{p:usual duality U}, the functor
$\sP_{\CU,\Nilp}: \Shv(\CU)\to \Shv_\Nilp(\CU)$
identifies with the dual of the tautological embedding 
$\iota_\CU:\Shv_\Nilp(\CU)\hookrightarrow \Shv(\CU)$.
\end{prop} 

\begin{proof}

We need to establish a canonical isomorphism
\begin{equation} \label{e:P u Shv}
(\on{Id}_{\Shv(\CU)}\otimes \sP_{\CU,\Nilp})(\on{u}_{\Shv(\CU)})
\simeq \on{u}_{\CU,\Nilp}
\end{equation}
(here $\on{u}_{\Shv(\CU)}$ is as in \secref{sss:u shv}) as objects of 
$$\Shv(\CU)\otimes \Shv_\Nilp(\CU)\subset \Shv(\CU)\otimes \Shv(\CU).$$

Let $\CF_1,\CF_2$ be a pair of compact objects of $\Shv(\CU)$, and let
us calculate
$$\CHom_{\Shv(\CU)\otimes \Shv(\CU)}(\CF_1\otimes \CF_2,-)$$
into both sides of \eqref{e:P u Shv}.

\medskip

Note that since $\CF_1$ is compact, the functor
$$\CHom_{\Shv(\CU)}(\CF_1,-): \Shv(\CU)\to \Vect$$
is continuous. Furthermore, the functor
$$\CHom_{\Shv(\CU)}(\CF_1,-)\otimes \on{Id}_{\Shv(\CU)}:\Shv(\CU)\otimes \Shv(\CU)\to \Shv(\CU)$$
sends $\on{u}_{\Shv(\CU)}$ to $\BD^{\on{Verdier}}(\CF_1)$. 

\medskip

Hence, for the left-hand side of \eqref{e:P u Shv}, we obtain
\begin{multline*} 
\CHom_{\Shv(\CU)\otimes \Shv(\CU)}(\CF_1\otimes \CF_2,(\on{Id}_{\Shv(\CU)}\otimes \sP_{\CU,\Nilp})(\on{u}_{\Shv(\CU)}))
\simeq \\
\simeq 
\CHom_{\Shv(\CU)}\left(\CF_2, (\CHom_{\Shv(\CU)}(\CF_1,-)\otimes \on{Id}_{\Shv(\CU)})\circ (\on{Id}_{\Shv(\CU)}\otimes \sP_{\CU,\Nilp})(\on{u}_{\Shv(\CU)})\right) 
\simeq \\
\simeq \CHom_{\Shv(\CU)}\left(\CF_2, \sP_{\CU,\Nilp} \circ (\CHom_{\Shv(\CU)}(\CF_1,-)\otimes \on{Id}_{\Shv(\CU)})(\on{u}_{\Shv(\CU)})\right) \simeq \\
\simeq \CHom_{\Shv(\CU)}(\CF_2,\sP_{\CU,\Nilp}(\BD^{\on{Verdier}}(\CF_1))).
\end{multline*} 

\medskip

Let us interpret $\on{u}_{\CU,\Nilp}$ as $(\on{Id}_\CU\boxtimes \sP_{\CU,\Nilp})(\on{u}_\CU)$. Note that since 
$\CF_1$ and $\CF_2$ are compact, for any $\CF\in \Shv(\CU\times \CU)$, we have
$$\CHom_{\Shv(\CU\times \CU)}(\CF_1\boxtimes \CF_2,\CF)
\simeq \CHom_{\Shv(\CU)}\left(\CF_2, (p_2)_\blacktriangle(p_1^!(\BD^{\on{Verdier}}(\CF_1))\sotimes \CF)\right).$$

Hence, for the right-hand side of \eqref{e:P u Shv}, we obtain
\begin{multline*} 
\CHom_{\Shv(\CU)\otimes \Shv(\CU)}(\CF_1\otimes \CF_2,\on{u}_{\CU,\Nilp}) \simeq 
%\CHom_{\Shv(\CU)\otimes \Shv(\CU)}(\CF_1\otimes \CF_2,(\on{Id}_\CU\boxtimes \sP_{\CU,\Nilp})(\on{u}_\CU)) \simeq \\ \simeq 
\CHom_{\Shv(\CU\times \CU)}(\CF_1\boxtimes \CF_2,(\on{Id}_\CU\boxtimes \sP_{\CU,\Nilp})(\on{u}_\CU)) \simeq \\
\simeq 
\CHom_{\Shv(\CU)}\left(\CF_2,(p_2)_\blacktriangle\left(p_1^!(\BD^{\on{Verdier}}(\CF_1))\sotimes (\on{Id}_\CU\boxtimes \sP_{\CU,\Nilp})(\on{u}_\CU)\right)\right)
\simeq  \\
\simeq \CHom_{\Shv(\CU)}\left(\CF_2, \sP_{\CU,\Nilp}\left((p_2)_\blacktriangle(p_1^!(\BD^{\on{Verdier}}(\CF_1))\sotimes \on{u}_\CU)\right)\right) \simeq \\
\simeq \CHom_{\Shv(\CU)}(\CF_2,\sP_{\CU,\Nilp}(\BD^{\on{Verdier}}(\CF_1))),$$
\end{multline*} 
as required. 

\end{proof}

\begin{cor}  \label{c:dual of projector U}
For $\CF_1\in \Shv_\Nilp(\CU)$ and $\CF_2\in \Shv(\CU)$, we have a canonical isomorphism
$$\on{C}^\cdot_\blacktriangle(\CU,\CF_1\sotimes \CF_2) \simeq
\on{C}^\cdot_\blacktriangle(\CU,\CF_1\sotimes \sP_{\CU,\Nilp}(\CF_2)).$$
\end{cor}

\ssec{Constraccessibility}

The material in this subsection is not needed for the main results of this paper,
so it can be skipped on the first pass. 

\sssec{}

We claim:

\begin{prop}  \label{p:PU right adjoint}
Let $\CU\subset \Bun_G$ be a universally $\Nilp$-cotrancative quasi-compact open substack. 
The following statements are equivalent:

\smallskip

\noindent{\em(i)} The functor $\iota_\CU^R$ is continuous;

\smallskip

\noindent{\em(ii)} The functor $\sP_{\CU,\Nilp}$ provides a right adjoint to $\iota_\CU$.

%\smallskip
%
%\noindent{\em(iii)} The natural transformation \eqref{e:right adjoint and proj} is an isomorphism.

\end{prop}

\begin{proof}

Clearly, % (iii) $\Rightarrow$ 
(ii) $\Rightarrow$ (i). We will prove that (i) implies (ii) using the following general assertion:

\begin{lem} \label{l:right adjoint and proj}
Let $\iota:\bC_1\to\bC$ be a fully faithful functor between compactly generated DG categories.
Suppose that $\iota$ preserves compactness, and let 
$$\iota^{\on{fake-op}}:\bC_1^\vee\to \bC^\vee,$$
be the ind-extension of $\iota^{\on{op}}:(\bC_1^c)^{\on{op}}\to (\bC^c)^{\on{op}}$. 
Then $(\iota^{\on{fake-op}})^\vee$ is the right adjoint of $\iota$.
\end{lem}

We apply this lemma as follows: take $\iota:\bC_1\to\bC$ to be $\iota_\CU:\Shv_\Nilp(\CU)\hookrightarrow \Shv(\CU)$.

\medskip

Assumption (i) means that $\Shv_\Nilp(\CU)$ is constraccessible, and the self-duality of $\Shv_\Nilp(\CU)$ with counit
\eqref{e:counit usual U}, at the level of compact objects is induced by the Verdier duality functor \eqref{e:Verdier functor}. This implies
that with respect to this self-duality of $\Shv_\Nilp(\CU)$ and the self-duality of $\Shv(\CU)$ given by
\eqref{e:Verdier}, the functor $\iota_\CU^{\on{fake-op}}$ identifies again with $\iota_\CU$.  

\medskip

Hence,
$$\iota_\CU^R \overset{\text{\lemref{l:right adjoint and proj}}}\simeq
(\iota^{\on{fake-op}}_\CU)^\vee \simeq \iota_\CU^\vee \overset{\text{\propref{p:dual of projector U}}}\simeq \sP_{\CU,\Nilp}.$$

 \end{proof}
 
 \sssec{}
 
From now on, for the duration of this subsection we will assume \cite[Conjecture 14.1.8]{AGKRRV}. This conjecture says
that the category $\Shv_\Nilp(\Bun_G)$ is generated by objects that are compact in the ambient
category $\Shv(\Bun_G)$. (Recall also that this conjecture 
is known to hold in the de Rham and Betti contexts, see \cite[Theorems 16.4.3 and 16.4.10]{AGKRRV}.)

\bigskip

Given \thmref{t:trunc}, this conjecture
is equivalent to the statement that for every universally Nilp-cotruncative quasi-compact open substack
$\CU\subset \Bun_G$, the pair $(\CU,\Nilp)$ is constraccessible\footnote{See \secref{sss:constraccess} for what this means.} 
(see \cite[Lemma F.8.10]{AGKRRV}). 

\medskip

Since each $\Shv_\Nilp(\CU)$ is compactly generated (\cite[Corollary 16.1.8]{AGKRRV}), the latter
statement is equivalent to saying that the right adjoint of the tautological embedding
$$\iota_\CU:\Shv_\Nilp(\CU)\hookrightarrow \Shv(\CU)$$
is continuous.

\medskip

Hence, by \propref{p:PU right adjoint}, we obtain that the functor $\sP_{\CU,\Nilp}$ identifies with the 
right adjoint of the embedding $\iota_\CU$.

\medskip

In addition, according to \cite[Proposition 17.2.3]{AGKRRV}, the functor 
$\sP_{\Bun_G,\Nilp}$, viewed as a functor $\Shv(\Bun_G)\to \Shv_\Nilp(\Bun_G)$, is the right adjoint
to the tautological embedding $\iota$.  

\sssec{}

Let $\CZ$ be an algebraic stack. Recall that the functors 
$$\on{Id}_\CZ\boxtimes \sP_{\Bun_G,\Nilp} \text{ and } \on{Id}_\CZ\boxtimes \sP_{\CU,\Nilp}$$
takes values in the subcategories
$$\Shv(\CZ)\otimes \Shv_{\Nilp}(\Bun_G)\subset \Shv(\CZ\times \Bun_G)  \text{ and } 
\Shv(\CZ)\otimes \Shv_{\Nilp}(\CU) \subset \Shv(\CZ\times \CU),$$
respectively (see \thmref{t:projector} and \corref{c:projector U}). 

\medskip

We claim:

\begin{prop} \label{p:P as adj}
Assuming \cite[Conjecture 14.1.8]{AGKRRV}, for any algebraic stack $\CZ$, the functors 
$$\on{Id}_\CZ\boxtimes \sP_{\Bun_G,\Nilp} \text{ and } \on{Id}_\CZ\boxtimes \sP_{\CU,\Nilp}$$
are the right adjoints of the embeddings
$$\Shv(\CZ)\otimes \Shv_{\Nilp}(\Bun_G)\hookrightarrow \Shv(\CZ\times \Bun_G) \text{ and }
\Shv(\CZ)\otimes \Shv_{\Nilp}(\CU)\hookrightarrow \Shv(\CZ\times \CU),$$
respectively.
\end{prop} 

\begin{proof}

We will prove the assertion for $\Bun_G$; the case of $\CU$ is similar. 

\medskip

First, it is easy to reduce the assertion to the case when $\CZ$ is quasi-compact, which
we will from now on assume.

\medskip

We need to show that for $\CF_\CZ \in \Shv(\CZ)^c$, $\CF_\Nilp\in \Shv_{\Nilp}(\Bun_G)$
and $\CF\in \Shv(\CZ\times \Bun_G)$ we have a canonical isomorphism
\begin{equation} \label{e:P adj ver}
\CHom(\CF_\CZ\boxtimes \CF_\Nilp,\CF) \simeq 
\CHom(\CF_\CZ\boxtimes \CF_\Nilp, (\on{Id}_\CZ\boxtimes \sP_{\Bun_G,\Nilp})(\CF)).
\end{equation}

Set
$$\CF':=(p_{\Bun_G})_*(p_\CZ^!(\BD^{\on{verdier}}(\CF_\CZ))\sotimes \CF)\in \Shv(\Bun_G).$$
Then the left-hand side in \eqref{e:P adj ver} identifies with
$$\CHom_{\Shv(\Bun_G)}(\CF_\Nilp,\CF').$$

Set
$$\CF'':=(p_{\Bun_G})_*(p_\CZ^!(\BD^{\on{verdier}}(\CF_\CZ))\sotimes (\on{Id}_\CZ\boxtimes \sP_{\Bun_G,\Nilp})(\CF))\in \Shv(\Bun_G).$$

Then the right-hand side in \eqref{e:P adj ver} identifies with
$$\CHom_{\Shv(\Bun_G)}(\CF_\Nilp,\CF'').$$

However, since $\CF_\CZ$ is compact (and hence so is $\BD^{\on{verdier}}(\CF_\CZ)$), the maps
$$(p_{\Bun_G})_\blacktriangle(p_\CZ^!(\BD^{\on{verdier}}(\CF_\CZ))\sotimes \CF)\to 
(p_{\Bun_G})_*(p_\CZ^!(\BD^{\on{verdier}}(\CF_\CZ))\sotimes \CF)$$
and
$$(p_{\Bun_G})_\blacktriangle(p_\CZ^!(\BD^{\on{verdier}}(\CF_\CZ))\sotimes (\on{Id}_\CZ\boxtimes \sP_{\Bun_G,\Nilp})(\CF))\to$$
$$\to (p_{\Bun_G})_*(p_\CZ^!(\BD^{\on{verdier}}(\CF_\CZ))\sotimes (\on{Id}_\CZ\boxtimes \sP_{\Bun_G,\Nilp})(\CF))$$
are isomorphisms.

\medskip

This implies that
$$\CF''\simeq \sP_{\Bun_G,\Nilp}(\CF').$$

Hence, the assertion follows from the $(\iota,\sP_{\Bun_G,\Nilp})$ adjunction.

\end{proof}

Iterating, from \propref{p:P as adj} we obtain: 

\begin{cor} \label{c:P P as adj}
Assuming \cite[Conjecture 14.1.8]{AGKRRV}, the functor 
$$\sP_{\Bun_G,\Nilp}\boxtimes \sP_{\Bun_G,\Nilp}:\Shv(\Bun_G\times \Bun_G)\to 
\Shv_{\Nilp}(\Bun_G)\otimes  \Shv_{\Nilp}(\Bun_G)$$
is the right adjoint of 
$$\Shv_{\Nilp}(\Bun_G)\otimes  \Shv_{\Nilp}(\Bun_G)\overset{\iota\otimes \iota}\longrightarrow
\Shv(\Bun_G)\boxtimes  \Shv(\Bun_G)\overset{\boxtimes}\hookrightarrow \Shv(\Bun_G\times \Bun_G).$$
\end{cor}

\sssec{}

Note that by passing to right adjoints in the commutative diagrams 
$$
\CD
\Shv_\Nilp(\CU) @>{\iota_\CU}>> \Shv(\CU) \\
@V{j_!}VV @VV{j_!}V \\
\Shv_\Nilp(\Bun_G) @>{\iota}>> \Shv(\Bun_G)
\endCD
$$
and
$$
\CD
\Shv_\Nilp(\CU) @>{\iota_\CU}>> \Shv(\CU) \\
@A{j^*}AA @AA{j^*}A \\
\Shv_\Nilp(\Bun_G) @>{\iota}>> \Shv(\Bun_G),
\endCD
$$
we obtain the isomorphisms
\begin{equation} \label{e:j and P}
j^*\circ \sP_{\Bun_G,\Nilp}\simeq \sP_{\CU,\Nilp}\circ j^*
\end{equation}
and 
\begin{equation} \label{e:P and j}
j_*\circ \sP_{\CU,\Nilp} \simeq \sP_{\Bun_G,\Nilp}\circ j_*.
\end{equation}

We claim that a stronger assertion holds:

\begin{cor} \label{c:j and P}
Assuming \cite[Conjecture 14.1.8]{AGKRRV}, the isomorphisms \eqref{e:j and P} and \eqref{e:P and j} hold as 
functors defined by kernels, where we view $\sP_{\Bun_G,\Nilp}$ and $\sP_{\CU,\Nilp}$ 
as endofunctors defined by kernels of $\Shv(\Bun_G)$ and $\Shv(\CU)$, 
respectively. 
\end{cor}

\begin{proof}

We will prove the assertion for \eqref{e:j and P}; the assertion for \eqref{e:P and j}
is similar.

\medskip

We have to construct an isomorphism 
$$\on{Id}_\CZ\boxtimes (j^*\circ \sP_\CU) \simeq \on{Id}_\CZ\boxtimes (\sP_\CU\circ j^*)$$
for any algebraic stack $\CZ$. 

\medskip

This follows formally from \propref{p:P as adj} by passing to right adjoints of the composition 
\begin{equation} \label{e:j and P adj 1}
\Shv(\CZ)\otimes \Shv_{\Nilp}(\CU) \overset{\on{Id}\otimes j_!}\longrightarrow
\Shv(\CZ)\otimes \Shv_{\Nilp}(\Bun_G) \hookrightarrow \Shv(\CZ\times \Bun_G),
\end{equation}
which is the same as
\begin{equation} \label{e:j and P adj 2}
\Shv(\CZ)\otimes \Shv_{\Nilp}(\CU)\hookrightarrow \Shv(\CZ\times \CU)
\overset{(\on{id}\times j)_!}\longrightarrow \Shv(\CZ\times \Bun_G).
\end{equation}

\end{proof}

\ssec{The ``co" category for $\Bun_G$ with nilpotent singular support} \label{ss:co BunG}

In order to talk about Verdier duality on the (non quasi-compact) algebraic stack
$\Bun_G$, we need to study the category $\Shv(\Bun_G)_{\on{co}}$, see \secref{ss:co}.

\medskip

In this subsection we will discuss the pattern of Hecke action on $\Shv(\Bun_G)_{\on{co}}$. 

\medskip

We also introduce the ``co"-version of the subcategory with nilpotent singular support
$$\Shv_\Nilp(\Bun_G)_{\on{co}} \subset \Shv(\Bun_G)_{\on{co}}.$$

\sssec{}

Recall the category $\Shv(\Bun_G)_{\on{co}}$, defined as 
$$\underset{\CU}{\on{colim}_*}\, \Shv(\CU),$$
where the colimit is taken over the poset of cotruncative quasi-compact
open substacks of $\Bun_G$, see \secref{sss:co}. 

\sssec{}

The Hecke action on $\Shv(\Bun_G)$ gives rise to a Hecke action on $\Shv(\Bun_G)_{\on{co}}$:

\medskip

Let $V$ be an object of $\Rep(\cG)^{\otimes I}$. Note that the Hecke functor 
$$\sH(V,-):\Shv(\Bun_G)\to \Shv(\Bun_G\times X^I)$$
is defined by the kernel  
$$(\on{Id}_{\Bun_G}\boxtimes \sH(V,-))(\on{u}_{\Bun_G,\on{co}_1})\in \Shv(\Bun_G\times \Bun_G\times X^I)_{\on{co}_1}.$$

Consider now the object 
\begin{equation} \label{e:Hecke co}
\sigma \left((\on{Id}_{\Bun_G}\boxtimes \sH(V^\tau,-))(\on{u}_{\Bun_G,\on{co}_1})\right)\in 
\Shv(\Bun_G\times \Bun_G\times X^I)_{\on{co}_2},
\end{equation}
where $\sigma$ is the transposition of the two factors of $\Bun_G$. 

\medskip
 
We let  
$$\sH(V,-)_{\on{co}}:\Shv(\Bun_G)_{\on{co}}\to \Shv(\Bun_G\times X^I)_{\on{co}}$$
be the functor defined by the kernel \eqref{e:Hecke co}. 

\sssec{}

Recall the functor 
$$\on{Id}^{\on{naive}}_{\Bun_G}:\Shv(\Bun_G)_{\on{co}}\to \Shv(\Bun_G),$$
see \secref{sss:naive BunG}. We claim:

\begin{lem} \label{l:Id nv and Hecke}
The diagram (of functors defined by kernels)
$$
\CD
\Shv(\Bun_G)_{\on{co}} @>{\sH(V,-)_{\on{co}}}>>  \Shv(\Bun_G\times X^I)_{\on{co}} \\
@V{\on{Id}^{\on{naive}}_{\Bun_G}}VV @VV{\on{Id}^{\on{naive}}_{\Bun_G}\boxtimes \on{Id}_{X^I}}V \\
\Shv(\Bun_G) @>{\sH(V,-)}>>  \Shv(\Bun_G\times X^I)
\endCD
$$
commutes.
\end{lem}

\begin{proof}

Follows by diagram chase from \propref{p:tau and sigma}.

\end{proof}

\sssec{} \label{sss:Hecke Ran co}

A similar discussion applies to the action of $\Rep(\cG)_\Ran$. Namely, for $\CV \in \Rep(\cG)_\Ran$
we have a functor defined by a kernel
$$\sH_{\CV,\on{co}}:\Shv(\Bun_G)_{\on{co}}\to \Shv(\Bun_G)_{\on{co}},$$
and the diagram (of functors defined by kernels)
\begin{equation} \label{e:Id nv and Hecke bis}
\CD
\Shv(\Bun_G)_{\on{co}} @>{\sH_{\CV,\on{co}}}>>  \Shv(\Bun_G)_{\on{co}} \\
@V{\on{Id}^{\on{naive}}_{\Bun_G}}VV @VV{\on{Id}^{\on{naive}}_{\Bun_G}}V \\
\Shv(\Bun_G) @>{\sH_\CV}>>  \Shv(\Bun_G)
\endCD
\end{equation}
commutes (indeed, this formally follows from \lemref{l:Id nv and Hecke}). 

\sssec{}

We claim:

\begin{lem} \label{l:dual of Hecke}
For $\CV \in \Rep(\cG)_\Ran$, we have a canonical isomorphism
$$(\sH_{\CV,\on{co}}\boxtimes \on{Id}_{\Bun_G})(\on{u}_{\Bun_G,\on{co}_1})
\simeq 
(\on{Id}_{\Bun_G}\boxtimes \sH_{\CV^\tau})(\on{u}_{\Bun_G,\on{co}_1})$$
as objects of $\Shv(\Bun_G\times \Bun_G)_{\on{co}_1}$.
\end{lem}

\begin{proof}

The assertion of the lemma is equivalent to an isomorphism
$$\sigma((\on{Id}_{\Bun_G}\boxtimes \sH_{\CV,\on{co}})(\on{u}_{\Bun_G,\on{co}_2}))
\simeq (\on{Id}_{\Bun_G}\boxtimes \sH_{\CV^\tau})(\on{u}_{\Bun_G,\on{co}_1}).$$

The fact that the images of both sides under the forgetful functor
$$\on{Id}^{\on{naive}}_{\Bun_G}\boxtimes \on{Id}_{\Bun_G}:\Shv(\Bun_G\times \Bun_G)_{\on{co}_1}\to \Shv(\Bun_G\times \Bun_G)$$
are canonically isomorphic is the content of \propref{p:tau and sigma} (using the commutative diagram \eqref{e:Id nv and Hecke bis}
as functors defined by kernels).

\medskip

In order to upgrade this isomorphism to an isomorphism that takes place in $\Shv(\Bun_G\times \Bun_G)_{\on{co}_1}$
we argue as follows:

\medskip

We can assume that the object $\CV\in \Rep(\cG)_\Ran$ is compact. It suffices to establish a compatible
family of isomorphisms
\begin{equation} \label{e:Hecke on U}
(\on{Id}_{\Bun_G}\times j)^*((\sH_{\CV,\on{co}}\boxtimes \on{Id}_{\Bun_G})(\on{u}_{\Bun_G,\on{co}_1}))\simeq 
(\on{Id}_{\Bun_G}\times j)^*((\on{Id}_{\Bun_G}\boxtimes \sH_{\CV^\tau})(\on{u}_{\Bun_G,\on{co}_1}))
\end{equation}
as objects in $\Shv(\Bun_G\times \CU)_{\on{co}}$ for quasi-compact open substacks $\CU\overset{j}\hookrightarrow \Bun_G$. 

\medskip

However, for $\CV$ compact and a fixed $\CU$, both sides in \eqref{e:Hecke on U} lie in the essential image
of the (fully faithful) functor
$$(j'\times \on{id}_\CU)_{*,\on{co}}:\Shv(\CU'\times \CU)\to \Shv(\Bun_G\times \CU)_{\on{co}}$$
for some quasi-compact open $\CU'\overset{j'}\hookrightarrow \Bun_G$.

\medskip

Hence, it is enough to show that both sides in \eqref{e:Hecke on U} become isomorphic after applying the
functor 
$$(j'\times \on{id}_\CU)^*_{\on{co}}:\Shv(\Bun_G\times \CU)_{\on{co}}\to \Shv(\CU'\times \CU).$$

However, 
$$(j'\times \on{id}_\CU)^*_{\on{co}}\circ (\on{Id}_{\Bun_G}\times j)^*\simeq
(j'\times j)^*\circ (\on{Id}^{\on{naive}}_{\Bun_G}\boxtimes \on{Id}_{\Bun_G})$$
and the assertion follows. 

\end{proof}

\sssec{}

Following \secref{sss:co N}, we define the category $\Shv_{\Nilp}(\Bun_G)_{\on{co}}$ as
$$\underset{\CU}{\on{colim}_*}\, \Shv_\Nilp(\CU),$$
where the colimit is taken over the poset of universally $\Nilp$-cotruncative quasi-compact
open substacks of $\Bun_G$, and the transition functors are given by $-_*$. 

\medskip

Since the above poset is filtered, the functor 
$$\iota_{\on{co}}:\Shv_{\Nilp}(\Bun_G)_{\on{co}}\to \Shv(\Bun_G)_{\on{co}}$$
comprised of the functors
$\iota_\CU:\Shv_\Nilp(\CU)\to \Shv(\CU)$, is fully faithful. 

\sssec{}

The following results from the definitions:

\begin{lem} \label{l:naive preserves Nilp}
The functor $\on{Id}^{\on{naive}}_{\Bun_G}$ sends 
$\Shv_\Nilp(\Bun_G)_{\on{co}}$ to $\Shv_\Nilp(\Bun_G)$.
\end{lem}

\ssec{Duality between $\Shv_\Nilp(\Bun_G)$ and $\Shv_\Nilp(\Bun_G)_{\on{co}}$}

In this subsection we will extend the Verdier self-duality of $\Shv_\Nilp(\CU)$ 
(where $\CU$ is a universally Nilp-contruncative quasi-compact open substack
of $\Bun_G$) to obtain a duality between $\Shv_\Nilp(\Bun_G)$ and $\Shv_\Nilp(\Bun_G)_{\on{co}}$.

\sssec{} \label{sss:Verdier on BunG}

Recall that according to \secref{sss:Verdier non-qc}, there is a canonical identification
\begin{equation} \label{e:duality with co BunG}
\Shv(\Bun_G)_{\on{co}} \simeq \Shv(\Bun_G)^\vee,
\end{equation} 
so that the pairing 
\begin{equation} \label{e:pairing with co BunG}
\Shv(\Bun_G)\otimes \Shv(\Bun_G)_{\on{co}}\to \Vect,
\end{equation} 
denoted $\on{ev} _{\Bun_G}$, is given by
$$\CF_1\in \Shv(\Bun_G),\,\CF_2\in \Shv(\Bun_G)_{\on{co}} \mapsto 
\on{C}^\cdot_\blacktriangle(\Bun_G,\CF_1\sotimes \CF_2),$$
where we view $\CF_1\sotimes \CF_2$ as an object of $\Shv(\Bun_G)_{\on{co}}$, and 
$\on{C}^\cdot_\blacktriangle(\Bun_G,-)$ denotes the corresponding functor
$$\Shv(\Bun_G)_{\on{co}}\to \Vect,$$
see \secref{sss:pairing non qc}. 

\sssec{}

From \lemref{l:dual of Hecke} and \secref{sss:sigma} we obtain:

\begin{cor} \label{c:dual of Hecke}
With respect to the duality \eqref{e:duality with co BunG}, the dual of the functor
$\sH_\CV$ identifies with $\sH_{\CV^\tau,\on{co}}$.
\end{cor}

\sssec{}

Combining \secref{sss:Verdier non-qc N} and \corref{c:usual duality U}, we obtain: 

\begin{cor} \label{c:duality with co BunG Nilp}
The restriction of the pairing \eqref{e:pairing with co BunG} along
$$\Shv_\Nilp(\Bun_G)_{\on{co}} \otimes \Shv_\Nilp(\Bun_G) \to \Shv(\Bun_G)_{\on{co}} \otimes \Shv(\Bun_G)$$
defines an equivalence
\begin{equation} \label{e:duality with co BunG Nilp}
\Shv_\Nilp(\Bun_G)_{\on{co}} \simeq \Shv_\Nilp(\Bun_G)^\vee. 
\end{equation} 
\end{cor}

%is characterized as follows:
%
%\medskip
%
%For a cotruncative open substack $\CU \overset{j}\hookrightarrow \Bun_G$, the composite 
%pairing
%$$\Shv(\CU) \otimes \Shv(\CU) 
%\overset{j_{*,\on{co}}\otimes j}\longrightarrow \Shv(\Bun_G)_{\on{co}} \otimes \Shv(\Bun_G)\to \Vect$$
%is the pairing \eqref{e:pairing sheaves Y}. 

\sssec{}

Our next step it to define a projector from $\Shv(\Bun_G)_{\on{co}}$
onto $\Shv_\Nilp(\Bun_G)_{\on{co}}$.

\medskip

Let $\sP_{\on{co}}$ denote the endofunctor of $\Shv(\Bun_G)_{\on{co}}$ equal to $\sH_{\sR,\on{co}}$
(see \secref{sss:Hecke Ran co} for the notation), where $\sR\in \Rep(\cG)_\Ran$ is as in \secref{sss:the projector}. We claim:

\medskip

\begin{prop} \label{p:dual of P}
With respect to the identifications \eqref{e:duality with co BunG} and \eqref{e:duality with co BunG Nilp},
we have:

\smallskip

\noindent{\em(i)} The functor dual to $\sP:\Shv(\Bun_G)\to \Shv_\Nilp(\Bun_G)$ identifies with
$$\iota_{\on{co}}:\Shv_\Nilp(\Bun_G)_{\on{co}}\to \Shv(\Bun_G)_{\on{co}}.$$

\smallskip

\noindent{\em(ii)} The essential image of the functor $\sP_{\on{co}}$ lies in $\Shv_\Nilp(\Bun_G)_{\on{co}}$,
and the resulting functor $$\sP_{\on{co}}:\Shv(\Bun_G)_{\on{co}}\to \Shv_\Nilp(\Bun_G)_{\on{co}}$$ is
the functor dual of $\iota:\Shv_\Nilp(\Bun_G)\to \Shv(\Bun_G)$.

\end{prop}

\begin{proof}

Point (i) follows by unwinding the definitions from \corref{c:dual of projector U}. 

\medskip

For point (ii),
given point (i), it suffices to show that the dual of $\sP$, viewed as an endofunctor of 
$\Shv(\Bun_G)$, is $\sP_{\on{co}}$, viewed as an endofunctor of 
$\Shv(\Bun_G)_{\on{co}}$.

\medskip

Since $\sP$ is defined by the kernel
$$(\on{Id}_{\Bun_G}\boxtimes \sH_\sR)(\on{u}_{\Bun_G,\on{co}_1})\in \Shv(\Bun_G\times \Bun_G)_{\on{co}_1},$$
from \corref{c:dual of Hecke}, we obtain that its dual is the functor $\sH_{\sR^\tau,\on{co}}$.

\medskip

Now the required assertion follows from the isomorphism $\sR^\tau\simeq \sR$.

\end{proof}

\begin{cor} \label{c:projector co}
The endofunctor $\sP_{\on{co}}$ of $\Shv(\Bun_G)_{\on{co}}$ is the projector onto the full subcategory
$$\Shv_{\Nilp}(\Bun_G)_{\on{co}}\subset \Shv(\Bun_G)_{\on{co}}.$$
\end{cor}

%\sssec{}
%
%Let
%$$\Shv(\Bun_G)_{\on{co}}^{\on{Hecke-lisse}}\subset \Shv(\Bun_G)_{\on{co}}$$
%be the full subcategory, consisting of objects $\CF$ such that
%$$\sH(V,\CF)_{\on{co}}\in \Shv(\Bun_G)_{\on{co}}\otimes \qLisse(X)\subset \Shv(\Bun_G\times X)_{\on{co}}, \quad \forall\, V\in \Rep(\cG).$$
%
%We claim:
%
%\begin{cor} \label{c:Hecke finite co}
%The subcategory $$\Shv(\Bun_G)_{\on{co}}^{\on{Hecke-lisse}}\subset \Shv(\Bun_G)_{\on{co}}$$ coincides 
%with $\Shv_{\Nilp}(\Bun_G)_{\on{co}}$.
%\end{cor}
%
%\begin{proof}
%
%It follows as in \cite[Sect. 14.5]{AGKRRV} that $\sP_{\on{co}}$ acts as a projector onto
%$\Shv(\Bun_G)_{\on{co}}^{\on{Hecke-lisse}}$. Now
%the assertion follows from \corref{c:projector co}. 
%
%\end{proof}

\ssec{The projector and the diagonal, revisited}

In this subsection, we will refine the results of \secref{ss:proj and diag}, when instead of
the object 
$$\on{u}^{\on{naive}}_{\Bun_G}\in \Shv(\Bun_G\times \Bun_G),$$
we consider its refined version, namely,
$$\on{u}_{\Bun_G,\on{co}_2}\in \Shv(\Bun_G\times \Bun_G)_{\on{co}_2}.$$

\sssec{}

First, we claim:

\begin{prop} \label{p:projector co}
The objects
$$(\on{Id}_{\Bun_G}\boxtimes \sP_{\on{co}})(\on{u}_{\Bun_G,\on{co}_2}),\,\, (\sP\boxtimes \sP_{\on{co}})(\on{u}_{\Bun_G,\on{co}_2}),\,\,
(\sP\boxtimes \on{Id}_{\Bun_G})(\on{u}_{\Bun_G,\on{co}_2})$$
are isomorphic.
\end{prop}

\begin{proof}

The isomorphism follows from \lemref{l:dual of Hecke} using the isomorphism
$\sR^\tau\simeq \sR$ in the same way as \propref{p:proj identity} follows
from \propref{p:tau and sigma}. 
\end{proof}

\sssec{}

Denote the object appearing in \propref{p:projector co} by $\on{u}_{\Bun_G,\Nilp,\on{co}_2}$.
We claim:

\begin{prop}
The object $\on{u}_{\Bun_G,\Nilp,\on{co}_2}$ belongs to
$$\Shv_{\Nilp}(\Bun_G)\otimes  \Shv_\Nilp(\Bun_G)_{\on{co}}
\subset \Shv(\Bun_G)\otimes  \Shv(\Bun_G)_{\on{co}}\subset \Shv(\Bun_G\times \Bun_G)_{\on{co}_2}.$$
\end{prop}

\begin{proof}

We first show that $(\sP\boxtimes \on{Id}_{\Bun_G})(\on{u}_{\Bun_G,\on{co}_2})$ belongs to
$$\Shv_{\Nilp}(\Bun_G)\otimes  \Shv(\Bun_G)_{\on{co}}\subset \Shv(\Bun_G\times \Bun_G)_{\on{co}_2}.$$

For that, it suffices to show that for every quasi-compact cotruncative open substack 
$\CU\overset{j}\hookrightarrow \Bun_G$, the object 
$$(\on{Id}_{\Bun_G}\boxtimes j^?)\circ (\sP\boxtimes \on{Id}_{\Bun_G})(\on{u}_{\Bun_G,\on{co}_2})\simeq
(\sP\boxtimes j^?)(\on{u}_{\Bun_G,\on{co}_2})\simeq
(\sP\boxtimes \on{Id}_{\Bun_G}) \circ (\on{Id}_{\Bun_G}\boxtimes j^?)(\on{u}_{\Bun_G,\on{co}_2})$$
belongs to
$$\Shv_{\Nilp}(\Bun_G)\otimes  \Shv(\CU)\subset \Shv(\Bun_G\times \CU),$$
where $j^?$ is the functor defined by a kernel as in \secref{sss:truncative}. 
However, this follows from Theorems \ref{t:projector} and \ref{t:Hecke action Nilp 2}.

\medskip

Now, the fact that  
$$(\sP\boxtimes \sP_{\on{co}})(\on{u}_{\Bun_G,\on{co}_2})
\simeq (\on{Id}_{\Bun_G}\boxtimes \sP_{\on{co}})\circ (\sP\boxtimes \on{Id}_{\Bun_G})(\on{u}_{\Bun_G,\on{co}_2})$$
belongs to
$$\Shv_{\Nilp}(\Bun_G)\otimes  \Shv_\Nilp(\Bun_G)_{\on{co}}\subset \Shv_{\Nilp}(\Bun_G)\otimes  \Shv(\Bun_G)_{\on{co}}$$
follows from \corref{c:projector co}.

\end{proof}

\sssec{}

We now claim: 

\begin{prop} \label{p:unit Verdier BunG}
The object 
$$\on{u}_{\Bun_G,\Nilp,\on{co}_2}\in \Shv_{\Nilp}(\Bun_G)\otimes  \Shv_\Nilp(\Bun_G)_{\on{co}}$$
is the unit of the duality \eqref{e:duality with co BunG Nilp}.
\end{prop}

\begin{proof}

Follows formally from \propref{p:projector co} in the same way as in the proof of
\propref{p:usual duality U}.

\end{proof}

\ssec{The miraculous functor on $\Bun_G$}

We now input another ingredient into our story, the \emph{miraculous functor},
see Sects. \ref{ss:Mir} and \ref{sss:Mir non qc}. 

\sssec{} \label{sss:ps-u}

Consider the object
$$\on{ps-u}_{\Bun_G}:=(\Delta_{\Bun_G})_!(\ul\sfe_{\Bun_G})\in 
\Shv(\Bun_G\times \Bun_G).$$

Following \secref{sss:Mir non qc}, we denote the resulting functor 
$$\Shv(\Bun_G)_{\on{co}}\to \Shv(\Bun_G)$$
by $\Mir_{\Bun_G}$.

%\medskip
%
%A similar discussion applies, when instead of $\Bun_G$ we consider its quasi-compact
%open substack $\CU$.

\sssec{}

First, we claim:

\begin{lem} \label{l:Mir and Hecke}
For a finite set $I$ and $V\in \Rep(\cG)^{\otimes I}$, the diagram 
\begin{equation} \label{e:Mir and Hecke new}
\CD
\Shv(\Bun_G)_{\on{co}} @>{\sH(V,-)_{\on{co}}}>>  \Shv(\Bun_G\times X^I)_{\on{co}} \\
@V{\Mir_{\Bun_G}}VV @VV{\Mir_{\Bun_G}\boxtimes \on{Id}_{X^I}}V \\
\Shv(\Bun_G) @>{\sH(V,-)}>>  \Shv(\Bun_G\times X^I)
\endCD
\end{equation} 
commutes, as functors defined by kernels. 
\end{lem}

\begin{proof}

Using \secref{sss:Hecke ULA} and \eqref{e:sigma and tau}, we obtain that both circuits of the diagram are functors defined by the kernel 
$$(\hl\times \hr\times s)_!(\on{Sat}_I(V))[-2\dim(\Bun_G)]\in \Shv(\Bun_G\times \Bun_G\times X^I).$$

\end{proof}

\begin{cor} \label{c:Hecke Mir}
The functor $\Mir_{\Bun_G}$ intertwines the actions of $\Rep(\cG)_\Ran$ on 
$\Shv(\Bun_G)_{\on{co}}$ and $\Shv(\Bun_G)$, respectively, as functors defined by kernels. 
\end{cor}

In particular, we obtain: 

\begin{cor} \label{c:Mir and P}
The functor $\Mir_{\Bun_G}$ intertwines the action of $\sP_{\on{co}}$ on 
$\Shv(\Bun_G)_{\on{co}}$ and the action of $\sP$ on $\Shv(\Bun_G)$. 
\end{cor}

\sssec{}

We now claim:

\begin{prop} \label{p:Nilp Mir compat}
The functor $\Mir_{\Bun_G}$ sends $\Shv_\Nilp(\Bun_G)_{\on{co}}$ to $\Shv_\Nilp(\Bun_G)$.
\end{prop}

\begin{proof}

Follows from \corref{c:Mir and P}, combined with \corref{c:projector co} and 
\thmref{t:projector} (for $\CZ=\on{pt}$).

\end{proof}

\sssec{}

Let $\CU\overset{j}\hookrightarrow \Bun_G$ be a quasi-compact open substack. Consider the
corresponding endofunctor 
$$\Mir_\CU:\Shv(\CU)\to \Shv(\CU),$$
see \secref{sss:Mir qc}.

\medskip

Assume that $\CU$ is cotruncative. Then, according to \secref{sss:Mir and j}, we have
\begin{equation} \label{e:Mir and U}
\Mir_{\Bun_G}\circ j_{*,\on{co}}\simeq j_!\circ \Mir_\CU.
\end{equation}

\sssec{}

We claim:

\begin{prop}  \label{p:Mir Nilp U}
Assume that $\CU$ is $\Nilp$-cotruncative. Then the functor $\Mir_\CU$ sends 
$\Shv_\Nilp(\CU)\to \Shv_\Nilp(\CU)$. 
\end{prop}

\begin{proof}

We have
$$\CF\in \Shv_\Nilp(\CU)\, \Leftrightarrow\, j_{*,\on{co}}(\CF)\in \Shv_\Nilp(\Bun_G)_{\on{co}}$$
and
$$\Mir_\CU(\CF)\in \Shv_\Nilp(\CU) \, \Leftrightarrow\, j_! \circ \Mir_\CU(\CF)\in \Shv_\Nilp(\Bun_G).$$

Using \eqref{e:Mir and U}, the assertion follows from \propref{p:Nilp Mir compat}.

\end{proof}

\ssec{The miraculous property of $\Bun_G$}  \label{ss:BunG Mir}

\sssec{} \label{sss:BunG Mir}

We now recall that in the paper \cite{Ga1} it was shown that the functor
$$\Mir_{\Bun_G}:\Dmod(\Bun_G)_{\on{co}}\to \Dmod(\Bun_G)$$
is an equivalence.

\medskip

However, the proof in {\it loc. cit.} applies in any sheaf-theoretic context.
Furthermore, the same proof shows that
$$\on{Id}_\CZ\boxtimes \Mir_{\Bun_G}:\Shv(\CZ\times \Bun_G)_{\on{co}_2}\to \Shv(\CZ\times \Bun_G)$$
is an equivalence for any algebraic stack $\CZ$. 

\medskip

This implies that $\Mir_{\Bun_G}$ admits an inverse \emph{as a functor defined by a kernel}.
Hence, it makes sense to consider
$$(\on{Id}_\CZ\boxtimes \Mir_{\Bun_G}^{-1}):\Shv(\CZ\times \Bun_G)\to \Shv(\CZ\times \Bun_G)_{\on{co}_2},$$
for any algebraic stack $\CZ$.  In other words, $\Bun_G$ is \emph{miraculous} in
the terminology of \secref{sss:Mir stacks non qc}. 

\sssec{} \label{sss:D Mir BunG}

Following \secref{sss:Mir stacks non qc bis}, let us denote by $\on{ev}^{\on{Mir}}_{\Bun_G}$
the composite functor
\begin{equation} \label{e:Mir BunG}
\Shv(\Bun_G)\otimes \Shv(\Bun_G) \overset{\on{Id}\otimes \Mir_{\Bun_G}^{-1}}\longrightarrow 
\Shv(\Bun_G)\otimes \Shv(\Bun_G)_{\on{co}}\overset{\on{ev} _{\Bun_G}}\longrightarrow \Vect,
\end{equation} 

Combining Sects. \ref{sss:BunG Mir} and \ref{sss:Verdier on BunG}, we obtain that the functor
\eqref{e:Mir BunG} is the counit of a self-duality on $\Shv(\Bun_G)$.

\medskip

We refer to it as the \emph{miraculous self-duality} of $\Shv(\Bun_G)$.

\sssec{} \label{sss:Mir Nilp}

We claim: 

\begin{prop} \label{p:Mir inv Nilp}
The functors $\Mir_{\Bun_G}$ and $\Mir_{\Bun_G}^{-1}$ send the subcategories 
$$\Shv_\Nilp(\Bun_G)_{\on{co}} \subset \Shv(\Bun_G)_{\on{co}}  \text{ and } \Shv_\Nilp(\Bun_G) \subset \Shv(\Bun_G)$$
to one another.
\end{prop}

\begin{proof}

Follows from the fact that the functor $\Mir_{\Bun_G}$ intertwines the functors $\sP$ and
$\sP_{\on{co}}$ (see \corref{c:Mir and P}), combined with \thmref{t:projector} (for $\CZ=\on{pt}$)
and \corref{c:projector co}.

\end{proof} 

\begin{cor} \label{c:Mir Nilp a priori}
The functor $\Mir_{\Bun_G}$ induces an equivalence
$$\Shv_\Nilp(\Bun_G)_{\on{co}}\to \Shv_\Nilp(\Bun_G).$$
\end{cor}

\sssec{} \label{sss:Mir duality Nilp}

Combining Corollaries \ref{c:Mir Nilp a priori} with \ref{c:duality with co BunG Nilp}, we obtain:

\begin{cor} \label{c:Mir duality Nilp}
The composition
$$\Shv_\Nilp(\Bun_G)\otimes \Shv_\Nilp(\Bun_G) 
\hookrightarrow \Shv(\Bun_G)\otimes \Shv(\Bun_G)\overset{\on{ev}^{\on{Mir}}_{\Bun_G}}\longrightarrow \Vect$$
is the counit of a self-duality on $\Shv_\Nilp(\Bun_G)$. 
\end{cor}

We refer to the self-duality of $\Shv_\Nilp(\Bun_G)$ established in \corref{c:Mir duality Nilp} as
the \emph{miraculous self-duality} of $\Shv_\Nilp(\Bun_G)$. 

\sssec{} \label{sss:U Mir}

The fact that $\Bun_G$ is miraculous formally implies that any cotruncative 
quasi-compact open substack $\CU\subset \Bun_G$ is also miraculous, see 
\secref{sss:Mir inherited}. 

\medskip

In particular, the functor 
$$\Mir_\CU:\Shv(\CU)\to \Shv(\CU)$$
is invertible also as a functor defined by a kernel, and so we can consider the functor 
$$(\on{Id}_\CZ\boxtimes \Mir_\CU^{-1}):\Shv(\CZ\times \CU)\to \Shv(\CZ\times \CU),$$
for any $\CZ$. 

\medskip

Similarly to \secref{sss:D Mir BunG}, we obtain that the functor, denoted $\on{ev}^{\on{Mir}}_\CU$,
$$\Shv(\CU)\otimes \Shv(\CU) \overset{\on{Id}\otimes \Mir_{\CU}^{-1}}\longrightarrow 
\Shv(\CU)\otimes \Shv(\CU)\overset{\on{ev} _\CU}\longrightarrow \Vect$$
is the counit of a self-duality on $\Shv(\CU)$.

\sssec{}

We claim:

\begin{prop} \label{p:Mir inv Nilp U}
For a quasi-compact Nilp-cotruncative $\CU\overset{j}\hookrightarrow \Bun_G$, the functor
 $\Mir_\CU^{-1}$ sends $\Shv_\Nilp(\CU)$ to $\Shv_\Nilp(\CU)$.
\end{prop}

\begin{proof}

Follows from \propref{p:Mir inv Nilp} in the same way as \propref{p:Mir Nilp U}
follows from \propref{p:Nilp Mir compat}.

\end{proof}

Finally, similarly to \corref{c:Mir duality Nilp}, we obtain:
\begin{cor} \label{c:Mir duality Nilp U}
The composition
$$\Shv_\Nilp(\CU)\otimes \Shv_\Nilp(\CU) \hookrightarrow 
\Shv(\CU)\otimes \Shv(\CU)\overset{\on{ev}^{\on{Mir}}_{\CU}}\longrightarrow \Vect$$
is the counit of a self-duality on $\Shv_\Nilp(\CU)$. 
\end{cor}

We refer to the self-duality of $\Shv_\Nilp(\CU)$ established in \corref{c:Mir duality Nilp} as
the \emph{miraculous self-duality} of $\Shv_\Nilp(\CU)$.

\section{The non-standard duality for $\Shv_\Nilp(\Bun_G)$} \label{s:pairing}

In this section we will prove what can be considered as the main result of this paper, \thmref{t:non-standard duality}.
It states that the category $\Shv_\Nilp(\Bun_G)$ is self-dual via a certain non-standard procedure. 

\medskip

The proof of this theorem will be rather straightforward, modulo the work done in \cite{AGKRRV}. 

\medskip

We will then connect this non-standard duality with the miraculous self-duality of \secref{sss:Mir duality Nilp}. 

\ssec{The projection of the pseudo-diagonal object}

In this subsection we study the object obtained by applying the projector $\sP=\sP_{\Bun_G,\Nilp}$ to the
\emph{pseudo-diagonal} object on $\Bun_G$. 

\sssec{}

Let $\sP_{\Bun_G,\Nilp}$ is as in \secref{sss:the projector}. I.e., this is the idempotent on $\Shv(\Bun_G)$ that corresponds to the
precomposition of the fully faithful embedding
$$\Shv_\Nilp(\Bun_G)\hookrightarrow \Shv(\Bun_G)$$
with its left inverse. 

\medskip

Recall that $\sP_{\Bun_G,\Nilp}$ is a functor defined by a kernel. In particular, it makes sense to consider the
endofunctors of $\Shv(\Bun\times\Bun_G)$ given by
$$\sP_{\Bun_G,\Nilp}\boxtimes \on{Id}_{\Bun_G},\,\, \sP_{\Bun_G,\Nilp}\boxtimes \sP_{\Bun_G,\Nilp} \text{ and } 
\on{Id}_{\Bun_G}\boxtimes \sP_{\Bun_G,\Nilp}.$$
%see \eqref{e:order does not matter}. 

\sssec{} \label{sss:ps u Bun}

Let
$$\on{ps-u}_{\Bun_G,\Nilp}\in \Shv(\Bun_G\times \Bun_G)$$
denote the object 
$$(\sP_{\Bun_G,\Nilp}\boxtimes \sP_{\Bun_G,\Nilp})(\on{ps-u}_{\Bun_G}),$$
where 
$$\on{ps-u}_{\Bun_G}:=(\Delta_{\Bun_G})_!(\ul\sfe_{\Bun_G}),$$
see \secref{sss:ps-u}. 

\medskip

Note that by \corref{c:double projector}, the functor $\sP_{\Bun_G,\Nilp}\boxtimes \sP_{\Bun_G,\Nilp}$ is the idempotent corresponding to the
embedding
$$\Shv_\Nilp(\Bun_G) \otimes \Shv_\Nilp(\Bun_G)\to \Shv(\Bun_G\times \Bun_G).$$

\medskip

Hence, we can think of $\on{ps-u}_{\Bun_G,\Nilp}$ as an object of
$$\Shv_\Nilp(\Bun_G) \otimes \Shv_\Nilp(\Bun_G).$$

\sssec{}

We claim:

\begin{prop} \label{p:Nilp !-diag}
We have canonical isomorphisms 
\begin{multline*} 
(\sP_{\Bun_G,\Nilp}\otimes \on{Id}_{\Bun_G})(\on{ps-u}_{\Bun_G})\simeq 
(\sP_{\Bun_G,\Nilp}\otimes \sP_{\Bun_G,\Nilp})(\on{ps-u}_{\Bun_G})\simeq \\
\simeq (\on{Id}_{\Bun_G}\otimes \sP_{\Bun_G,\Nilp})(\on{ps-u}_{\Bun_G}).
\end{multline*} 
\end{prop} 

\begin{proof}

Repeats the proof of \propref{p:proj identity}, taking into account \propref{p:Hecke codefined}.

\end{proof}

\ssec{The main result}

In this subsection we state and prove the main result of this paper, \thmref{t:non-standard duality}. 

\sssec{}

Let $\on{ev}^l_{\Bun_G}$ denote the functor
$$\Shv(\Bun_G)\otimes \Shv(\Bun_G)\to \Vect, \quad
\CF_1,\CF_2\mapsto \on{C}^\cdot_c({\Bun_G},\CF_1\overset{*}\otimes \CF_2).$$

The main result of this paper is:

\begin{thm} \label{t:non-standard duality} 
The functors 
$$\Shv_{\Nilp}(\Bun_G)\otimes \Shv_{\Nilp}(\Bun_G)\hookrightarrow 
\Shv(\Bun_G)\otimes \Shv(\Bun_G)\overset{\on{ev}^l_{\Bun_G}}\longrightarrow \Vect$$
and 
$$\on{ps-u}_{{\Bun_G},\Nilp}\in \Shv_{\Nilp}({\Bun_G})\otimes \Shv_{\Nilp}({\Bun_G})$$
define a duality datum. 
\end{thm}

\begin{rem}
As we shall see, the proof of \thmref{t:non-standard duality} will be essentially 
a formal manipulation (modulo \propref{p:Nilp !-diag}),
given the highly non-trivial fact that $\on{ps-u}_{\Bun_G,\Nilp}$ actually belongs to 
$\Shv_\Nilp(\Bun_G) \otimes \Shv_\Nilp(\Bun_G)$. The latter result uses the full strength of
the main results of \cite{AGKRRV}.

\end{rem}

\begin{rem}

The peculiarity of \thmref{t:non-standard duality} is that is straddles two different paradigms in which duality takes
place:

\medskip

In the constructible sheaf-theoretic contexts \emph{and} for D-modules, for a scheme $\CY$
we have the perfect pairing
\begin{equation} \label{e:usual pairing Y}
\on{ev} _\CY:\Shv(\CY)\otimes \Shv(\CY) \to \Vect, \quad \CF_1,\CF_2\mapsto \on{C}^\cdot(\CY,\CF_1\sotimes \CF_2),
\end{equation} 
where the corresponding contravariant equivalence 
$$(\Shv(\CY)^c)^{\on{op}}\to \Shv(\CY)^c$$
is given by Verdier duality. (The same is valid when $\CY$ is a Verdier-compatible 
quasi-compact algebraic stack,
%\footnote{Under a mild assumption, see \cite[Sect. C.3]{AGKRRV}}
but we need to replace $\on{C}^\cdot(\CY,-)$ by $\on{C}^\cdot_\blacktriangle(\CY,-)$.)

\medskip

In the Betti context, for the category $\Shv^{\on{all}}(\CY)$ of all sheaves of $\sfe$-vector spaces (not necessarily constructible ones), we have
a perfect pairing
\begin{equation} \label{e:Betti pairing Y}
\on{ev}^l_\CY:\Shv^{\on{all}}(\CY)\otimes \Shv^{\on{all}}(\CY) \to \Vect, \quad \CF_1,\CF_2\mapsto \on{C}^\cdot_c(\CY,\CF_1\overset{*}\otimes \CF_2).
\end{equation} 

\medskip

The content of \thmref{t:non-standard duality} is that an analog of the latter pairing defines a self-duality on $\Shv_{\Nilp}({\Bun_G})$.
We will see that this non-standard duality actually coincides with the miraculous duality of \secref{sss:Mir duality Nilp}. 
\end{rem}

\begin{rem}

It is easy to see that if $\CY$ is a smooth and proper scheme, then the pairing \eqref{e:Betti pairing Y} does induce a perfect pairing 
$$\Shv_{\{0\}}(\CY)\otimes \Shv_{\{0\}}(\CY)\to \Vect,$$
where $\Shv_{\{0\}}(-)$ is the same as $\qLisse(-)$. 

\medskip

Further, in \secref{sss:ex smooth prel} we will see that if $\CY$ is a (quasi-compact) stack with finitely many
isomorphism classes of $k$-points, then \eqref{e:Betti pairing Y} \emph{is} a perfect perfect
$$\Shv(\CY)\otimes \Shv(\CY)\to \Vect.$$

Thus, the pair $(\Bun_G,\Nilp)$ behaves in a way analogous to both $(\CY,\{0\})$ for a smooth and proper $\CY$
and $(\CY,T^*(\CY))$ for $\CY$ with finitely many
isomorphism classes of $k$-points.

\medskip

In fact, this analogy is even stronger, see \secref{sss:ex smooth}.

\end{rem}

%\begin{rem} 
%
%One can ask the following general question: let $\CY$ be a (quasi-compact) algebraic stack and let
%$\CN\subset T^*(\CY)$ be a conical subset. Consider the corresponding subcategory $\Shv_\CN(\CY)\subset \Shv(\CY)$. 
%
%\medskip
%
%We will either assume that 
%pair $(\CY,\CN)$ is \emph{constraccessible} and \emph{renormalization adapted}
%(see \cite[Sect. C.5]{AGKRRV} for what this means), or we will work instead of $\Shv_\CN(\CY)$ with
%its variant $\Shv_\CN(\CY)^{\on{access}}$ (see \cite[Sect. C.5.3]{AGKRRV}). 
%
%\medskip
%
%The pairing \eqref{e:usual pairing Y}, restricted to $\Shv_\CN(\CY)$, gives rise to a perfect pairing
%\begin{equation} \label{e:usual pairing Y N}
%\Shv_\CN(Y)\otimes \Shv_\CN(\CY) \to \Vect, \quad \CF_1,\CF_2\mapsto \on{C}^\cdot_\blacktriangle(\CY,\CF_1\sotimes \CF_2)'
%\end{equation} 
%
%\medskip
%
%We can now ask when is 
%\begin{equation} \label{e:Betti pairing Y N}
%\Shv_\CN(Y)\otimes \Shv_\CN(\CY) \to \Vect, \quad \CF_1,\CF_2\mapsto \on{C}^\cdot_c(Y,\CF_1\overset{*}\otimes \CF_2)
%\end{equation} 
%a perfect pairing. 
%
%\medskip
%
%One case in which it is, is when $\CY=Y$ is a proper scheme, and $\CN=\{0\}$, so $\Shv_\CN(\CY)=\qLisse(Y)$. In this case,
%the two pairings \eqref{e:usual pairing Y N} and \eqref{e:Betti pairing Y N} differ by a cohomological shift.

\sssec{Proof of \thmref{t:non-standard duality}}

In order to prove \thmref{t:non-standard duality}, we need to show that for $\CF\in \Shv_{\Nilp}({\Bun_G})$, the object
\begin{multline} \label{e:pairing with unit 0}
(\on{Id}\otimes \on{ev}^l_{\Bun_G})(\on{ps-u}_{{\Bun_G},\Nilp}\otimes \CF)\simeq \\
\simeq (\on{Id}\otimes \on{C}^\cdot_c({\Bun_G},-))\circ (\on{Id}\otimes (\Delta_{{\Bun_G}})^*)(\on{ps-u}_{{\Bun_G},\Nilp}\boxtimes \CF)\in \Shv_{\Nilp}({\Bun_G})
\end{multline} 
identifies canonically with $\CF$, where in the right-hand side of the above formula we consider $\on{ps-u}_{{\Bun_G},\Nilp}\boxtimes \CF$
as an object of the category
$$\Shv(\Bun_G)\otimes \Shv(\Bun_G\times \Bun_G).$$

\medskip

Using \propref{p:Nilp !-diag}, we interpret $\on{ps-u}_{\Bun_G,\Nilp}$ as $(\sP_{\Bun_G,\Nilp}\boxtimes \on{Id}_{\Bun_G})(\on{ps-u}_{\Bun_G})$.
We now use the fact that the functor $\sP_{\Bun_G,\Nilp}$ is \emph{co}defined by a kernel, see \propref{p:Hecke codefined}.

\medskip

Let us view $\on{C}^\cdot_c({\Bun_G},-)$ and $(\Delta_{{\Bun_G}})^*$ also as functors codefined by kernels. 
Hence, we can rewrite
the right-hand side in \eqref{e:pairing with unit 0} as
$$(\on{Id}_{\Bun_G}\boxtimes \on{C}^\cdot_c({\Bun_G},-))\circ (\on{Id}_{\Bun_G}\boxtimes (\Delta_{{\Bun_G}})^*) \circ 
(\sP_{\Bun_G,\Nilp}\boxtimes \on{Id}_{\Bun_G}\boxtimes \on{Id}_{\Bun_G})(\on{ps-u}_{\Bun_G}\boxtimes \CF),$$
where we now view $\on{ps-u}_{\Bun_G}\boxtimes \CF$ 
as an object of $\Shv(\Bun_G\times \Bun_G\times \Bun_G)$. 

\medskip

We can further rewrite the above expression as
$$\sP_{\Bun_G,\Nilp} \circ (\on{Id}_{\Bun_G}\boxtimes \on{C}^\cdot_c({\Bun_G},-))\circ (\on{Id}_{\Bun_G}\boxtimes (\Delta_{{\Bun_G}})^*) 
(\on{ps-u}_{\Bun_G}\boxtimes \CF).$$

We note that we have, tautologically,
$$(\on{Id}_{\Bun_G}\boxtimes \on{C}^\cdot_c({\Bun_G},-))\circ (\on{Id}_{\Bun_G}\boxtimes (\Delta_{{\Bun_G}})^*) 
(\on{ps-u}_{\Bun_G}\boxtimes \CF)\simeq \CF.$$

Finally, we have
$$\sP_{\Bun_G,\Nilp}(\CF)\simeq \CF,$$
since $\CF\in \Shv_{\Nilp}({\Bun_G})$.

\qed[\thmref{t:non-standard duality}]

\ssec{Relation to the miraculous functor}

In this subsection, we will relate the non-standard duality of \thmref{t:non-standard duality} with the
miraculous duality of \secref{sss:Mir duality Nilp}. 

\sssec{}

On the one hand, according to \thmref{t:non-standard duality}, we have a canonical identification
\begin{equation} \label{e:non-standard duality again}
\Shv_{\Nilp}({\Bun_G}) \simeq \Shv_{\Nilp}({\Bun_G})^\vee.
\end{equation}

On the other hand, according to \corref{c:duality with co BunG Nilp}
we have a canonical identification 
\begin{equation} \label{e:standard duality again}
\Shv_{\Nilp}({\Bun_G})_{\on{co}} \simeq \Shv_{\Nilp}({\Bun_G})^\vee.
\end{equation}

\sssec{}

We claim:

\begin{thm} \label{t:miraculous and dualities}
The identifications \eqref{e:non-standard duality again} and \eqref{e:standard duality again}
are intertwined by the miraculous functor
$$\Mir_{\Bun_G}: \Shv_{\Nilp}({\Bun_G})_{\on{co}} \overset{\sim}\to \Shv_{\Nilp}({\Bun_G}).$$
\end{thm}

\begin{proof}

It suffices to show that the functor
$$\on{Id}_{\Shv_\Nilp(\Bun_G)}\otimes \Mir_{\Bun_G}:
\Shv_{\Nilp}({\Bun_G})\otimes \Shv_{\Nilp}({\Bun_G})_{\on{co}}\to \Shv_{\Nilp}({\Bun_G})\otimes \Shv_{\Nilp}({\Bun_G})$$
sends the unit of the \eqref{e:standard duality again} duality to the unit of the \eqref{e:non-standard duality again}
duality.

\medskip

By \propref{p:unit Verdier BunG}, the unit of the \eqref{e:standard duality again} duality is the object
$$\on{u}_{\Bun_G,\Nilp,\on{co}_2}\simeq (\sP\boxtimes \on{Id}_{\Bun_G})(\on{u}_{\Bun_G,\on{co}_2}),$$
and the unit of the duality \eqref{e:non-standard duality again} is the object
$$\on{ps-u}_{{\Bun_G},\Nilp} \simeq (\sP\boxtimes \on{Id}_{\Bun_G})(\on{ps-u}_{\Bun_G}).$$

Hence, it suffices to establish an isomorphism
$$(\on{Id}_{\Bun_G}\boxtimes \Mir_{\Bun_G})\circ 
(\sP\boxtimes \on{Id}_{\Bun_G})(\on{u}_{\Bun_G,\on{co}_2}) \simeq 
(\sP\boxtimes \on{Id}_{\Bun_G})(\on{ps-u}_{\Bun_G}).$$

The latter isomorphism follows from the tautological isomorphism
$$(\on{Id}_{\Bun_G}\boxtimes \Mir_{\Bun_G})(\on{u}_{\Bun_G,\on{co}_2}) \simeq \on{ps-u}_{\Bun_G}.$$

\end{proof}

\sssec{}

From \thmref{t:miraculous and dualities}, we obtain:

\begin{cor} \label{c:two pairings BunG}
The diagram 
$$
\CD 
\Shv_{\Nilp}(\Bun_G)\otimes \Shv_{\Nilp}(\Bun_G)_{\on{co}} @>>> 
\Shv(\Bun_G)\otimes \Shv(\Bun_G)_{\on{co}} @>{\on{ev} _{\Bun_G}}>>  \Vect \\
@V{\on{Id}\otimes \Mir_{\Bun_G}}VV & & @VV{\on{Id}}V \\
\Shv_{\Nilp}(\Bun_G)\otimes \Shv_{\Nilp}(\Bun_G) @>>> 
\Shv(\Bun_G)\otimes \Shv(\Bun_G) @>{\on{ev}^l_{\Bun_G}}>>  \Vect
\endCD
$$
commutes.
\end{cor}

\sssec{}

From \corref{c:two pairings BunG} we obtain: 

\begin{cor} \label{c:non-st vs mir}
The pairing 
$$\Shv_{\Nilp}(\Bun_G)\otimes \Shv_\Nilp(\Bun_G)  \hookrightarrow 
\Shv(\Bun_G)\otimes \Shv(\Bun_G)\overset{\on{ev}^l_{\Bun_G}}\longrightarrow \Vect$$
identifies canonically with 
$$\Shv_\Nilp(\Bun_G)\otimes \Shv_\Nilp(\Bun_G) 
\hookrightarrow \Shv(\Bun_G)\otimes \Shv(\Bun_G)\overset{\on{ev}^{\on{Mir}}_{\Bun_G}}\longrightarrow \Vect.$$
\end{cor}

In other words, this corollary says that the non-standard duality of \thmref{t:non-standard duality}
identifies canonically with the miraculous duality of \secref{sss:Mir duality Nilp}. 

\ssec{Pairings against sheaves with nilpotent singular support}

In this subsection we will amplify the assertion of \corref{c:non-st vs mir}.

\sssec{}

We claim:

\begin{thm} \label{t:two pairings amplified}
The pairings
$\on{ev}^l_{\Bun_G}$ and $\on{ev}^{\on{Mir}}_{\Bun_G}$ 
$$\Shv(\Bun_G)\otimes \Shv(\Bun_G) \to \Vect$$
agree on the subcategories 
$$\Shv_\Nilp(\Bun_G)\otimes \Shv(\Bun_G) \subset 
\Shv(\Bun_G)\otimes \Shv(\Bun_G) \supset \Shv(\Bun_G)\otimes \Shv_\Nilp(\Bun_G).$$
\end{thm}

\bigskip

\begin{rem} \label{r:one factor enough}
We can reformulate \thmref{t:two pairings amplified} as saying that the following diagrams commute: 
$$
\CD 
\Shv(\Bun_G)\otimes \Shv_{\Nilp}(\Bun_G)_{\on{co}} @>>> 
\Shv(\Bun_G)\otimes \Shv(\Bun_G)_{\on{co}} @>{\on{ev} _{\Bun_G}}>>  \Vect \\
@V{\on{Id}\otimes \Mir_{\Bun_G}}VV & & @VV{\on{Id}}V \\
\Shv(\Bun_G)\otimes \Shv_{\Nilp}(\Bun_G) @>>> 
\Shv(\Bun_G)\otimes \Shv(\Bun_G) @>{\on{ev}^l_{\Bun_G}}>>  \Vect.
\endCD
$$
and
$$
\CD 
\Shv_{\Nilp}(\Bun_G)\otimes \Shv(\Bun_G)_{\on{co}} @>>> 
\Shv(\Bun_G)\otimes \Shv(\Bun_G)_{\on{co}} @>{\on{ev} _{\Bun_G}}>>  \Vect \\
@V{\on{Id}\otimes \Mir_{\Bun_G}}VV & & @VV{\on{Id}}V \\
\Shv_{\Nilp}(\Bun_G)\otimes \Shv(\Bun_G) @>>> 
\Shv(\Bun_G)\otimes \Shv(\Bun_G) @>{\on{ev}^l_{\Bun_G}}>>  \Vect.
\endCD
$$

\medskip

I.e., unlike the commutative diagram in \corref{c:two pairings BunG}, 
we only need one of the factors to have singular support in $\Nilp$. 
\end{rem}

\sssec{}

The rest of this subsection is devoted to the proof of \thmref{t:two pairings amplified}. Since 
both pairings are swap-equivariant (see Sects. \ref{sss:D Mir sym} and \ref{sss:Mir stacks non qc bis}),
it suffices to show that the two pairings agree on 
$$\Shv_\Nilp(\Bun_G)\otimes \Shv(\Bun_G) \subset 
\Shv(\Bun_G)\otimes \Shv(\Bun_G).$$

\sssec{}

By \propref{p:dual of P}(i), for $\CF_1\in \Shv_\Nilp(\Bun_G)$ and 
$\CF_2\in  \Shv(\Bun_G)$, we have
$$\on{ev}^{\on{Mir}}_{\Bun_G}(\CF_1,\CF_2):=
\on{ev} _{\Bun_G}(\Mir^{-1}_{\Bun_G}(\CF_1),\CF_2) \simeq
\on{ev} _{\Bun_G}(\Mir^{-1}_{\Bun_G}(\CF_1),\sP(\CF_2)),$$ 
and by \corref{c:non-st vs mir}, the latter is canonically isomorphic to
$$\on{ev}^l_{\Bun_G}(\CF_1,\sP(\CF_2)).$$ 

Hence, in order to prove \thmref{t:two pairings amplified}, it suffices to show that for
$\CF_1\in \Shv_\Nilp(\Bun_G)$ and $\CF_2\in  \Shv(\Bun_G)$, we have a canonical
isomorphism
$$\on{ev}^l_{\Bun_G}(\CF_1,\CF_2) \simeq \on{ev}^l_{\Bun_G}(\CF_1,\sP(\CF_2)).$$

This would follow from the next assertion:

\begin{prop} \label{p:*-pairing and P}
For any $\CF_1,\CF_2\in  \Shv(\Bun_G)$, we have a canonical isomorphism
$$\on{ev}^l_{\Bun_G}(\sP(\CF_1),\CF_2) \simeq \on{ev}^l_{\Bun_G}(\CF_1,\sP(\CF_2)).$$
\end{prop}

\sssec{Proof of \propref{p:*-pairing and P}}

Recall that $\sP=\sH_\sR$ for $\sR$ as in \secref{sss:the projector}. Recall also that 
$\sR\simeq \sR^\tau$. Hence, \propref{p:*-pairing and P} follows from the next
general assertion:

\begin{lem} \label{l:l pair and Hecke}
For $\CF_1,\CF_2\in  \Shv(\Bun_G)$ and $\CV\in \Rep(\cG)_\Ran$, we have a canonical
isomorphism
$$\on{ev}^l_{\Bun_G}(\sH_{\CV}(\CF_1),\CF_2) \simeq \on{ev}^l_{\Bun_G}(\CF_1,\sH_{\CV^\tau}(\CF_2)).$$
\end{lem}

\begin{proof}

We can assume that $\CF_1,\CF_2,\CV$ are compact, so both sides in the lemma are finite-dimensional
vector spaces. Hence, it is sufficient to identify the corresponding dual vector spaces. 

\medskip

For $\CF'_1,\CF'_2\in  \Shv(\Bun_G)$, we have
$$\on{ev}^l_{\Bun_G}(\CF'_1,\CF'_2)^\vee:=
\on{C}^\cdot_c(\Bun_G,\CF'_1\overset{*}\otimes \CF'_2)^\vee\simeq
\CHom_{\Shv(\Bun_G\times \Bun_G)}(\CF'_1\boxtimes \CF'_2, \on{u}^{\on{naive}}_{\Bun_G}).$$

Recall (see \cite[Sect. 11.3]{AGKRRV}) that the category $\Rep(\cG)_\Ran$ is rigid. Let $\CV^\vee$ denote the
monoidal dual of $\CV$. 

\medskip

We have
$$\CHom_{\Shv(\Bun_G\times \Bun_G)}(\sH_\CV(\CF_1)\boxtimes \CF_2, \on{u}^{\on{naive}}_{\Bun_G})\simeq
\CHom_{\Shv(\Bun_G\times \Bun_G)}(\CF_1\boxtimes \CF_2, 
(\sH_{\CV^\vee}\boxtimes \on{Id}_{\Bun_G})(\on{u}^{\on{naive}}_{\Bun_G}))$$
and
$$\CHom_{\Shv(\Bun_G\times \Bun_G)}(\CF_1\boxtimes \sH_{\CV^\tau}(\CF_2), \on{u}^{\on{naive}}_{\Bun_G})
\simeq 
\CHom_{\Shv(\Bun_G\times \Bun_G)}(\CF_1\boxtimes \CF_2, 
(\on{Id}_{\Bun_G}\boxtimes \sH_{\CV^\tau{}^\vee})(\on{u}^{\on{naive}}_{\Bun_G})).$$

Note also that 
$$\CV^\tau{}^\vee\simeq \CV^\vee{}^\tau.$$

Hence, the assertion of the lemma follows from the isomorphism
$$(\sH_{\CV^\vee}\boxtimes \on{Id}_{\Bun_G})(\on{u}^{\on{naive}}_{\Bun_G}) \simeq
(\on{Id}_{\Bun_G}\boxtimes \sH_{\CV^\vee{}^\tau})(\on{u}^{\on{naive}}_{\Bun_G})$$
of \propref{p:tau and sigma}.

\end{proof}

\ssec{Pairings of cuspidal sheaves}

In this subsection we will discuss an application of \corref{c:two pairings BunG} to two ways to define 
pairings on cuspidal objects of $\Shv(\Bun_G)$.

\sssec{}

Let
$$\Shv(\Bun_G)_{\on{cusp}}\overset{\jmath}\hookrightarrow \Shv(\Bun_G)$$
be the full subcategory of \emph{cuspidal} objects, see \cite[Sect. 1.4]{DrGa2}. 

\medskip

Recall that every object
of $\Shv(\Bun_G)_{\on{cusp}}$ is a clean extension from a particular open substack
$\CU_0\subset \Bun_G$, whose intersection with every connected component of
$\Bun_G$ is quasi-compact (see \cite[Proposition 1.4.6]{DrGa2}). 
In particular, we have a canonical fully faithful embedding
$$\Shv(\Bun_G)_{\on{cusp}} \overset{\jmath_{\on{co}}}\hookrightarrow \Shv(\Bun_G)_{\on{co}}.$$

\medskip

Furthermore, in \cite[Theorem 2.2.7 and Corollary 3.3.2]{Ga1}, the following isomorphism is established:
\begin{equation} \label{e:Mir and cusp}
\Mir_{\Bun_G}\circ \jmath_{\on{co}}\simeq \jmath[-2(\dim(\Bun_G))-\dim(Z_G)],
\end{equation}
where $\dim(Z_G)$ is the center of $G$. 

\sssec{}

Combining \eqref{e:Mir and cusp} with \corref{c:two pairings BunG} we obtain

\begin{cor} \label{c:cuspidal}
Let 
$\CF_1,\CF_2\in \Shv_{\Nilp}(\Bun_G)$ be two objects with 
$$\CF_1=\jmath(\CF_{1,\on{cusp}}), \quad \CF_{1,\on{cusp}}\in \Shv(\Bun_G)_{\on{cusp}}.$$
Then there is a natural isomorphism
$$\on{C}^\cdot_c(\Bun_G,\CF_1\overset{*}\otimes \CF_2)[-2(\dim(\Bun_G))-\dim(Z_G)] \simeq 
\on{C}^\cdot_\blacktriangle(\Bun_G,\jmath_{\on{co}}(\CF_{1,\on{cusp}})\sotimes \CF_2).$$
\end{cor} 

\begin{rem}

\corref{c:cuspidal} justifies a degree of agnosticism that some authors used when considering
pairings on automorphic sheaves: do we want to use 
$$\on{C}^\cdot_c(\Bun_G,-\overset{*}\otimes -),$$
which is more natural in the Betti setting, or
$$\on{C}^\cdot_\blacktriangle(\Bun_G,-\sotimes -),$$
which is more natural in the de Rham setting? 

\medskip

Now, \corref{c:cuspidal} says that as long as one of the objects is cuspidal, the two give
the same result (up to a cohomological shift). 

\end{rem} 

\begin{rem}
Similarly to Remark \ref{r:one factor enough}, for the validity of \corref{c:cuspidal}, it is sufficient 
to require that only one of the objects $\CF_1,\CF_2$ have nilpotent singular support.
\end{rem}

\ssec{Restricting to quasi-compact substacks of $\Bun_G$}

In this subsection we will discuss a variant of \corref{c:two pairings BunG} for 
quasi-compact open substacks of $\Bun_G$. 

\sssec{} \label{sss:ps-u U Nilp}

Let $\CU\overset{j}\hookrightarrow \Bun_G$ be a Nilp-cotruncative quasi-compact open substack $\Bun_G$. 

\medskip

Set
$$\on{ps-u}_{\CU,\Nilp}:=(j\times j)^*(\on{ps-u}_{\Bun_G,\Nilp})\in \Shv(\CU\times \CU).$$

Since $\on{ps-u}_{\Bun_G,\Nilp}$ belongs to $\Shv_\Nilp(\Bun_G)\otimes \Shv_\Nilp(\Bun_G)$, we obtain that 
$\on{ps-u}_{\CU,\Nilp}$ belongs to the full subcategory
$$\Shv_\Nilp(\CU)\otimes \Shv_\Nilp(\CU)\subset \Shv(\CU\times \CU).$$

\sssec{}

Let 
$$\on{ev}^l_{\CU}:\Shv(\CU)\otimes \Shv(\CU)\to \Vect$$
denote the functor 
$$\CF_1,\CF_2\mapsto \on{C}^\cdot_c({\CU},\CF_1\overset{*}\otimes \CF_2).$$

\medskip

From \thmref{t:non-standard duality}, we obtain:

\begin{cor} \label{c:non-standard duality U} 
The functors
$$\Shv_{\Nilp}(\CU)\otimes \Shv_{\Nilp}(\CU) \to \Shv(\CU)\otimes \Shv(\CU)
\overset{\on{ev}^l_{\CU}}\longrightarrow \Vect$$
and 
$$\on{ps-u}_{\CU,\Nilp}\in \Shv_{\Nilp}(\CU)\otimes \Shv_{\Nilp}(\CU)$$
define a duality datum. 
\end{cor}

\begin{proof}

We need to show that for $\CF\in \Shv_{\Nilp}(\CU)$, 
$$(\on{ev}^l_\CU\otimes \on{Id})(\CF\otimes \on{ps-u}_{\CU,\Nilp})\simeq \CF.$$

We have
\begin{multline*}
(\on{ev}^l_\CU\otimes \on{Id})(\CF\otimes \on{ps-u}_{\CU,\Nilp})=
(\on{ev}^l_\CU\otimes \on{Id})(\CF\otimes (j^*\otimes j^*)(\on{ps-u}_{\Bun_G,\Nilp}))\simeq \\
\simeq 
(\on{ev}^l_{\Bun_G}\otimes \on{Id})(j_!(\CF)\otimes (\on{Id}\otimes j^*)(\on{ps-u}_{\Bun_G,\Nilp})) 
\simeq \\
\simeq j^*\left((\on{ev}^l_{\Bun_G}\otimes \on{Id})(j_!(\CF)\otimes \on{ps-u}_{\Bun_G,\Nilp})\right) 
\overset{\text{\thmref{t:non-standard duality}}}\simeq j^*\circ j_!(\CF)\simeq \CF.
\end{multline*}

\end{proof}

\sssec{}

Parallel to \thmref{t:miraculous and dualities}, we have:

\bigskip

\begin{cor} \label{c:two pairings U}
Suppose that $\CU$ is universally $\Nilp$-contruncative (in particular, the category $\Shv_\Nilp(\CU)$
is self-dual via Verdier duality). Then:

\smallskip

\noindent{\em(a)}
The two identifications
$$\Shv_\Nilp(\CU)\underset{\sim}{\overset{\sim}\rightrightarrows} \Shv_\Nilp(\CU)^\vee$$
of \propref{p:usual duality U} and \corref{c:non-standard duality U}, respectively,
are intertwined by the miraculous functor $\Mir_\CU$. 

\smallskip

\noindent{\em(b)} 
The diagram
$$
\CD 
\Shv_{\Nilp}(\CU)\otimes \Shv_{\Nilp}(\CU) @>>> \Shv(\CU) \otimes \Shv(\CU) @>{\on{ev} _\CU}>>  \Vect \\
@V{\on{Id}\otimes \Mir_{\CU}}VV & & @VV{\on{Id}}V \\
\Shv_{\Nilp}(\CU)\otimes \Shv_{\Nilp}(\CU)  @>>> \Shv(\CU)\otimes \Shv(\CU) @>{\on{ev}^l_\CU}>>  \Vect
\endCD
$$
commutes.

\smallskip

\noindent{\em(c)} The self-duality of point (a) identifies with the miraculous self-duality of
\corref{c:Mir duality Nilp U}. 

\end{cor}

\begin{proof}

Point (a) follows in the same way as \thmref{t:miraculous and dualities}. Point (b)
follows formally from point (a). Point (c) follows formally from point (b).

\end{proof}

%\begin{rem}
%A remark parallel to Remark \ref{r:one factor enough} applies in the situation of 
%\corref{c:two pairings U} as well, see \corref{c:unit F U}.
%\end{rem}

%\sssec{}
%
%Let us place ourselves in the context of \secref{sss:constraccess and P}. In particular, we will
%be assuming that the pairs $(\Bun_G,\Nilp)$ and $(\CU,\Nilp)$ are constraccessible. 
%
%\medskip
%
%In particular, isomorphism \eqref{e:j and P} implies that the object $\on{ps-u}_{\CU,\Nilp}$ is isomorphic
%to each of the following three objects 
%$$(\sP_{\CU,\Nilp}\otimes \on{Id}_\CU)(\on{ps-u}_\CU),\,\, 
%(\sP_{\CU,\Nilp}\otimes \sP_{\CU,\Nilp})(\on{ps-u}_\CU),\,\, (\on{Id}_\CU\otimes \sP_{\CU,\Nilp})(\on{ps-u}_\CU).$$

\sssec{}

Finally, we claim that we have the following analog of \thmref{t:two pairings amplified}: 

\begin{cor} \label{c:two pairings amplified U}
The pairings
$\on{ev}^l_{\CU}$ and $\on{ev}^{\on{Mir}}_{\CU}$ 
$$\Shv(\CU)\otimes \Shv(\CU) \to \Vect$$
agree on the subcategory  
$$\Shv_\Nilp(\CU)\otimes \Shv(\CU) \subset 
\Shv(\CU)\otimes \Shv(\CU).$$
\end{cor}

\begin{proof}

Consider objects $\CF_1\in \Shv_\Nilp(\CU)$ and $\CF_2\in \Shv(\CU)$. We have
\begin{multline*}
\on{ev}^l_{\CU}(\CF_1,\CF_2)\simeq 
\on{ev}^l_{\Bun_G}(j_!(\CF_1),j_!(\CF_2)) \overset{\text{\eqref{t:two pairings amplified}}} \simeq
\on{ev}_{\Bun_G}(j_!(\CF_1),\Mir^{-1}_{\Bun_G}\circ j_!(\CF_2)) \simeq \\
\simeq \on{ev}_{\Bun_G}(j_!(\CF_1),j_{*,\on{co}}\circ \Mir^{-1}_{\CU}(\CF_2)) \simeq \on{ev}_{\CU}(j^*\circ j_!(\CF_1),\Mir^{-1}_{\CU}(\CF_2))\simeq 
\on{ev}_{\CU}(\CF_1,\Mir^{-1}_{\CU}(\CF_2)),
\end{multline*}
as required. 

\end{proof}

\begin{rem}
As in Remark \ref{r:one factor enough}, we can reformulate \corref{c:two pairings amplified U} as saying that the following diagram commutes: 
$$
\CD 
\Shv_\Nilp(\CU)\otimes \Shv(\CU) @>>> 
\Shv(\CU)\otimes \Shv(\CU) @>{\on{ev} _{\CU}}>>  \Vect \\
@V{\on{Id}\otimes \Mir_{\CU}}VV & & @VV{\on{Id}}V \\
\Shv_\Nilp(\CU)\otimes \Shv(\CU) @>>> 
\Shv(\CU)\otimes \Shv(\CU) @>{\on{ev}^l_{\CU}}>>  \Vect.
\endCD
$$

\medskip

I.e., unlike the commutative diagram in \corref{c:two pairings U}(b), 
we only need one of the factors to have singular support in $\Nilp$. 
\end{rem}

\section{An intrinsic characterization of sheaves with nilpotent singular support} \label{s:adj}

The goal of this section is to show that the full subcategory
$\Shv_\Nilp(\Bun_G)^{\on{constr}}\subset \Shv(\Bun_G)$ can be characterized by an intrinsic 
categorical property (see \thmref{t:char of Nilp adj}): 

\medskip

This theorem says that a constructible object of $\Shv(\Bun_G)$ belongs to $\Shv_\Nilp(\Bun_G)$ if and only if
the functor
$$\CF'\mapsto \on{C}^\cdot_\blacktriangle(\Bun_G,\CF\sotimes \CF')$$
admits a right adjoint \emph{as a functor defined by a kernel}\footnote{In the case of D-modules, the above condition
is equivalent to the functor in question preserving compactness.}. 

\ssec{Comparison of two pairings, revisited}

In this subsection we revisit the isomorphism of \corref{c:two pairings U}(b) from the point of view
of the material in \secref{s:ker}.

\sssec{}

Recall that, according to \secref{sss:Mir and ev}, if $\CY$ is a quasi-compact algebraic stack, we have 
a natural transformation
\begin{equation} \label{e:ident pair 2 gen}
\on{ev}^l_\CY(\CF_1,\Mir_\CY(\CF_2)) \to \on{ev}_\CY(\CF_1,\CF_2)
\end{equation}
as functors
$$\Shv(\CY)\otimes \Shv(\CY)\to \Vect.$$

Using \secref{sss:Mir and j}, this natural transformation automatically extends to the case when $\CY$ is not necessarily
quasi-compact, when we view both sides as functors
$$\Shv(\CY)\otimes \Shv(\CY)_{\on{co}}\to \Vect.$$

\medskip

In particular, we obtain a natural transformation
\begin{equation} \label{e:ident pair 2}
\on{ev}^l_{\Bun_G}(\CF_1,\Mir_{\Bun_G}(\CF_2)) \to \on{ev}_{\Bun_G}(\CF_1,\CF_2), \quad \CF_1\in \Shv(\Bun_G),\, \CF_2\in \Shv(\Bun_G)_{\on{co}}.
\end{equation}

\sssec{}

Recall now that according to \corref{c:two pairings BunG}, we have a canonical isomorphism
\begin{equation} \label{e:ident pair 1}
\on{C}^\cdot_c(\Bun_G,\CF_1\overset{*}\otimes \Mir_\CU(\CF_2)) \simeq 
\on{C}^\cdot_\blacktriangle(\Bun_G,\CF_1\sotimes \CF_2), \quad 
\CF_1\in  \Shv_\Nilp(\Bun_G),\,\, \CF_2\in \Shv_\Nilp(\Bun_G)_{\on{co}}.
\end{equation}

\medskip
 
The goal of this subsection is to prove the following assertion:

\begin{prop} \label{p:ident pair 1 2}
For $\CF_1\in  \Shv_\Nilp(\Bun_G)$ and $\CF_2\in \Shv_\Nilp(\Bun_G)_{\on{co}}$, the 
map \eqref{e:ident pair 2} identifies with the
isomorphism  \eqref{e:ident pair 1}.
\end{prop}

\bigskip
 
Note that \propref{p:ident pair 1 2} formally implies:

\bigskip

\begin{cor} \label{c:ident pair 1 2}
Let $\CU$ be a universally $\Nilp$-cotruncative quasi-compact open substack of $\Bun_G$. 
Then for $\CF_1,\CF_2\in \Shv_\Nilp(\CU)$, the map 
\begin{equation} \label{e:ident pair 2 U}
\on{ev}^l_\CU(\CF_1,\Mir_\CY(\CF_2)) \to \on{ev}_\CU(\CF_1,\CF_2)
\end{equation}
of \eqref{e:ident pair 2 gen} identifies with the
isomorphism of \corref{c:two pairings U}(b).
\end{cor}

\bigskip
 
The rest of this subsection is devoted to the proof of \propref{p:ident pair 1 2}.

\sssec{}

According to \secref{sss:Mir and ev}, the map \eqref{e:ident pair 2} is given by 
\begin{multline*}
\on{ev}^l_{\Bun_G} \circ 
(\on{Id}_{\Bun_G\times \Bun_G}\boxtimes \on{ev}_{\Bun_G})(\CF_1\boxtimes \on{ps-u}_{\Bun_G}\boxtimes \CF_2)\to \\
\to \on{ev}_{\Bun_G} \circ
(\on{ev}^l_{\Bun_G}\boxtimes \on{Id}_{\Bun_G\times \Bun_G})(\CF_1\boxtimes \on{ps-u}_{\Bun_G}\boxtimes \CF_2)
\end{multline*}
where
$$\on{ev}^l_{\Bun_G}:=\on{C}^\cdot_c(\Bun_G,-)\circ \Delta^*_{\Bun_G} \text{ and }
\on{ev}_{\Bun_G}:=\on{C}^\cdot_\blacktriangle(\Bun_G,-)\circ \Delta^!_{\Bun_G},$$
as functors 
$$\Shv(\Bun_G\times \Bun_G)\to \Shv(\on{pt}) \text{ and }
\Shv(\Bun_G\times \Bun_G)_{\on{co}_2}\to \Shv(\on{pt}),$$
codefined and defined by kernels, respectively. 

\medskip

Consider now the corresponding map
\begin{multline*}
\on{ev}^l_{\Bun_G}\circ 
(\on{Id}_{\Bun_G\times \Bun_G}\boxtimes \on{ev}_{\Bun_G})(\CF_1\boxtimes \on{ps-u}_{\Bun_G,\Nilp}\boxtimes \CF_2)\to \\
\to \on{ev}_{\Bun_G} \circ
(\on{ev}^l_{\Bun_G}\boxtimes \on{Id}_{\Bun_G\times \Bun_G})(\CF_1\boxtimes \on{ps-u}_{\Bun_G,\Nilp}\boxtimes \CF_2).
\end{multline*}

We claim that we have a commutative diagram

\bigskip

\begin{equation} \label{e:ev com diag Nilp}
\xy
(0,0)*+{\on{ev}^l_{\Bun_G}\circ 
(\on{Id}_{\Bun_G\times \Bun_G}\boxtimes \on{ev}_{\Bun_G})(\CF_1\boxtimes \on{ps-u}_{\Bun_G}\boxtimes \CF_2)}="A";
(50,-20)*+{\on{ev}_{\Bun_G} \circ
(\on{ev}^l_{\Bun_G}\boxtimes \on{Id}_{\Bun_G\times \Bun_G})(\CF_1\boxtimes \on{ps-u}_{\Bun_G}\boxtimes \CF_2) }="B";
(0,-50)*+{\on{ev}^l_{\Bun_G}\circ 
(\on{Id}_{\Bun_G\times \Bun_G}\boxtimes \on{ev}_{\Bun_G})(\CF_1\boxtimes \on{ps-u}_{\Bun_G,\Nilp}\boxtimes \CF_2)}="C";
(50,-80)*+{\on{ev}_{\Bun_G} \circ
(\on{ev}^l_{\Bun_G}\boxtimes \on{Id}_{\Bun_G\times \Bun_G})(\CF_1\boxtimes \on{ps-u}_{\Bun_G,\Nilp}\boxtimes \CF_2)}="D";
{\ar@{->}^{\text{\eqref{e:left to right}}} "A";"B"};
{\ar@{->}_{\text{\eqref{e:left to right}}} "C";"D"};
{\ar@{->}^{\sim} "C";"A"};
{\ar@{->}_{\sim} "D";"B"};
\endxy
\end{equation}
%
%\CD
%\on{ev}^l_{\Bun_G}\circ 
%(\on{Id}_{\Bun_G\times \Bun_G}\boxtimes \on{ev}_{\Bun_G})(\CF_1\boxtimes \on{ps-u}_{\Bun_G}\boxtimes \CF_2) @>>>
%\on{ev}_{\Bun_G} \circ
%(\on{ev}^l_{\Bun_G}\boxtimes \on{Id}_{\Bun_G\times \Bun_G})(\CF_1\boxtimes \on{ps-u}_{\Bun_G}\boxtimes \CF_2) \\
%@AAA @AAA \\
%\on{ev}^l_{\Bun_G}\circ 
%(\on{Id}_{\Bun_G\times \Bun_G}\boxtimes \on{ev}_{\Bun_G})(\CF_1\boxtimes \on{ps-u}_{\Bun_G,\Nilp}\boxtimes \CF_2) @>>>
%\on{ev}_{\Bun_G} \circ
%(\on{ev}^l_{\Bun_G}\boxtimes \on{Id}_{\Bun_G\times \Bun_G})(\CF_1\boxtimes \on{ps-u}_{\Bun_G,\Nilp}\boxtimes \CF_2)
%\endCD
with vertical arrows being isomorphisms.

\medskip

The existence of this commutative diagram implies the assertion of \propref{p:ident pair 1 2}. Indeed, unwinding
the definitions, we obtain that slanted bottom arrow in \eqref{e:ev com diag Nilp} fits into the commutative diagram
$$
\xy
(0,-40)*+{\on{C}^\cdot_c(\Bun_G,\CF_1\overset{*}\otimes \Mir_{\Bun_G}(\CF_2))}="A";
(50,-70)*+{\on{C}^\cdot_\blacktriangle(\Bun_G,\CF_1\sotimes \CF_2).}="B";
(0,0)*+{\on{ev}^l_{\Bun_G}\circ 
(\on{Id}_{\Bun_G\times \Bun_G}\boxtimes \on{ev}_{\Bun_G})(\CF_1\boxtimes \on{ps-u}_{\Bun_G,\Nilp}\boxtimes \CF_2)}="C";
(50,-30)*+{\on{ev}_{\Bun_G} \circ
(\on{ev}^l_{\Bun_G}\boxtimes \on{Id}_{\Bun_G\times \Bun_G})(\CF_1\boxtimes \on{ps-u}_{\Bun_G,\Nilp}\boxtimes \CF_2)}="D";
{\ar@{->}^{\text{\eqref{e:ident pair 1}}} "A";"B"};
{\ar@{->}_{\text{\eqref{e:left to right}}} "C";"D"};
{\ar@{->}^{\sim} "C";"A"};
{\ar@{->}^{\sim} "D";"B"};
\endxy
$$
%
%\CD
%(\on{Id}_{\Bun_G\times \Bun_G}\boxtimes \on{ev}_{\Bun_G})(\CF_1\boxtimes \on{ps-u}_{\Bun_G,\Nilp}\boxtimes \CF_2) @>>>
%\on{ev}_{\Bun_G} \circ
%(\on{ev}^l_{\Bun_G}\boxtimes \on{Id}_{\Bun_G\times \Bun_G})(\CF_1\boxtimes \on{ps-u}_{\Bun_G,\Nilp}\boxtimes \CF_2) \\
%@V{\sim}VV @VV{\sim}V \\
%\on{C}^\cdot_c(\Bun_G,\CF_1\overset{*}\otimes \Mir_{\Bun_G}(\CF_2)) @>{\text{\eqref{e:ident pair 1}}}>{\sim}> 
%\on{C}^\cdot_\blacktriangle(\Bun_G,\CF_1\sotimes \CF_2). 
%\endCD
%$$

\sssec{}

Thus, it remains to establish the existence of \eqref{e:ev com diag Nilp}. 

\medskip

Let $\CF$ be an arbitrary object of $\Shv(\Bun_G\times \Bun_G)$, and let 
$\CV_1,\CV_2$ be two objects of $\Rep(\cG)_\Ran$. Since the Hecke functors
$\sH_{\CV_1}$ and $\sH_{\CV_2}$ are defined and codefined by kernels, 
we have a commutative diagram

\smallskip

\begin{equation} \label{e:ev com gen}
\xy
(0,-60)*+{\on{ev}^l_{\Bun_G}\circ 
(\on{Id}_{\Bun_G\times \Bun_G}\boxtimes \on{ev}_{\Bun_G})(\CF_1\boxtimes (\sH_{\CV_1}\boxtimes \sH_{\CV_2})(\CF) \boxtimes \CF_2)}="A";
(55,-100)*+{\on{ev}_{\Bun_G} \circ
(\on{ev}^l_{\Bun_G}\boxtimes \on{Id}_{\Bun_G\times \Bun_G})(\CF_1\boxtimes (\sH_{\CV_1}\boxtimes \sH_{\CV_2})(\CF) \boxtimes \CF_2)}="B";
(0,0)*+{\on{ev}^l_{\Bun_G}\circ 
(\on{Id}_{\Bun_G\times \Bun_G}\boxtimes \on{ev}_{\Bun_G})(\sH_{\CV^\tau_1}(\CF_1)\boxtimes \CF \boxtimes \sH_{\CV^\tau_2,\on{co}}(\CF_2))}="C";
(55,-40)*+{\on{ev}_{\Bun_G} \circ
(\on{ev}^l_{\Bun_G}\boxtimes \on{Id}_{\Bun_G\times \Bun_G})(\sH_{\CV^\tau_1}(\CF_1)\boxtimes \CF \boxtimes \sH_{\CV^\tau_2,\on{co}}(\CF_2)).}="D";
{\ar@{->}^{\text{\eqref{e:left to right}}} "A";"B"};
{\ar@{->}_{\text{\eqref{e:left to right}}} "C";"D"};
{\ar@{->}^{\text{\lemref{l:l pair and Hecke}+\lemref{l:dual of Hecke}}}_{\sim} "A";"C"};
{\ar@{->}_{\text{\lemref{l:l pair and Hecke}+\lemref{l:dual of Hecke}}}^{\sim} "B";"D"};
\endxy
\end{equation}

Taking $\CF=\on{ps-u}_{\Bun_G}$ and $\CV_1=\CV_2=\sR$, we obtain that the terms in \eqref{e:ev com gen}
identify with the terms in \eqref{e:ev com diag Nilp}, establishing the existence of the latter diagram.

\qed[\propref{p:ident pair 1 2}] 

\begin{rem}

Note that the same argument shows that the map \eqref{e:ident pair 2} identifies with the isomorphism of
\thmref{t:two pairings amplified} when one of the objects $\CF_1$ and $\CF_2$ has nilpotent singular support. 

\medskip

The same remark applies when instead of $\Bun_G$ we take a universally contruncative open substack
$\CU\subset \Bun_G$.

\end{rem}

\ssec{Kernels defined by objects from $\Shv_\Nilp(\Bun_G)$} \label{ss:ker Nilp}

In this subsection we will combine \propref{p:ident pair 1 2} with \thmref{t:right adj crit} and deduce
that (constructible) objects from $\Shv_\Nilp(\Bun_G)$ admit right adjoints, when viewed as functors
defined by kernels.

\sssec{}

Let $\CU\subset \Bun_G$ be a universally $\Nilp$-contruncative quasi-compact open substack. We are going to prove: 

\begin{thm} \label{t:Nilp admits right adjoint U}
Let $\CF\in \Shv(\CU)$ be an object contained in $\Shv_\Nilp(\CU)^{\on{constr}}$. 
Then $\CF$ admits a right adjoint, viewed as a functor 
$$\Shv(\CU)\to \Vect,$$
defined by a kernel. 
\end{thm}

\begin{proof}

We will apply the criterion of \thmref{t:right adj crit} to $\CF$. (Note that the condition in \secref{sss:safety for adj}
holds automatically, since $\CY_2=\on{pt}$).

\medskip

By the equivalence  (i) $\Leftrightarrow$ (iii) in \thmref{t:right adj crit}, it suffices to show that for any 
$\CF'\in \Shv_\Nilp(\CU)$, the map
$$\on{ev}_\CU^l(\CF,\CF')\overset{\text{\eqref{e:left-to-right trans bis}}}\longrightarrow \on{ev}_\CU(\CF,\on{Id}^l_\CU(\CF'))$$
is an isomorphism. 

\medskip

Recall that the stack $\CU$ is miraculous and that the functor $\Mir_\CU$ defines a self-equivalence 
on $\Shv_\Nilp(\CU)$. Hence, we can assume that $\CF'$ is of the form $\Mir_\CU(\CF'')$ for
$\CF''\in \Shv_\Nilp(\CU)$.

\medskip

Consider the composition
$$\on{ev}^l_\CU(\CF,\CF')=\on{ev}^l_\CU(\CF,\Mir_\CU(\CF''))\overset{\text{\eqref{e:left-to-right trans bis}}}\longrightarrow
\on{ev}_\CU(\CF,\on{Id}^l_\CU \circ \Mir_\CU(\CF''))\overset{\text{\eqref{e:back to Q}}}\longrightarrow 
\on{ev}_\CU(\CF,\CF'').$$

By \secref{sss:two versions left-to-right trans}, this composite map identifies with the map \eqref{e:ident pair 2 U}, and hence
is an isomorphism by \corref{c:ident pair 1 2}.  Now, the map
$$\on{Id}^l_\CU \circ \Mir_\CU(\CF'')\to \CF''$$
is an isomorphism since $\CU$ is miraculous (see \secref{sss:Mir and Id l}). 

\medskip

This implies that the map 
$$\on{ev}^l_\CU(\CF,\Mir_\CU(\CF''))\overset{\text{\eqref{e:left-to-right trans bis}}}\longrightarrow
\on{ev}_\CU(\CF,\on{Id}^l_\CU \circ \Mir_\CU(\CF''))$$
is an isomorphism, as required.

\end{proof}

\sssec{}

Let $\CF$ be as in \thmref{t:Nilp admits right adjoint U}. Let us describe the right adjoint $\CF^R$ explicitly.
Namely, by \thmref{t:right adj crit}, we have
$$\CF^R\simeq \on{Id}^l_\CU \circ \BD^{\on{Verdier}}(\CF)\simeq \Mir_\CU^{-1}\circ \BD^{\on{Verdier}}(\CF),$$
which also identifies with 
$$\BD^{\on{Verdier}}\circ \Mir_\CU(\CF),$$
see Remark \ref{r:right adj when exists}. 

\medskip

The unit of the adjunction is the map 
\begin{equation} \label{e:unit adj U}
\on{u}_\CU \to \CF \boxtimes (\on{Id}^l_\CU \circ \BD^{\on{Verdier}}(\CF))
\end{equation}
obtained from the tautological map
$$\on{ps-u}_\CU \to \CF \boxtimes  \BD^{\on{Verdier}}(\CF)$$
by applying the functor $\on{Id}_\CU\boxtimes \on{Id}^l_\CU$. 

\medskip

The counit is the map
\begin{multline} \label{e:counit adj U}
\on{ev}_\CU(\CF, \Mir_\CU^{-1} \circ \BD^{\on{Verdier}}(\CF)) \overset{\text{\corref{c:two pairings U}(b)}}\simeq \\
\simeq \on{ev}^l_\CU(\CF, \Mir_\CU \circ \Mir_\CU^{-1} \circ \BD^{\on{Verdier}}(\CF))  =
\on{ev}^l_\CU(\CF, \BD^{\on{Verdier}}(\CF)) \to  \sfe,
\end{multline} 
where the last arrow is obtained by adjunction.

\sssec{}

We now claim: 

\begin{cor} \label{c:Nilp codefined U}
For any $\CF\in \Shv_{\Nilp}(\CU)$, the functor 
$$\sF: \Shv_{\Nilp}(\CU)\to \Vect, \quad \CF'\mapsto \on{C}^\cdot_\blacktriangle(\CU,\CF\sotimes \CF')$$
is defined and codefined by a kernel. The codefining object identifies canonically with
$$\CG:=\Mir_\CU (\CF)\in  \Shv_{\Nilp}(\CU) \subset \Shv(\CU).$$
\end{cor}

\begin{proof}

If $\CF$ is constructible, the assertion follows by combining Theorems \ref{t:Nilp admits right adjoint U} and 
\ref{t:right adj main}(a).

\medskip

We will now show how to reduce the assertion to the case when $\CF$ is constructible. 

\medskip

We need check that the map 
\begin{equation} \label{e:F and G on U}
(\on{Id}_\CZ\boxtimes \sG^l)(\CF') \to \on{Id}_\CZ\boxtimes \sF(\CF'),
\end{equation}
is an isomorphism for $\CF'\in \Shv(\CZ\times \CU)$.

\medskip

Note that both sides in \eqref{e:F and G on U} commute with colimits in $\CF$ and $\CF'$. 
Hence, we can assume that $\CF'$ is bounded above. 

\medskip

If $\CF'$ is bounded above, both sides of \eqref{e:F and G on U} maps isomorphically to the limits
over $n\in \BN$ of the corresponding expressions, when we replace $\CF$ by $\tau^{\geq -n}(\CF)$. 

\medskip

Thus, we can assume that $\CF$ is bounded below. In this case, $\CF$ is a colimit of
constructible objects in $\Shv_{\Nilp}(\CU)$. Hence, the assertion that \eqref{e:F and G on U}
is an isomorphism follows from the constructible case, by passing to colimits in the $\CF$
argument.  

\end{proof}

\sssec{}

We now consider the situation with all of $\Bun_G$. We claim:

\begin{thm} \label{t:Nilp admits right adjoint BunG}
Let $\CF\in \Shv(\Bun_G)$ be an object contained in $\Shv_\Nilp(\Bun_G)^{\on{constr}}$. 
Then it admits a right adjoint, when viewed as a functor
$$\Shv(\Bun_G)_{\on{co}}\to \Vect.$$
\end{thm}

\begin{proof}

Set 
$$\CF^R:=\Mir^{-1}_{\Bun_G}\circ \BD^{\on{Verdier}}(\CF)\in \Shv(\Bun_G)_{\on{co}},$$
where we view $\BD^{\on{Verdier}}$ as a functor
$$(\Shv(\Bun_G)^{\on{constr}})^{\on{op}}\to  \Shv(\Bun_G)^{\on{constr}}.$$

We will show that $\CF^R$ provides an adjoint of $\CF$. The unit of the adjunction is a map
\begin{equation} \label{e:unit BunG}
\on{u}_{\Bun_G,\on{co}_2}\to \CF \boxtimes \CF^R
\end{equation}
constructed as follows:

\medskip

For every universally $\Nilp$-contruncative quasi-compact open substack $\CU\overset{j}\hookrightarrow \Bun_G$, the 
$j^*\boxtimes j^?$ restriction of \eqref{e:unit BunG} is the map
\begin{multline*}
(j^*\boxtimes j^?)(\on{u}_{\Bun_G,\on{co}_2}) \simeq
(\on{Id}_{\Bun_G}\boxtimes j^?)\circ (j^*\boxtimes \on{Id}_{\Bun_G})(\on{u}_{\Bun_G,\on{co}_2}) \simeq
(\on{Id}_{\Bun_G}\boxtimes j^?)\circ (\on{Id}_\CU \boxtimes j_{*,\on{co}})(\on{u}_\CU) \simeq \\
\simeq (\on{Id}_{\Bun_G}\boxtimes (j^?\circ j_*))(\on{u}_\CU)  \simeq \on{u}_\CU \overset{\text{\eqref{e:unit adj U}}}\longrightarrow 
 j^*(\CF)\boxtimes (\Mir_\CU^{-1}\circ \BD^{\on{Verdier}}\circ j^*(\CF)) \simeq  \\
\simeq  j^*(\CF)\boxtimes (\Mir_\CU^{-1}\circ j^* \circ \BD^{\on{Verdier}}(\CF)) 
\simeq  j^*(\CF)\boxtimes (j^? \circ \Mir^{-1}_{\Bun_G} \circ \BD^{\on{Verdier}}(\CF))  \simeq 
 (j^*\boxtimes j^?)(\CF \boxtimes \CF^R).
 \end{multline*}

\medskip

The counit of the adjunction is a map
\begin{equation} \label{e:counit BunG}
\on{C}^\cdot_\blacktriangle(\Bun_G,\CF\sotimes \CF^R)\to \sfe
\end{equation}
is equal to the composition

\medskip

$$\on{C}^\cdot_\blacktriangle(\Bun_G,\CF\sotimes \CF^R) 
\overset{\text{\corref{c:two pairings BunG}}}\simeq 
\on{C}^\cdot_c(\Bun_G,\CF\overset{*}\otimes \BD^{\on{Verdier}}(\CF)) \to \sfe,$$
where the latter map is obtained by adjunction from the map
$$\CF\boxtimes \BD^{\on{Verdier}}(\CF)\to (\Delta_{\Bun_G})_*(\omega_{\Bun_G}).$$

\end{proof}

\begin{cor} \label{c:Nilp codefined BunG}
For any $\CF\in \Shv_{\Nilp}(\Bun_G)_{\on{co}}$, the functor 
$$\sF: \Shv_{\Nilp}(\Bun_G)\to \Vect, \quad \CF'\mapsto \on{C}^\cdot_\blacktriangle(\Bun_G,\CF\sotimes \CF')$$
is defined and codefined by a kernel. The codefining object identifies canonically with
$$\CG:=\Mir_{\Bun_G}(\CF)\in  \Shv_{\Nilp}(\Bun_G) \subset \Shv(\Bun_G).$$
\end{cor}

\begin{proof}

Repeats that of \corref{c:Nilp codefined U} (indeed, it is easy to see that 
\thmref{t:right adj main} extends from the case of quasi-compact algebraic stacks to truncatable ones).

\end{proof}

\ssec{Pairings against sheaves with nilpotent singular support, revisited}

\sssec{}

Let $\CU$ be a universally $\Nilp$-contruncative quasi-compact open substack of $\Bun_G$. 

\medskip

Let $\CY$ be an algebraic stack. Note that for $\CF\in \Shv(\CU)$ and $\CF'\in \Shv(\CY\times \CU)$, we have a canonically defined map
\begin{equation} \label{e:pairing Y U}
(p_\CY)_!(\CF'\overset{*}\otimes p_\CU^*(\Mir_\CU(\CF)))\to (p_\CY)_\blacktriangle(\CF' \sotimes p_\CU^!(\CF))
\end{equation}
as functors 
$$\Shv(\CY\times \CU)\to \Shv(\CY),$$
see \eqref{e:left-to-right trans}. 

\medskip

From \corref{c:Nilp codefined U} we obtain: 

\begin{cor} \label{c:pairing Y U}
The natural transformation \eqref{e:pairing Y U} is an isomorphism if $\CF\in \Shv_\Nilp(\CU)$.
\end{cor}

\begin{rem} \label{r:pairing Y U}
Recall that \corref{c:two pairings amplified U} says that we have \emph{a} canonical isomorphism
\begin{equation} \label{e:pairing Y U again}
(p_\CY)_!(\CF'\overset{*}\otimes p_\CU^*(\Mir_\CU(\CF)))\to (p_\CY)_\blacktriangle(\CF'\sotimes p_\CU^!(\CF)),
\end{equation}
when $\CY=\on{pt}$. In fact, the same proof shows that there exists an isomorphism as in \eqref{e:pairing Y U again}
for any $\CY$.

\medskip

The additional information provided by \corref{c:pairing Y U} is that this isomorphism comes from the generally
defined natural transformation \eqref{e:left-to-right trans}. 

\end{rem}

\sssec{}

%Note that in addition to the map \eqref{e:pairing Y U}, we also have a natural transformation
%\begin{equation} \label{e:pairing Y U bis}
%(p_\CY)_!((\on{Id}_\CY\boxtimes \Mir_\CU)(\CF')\overset{*}\otimes p_\CU^*(\CF))\to (p_\CY)_\blacktriangle(\CF'\sotimes p_\CU^!(\CF))
%\end{equation}

We now claim: 

\begin{cor} \label{c:pairing Y U bis}
There exists a canonical isomorphism
$$(p_\CY)_!((\on{Id}_\CY\boxtimes \Mir_\CU)(\CF')\overset{*}\otimes p_\CU^*(\CF))\simeq (p_\CY)_\blacktriangle(\CF'\sotimes p_\CU^!(\CF)),
\quad \CF'\in \Shv(\CY\times \CU),\, \CF\in \Shv_\Nilp(\CU).$$
\end{cor}

\begin{proof}

Write 
\begin{multline*}
(p_\CY)_!((\on{Id}_\CY\boxtimes \Mir_\CU)(\CF')\overset{*}\otimes p_\CU^*(\CF))=
(p_\CY)_!\left((\on{Id}_\CY\boxtimes \Mir_\CU)(\CF')\overset{*}\otimes p_\CU^*(\Mir_\CU(\Mir^{-1}_\CU(\CF)))\right)
\overset{\text{\corref{c:pairing Y U}}}\simeq \\
\simeq (p_\CY)_\blacktriangle((\on{Id}_\CY\boxtimes \Mir_\CU)(\CF')\overset{!}\otimes p_\CU^!(\Mir^{-1}_\CU(\CF))).
\end{multline*}

However, since the object $\on{ps-u}_{\CU}$ is swap-equivariant, the latter expression identifies with 
$$(p_\CY)_\blacktriangle\left(\CF' \sotimes p_\CU^!(\Mir^{-1}_\CU(\Mir_\CU(\CF)))\right)\simeq
(p_\CY)_\blacktriangle(\CF'\sotimes p_\CU^!(\CF)),$$
as required.

\end{proof}

\sssec{}

We now revisit the pairings on all of $\Bun_G$. Let $\CY$ be as above. For $\CF\in \Shv(\Bun_G)_{\on{co}}$ 
consider the natural transformation
\begin{equation} \label{e:pairing Y BunG}
(p_\CY)_!(\CF' \overset{*}\otimes p_{\Bun_G}^*(\Mir_{\Bun_G}(\CF))) \to 
(p_\CY)_\blacktriangle(\CF' \sotimes p_{\Bun_G}^!(\CF)),
\end{equation}
as functors $\Shv(\CY\times \Bun_G)\to \Shv(\CY)$, see \eqref{e:left-to-right trans}.

\medskip

From \corref{c:Nilp codefined BunG} we obtain: 

\begin{cor} \label{c:pairing Y BunG}
The natural transformation \eqref{e:pairing Y BunG} is an isomorphism if $\CF\in \Shv_\Nilp(\Bun_G)$.
\end{cor}

\sssec{}

Finally, we claim: 

\begin{cor} \label{c:pairing Y BunG bis}
For $\CF\in  \Shv_\Nilp(\Bun_G)$ and any algebraic stack $\CY$, there is a canonical isomorphism 
isomorphism of functors $\Shv(\CY\times \Bun_G)_{\on{co}}\to \Shv(\CY)$
\begin{equation} \label{e:pairing Y BunG bis}
(p_\CY)_!((\on{Id}_\CY\boxtimes \Mir_{\Bun_G})(\CF') \overset{*}\otimes p_{\Bun_G}^*(\CF)) \simeq
(p_\CY)_\blacktriangle(\CF' \sotimes p_{\Bun_G}^!(\CF)), \quad
\CF'\in \Shv(\CY\times \Bun_G).
\end{equation}
\end{cor} 

\begin{proof}

Follows from \corref{c:pairing Y BunG} in the same way as \corref{c:pairing Y U bis} follows from \corref{c:pairing Y U}. 

\end{proof}

\begin{rem}
Note that in the case when $\CY=\on{pt}$, 
the assertion of Corollaries \ref{c:pairing Y BunG bis} and \ref{c:pairing Y BunG}
has been already established in \thmref{t:two pairings amplified}. 

\medskip

Similarly to Remark \ref{r:pairing Y U}, the proof of \thmref{t:two pairings amplified}
can be extended to include the case of an arbitrary $\CY$. Furthermore, one can
show that the resulting isomorphisms identify with those of Corollaries \ref{c:pairing Y BunG bis} and \ref{c:pairing Y BunG}.

\end{rem}

\ssec{A refined version of \thmref{t:Hecke action Nilp 2}}

In this subsection we discuss an amplification of the categorical K\"unneth formula,
\thmref{t:Hecke action Nilp 2}. 

\sssec{}

In this subsection, we will use \thmref{t:Nilp admits right adjoint U} to prove the following result:

\begin{thm} \label{t:Kunneth fixed Nilp}
Let $\CZ$ be an algebraic stack, and let $\CN$ be a conical half-dimensional closed subset of $T^*(\CZ)$. Then the
functor
$$\Shv_\CN(\CZ)\otimes \Shv_\Nilp(\Bun_G)\to \Shv_{\CN\times \Nilp}(\CZ\times \Bun_G)$$
is an equivalence.
\end{thm}

By passing to the limit, the statement of \thmref{t:Kunneth fixed Nilp} is obtained from the following:

\begin{thm} \label{t:Kunneth fixed Nilp U}
Let $\CZ$ be an algebraic stack, and let $\CN$ be a half-dimensional closed 
subset of $T^*(\CZ)$. Then for 
any universally $\Nilp$-cotruncative open substack $\CU\subset \Bun_G$, the
functor
$$\Shv_\CN(\CZ)\otimes \Shv_\Nilp(\CU)\to \Shv_{\CN\times \Nilp}(\CZ\times \CU)$$
is an equivalence.
\end{thm}

We proceed to the proof of \thmref{t:Kunneth fixed Nilp U}. 

\sssec{}

Given \corref{c:Kunneth for cotrunc}, in order to prove \thmref{t:Kunneth fixed Nilp U}, we have to show
the following:

\medskip

Let $\CF'$ be an object in $\Shv(\CZ)\otimes  \Shv_\Nilp(\CU)$, whose image along the functor
\begin{equation} \label{e:boxtimes Z U}
\Shv(\CZ)\otimes  \Shv_\Nilp(\CU)\overset{\boxtimes}\to \Shv(\CZ\times \CU)
\end{equation}
is contained in $\Shv_{\CN\times \Nilp}(\CZ\times \CU)$. Then $\CF'\in \Shv_\CN(\CZ)\otimes  \Shv_\Nilp(\CU)$.

\medskip

To prove this, by \propref{p:usual duality U}, it suffices to show that for every $\CF\in  \Shv_\Nilp(\CU)$, the object
$$(\on{Id}\otimes \on{ev}_\CU)(\CF' \otimes \CF)\in \Shv(\CZ)$$
belongs to $\Shv_\CN(\CZ)$. 

\medskip

We rewrite
$$(\on{Id}\otimes \on{ev}_\CU)(\CF' \otimes \CF) 
\simeq (p_1)_\blacktriangle(\CF' \sotimes p_2^!(\CF)),$$
where in the right-hand side, we denote by the same character $\CF'$
the image of $\CF'$ along \eqref{e:boxtimes Z U}.

\medskip

Thus, it suffices to show that for $\CF'\in \Shv_{\CN\times T^*(\Bun_G)}(\CZ\times \CU)$ and 
$\CF\in \Shv_\Nilp(\CU)$, the object 
$$(\on{Id}_\CZ\boxtimes \sF)(\CF')\in \Shv(\CZ)$$
belongs to $\Shv_\CN(\CZ)$, where $\sF$ is the functor defined by $\CF$ as a kernel. 

\medskip

According to \corref{c:Nilp codefined U}, the functor $\sF$ is defined and codefined by a kernel. 
The required assertion follows now from \corref{c:preserve sing supp}. 

\qed[\thmref{t:Kunneth fixed Nilp U}]

\ssec{A converse statement}

In this subsection we will state an assertion that provides a converse to \corref{c:Nilp codefined BunG},
and as a result also to \thmref{t:Nilp admits right adjoint BunG}. 

\sssec{}

We claim:

\begin{thm} \label{t:char of Nilp codefined}
Let $\CF$ be an object of $\Shv(\Bun_G)_{\on{co}}$ such that the corresponding functor
\begin{equation} \label{e:left functor new}
\sF:=\on{C}^\cdot_\blacktriangle(\Bun_G,\CF\sotimes -), \quad \Shv(\Bun_G)\to \Vect
\end{equation}
is defined and codefined by a kernel. Then
$\CF\in \Shv_\Nilp(\Bun_G)_{\on{co}}$.
\end{thm}

This theorem will be proved in \secref{ss:proof of char of Nilp}. In the rest of this subsection we will derive some corollaries. 

\sssec{}

As a consequence, we obtain: 

\begin{thm} \label{t:char of Nilp adj}
Let $\CF$ be an object of $\Shv(\Bun_G)^{\on{constr}}$ such that the corresponding functor
$$\sF:=\on{C}^\cdot_\blacktriangle(\Bun_G,\CF\sotimes -), \quad \Shv(\Bun_G)_{\on{co}}\to \Vect$$
admits a right adjoint \emph{as a functor defined by a kernel}. Then
$\CF\in \Shv_\Nilp(\Bun_G)$.
\end{thm}

\begin{proof}

Consider the object 
$$\CG:=\Mir^{-1}_{\Bun_G}(\CF)\in \Shv(\Bun_G)_{\on{co}}.$$ 

The assumption on $\CF$ implies that $\CG$, viewed as a functor $\Shv(\Bun_G)\to \Vect$
defined by a kernel, also admits a right adjoint.
By \thmref{t:right adj main}(a) (see Remark \ref{r:right adj main non qc}), the resulting functor 
$$\sG:\Shv(\Bun_G)\to \Vect$$
is defined and codefined by a kernel. 

\medskip

Hence, from \thmref{t:char of Nilp codefined}, we obtain that $\CG\in \Shv_\Nilp(\Bun_G)_{\on{co}}$.
Hence, 
$$\CF=\Mir_{\Bun_G}(\CG)\in \Shv_\Nilp(\Bun_G).$$

\end{proof}

%\begin{rem}
%In the de Rham context, we will prove a slightly stronger assertion: namely, 
%we will not have to assume that $\CF$ is compact.
%\end{rem}

%\sssec{}
%
%First, we observe: 
%
%\begin{cor} \label{c:char of Nilp bis}
%Let $\CF$ be an object of $\Shv(\Bun_G)^c_{\on{co}}$ such that the corresponding functor
%\begin{equation} \label{e:left functor new co}
%\sF:=\on{C}^\cdot_\blacktriangle(\Bun_G,\CF\sotimes -), \quad \Shv(\Bun_G)\to \Vect
%\end{equation}
%admits a right adjoint \emph{as a functor defined by a kernel}. Then
%$\CF\in \Shv_\Nilp(\Bun_G)_{\on{co}}$.
%\end{cor}
%
%\begin{proof}
%
%Since $\Shv(\Bun_G)$ is miraculous, i.e., the functor $\Mir_{\Bun_G}$ has an inverse
%as a functor defined by a kernel, it suffices to show that if $\CF$ is such that the functor
%$$\on{C}^\cdot_\blacktriangle(\Bun_G,\CF\sotimes \Mir_{\Bun_G}(-)):\Shv(\Bun_G)_{\on{co}}\to \Vect$$
%admits a right adjoint as a functor defined by a kernel, then $\CF\in \Shv_\Nilp(\Bun_G)_{\on{co}}$.
%
%\medskip
%
%We rewrite the above expression as 
%$$\on{C}^\cdot_\blacktriangle(\Bun_G,\Mir_{\Bun_G}(\CF)\sotimes -):\Shv(\Bun_G)_{\on{co}}\to \Vect,$$
%and from \thmref{t:char of Nilp} we obtain that $\Mir_{\Bun_G}(\CF)\in \Shv_\Nilp(\Bun_G)$. Now the assertion follows from 
%\propref{p:Mir inv Nilp}.
%
%\end{proof}

\sssec{} 

Combining Theorems \ref{t:char of Nilp adj}, \ref{t:Nilp admits right adjoint BunG} and \ref{t:right adj crit}
(see Remark \ref{r:right adj crit non-qc}), we obtain:

\begin{cor} \label{c:equiv BunG}
For $\CF\in \Shv(\Bun_G)^{\on{constr}}$ the following conditions are equivalent:

\smallskip

\noindent{\em(i)} $\CF\in \Shv_\Nilp(\Bun_G)$;

\smallskip

\noindent{\em(ii)}
The functor $\sF$ admits \emph{a} right adjoint as a functor defined by a kernel.

\smallskip

\noindent{\em(iii)} The map 
$$\on{C}^\cdot_c(\Bun_G,\CF\overset{*}\otimes \BD(\CF)) \to 
\on{C}^\cdot_\blacktriangle(\Bun_G,\CF\sotimes (\Mir^{-1}_{\Bun_G}\circ \BD^{\on{Verdier}}(\CF)))$$ 
of \eqref{e:ident pair 2} is an isomorphism, where $\BD^{\on{Verdier}}$ is understood as a functor
$$(\Shv(\Bun_G)^{\on{constr}})^{\on{op}}\to \Shv(\Bun_G)^{\on{constr}}.$$

\end{cor}

\sssec{}

Let us momentarily place ourselves in the context of (non-necessarily holonomic) D-modules. I.e., $\Shv(-)$
will denote the category of ind-holonoimic D-modules, viewed as a subcategory of the category $\Dmod(-)$
of all D-modules. Note that the inclusion 
$$\Shv_\Nilp(\Bun_G)\subset \Dmod_\Nilp(\Bun_G)$$ 
is an equivalence, since $\Nilp\subset T^*(\Bun_G)$ is Lagrangian.

\medskip

In this case, the condition that the functor
$\sF$ admits a right adjoint as a functor defined by a kernel is equivalent to the condition that it admits a 
continuous right adjoint as a plain functor
$$\Dmod(\Bun_G)_{\on{co}}\to \Vect,$$
which is in turn equivalent to the condition that it sends compact objects to finite-dimensional vector spaces.

\sssec{}

We clam: 

\begin{thm} 
For $\CF\in \Shv(\Bun_G)^{\on{constr}}\subset \Dmod(\Bun_G)$ the following conditions are equivalent:

\smallskip

\noindent{\em(i)} $\CF\in \Shv_\Nilp(\Bun_G)$;

\smallskip

\noindent{\em(ii)} For any $\CF'\in \Dmod(\Bun_G)^c_{\on{co}}$, the object
$$\on{C}_\blacktriangle(\Bun_G,\CF\sotimes \CF')\in \Vect$$
is finite-dimensional. 

\smallskip

\noindent{\em(ii')} For any $\CF'\in \Dmod(\Bun_G)^c$, the object
$$\CHom_{\Dmod(\Bun_G)}(\CF',\CF)\in \Vect$$
is finite-dimensional. 

\smallskip

\noindent{\em(ii'')} For any $\CF'\in \Dmod(\Bun_G)^c_{\on{co}}$, the object
$$\CHom_{\Dmod(\Bun_G)}(\CF,\on{Id}^{\on{naive}}_{\Bun_G}(\CF'))\in \Vect$$
is finite-dimensional. 

\end{thm} 

\begin{proof}

The equivalence of (i) and (ii) follows from the combination of Theorems \ref{t:Nilp admits right adjoint BunG} and \ref{t:char of Nilp adj}. 
The equivalence of (ii) and (ii') is formal: for $\CF'\in \Dmod(\Bun_G)^c$
$$\CHom_{\Dmod(\Bun_G)}(\CF',\CF) \simeq \on{C}_\blacktriangle(\Bun_G,\CF\sotimes \BD^{\on{Verdier}}(\CF')),$$
where $\BD^{\on{Verdier}}$ is understood as an equivalence
$$(\Dmod(\Bun_G)^c)^{\on{op}}\to \Dmod(\Bun_G)^c_{\on{co}}.$$

The equivalence of (ii') and (ii'') follows by viewing $\BD^{\on{Verdier}}$ as an equivalence 
$$(\Shv(\Bun_G)^{\on{constr}})^{\on{op}}\simeq \Shv(\Bun_G)^{\on{constr}}.$$

\end{proof} 

\ssec{A version for a universally $\Nilp$-cotruncative quasi-compact open substack}

\sssec{}

Let $\CU$ be a universally $\Nilp$-cotruncative quasi-compact open substack of $\Bun_G$. 
As a formal corollary of \thmref{t:char of Nilp codefined}, we obtain:

\begin{cor} \label{c:char of Nilp U codefined}
Let $\CF$ be an object of $\Shv(\CU)$ such that the corresponding functor
$$\sF:=\on{C}^\cdot_\blacktriangle(\CU,\CF\sotimes -), \quad \Shv(\CU)\to \Vect$$
is defined and codefined by a kernel. Then
$\CF\in \Shv_\Nilp(\CU)$.
\end{cor}

From \thmref{t:char of Nilp adj}, we obtain:

\begin{cor} \label{c:char of Nilp U adj}
Let $\CF$ be an object of $\Shv(\CU)^{\on{constr}}$ such that the corresponding functor
$$\sF:=\on{C}^\cdot_\blacktriangle(\CU,\CF\sotimes -), \quad \Shv(\CU)\to \Vect$$
admits a right adjoint defined by a kernel. Then
$\CF\in \Shv_\Nilp(\CU)$.
\end{cor}

\sssec{}

Observe that by combining \corref{c:char of Nilp U adj} and Theorems \ref{t:Nilp admits right adjoint U} and \ref{t:right adj crit}, 
we obtain:

\begin{cor} \label{c:equiv U}
Let $\CU\subset \Bun_G$ be a universally $\Nilp$-cotruncative quasi-compact open substack.
Then for $\CF\in \Shv(\CU)^{\on{constr}}$ the following conditions are equivalent:

\smallskip

\noindent{\em(i)} $\CF\in \Shv_\Nilp(\CU)$;

\smallskip

\noindent{\em(ii)}
The functor $\sF$ admits \emph{a} right adjoint as a functor defined by a kernel.

\smallskip

\noindent{\em(iii)} The map 
$$\on{C}^\cdot_c(\CU,\CF\overset{*}\otimes \BD(\CF)) \to 
\on{C}^\cdot_\blacktriangle(\CU,\CF\sotimes (\Mir^{-1}_{\CU}\circ \BD(\CF)))$$ 
of \eqref{e:ident pair 2 gen} is an isomorphism. 

\end{cor}

\sssec{}

Let $\CY$ be an arbitrary quasi-compact algebraic stack. Let $\CF$ be an object of $\Shv(\CY)^{\on{constr}}$. 

\medskip

First, recall that according to \thmref{t:right adj crit}, conditions (ii) and (iii) in \corref{c:equiv U} are 
equivalent. 

\medskip

The statement of \corref{c:equiv U} suggests the following question: 

\begin{quest}  Under what conditions on $\CY$ does there exist a subset $\CN\subset T^*(\CY)$ so that 
conditions (ii) or/and (iii) as in \corref{c:equiv U} are equivalent to the condition that
$\CF\in \Shv_\CN(\CY)$?

\end{quest} 

\sssec{Example} \label{sss:ex smooth}

Let $\CY=Y$ be a proper smooth scheme. Then it is easy to see that the assertion of \corref{c:equiv U} 
holds for $\CN=\{0\}$.

\medskip

So, in some ways the subset $\Nilp\subset T^*(\Bun_G)$ plays the same role as the zero-section 
$\{0\}\subset T^*(Y)$ for a proper smooth scheme $Y$. 

\medskip

The other extreme case is when $\CY$ is an algebraic stack with finitely many isomorphism
classes of points. Then the assertion of \corref{c:equiv U} holds for $\CN=T^*(\CY)$. 

\medskip

Note that in both of the above examples, the pair $(\CY,\CN)$ is Serre; see Definition \ref{d:Serre}
for what this means.

\ssec{Proof of \thmref{t:char of Nilp codefined}} \label{ss:proof of char of Nilp}

\sssec{}

We will deduce \thmref{t:char of Nilp codefined} from the combination of the following two statements: 

\begin{prop} \label{p:smooth on curve 1}
Assume that $\CF$ is such that the functor \eqref{e:left functor new} is defined and codefined by a kernel. 
Then for any $\CF'\in \Shv(\Bun_G)$ and $V\in \Rep(\cG)$, 
\begin{equation} \label{e:pairing smooth}
(p_2)_\blacktriangle(\sH(V,\CF)_{\on{co}}\sotimes p_1^!(\CF'))\in \Shv_{\{0\}}(X)=\qLisse(X).
\end{equation}
\end{prop}

\begin{prop} \label{p:smooth on curve 2}
Let $\CY$ be an algebraic stack, and let $\CF_\CY\in \Shv(\CY\times X)$ be an object, such that for every
geometric point $\bi_y:\Spec(k')\to \CY$, the object
$$(\bi_y \times \on{id})^!(\CF_\CY)\in \Shv(X'), \quad X':=\Spec(k')\underset{\Spec(k)}\times X$$
belongs to $\qLisse(X')$. Then $\CF_\CY$ lies in the essential image of the functor 
$$\Shv(\CY)\otimes \qLisse(X) \hookrightarrow \Shv(\CY\times X).$$
\end{prop}

%With no restriction of generality, we can assume that $V\in \Rep(\cG)^c$.

%the following characterization of the subcategory
%$$\Shv_\Nilp(\Bun_G)\subset \Shv(\Bun_G),$$
%see \cite{AGKRRV}:
%
%\begin{thm} \label{t:Hecke}
%An object $\CF\in \Shv(\Bun_G)$ belongs to $\Shv_\Nilp(\Bun_G)$ if and only if for every
%$V\in \Rep(\cG)^c$, we have
%$$\on{H}(V,\CF)\in \Shv_{T^*(\Bun_G)\times \{0\}}(\Bun_G\times X).$$
%\end{thm}

%\begin{cor}
%An object $\CF\in \Shv(\Bun_G)^c_{\on{co}}$ belongs to $\Shv_\Nilp(\Bun_G)_{\on{co}}$ if and only if for every
%$V\in \Rep(\cG)^c$, we have
%$$\on{Id}^{\on{naive}}\circ \on{H}(V,\CF)\in \Shv_{T^*(\Bun_G)\times \{0\}}(\Bun_G\times X).$$
%\end{cor}

\sssec{}

Let us assume these propositions temporarily, and prove \thmref{t:char of Nilp codefined}.  

\medskip

Let
$$\CG:=\Mir_{\Bun_G}(\CF)\in \Shv(\Bun_G).$$

By \thmref{t:Hecke action Nilp 2} for $\CZ=\on{pt}$ (which is \cite[Theorem 14.4.3]{AGKRRV}), 
it suffices to show that 
for every $V\in \Rep(\cG)$, the object $\sH(V,\CG)$ belongs to the essential image of
\begin{equation} \label{e:ext prod again}
\Shv(\Bun_G)\otimes \qLisse(X) \hookrightarrow \Shv(\Bun_G\times X).
\end{equation}

\medskip

By \propref{p:smooth on curve 2}, it suffices to show that for every
geometric point $\bi_y:\Spec(k')\to \CY$, the object
$$(\bi_y \times \on{id})^!(\sH(V,\CG))\in \Shv(X'), \quad X':=\Spec(k')\underset{\Spec(k)}\times X$$
belongs to $\qLisse(X')$.

\medskip

Changing base from $k$ and $k'$, we can assume that $y$ is a closed point. We can rewrite
$$(\bi_y \times \on{id})^!(\sH(V,\CG)) \simeq
(p_2)_\blacktriangle(\sH(V,\CG)\sotimes p_1^!((\bi_y)_*(\sfe))),$$
and further as
$$(p_2)_\blacktriangle\left((\Mir_{\Bun_G}\boxtimes \on{Id})(\sH(V,\CF)_{\on{co}})\sotimes p_1^!((\bi_y)_*(\sfe))\right)
\simeq 
(p_2)_\blacktriangle\left(\sH(V,\CF)_{\on{co}}\sotimes p_1^!\left(\Mir_{\Bun_G}((\bi_y)_*(\sfe))\right)\right).$$

The required assertion follows now from \propref{p:smooth on curve 1}, applied to
$$\CF':=\Mir_{\Bun_G}((\bi_y)_*(\sfe)).$$

\qed[\thmref{t:char of Nilp codefined}]

\sssec{Proof of \propref{p:smooth on curve 1}}

We rewrite
$$(p_2)_\blacktriangle(\sH(V,\CF)_{\on{co}}\sotimes p_1^!(\CF')) \simeq
(p_2)_\blacktriangle(p_1^!(\CF)\sotimes \sH(V^\tau,\CF')).$$

With no restriction of generality, we can assume that $V\in \Rep(\cG)^c$ and $\CF'\in \Shv(\Bun_G)^c$.
In particular, $\CF'$ is constructible, and so is $\sH(V^\tau,\CF')$. 

\medskip

Recall now that objects of the form 
$$\sH(V',\CF')\in \Shv(\Bun_G\times X), \quad V'\in \Rep(\cG)^c,\, \CF'\in \Shv(\Bun_G)^{\on{constr}}$$
are ULA with respect to the projection 
$$p_2:\Bun_G\times X\to X.$$

Applying \corref{c:preserve sing supp}, we obtain that 
$$(p_2)_\blacktriangle(p_1^!(\CF)\sotimes \sH(V^\tau,\CF'))\in \qLisse(X),$$ as required.

\qed[Proof of \propref{p:smooth on curve 1}]

\sssec{Proof of \propref{p:smooth on curve 2}}

The functor \eqref{e:ext prod again} is fully faithful, and admits a continuous right adjoint, explicitly
described as follows:

\medskip

Identify $\qLisse(X)$ with its own dual via pairing $\on{ev}_X$
(here we use the fact that the pair $(X,\{0\})$ is duality-adapted, see Sects. \ref{sss:duality adapted} and \ref{sss:duality adapted curves}). 
Then the resulting functor
$$\Shv(\CY\times X)\otimes \qLisse(X) \simeq \Shv(\CY\times X)\otimes \qLisse(X)^\vee \to \Shv(\CY)$$
is given by
$$\CF_\CY\otimes E_X\in \Shv(\CY\times X)\otimes \qLisse(X) \mapsto (p_1)_*(\CF_\CY\sotimes p_2^!(E_X))\in \Shv(\CY).$$

Note also that every object in the essential image of \eqref{e:ext prod again} satisfies the condition of
\propref{p:smooth on curve 2}. Hence, it suffices to show that if $\CF_\CY$ 
is such that it satisfies the condition of \propref{p:smooth on curve 2} and
$$(p_1)_*(\CF_\CY\sotimes p_2^!(E_X))=0,\, \, \forall E_X\in \qLisse(X),$$
then $\CF_\CY=0$.

\medskip

To check this, it suffices to show that 
$$(\bi_y \times \on{id})^!(\CF_\CY)=0$$
for every geometric point $\bi_y:\Spec(k')\to \CY$. By the assumption on $\CF_\CY$, it suffices to show that
$$\on{C}^\cdot(X',(\bi_y \times \on{id})^!(\CF_\CY)\sotimes E_{X'})=0$$
for every $E_{X'}\in \qLisse(X')$. 

\medskip

Note that the base change functor $\qLisse(X)\to \qLisse(X')$ is an equivalence. Hence, the object $E_{X'}\in \qLisse(X')$
in the above formula is the base change of some $E_X\in \qLisse(X)$, and hence
$$\on{C}^\cdot(X',(\bi_y \times \on{id})^!(\CF_\CY)\sotimes E_{X'}) \simeq
\bi_y^!\left((p_1)_*(\CF_\CY\sotimes p_2^!(E_X))\right),$$
while the latter vanishes, by assumption. 

\qed[Proof of \propref{p:smooth on curve 2}]

\section{Serre functor on $\Shv_\Nilp(\Bun_G)$} \label{s:Serre}

In this section we will recast the results of \secref{s:pairing} in a different light, by relating 
the miraculous functor on $\Shv_\Nilp(\Bun_G)$ to the \emph{Serre functor}. 

\medskip

We will show that the category  $\Shv_\Nilp(\Bun_G)$ is Serre (see \secref{sss:Serre defn} for what this means), \emph{up to} the
$\Shv_\Nilp(\Bun_G)\rightsquigarrow  \Shv_\Nilp(\Bun_G)_{\on{co}}$ replacement. 
 
\ssec{The Serre functor}

In this subsection we recall the basic definitions pertaining to the Serre functor. 

\sssec{}

Let $\bC$ be a dualizable category. Throughout this paper we denote by
$$\on{u}_\bC\in \bC\otimes \bC^\vee$$
the unit of the duality.

\medskip

We denote by
$$\on{ev}_\bC:\bC\otimes \bC^\vee\to \Vect$$
the canonical pairing. 

\sssec{}

From now on we will assume that $\bC$ is compactly generated. We will identify the dual $\bC^\vee$ of $\bC$ with
$\on{Ind}((\bC^c)^{\on{op}})$. We will denote by $\BD$ the resulting contravariant equivalence
$$(\bC^c)^{\on{op}}\to (\bC^\vee)^c.$$
Under this identification, for $\bc\in \bC^c$ and $\bc'\in \bC$,
$$\on{ev}_\bC(\bc',\BD(\bc))\simeq \CHom_\bC(\bc,\bc').$$

\sssec{}

Recall that $\bC$ is said to be proper if the functor $\on{ev}_\bC$
preserves compactness\footnote{The more familiar formulation of this is that $\Hom_\bC(-,-)$
between compact objects is compact}.
This is equivalent to the assumption that the functor
$$\Vect \overset{\on{u}_\bC}\to \bC\otimes \bC^\vee$$
admits a \emph{left} adjoint, to be denoted 
\begin{equation} \label{e:L u}
\on{u}_\bC^L:  \bC\otimes \bC^\vee\to \Vect.
\end{equation} 

\sssec{} \label{sss:Serre}

Assume that $\bC$ is proper. The Serre endofunctor of $\bC$, denoted $\Se_\bC$,
is defined by the formula
$$\CHom_\bC(\bc_1,\Se_\bC(\bc)):=\CHom_\bC(\bc,\bc_1)^\vee, \quad \bc,\bc_1\in \bC^c.$$

\medskip

The following results from the definitions:

\begin{lem} \label{l:Serre functor}
Let $\bC$ be proper.
Then the functor $\on{u}_\bC^L$ identifies canonically with
$$\bC\otimes \bC^\vee \overset{\Se_\bC\otimes \on{Id}}\longrightarrow \bC\otimes \bC^\vee 
\overset{\on{ev}_\bC}\to \Vect.$$
\end{lem}

\sssec{} \label{sss:Serre defn}

A DG category $\bC$ is said to be \emph{Serre} if the functor $\Se_\bC$ is a self-equivalence.

\medskip

From \lemref{l:Serre functor}, we obtain:

\begin{cor} \label{c:Serre}
The following conditions are equivalent:

\smallskip

\noindent{\em(i)} $\bC$ is Serre;

\smallskip

\noindent{\em(ii)} The functor $\on{u}^L_\bC$ is the counit of \emph{another} duality
between $\bC$ and $\bC^\vee$. 

\end{cor}

\begin{cor}
Let $\bC$ be proper (and hence so is $\bC^\vee$). Then $\bC$ is Serre if and only if $\bC^\vee$ is Serre.
\end{cor}

\sssec{} \label{sss:D new}

Let $\bC$ be Serre. Thus, we obtain a \emph{new} identification between
the dual of $\bC$ and $\bC^\vee$. In particular, we obtain a \emph{new} contravariant equivalence
$$(\bC^c)^{\on{op}}\to (\bC^\vee)^c,$$
which we will denote by $\BD^{\on{new}}$. 

\medskip

By definition,
\begin{equation} \label{e:D new Serre}
\BD^{\on{new}}\simeq \BD\circ \Se_\bC.
\end{equation}

\sssec{Example}

Let $\bC=\QCoh(Y)$, where $Y$ is a smooth and proper scheme. We identify 
$$\QCoh(Y)\simeq \QCoh(Y)^\vee$$
via the \emph{Serre pairing}. I.e., the unit is given by
$$(\Delta_Y)_*(\omega_Y)\in \QCoh(Y\times Y)\simeq \QCoh(Y)\otimes \QCoh(Y)$$
(here $\omega_Y$ is the dualizing sheaf) and the counit by
$$\QCoh(Y)\otimes \QCoh(Y)\simeq \QCoh(Y\times Y)\overset{\Delta_Y^!}\longrightarrow
\QCoh(Y)\overset{\Gamma(Y,-)}\to \Vect.$$

Then the functor
$$\BD:(\on{Perf}(Y))^{\on{op}}\to \on{Perf}(Y)$$
is 
$$\CF\mapsto \CF^\vee\otimes \omega_Y,$$
where 
$$\CF^\vee:=\ul\Hom(\CF,\CO_Y).$$

\medskip

The Serre functor on $\QCoh(Y)$ is given by 
$$\CF\mapsto \CF\otimes \omega_Y,$$
in particular $\QCoh(Y)$ is Serre. 

\medskip

The functor $\BD^{\on{new}}$ is given by
$$\CF\mapsto \CF^\vee,$$
and the \emph{new} self-duality on $\QCoh(Y)$ has as counit 
$$\QCoh(Y)\otimes \QCoh(Y)\simeq \QCoh(Y\times Y)\overset{\Delta_Y^*}\longrightarrow
\QCoh(Y)\overset{\Gamma(Y,-)}\to \Vect,$$
and as unit 
$$(\Delta_Y)_*(\CO_Y)\in \QCoh(Y\times Y)\simeq \QCoh(Y)\otimes \QCoh(Y).$$

\ssec{Serre vs pseudo-identity}

In this subsection we reproduce some of the results of \cite{GaYo}, which relate the Serre
functor to the \emph{pseudo-identity} functor. 

\sssec{}

Let $\bC$ be a compactly generated category. 

\medskip

Applying the right Kan extension to $\BD$ along $(\bC^c)^{\on{op}}\hookrightarrow \bC^{\on{op}}$, 
and extend $\BD$ to a \emph{discontinuous} functor
$$\bC^{\on{op}}\to \bC^\vee.$$

Explicitly, for $$\bc=\underset{\alpha}{\on{colim}}\, \bc_{\alpha_i},$$
we have
$$\BD(\bc)=\underset{\alpha}{\on{lim}}\, \BD(\bc_{\alpha_i}).$$

\sssec{}

We identify
$$(\bC\otimes \bC^\vee)^\vee\simeq \bC^\vee \otimes \bC,$$
and define the object
\begin{equation} \label{e:ps-id object}
\on{ps-u}_\bC\in \bC\otimes \bC^\vee
\end{equation} 
to be $\BD(\on{u}_\bC)$.

\sssec{} \label{sss:ps-id}

We identify
\begin{equation} \label{e:kernels}
\bC \otimes \bC^\vee \simeq \on{Funct}_{\on{cont}}(\bC,\bC).
\end{equation} 

Under the identification the object $\on{u}_\bC\in \bC \otimes \bC^\vee$ 
corresponds to $\on{Id}_\bC\in \on{Funct}_{\on{cont}}(\bC,\bC)$.

\medskip

We define the pseudo-identity endofunctor of $\bC$
$$\on{Ps-Id}_\bC\in \on{Funct}_{\on{cont}}(\bC,\bC)$$ 
as the object corresponding under \eqref{e:kernels} to
$\on{ps-u}_\bC\in \bC\otimes \bC^\vee$.

\medskip

Note that we have:
$$(\on{Ps-Id}_\bC)^\vee \simeq \on{Ps-Id}_{\bC^\vee}.$$

\sssec{}

Let $\bC$ be Serre. In particular, by \corref{c:Serre}, the functor
\begin{equation} \label{e:new duality}
\bC \otimes \bC^\vee \overset{\on{u}^L_\bC}\to \Vect
\end{equation}
is the counit of a duality. 

\medskip

We claim:

\begin{prop} \label{p:unit new duality}
The unit of the duality \eqref{e:new duality} is given by $\on{ps-u}_\bC$.
\end{prop} 

\begin{proof}

We have to show that there exists a canonical isomorphism
$$\CHom_{\bC\otimes \bC^\vee}(\bc_1\otimes \BD(\Se_\bC(\bc_2)),\on{ps-u}_\bC)
\simeq \CHom_\bC(\bc_1,\bc_2), \quad \bc_1,\bc_2\in \bC^c.$$

\medskip

The left-hand side identifies, by definition, with
$$\CHom_{\bC\otimes \bC^\vee}(\on{u}_\bC,\Se_\bC(\bc_2)\otimes \BD(\bc_1)),$$
which is the same as the space of natural transformations 
$$\bc \to \CHom_{\bC}(\bc_1,\bc)\otimes \Se_\bC(\bc_2), \quad \bc\in \bC^c$$
We rewrite the latter as the space of natural transformations
$$\CHom_{\bC}(\bc_1,\bc)^\vee\to \CHom_{\bC}(\bc,\Se_\bC(\bc_2)), \quad \bc\in \bC^c,$$
i.e.,
$$\CHom_{\bC}(\bc_1,\bc)^\vee\to \CHom_{\bC}(\bc_2,\bc)^\vee, \quad \bc\in \bC^c,$$
which is the same as
$$\CHom_{\bC}(\bc_2,\bc) \to \CHom_{\bC}(\bc_1,\bc), \quad \bc\in \bC^c.$$

By Yoneda, the latter is the same as $\CHom_\bC(\bc_1,\bc_2)$, as required. 

\end{proof} 

\sssec{}

To summarize, if $\bC$ is Serre, we have the commutative diagrams
\begin{equation} \label{e:old and new pairing}
\CD
\bC\otimes \bC^\vee @>{\on{ev}_\bC}>> \Vect \\
@A{\Se_\bC\otimes \on{Id}}AA @AA{\on{Id}}A \\
\bC\otimes \bC^\vee @>{\on{u}^L_\bC}>> \Vect 
\endCD
\end{equation}
and
\begin{equation} \label{e:old and new unit}
\CD
\Vect @>{\on{u}_\bC}>> \bC\otimes \bC^\vee \\
@A{\on{Id}}AA  @AA{\Se_\bC\otimes \on{Id}}A \\
\Vect @>{\on{ps-u}_\bC}>> \bC\otimes \bC^\vee,
\endCD
\end{equation}
where the horizontal arrows are the counits and units of the corresponding dualities, 
respectively. 

\medskip

In particular, we recover the following result of \cite[Proposition 1.5.2]{GaYo}:

\begin{cor} \label{c:Serre inverse to Ps}
Assume that $\bC$ is Serre. Then the functor $\on{Ps-Id}_\bC$ is the inverse of $\Se_\bC$. 
\end{cor}

\sssec{} \label{sss:summary diag}

Thus, if $\bC$ is Serre, along with diagram \eqref{e:old and new pairing} 
%and \eqref{e:old and new unit} 
we also obtain the commutative diagram
\begin{equation} \label{e:old and new pairing bis}
\CD
\bC\otimes \bC^\vee @>{\on{u}^L_\bC}>> \Vect  \\
@A{\on{Ps-Id}_\bC\otimes \on{Id}}AA @AA{\on{Id}}A \\
\bC\otimes \bC^\vee @>{\on{ev}_\bC}>> \Vect,
\endCD
\end{equation}
%and
%\begin{equation} \label{e:old and new unit bis}
%\CD
%\Vect @>{\on{ps-u}_\bC}>> \bC\otimes \bC^\vee \\
%@A{\on{Id}}AA  @AA{\on{Ps-Id}_\bC\otimes \on{Id}}A \\
%\Vect @>{\on{u}_\bC}>> \bC\otimes \bC^\vee.
%\endCD
%\end{equation}

\sssec{}

Recall the functor
$$\BD^{\on{new}}:(\bC^c)^{\on{op}}\to (\bC^\vee)^c,$$
see \eqref{e:D new Serre}. We have
\begin{equation} \label{e:D new Mir}
\BD\simeq \BD^{\on{new}}\circ \on{Ps-Id}_\bC.
\end{equation}

\ssec{Duality-adapted pairs, complements}

In this subsection we collect some background material on the notion of 
\emph{duality-adapted pair}, see \secref{sss:duality adapted}.  

\sssec{}

Let $\CY$ be a quasi-compact algebraic stack (see \secref{sss:stacks} for our assumptions). Let $\CN$ be a 
closed conical subset in $T^*(\CY)$. 
Consider the corresponding full subcategory
$$\Shv_\CN(\CY) \overset{\iota_\CY}\hookrightarrow \Shv(\CY).$$

We will assume that the pair $(\CY,\CN)$ is \emph{duality-adapted}, see \secref{sss:duality adapted},
i.e., the functor
$$\Shv_\CN(\CY) \otimes \Shv_\CN(\CY)  \to \Shv(\CY) \otimes \Shv(\CY) 
\overset{\on{ev} _\CY}\longrightarrow \Vect$$
is the counit of a self-duality
\begin{equation} \label{e:usual self duality Y}
\Shv_\CN(\CY)^\vee \simeq \Shv_\CN(\CY)
\end{equation}

\sssec{} \label{sss:proj Y N}

Let $$\sP_{\CY,\CN}:\Shv(\CY)\to \Shv_\CN(\CY)$$
denote the functor dual to the embedding
$$\iota_\CY:\Shv_\CN(\CY) \hookrightarrow \Shv(\CY)$$
with respect the self-dualities
$$\Shv(\CY)^\vee \simeq \Shv(\CY) \text{ and } \Shv_\CN(\CY)^\vee \simeq \Shv_\CN(\CY)$$
of \eqref{e:Verdier} and \eqref{e:usual self duality Y}, respectively. 

\medskip

We will sometimes view $\sP_{\CY,\CN}$ as an endofunctor of $\Shv(\CY)$, by composing it
with the embedding $\iota_\CY$. 

\sssec{} \label{sss:u Y N}

Let 
$$\on{u}_{\CY,\CN}\in \Shv_\CN(\CY)\otimes \Shv_\CN(\CY)$$
be the unit of the self-duality \eqref{e:usual self duality Y} on $\Shv_\CN(\CY)$.

\medskip

We will sometimes view $\on{u}_{\CY,\CN}$ as an object of $\Shv(\CY\times \CY)$ via
the embedding
$$\Shv_\CN(\CY)\otimes \Shv_\CN(\CY) \overset{\iota_\CY\otimes \iota_\CY}\longrightarrow
\Shv(\CY)\otimes \Shv(\CY) \overset{\boxtimes}\longrightarrow \Shv(\CY\times \CY).$$

\begin{rem}

Note that when $\CN=T^*(\CY)$, we have
$$\on{u}_{\CY,\CN}=\on{u}_{\Shv(\CY)}.$$

\end{rem}

\sssec{} \label{sss:P Y N}

It follows from the definitions that $\sP_{\CY,\CN}$, viewed as an endofunctor of $\Shv(\CY)$,
identifies with
$$\CF \mapsto (p_2)_\blacktriangle(p_1^!(\CF)\sotimes \on{u}_{\CY,\CN}).$$

Thus, we extend $\sP_{\CY,\CN}$ to a functor defined by a kernel in the sense of
\secref{ss:ker} (with the kernel being $\on{u}_{\CY,\CN}$, viewed as an object of $\Shv(\CY\times \CY)$). 
In particular, we obtain the endofunctor
$$\on{Id}_\CZ\boxtimes \sP_{\CY,\CN}$$
of $\Shv(\CZ\times \CY)$ for any algebraic stack $\CZ$, so that for $\CZ=\on{pt}$ we recover
the original $\sP_{\CY,\CN}$. 

\medskip

\begin{prop}
The endofunctor $\on{Id}_\CZ\boxtimes \sP_{\CY,\CN}$
is the projector onto the full subcategory
$$\Shv(\CZ)\otimes \Shv_\CN(\CY) \hookrightarrow 
\Shv(\CZ)\otimes \Shv(\CY)\overset{\boxtimes}\hookrightarrow \Shv(\CZ\times \CY).$$
\end{prop}

\begin{proof}

The fact that $\on{u}_{\CY,\CN}$ belongs to 
$$\Shv_\CN(\CY)\otimes \Shv_\CN(\CY) \subset \Shv(\CY)\otimes \Shv_\CN(\CY)$$
implies that the essential image of $\on{Id}_\CZ\boxtimes \sP_{\CY,\CN}$ belongs to
$$\Shv(\CZ)\otimes \Shv_\CN(\CY)\subset \Shv(\CZ\times \CY).$$

For any functor $\sQ$ defined by a kernel $\CQ\in \Shv(\CY_1\times \CY_2)$, 
the restriction of $\on{Id}_\CZ\boxtimes \sQ$ to 
$$\Shv(\CZ)\otimes \Shv(\CY_1)\subset \Shv(\CZ\times \CY_2)$$
identifies with
$$\on{Id}_{\Shv(\CZ)}\otimes \sQ,$$
viewed as a functor
$$\Shv(\CZ)\otimes \Shv(\CY_1)\to \Shv(\CZ)\otimes \Shv(\CY_2)\subset \Shv(\CZ\times \CY_2),$$
see \eqref{e:boxtimes vs otimes}.

\medskip

This implies that $\on{Id}_\CZ\boxtimes \sP_{\CY,\CN}$ acts as identity when restricted
to $\Shv(\CZ)\otimes \Shv_\CN(\CY)$.

\end{proof}

\begin{cor}
The endofunctor $\sP_{\CY,\CN}\boxtimes \sP_{\CY,\CN}$ of $\Shv(\CY\times \CY)$
is a projector onto
$$\Shv_\CN(\CY)\otimes \Shv_\CN(\CY)\subset \Shv(\CY\times \CY).$$
\end{cor}

\begin{proof}
Same as that of \corref{c:double projector}.
\end{proof}

\sssec{}

We now claim: 

\begin{lem} \label{l:diag Y N}
The object $\on{u}_{\CY,\CN}$ identifies with each of the following:
$$(\sP_{\CY,\CN}\boxtimes \on{Id}_\CY)(\on{u}_\CY),\,\, (\sP_{\CY,\CN}\boxtimes \sP_{\CY,\CN})(\on{u}_\CY),\,\, 
(\on{Id}_\CY \boxtimes \sP_{\CY,\CN})(\on{u}_\CY).$$
\end{lem}

\begin{proof}
The isomorphism
$$\on{u}_{\CY,\CN}\simeq (\on{Id}_\CY \boxtimes \sP_{\CY,\CN})(\on{u}_\CY)$$
is tautological, and the isomorphism with $(\sP_{\CY,\CN}\boxtimes \on{Id}_\CY)(\on{u}_\CY)$ follows by symmetry.

\medskip

We have
$$(\sP_{\CY,\CN}\boxtimes \sP_{\CY,\CN})(\on{u}_\CY)\simeq
(\sP_{\CY,\CN}\boxtimes \on{Id}_\CY)\circ (\on{Id}_\CY \boxtimes \sP_{\CY,\CN})(\on{u}_\CY)\simeq
(\sP_{\CY,\CN}\boxtimes \on{Id}_\CY)(\on{u}_{\CY,\CN}),$$
and the latter is isomorphic to $\on{u}_{\CY,\CN}$, since $\sP_{\CY,\CN}\boxtimes \on{Id}_\CY$ acts
as identity on 
$$\Shv_\CN(\CY)\otimes \Shv(\CY)\subset \Shv(\CY\times \CY).$$
\end{proof}

\ssec{Constraccessible pairs, complements}

In this subsection we collect some background material on the notion of 
\emph{constraccessible pair}, see \secref{sss:constraccess}.  

\medskip

In this subsection we assume that $\CY$ is quasi-compact. 

\sssec{}

Assume now that the pair $(\CY,\CN)$ is \emph{constraccessible}, see \secref{sss:constraccess}.
In particular, it is duality-adapted, and the equivalence
$$(\Shv_\CN(\CY)^c)^{\on{op}}\to \Shv_\CN(\CY)^c,$$
corresponding to \eqref{e:usual self duality Y}, is obtained by restriction from the Verdier duality functor
$$\BD^{\on{Verdier}}:(\Shv(\CY)^c)^{\on{op}}\to \Shv(\CY)^c.$$ 

\sssec{}

We claim:

\begin{lem} \label{l:right adj as dual of P}
For a quasi-compact algebraic stack $\CZ$, the functor $\on{Id}_\CZ\boxtimes \sP_{\CY,\CN}$, viewed as a functor
$$\Shv(\CZ\times \CY)\to \Shv(\CZ)\otimes \Shv_\CN(\CY),$$
is the right adjoint of 
$$\Shv(\CZ)\otimes \Shv_\CN(\CY) \overset{\on{Id}\otimes \iota_\CY}\longrightarrow 
\Shv(\CZ)\otimes \Shv(\CY) \hookrightarrow \Shv(\CZ\times \CY).$$
\end{lem}

\begin{proof}

Follows from \lemref{l:right adjoint and proj}, since the constraccessibility assumption
implies that
$$(\boxtimes \circ (\on{Id}\otimes \iota_\CY))^{\on{fake-op}}\simeq 
(\boxtimes \circ (\on{Id}\otimes \iota_\CY)).$$

\end{proof}

\begin{cor} \label{c:embed N sq}
The endofunctor $\sP_{\CY,\CN}\boxtimes \sP_{\CY,\CN}$ of $\Shv(\CY\times \CY)$,
viewed as a functor
$$\Shv(\CY\times \CY)\to \Shv_\CN(\CY)\otimes \Shv_\CN(\CY)$$
identifies with the right adjoint of the tautological embedding.
\end{cor}

\begin{proof}

Obtained by applying \lemref{l:right adj as dual of P} to
$$\on{Id}_\CY\boxtimes \sP_{\CY,\CN}: \Shv(\CY\times \CY)\to \Shv(\CY)\otimes \Shv_\CN(\CY)$$
and
$$\sP_{\CY,\CN}\otimes \on{Id}: \Shv(\CY)\otimes \Shv_\CN(\CY)\to \Shv_\CN(\CY)\otimes \Shv_\CN(\CY).$$

\end{proof}

\begin{cor} \label{c:u N}
The object $\on{u}_{\CY,\CN}$ identifies with the value on $\on{u}_\CY$ of the right adjoint to the embedding
\begin{equation} \label{e:embed N sq}
\Shv_\CN(\CY)\otimes \Shv_\CN(\CY)\hookrightarrow \Shv(\CY\times \CY).
\end{equation} 
\end{cor}

\ssec{Serre pairs}  \label{ss:Serre stack} 

In this subsection we introduce the notion of what it means for a pair $(\CY,\CN)$ to be \emph{Serre}.  

\sssec{} \label{sss:ps u N} 

Let $\CY$ be an algebraic stack (for now we are \emph{not} assuming that $\CY$ is quasi-compact). 
Let $\CN$ be a subset in $T^*(\CY)$. 

\medskip

Recall the notation 
$$\on{ps-u}_\CY:=\Delta_!(\ul\sfe_\CY)\in \Shv(\CY\times \CY).$$

Let 
$$\on{ps-u}_{\CY,\CN}\in  \Shv_\CN(\CY)\otimes \Shv_\CN(\CY)$$
be the value on $\on{ps-u}_\CY$ of the right adjoint of the embedding
$$\Shv_\CN(\CY)\otimes \Shv_\CN(\CY)\hookrightarrow 
\Shv(\CY)\otimes \Shv(\CY)\overset{\boxtimes}\hookrightarrow\Shv(\CY\times \CY).$$

\medskip

The counit of the adjunction defines a map
\begin{equation} \label{e:ps-u N to ps-u}
\on{ps-u}_{\CY,\CN}\to \on{ps-u}_\CY.
\end{equation}

%\begin{rem} \label{r:adj ps-u}
%Note that if $\CY$ is quasi-compact and the pair $(\CY,\CN)$ is constraccessible, 
%then it follows from \corref{c:embed N sq} that the object $\on{ps-u}_{\CY,\CN}$ 
%identifies also with 
%$$(\sP_{\CY,\CN}\boxtimes \sP_{\CY,\CN})(\on{ps-u}_{\CY}).$$
%\end{rem}

\sssec{}

Recall that 
$$\on{ev}^l_\CY:\Shv(\CY)\otimes \Shv(\CY)\to \Vect$$
denoted the pairing
$$\CF_1,\CF_2 \mapsto \on{C}^\cdot_c(\CY,\CF_1\overset{*}\otimes \CF_2).$$

By a slight abuse of notation, we will denote by the same symbol $\on{ev}^l_\CY$
its restriction along 
$$\Shv_\CN(\CY)\otimes \Shv_\CN(\CY)\to \Shv(\CY)\otimes \Shv(\CY).$$

\sssec{}

The map \eqref{e:ps-u N to ps-u} gives rise to a natural transformation from the functor
\begin{equation} \label{e:check axiom 2}
\Shv(\CY) \overset{\on{Id}\boxtimes \on{ps-u}_{\CY,\CN}}\longrightarrow 
\Shv(\CY\times \CY \times \CY) \overset{\ev_\CY^l\boxtimes \on{Id}_\CY}\longrightarrow \Shv(\CY)
\end{equation}
to the composition
$$\Shv(\CY) \overset{\on{Id}\boxtimes \on{ps-u}_{\CY}}\longrightarrow 
\Shv(\CY\times \CY \times \CY) \overset{\ev_\CY^l\boxtimes \on{Id}_\CY}\longrightarrow \Shv(\CY),$$
which is the identity functor. 

\medskip

Restricting to $\Shv_\CN(\CY)$, we obtain a natural transformation 
\begin{equation} \label{e:check axiom 3 new}
(\ev_\CY^l\otimes \on{Id})\circ (\on{Id}_\CY\boxtimes \on{ps-u}_{\CY,\CN})\to \on{Id}
\end{equation}

\medskip

\begin{defn} \label{d:Serre}
We shall say that the pair $(\CY,\CN)$ is \emph{Serre} if the natural transformation \eqref{e:check axiom 3 new}
becomes an isomorphism when restricted to $\Shv_\CN(\CY)$.
\end{defn}

\sssec{}

Thus, by definition, if $(\CY,\CN)$ is Serre, then the object
$$\on{ps-u}_{\CY,\CN}\in  \Shv_\CN(\CY)\otimes \Shv_\CN(\CY)$$
and the pairing
$$\on{ev}^l_\CY:\Shv_\CN(\CY)\otimes \Shv_\CN(\CY)\to \Vect$$
define the unit and the counit for a self-duality on $\Shv_\CN(\CY)$.

\sssec{Examples} \label{sss:ex smooth prel}

Let $\CY$ be a smooth proper scheme. Assume that the pair $(\CY,\{0\})$
is constraccessible. Then the pair $(\CY,\{0\})$ is Serre. 

\medskip

Let us now consider another extreme: let $\CY$ be an algebraic stack with finitely many isomorphism
classes of points (e.g., we can take $\CY$ to be $N\backslash G/B$ or an open substack $\Bun_G$ 
for a curve of genus $0$). We claim that in case the pair $(\CY,T^*(\CY))$ is Serre. Indeed, this follows
from the fact that for $\CY$ of the above form the functor
$$\Shv(\CY) \otimes \Shv(\CY)\to \Shv(\CY\times \CY)$$ 
is an equivalence, so 
$$\on{ps-u}_{\CY,\CN}\simeq \on{ps-u}_\CY.$$

\sssec{}

Recall that the category $\Shv(\CY)$ is compactly generated by objects of the form $g_!(\CF)$, for
$$g:S\to \CY,$$
where $S$ is an affine scheme, $g$ is smooth and $\CF\in \Shv(Z)^c=\Shv(Z)^{\on{constr}}$
(see \secref{sss:stacky}). It is easy to see that the $\CHom$  space between objects of the above form 
is compact. This implies that the category $\Shv(\CY)$ is proper. 

\medskip

Assume that the pair
$(\CY,\CN)$ is constraccessible, i.e., $\Shv_\CN(\CY)$ is generated by objects that are compact in $\Shv(\CY)$. 
We obtain that the category $\Shv_\CN(\CY)$ is also proper. Hence, we can consider the Serre functor
$\Se_{\Shv_\CN(\CY)}$.

\sssec{}

We claim:

\begin{thm} \label{t:Serre Serre}
Assume that $\CY$ is quasi-compact, and that $(\CY,\CN)$ is constraccessible and 
Serre. Then the category $\Shv_\CN(\CY)$ is Serre. Moreover
under the identification 
$\Shv_\CN(\CY)^\vee\simeq \Shv_\CN(\CY)$
of \eqref{e:usual self duality Y}, the object 
$$\on{ps-u}_{\CY,\CN}\in  \Shv_\CN(\CY)\otimes \Shv_\CN(\CY)$$
corresponds to the object 
$$\on{ps-u}_{\Shv_\CN(\CY)}\in \Shv_\CN(\CY)\otimes \Shv_\CN(\CY)^\vee$$
of \eqref{e:ps-id object}, and the pairing
$$\on{ev}^l_\CY:\Shv_\CN(\CY)\otimes \Shv_\CN(\CY)\to \Vect$$
corresponds to the pairing
$$\on{u}_{\Shv_\CN(\CY)}^L:\Shv_\CN(\CY)\otimes \Shv_\CN(\CY)^\vee\to \Vect$$
of \eqref{e:L u}. 
\end{thm} 

\begin{proof}

By \corref{c:Serre} and \propref{p:unit new duality}, we only have to show that the functor $\on{ev}^l_\CY$ is the left adjoint of
$$\Vect \overset{\on{u}_{\CY,\CN}}\longrightarrow \Shv_\CN(\CY)\otimes \Shv_\CN(\CY).$$

However, this is formal from \corref{c:u N}. 

\end{proof}

\ssec{Relation to the miraculous functor}

Let $\CY$ be quasi-compact, and let $(\CY,\CN)$ be constraccessible. 

\medskip

In this subsection
we will relate the Serre functor on $\Shv_\CN(\CY)$ to the miraculous functor. 

\sssec{} \label{sss:ps u N as proj}

Consider the object 
$$\on{ps-u}_{\CY,\CN}\in  \Shv_\CN(\CY)\otimes \Shv_\CN(\CY)\subset \Shv(\CY\times \CY),$$
see \secref{sss:ps u N}. According to \corref{c:embed N sq}, we have
$$\on{ps-u}_{\CY,\CN} \simeq (\sP_{\CY,\CN}\otimes \sP_{\CY,\CN})(\on{ps-u}_\CY).$$

\sssec{}

Let 
$$\Mir_\CY:\Shv(\CY)\to \Shv(\CY)$$
be the \emph{miraculous} functor, i.e., the functor defined by the kernel
$$\on{ps-u}_\CY\in \Shv(\CY\times \CY),$$
see \secref{sss:Mir qc}. 

\sssec{}

From \secref{sss:ps u N as proj} we obtain:

\begin{lem} \label{l:Mir on N}
The endofunctor of $\Shv(\CY)$ defined by $\on{ps-u}_{\CY,\CN}$ as a kernel is the composite 
$$\Shv(\CY) \overset{\sP_{\CY,\CN}}\longrightarrow \Shv(\CY)\overset{\Mir_\CY}\longrightarrow \Shv(\CY)\overset{\sP_{\CY,\CN}}\to \Shv(\CY).$$
\end{lem}

\sssec{}

We give the following definition: 

\begin{defn}
We will say that $\CN\subset T^*(\CY)$ is \emph{miraculous-compatible} if the functor $\Mir_\CY$
preserves the subcategory $\Shv_\CN(\CY)\subset \Shv(\CY)$.
\end{defn}  

If $\CN$ is miraculous-compatible, we will denote by $\Mir_{\CY,\CN}$ the resulting endofunctor of $\Shv_\CN(\CY)$. 

\sssec{}

From \lemref{l:Mir on N}, we obtain: 

\begin{cor} \label{c:Mir on N}
Suppose that $\CN$ be miraculous-compatible. Then:

\smallskip

\noindent{\em(a)}
The endofunctor of $\Shv_\CN(\CY)$ defined by $\on{ps-u}_{\CY,\CN}$ 
identifies canonically with $\Mir_{\CY,\CN}$. 

\smallskip

\noindent{\em(b)} We have canonical isomorphisms:
$$(\sP_{\CY,\CN}\otimes \on{Id} _{\CY})(\on{ps-u}_\CY)\simeq \on{ps-u}_{\CY,\CN}\simeq 
(\on{Id} _{\CY} \otimes \sP_{\CY,\CN})(\on{ps-u}_\CY).$$

\smallskip

\noindent{\em(c)} We have canonical isomorphisms of endofunctors of $\Shv(\CY)$ 
$$\Mir_\CY\circ \sP_{\CY,\CN} \simeq  \sP_{\CY,\CN} \circ \Mir_\CY.$$

\end{cor}

\begin{proof}

Point (a) is immediate from \lemref{l:Mir on N}. For point (b), by symmetry, it suffices to 
prove the first isomorphism. Since
$$(\sP_{\CY,\CN}\otimes \on{Id} _{\CY})(\on{ps-u}_\CY)\in \Shv_\CN(\CY)\otimes \Shv(\CY)$$
(by \lemref{l:right adj as dual of P}), it suffices to show that
$$(\on{ev}_\CY \otimes \on{Id})(\CF\otimes (\sP_{\CY,\CN}\otimes \on{Id} _{\CY})(\on{ps-u}_\CY))\simeq 
(\on{ev}_\CY \otimes \on{Id})(\CF\otimes (\sP_{\CY,\CN}\otimes \sP_{\CY,\CN})(\on{ps-u}_\CY)), \quad \CF\in \Shv_\CN(\CY).$$

The left-hand side identifies with
$$(\on{ev}_\CY \otimes \on{Id})(\CF\otimes \on{ps-u}_\CY)\simeq \Mir_\CY(\CF),$$
while the right-hand side also identifies with $\Mir_\CY(\CF)$, by point (a). 

\medskip

Point (c) is formal from point (b). 

\end{proof}

\sssec{}

Finally, we claim:

\begin{cor} \label{c:Mir inverse to Serre}
Suppose that $\CN$ is miraculous-compatible and that $(\CY,\CN)$ is Serre. Then: 

\smallskip

\noindent{\em(a)} $\Mir_{\CY,\CN}\simeq \on{Ps-Id}_{\Shv_\CN(\CY)}$;

\smallskip

\noindent{\em(b)} $\Mir_{\CY,\CN}\simeq (\Se_{\Shv_\CN(\CY)})^{-1}$.

\smallskip

\noindent{\em(c)} The diagram
$$
\CD
\Shv_\CN(\CY)  \otimes \Shv_\CN(\CY)  @>{\on{ev}^l_\CY}>> \Vect \\
@A{\Mir_{\CY,\CN}\otimes \on{Id}}AA @AA{\on{Id}}A   \\
\Shv_\CN(\CY)  \otimes \Shv_\CN(\CY)  @>{\on{ev}_\CY}>> \Vect 
\endCD
$$
commutes.
\end{cor} 

\begin{proof}

Point (a) follows by combining \corref{c:Mir on N}(a) and \thmref{t:Serre Serre}. Point (b)
follows from \corref{c:Serre inverse to Ps}. Point (c) follows from \secref{sss:summary diag}. 

\end{proof}

\begin{cor}
Suppose that $\CN$ is miraculous-compatible and that $(\CY,\CN)$ is Serre. Then 
the endofunctor $\Mir_{\CY,\CN}$ of $\Shv_\CN(\CY)$ is an equivalence.
\end{cor}

\begin{rem}
Suppose that in the setting of \corref{c:Mir inverse to Serre} above, the stack $\CY$
is miraculous (see \secref{sss:Mir stacks qc} for what this means). Recall the notation 
$$\BD^{\Mir}:(\Shv(\CY)^c)^{\on{op}}\to \Shv(\CY)^c,$$
see \secref{sss:D Mir}. Recall also the notation $\BD^{\on{new}}$, see \secref{sss:D new}. We obtain: 
$$\BD^{\Mir}\simeq \BD^{\on{new}}.$$

\end{rem}

\ssec{Relation to the miraculous functor, continued}

We retain the assumptions of the previous subsection, and additionally require that 
$\CN$ is miraculous-compatible and that $(\CY,\CN)$ is Serre. 

\sssec{}

Note that \corref{c:Mir inverse to Serre} implies that we have a canonical isomorphism
\begin{equation} \label{e:two pairings Y N}
\ev^l_\CY(\CF_1,\Mir_\CY(\CF_2)) \simeq \ev_\CY(\CF_1,\CF_2), \quad \CF_1,\CF_2\in \Shv_\CN(\CY).
\end{equation}

We claim:

\begin{prop} \label{p:two pairings Y N}
The isomorphism \eqref{e:two pairings Y N} identifies with the natural transformation 
\eqref{e:ident pair 2 gen}.
\end{prop}

\begin{proof}

Repeats the proof of \propref{p:ident pair 1 2}, with the difference that the vertical arrows in 
the counterpart of diagram \eqref{e:ev com diag Nilp} are furnished by that natural transformation
$$\sP_{\CY,\CN}\to \on{Id},$$
given by the counit of the adjunction.

\end{proof}

\sssec{}

As in \thmref{t:Nilp admits right adjoint U}, from \propref{p:two pairings Y N} we obtain:

\begin{cor} Let $\CF$ be an object in $\Shv_\CN(\CY)^{constr}$. Then the functor
$$\sF:\Shv(\CY)\to \Vect$$
that it defines admits a right adjoint, as a functor defined by a kernel.
\end{cor}

Further, as in \corref{c:Nilp codefined U}, we obtain: 

\begin{cor} Let $\CF$ be an object in $\Shv_\CN(\CY)$. Then the functor
$$\sF:\Shv(\CY)\to \Vect$$
that it defines, is defined and codefined by a kernel.
\end{cor}

As in \corref{c:pairing Y U}, we obtain: 

\begin{cor} Let $\CF$ be an object in $\Shv_\CN(\CY)$. Then for an algebraic stack $\CZ$, the natural transformation
$$(p_\CZ)_!(\CF' \overset{*}\otimes p_\CY^*(\Mir_\CY(\CF)))\to 
(p_\CZ)_\blacktriangle(\CF' \sotimes p_\CY^!(\CF)), \quad \CF'\in \Shv(\CZ\times \CY)$$
is an isomorphism. 
\end{cor}

Finally, as in \corref{c:pairing Y U bis}, we obtain:

\begin{cor} Let $\CY$ be miraculous, and let  $\CF$ be an object in $\Shv_\CN(\CY)$. 
Then for an algebraic stack $\CZ$, we have a canonical isomorphism 
the natural transformation
$$(p_\CZ)_!((\on{Id}_\CZ \boxtimes \Mir_\CY)(\CF') \overset{*}\otimes p_\CY^*(\CF))\simeq 
(p_\CZ)_\blacktriangle(\CF' \sotimes p_\CY^!(\CF)), \quad \CF'\in \Shv(\CZ\times \CY).$$
\end{cor}

\ssec{The (non)-Serre property of $\Bun_G$}

In this subsection we will study the Serre property of the category $\Shv_\Nilp(\Bun_G)$, and also of $\Shv_\Nilp(\CU)$,
for universally Nilp-cotruncative quasi-compact open substacks $\CU\subset \Bun_G$. 

\sssec{}

Let $\CU\subset \Bun_G$ be a universally Nilp-cotruncative quasi-compact open substack. Note that 
it follows from \propref{p:usual duality U} that the object that was denoted $\on{u}_{\CU,\Nilp}$ in
\secref{sss:u U Nilp} is the same as the object that appears in \secref{sss:u Y N}
(for $\CY=\CU$ and $\CN=\Nilp$). 

\medskip

Furthermore, it follows from \propref{p:u naive projectors} that the endofunctor of $\Shv(\CU)$
denoted $\sP_{\CU,\Nilp}$ in \secref{sss:proj U} identifies, \emph{as an endofunctor defined by a kernel}, 
with the endofunctor that appears in \secref{sss:P Y N} (for $\CY=\CU$ and $\CN=\Nilp$).  

\medskip

Hence, the notations $\on{u}_{\CU,\Nilp}$ and $\sP_{\CU,\Nilp}$ are unambiguous. 

\sssec{}

From now on, for the duration of this subsection, we will assume \cite[Conjecture 14.1.8]{AGKRRV}. 
I.e., we will assume that $\Shv_{\Nilp}(\Bun_G)$ is generated by objects that are compact in 
$\Shv(\Bun_G)$. By \cite[Lemma F.8.10]{AGKRRV}, this implies, that for every $\CU$ as above, the pair 
$(\CU,\Nilp)$ is constraccessible. 

\medskip

Applying \corref{c:j and P}, we obtain that the object $\on{ps-u}_{\CU,\Nilp}$ introduced in \secref{sss:ps-u U Nilp} 
identifies with each of the following objects
$$(\sP_{\CU,\Nilp}\otimes \on{Id}_\CU)(\on{ps-u}_\CU),\,\, 
(\sP_{\CU,\Nilp}\otimes \sP_{\CU,\Nilp})(\on{ps-u}_\CU),\,\, (\on{Id}_\CU\otimes \sP_{\CU,\Nilp})(\on{ps-u}_\CU).$$

\medskip

Combining with \secref{sss:ps u N as proj}, we obtain that $\on{ps-u}_{\CU,\Nilp}$
of \secref{sss:ps-u U Nilp} is the same as the object
that appears in \secref{sss:ps u N} (for $\CY=\CU$ and $\CN=\Nilp$).  

\medskip

Hence, the notation $\on{ps-u}_{\CU,\Nilp}$ is also unambiguous. 

\sssec{}

We now claim:

\begin{cor} \label{c:U is Serre}
The pair $(\CU,\Nilp)$ is Serre.
\end{cor}

\begin{proof}

By \corref{c:non-standard duality U}, we know that 
$$\on{ev}^l_{\CU}:\Shv_{\Nilp}({\CU})\otimes \Shv_{\Nilp}(\CU)\to \Vect, \quad
\CF_1,\CF_2\mapsto \on{C}^\cdot_c({\CU},\CF_1\overset{*}\otimes \CF_2)$$
and the object $\on{ps-u}_{\CU,\Nilp}$ define a duality datum. 

\medskip

In particular, we obtain a canonical isomorphism
$$(\ev_\CU^*\otimes \on{Id})\circ (\on{Id}\otimes \on{ps-u}_{\CU,\Nilp})\simeq \on{Id}.$$

However, unwinding the definitions, one shows that the above isomorphism agrees
with the natural transformation in Definition \ref{d:Serre}. 

\end{proof}

\sssec{}

We now consider the non quasi-compact stack $\Bun_G$. From \corref{c:P P as adj} we obtain: 

\begin{cor}
The object $\on{ps-u}_{\Bun_G,\Nilp}$ of \secref{sss:ps u Bun} identifies canonically with the
value of the right adjoint to 
$$\Shv_{\Nilp}(\Bun_G) \otimes \Shv_{\Nilp}(\Bun_G)\hookrightarrow \Shv(\Bun_G\times \Bun_G)$$
on $\on{ps-u}_{\Bun_G}$.
\end{cor}

\medskip

Hence, as in \corref{c:U is Serre}, from \thmref{t:non-standard duality}, we obtain that 
the pair $(\Bun_G,\Nilp)$ is Serre.

\sssec{}

The assumption that \cite[Conjecture 14.1.8]{AGKRRV} holds implies, in particular, that 
the category $\Shv_\Nilp(\Bun_G)$ is proper. Hence, we can consider the Serre functor
$$\Se_{\Shv_\Nilp(\Bun_G)}:\Shv_\Nilp(\Bun_G)\to \Shv_\Nilp(\Bun_G).$$

We claim:

\begin{thm} \label{t:Serre on BunG}
The functor $\Se_{\Shv_\Nilp(\Bun_G)}$ is canonically isomorphic to the composition
$$\Shv_\Nilp(\Bun_G) \overset{\Mir^{-1}_{\Bun_G}}\longrightarrow 
\Shv_\Nilp(\Bun_G)_{\on{co}} \overset{\on{Id}^{\on{naive}}_{\Bun_G}}\longrightarrow  \Shv_\Nilp(\Bun_G).$$
\end{thm}

%Thus, \thmref{t:Serre on BunG} says that the category $\Shv_\Nilp(\Bun_G)$ is \emph{not} Serre;
%its failure to be so is exactly accounted for by the functor $\on{Id}^{\on{naive}}$.

\begin{proof}

Let $\CU\overset{j}\hookrightarrow \Bun_G$ be a universally Nilp-cotruncative quasi-compact open substack.
Unwinding the definitions, we obtain a canonical isomorphism
\begin{equation} \label{e:Se and j}
\Se_{\Shv_\Nilp(\Bun_G)} \circ j_! \simeq j_* \circ \Se_{\Shv_\Nilp(\CU)}.
\end{equation}

Furthermore, for an inclusion $\CU_1\overset{j_{1,2}}\hookrightarrow \CU_2$, we have
\begin{equation} \label{e:Se and j12}
\Se_{\Shv_\Nilp(\CU_2)} \circ (j_{1,2})_! \simeq (j_{1,2})_*\circ \Se_{\Shv_\Nilp(\CU_1)},
\end{equation}
compatible with \eqref{e:Se and j}.

\medskip

By \eqref{e:Mir and U}, we also have a system of isomorphisms
$$\on{Id}^{\on{naive}}_{\Bun_G}\circ \Mir^{-1}_{\Bun_G} \circ j_! \simeq j_* \circ \Mir^{-1}_{\CU},$$ 
compatible with
$$\Mir^{-1}_{\CU_2} \circ (j_{1,2})_! \simeq (j_{1,2})_* \circ \Mir^{-1}_{\CU_1}.$$ 

\medskip

Hence, it suffices to show that the system of isomorphisms
$$\Se_{\Shv_\Nilp(\CU)}\simeq \Mir^{-1}_{\CU,\Nilp}$$
of \corref{c:Mir inverse to Serre}(b) makes the system of diagrams 
$$
\CD
\Se_{\Shv_\Nilp(\CU_2)} \circ (j_{1,2})_!  @>{\sim}>> (j_{1,2})_*\circ \Se_{\Shv_\Nilp(\CU_1)} \\
@V{\sim}VV @V{\sim}VV \\
\Mir^{-1}_{\CU_2} \circ (j_{1,2})_! @>{\sim}>> (j_{1,2})_* \circ \Mir^{-1}_{\CU_1}.
\endCD
$$

However, this follows by unwinding the definitions. 

\end{proof}

\sssec{}

It follows from \thmref{t:Serre on BunG} that the functor $\Se_{\Shv_\Nilp(\Bun_G)}$ sends compact objects
of $\Shv_\Nilp(\Bun_G)$ to objects that lie in the essential image of the \emph{fully faithful} functor
$$(\Shv_\Nilp(\Bun_G)_{\on{co}})^c\hookrightarrow \Shv_\Nilp(\Bun_G)_{\on{co}} 
\overset{\on{Id}^{\on{naive}}_{\Bun_G}}\longrightarrow  \Shv_\Nilp(\Bun_G).$$

Denote by
$$\Se_{\Shv_\Nilp(\Bun_G),\on{co}}$$
the ind-extension of the resulting functor
$$\Shv_\Nilp(\Bun_G)\to \Shv_\Nilp(\Bun_G)_{\on{co}}.$$

From \thmref{t:Serre on BunG}, we obtain:

\begin{cor}
The functor $\Se_{\Shv_\Nilp(\Bun_G),\on{co}}$ identifies canonically with $(\Mir^{-1}_{\Bun_G})|_{\Shv_\Nilp(\Bun_G)}$.
\end{cor}

\appendix

\section{Sheaves on stacks} \label{s:sheaves}

In this section we review the theory of $\Shv(-)$ on stacks, especially the aspects that
have to do with Verdier duality, used in the bulk of the paper. 

\ssec{The basics}

In this subsection we name the main players in $\Shv(-)$ on stacks. 

\sssec{} \label{sss:sheaves}

In this paper we work with a constructible sheaf theory, denoted $\Shv$, which we view as a functor
$$(\affSch_{\on{ft}/k})^{\on{op}}\to \DGCat,$$
see \cite[Sects. 1.1]{AGKRRV}, where for $f:S_1\to S_2$, the functor
$\Shv(S_2)\to \Shv(S_1)$ is $f^!$. 

\medskip

For a given affine scheme $S$, the category $\Shv(S)$ is defined as the ind-completion of the category
$\Shv(S)^{\on{constr}}$ of constructible sheaves, so that $\Shv(S)^c$ recovers $\Shv(S)^{\on{constr}}$. 

\sssec{}

We extend $\Shv(-)$ from affine schemes to algebraic stacks by the procedure of right Kan extension.
Explicitly, for an algebraic stack $\CY$,
\begin{equation} \label{e:sheaves on stacks}
\Shv(\CY):= \underset{S\to \CY}{\on{lim}^!}\, \Shv(S),
\end{equation}
where the index category is any of the following:

\begin{itemize}

\item All affine schemes $S$ of finite type over $\CY$;

\item Affine schemes $S$ that map smoothly to $\CY$, and all maps $f:S_1\to S_2$;

\item Affine schemes $S$ that map smoothly to $\CY$, and smooth maps $f:S_1\to S_2$. 

\end{itemize}

In the formation of the limit, the transition functors $\Shv(S_2)\to \Shv(S_1)$ are given by $f^!$. 

\sssec{}

One can also rewrite $\Shv(\CY)$ as 
\begin{equation} \label{e:sheaves on stacks *}
\underset{S\to \CY}{\on{lim}^*}\, \Shv(S),
\end{equation}
over the same choice of index categories, but with transition functors $\Shv(S_2)\to \Shv(S_1)$ given by $f^*$. 

\medskip

The two limits are equivalent via the following procedure: use the third index category in both cases. Applying the cohomological
shift by $[2\dim(S/\CY)]$ on each $\Shv(S)$, we can isomorph the limit \eqref{e:sheaves on stacks *} to
\begin{equation} \label{e:sheaves on stacks * shift}
\underset{S\to \CY}{\on{lim}}\!^{*,\on{shift}}\, \Shv(S),
\end{equation}
where the transition functors are now $f^*[2\dim(S_2/S_1)]$. 

\medskip

Now, the limit \eqref{e:sheaves on stacks * shift} is equivalent to \eqref{e:sheaves on stacks} term-wise. 

\sssec{}

For a stack $\CY$, we let $\Shv(\CY)^{\on{constr}}$ be the full subcategory of $\Shv(\CY)$ equal to
$\underset{S\to \CY}{\on{lim}}\, \Shv(S)^{\on{constr}}$ (with respect to any of the above three index
categories and either $-^!$ or $-^*$ version). 

\medskip

For an algebraic stack $\CY$, we let 
$$\ul\sfe_\CY\in \Shv(\CY)^{\on{constr}}$$
denote the constant sheaf. 

\sssec{}

Built into the definition of $\Shv(-)$ are the functors
$$f^!,f^*:\Shv(\CY_2)\to \Shv(\CY_1)$$
for a morphism $f:\CY_2\to \CY_1$ between algebraic stacks.

\medskip

In addition, the $(-_!,-^*)$ (resp.,  $(-_*,-^!)$ base change isomorphisms for maps between
schemes give rise to functors
$$f_!,f_*:\Shv(\CY_1)\to \Shv(\CY_2)$$
when $f$ is schematic. It follows from the construction that $f_!$ (resp., $f_*$) is the left (resp., right)
adjoint of $f^!$ (resp., $f^*$). 

\medskip

We will explain how extend this definition for maps that are not necessarily schematic
in \secref{sss:dir images 1}-\secref{sss:dir images 2} below. 

\sssec{} \label{sss:stacky}

As is explained in \cite[Sect. F.1.1]{AGKRRV}, for any algebraic stack $\CY$, the category 
$\Shv(\CY)$ is compactly generated. 
Namely, we can take as generators objects of the form $g_!(\CF)$ for 
$$g:S\to \CY, \quad \CF\in \Shv(S)^c, \quad S\in \affSch_{\on{ft}/k}.$$

\medskip

We have an inclusion
$$\Shv(\CY)^c\subset \Shv(\CY)^{\on{constr}},$$
but this inclusion is in general \emph{not} an equality. For example,
the object $\ul\sfe_\CY$ is not compact for $\CY=B\BG_m$. 

\medskip

It easy to see, however, that the subcategory $\Shv(\CY)^c$ is preserved by the
$\overset{*}\otimes$ operation with objects in $\Shv(\CY)^{\on{constr}}$
(see \eqref{e:Verdier duality ten} below). 

\sssec{} \label{sss:dir images 1}

Let $f:\CY_1\to \CY_2$ be a morphism between algebraic stacks. We define the functor
$$f_!:\Shv(\CY_1)\to \Shv(\CY_2)$$
as the left adjoint of $f^!$. 

\medskip

To show that $f_!$ exists, it is enough to show that it is defined on the generators of 
$\Shv(\CY_1)$. Taking the generators described in \secref{sss:stacky}, we have
$$f_!(g_!(\CF))\simeq (f\circ g)_!(\CF),$$
where the right-hand side is well-defined because the morphism $f\circ g$ is schematic.

\medskip

The same argument shows that the $-_!$ pushforward 
satisfies base change against the $-^*$ pullback. 

\sssec{} \label{sss:dir images 2}

For $f$ as above, we define the functor 
$$f_*:\Shv(\CY_1)\to \Shv(\CY_2)$$
as the right adjoint of $f^*$. 

\medskip

This functor exists for general reasons\footnote{I.e., Lurie's Adjoint Functor Theorem, which in particular
says that any colimit preserving functor between presentable DG categories admits a right adjoint.},
however it may be discontinuous (because the functor $f^*$ does not necessarily preserve compactness). 

\medskip

However,  $-_*$ pushforward does satisfy base change against $-^!$ pullback. This follows
by passing to right adjoints from the $(-_!,-^*)$ base change. 

\sssec{}

For any algebraic stack $\CY$, Verdier duality defines a contravariant self-equivalence
\begin{equation} \label{e:Verdier inv}
\BD^{\on{Verdier}}:(\Shv(\CY)^{\on{constr}})^{\on{op}}\to \Shv(\CY)^{\on{constr}}.
\end{equation} 

Its basic property is that for $\CF\in \Shv(\CY)^{\on{constr}}$
\begin{equation} \label{e:Verdier duality ten}
\CHom_{\Shv(\CY)}(\CF\overset{*}\otimes \CF_1, \CF_2)
\simeq \CHom_{\Shv(\CY)}(\CF_1, \BD^{\on{Verdier}}(\CF)\sotimes \CF_2), \quad \CF_1,\CF_2\in  \Shv(\CY).
\end{equation} 

We will denote by $\omega_\CY\in \Shv(\CY)$ the dualizing sheaf on $\CY$:
$$\BD^{\on{Verdier}}(\ul\sfe_\CY).$$

\ssec{Verdier-compatible stacks}

In this subsection we will assume that algebraic stacks are quasi-compact. 

\medskip

We review the notion of what it means for an algebraic stack to be Verdier compatible.
On the one hand, this property confers some particularly favorable properties on to $\Shv(\CY)$.
On the other hand, it is quite ubiquitous. 

\sssec{} \label{sss:Verdier compatible}

It is not clear whether for a general algebraic stack $\CY$, the Verdier involution \eqref{e:Verdier inv}
always sends $\Shv(\CY)^c$ to $\Shv(\CY)^c$. 

\medskip

If this is the case, following \cite[Sect. F.2.6]{AGKRRV}, we shall say that $\CY$ is \emph{Verdier compatible}. 
This property always holds under the assumption on algebraic stacks imposed in \secref{sss:stacks}, see 
\cite[Theorem F.2.8]{AGKRRV}.

\medskip

Thus, we will assume that all stacks involved are Verdier compatible. 

\sssec{} \label{sss:Verdier compatible bis}

Since Verdier duality swaps $-_*$ and $-_!$, we obtain that objects of the form $g_*(\CF)$, where 
$$g:S\to \CY, \quad \CF\in \Shv(S)^c, \quad S\in \affSch_{\on{ft}/k}$$
are compact in and generate $\Shv(\CF)$. 

\medskip

This implies that for a morphism $f:\CY_1\to \CY_2$, the functor $f_*$ sends 
$\Shv(\CY_1)^c$ to $\Shv(\CY_2)^c$. 

\medskip

This also implies that the subcategory $\Shv(\CY)^c$ is preserved by the
$\sotimes$ operation with objects in $\Shv(\CY)^{\on{constr}}$. 

\sssec{} \label{sss:renorm dir im}

For a map between stacks $f:\CY_1\to \CY_2$, the usual functor of
direct image $f_*: \Shv(\CY_1)\to \Shv(\CY_2)$ is not in general colimit-preserving. 
We define the \emph{renormalized} functor of direct image, denoted 
$$f_\blacktriangle:\Shv(\CY_1)\to \Shv(\CY_2),$$
to be the unique colimit-preserving functor such that restricts to $f_*$ on $\Shv(\CY_1)^c$. 
We have a tautologically defined natural transformation
$$f_\blacktriangle\to f_*,$$
which is an isomorphism if $f$ is schematic. 

\medskip

If follows from \secref{sss:Verdier compatible bis} that the functor $f_\blacktriangle$
preserves compactness and satisfies the projection formula. 

\medskip

This in turn implies that for a pair of composable
morphisms
$$\CY_1 \overset{f_{1,2}}\to \CY_2  \overset{f_{2,3}}\to \CY_3,$$
the natural transformation
$$(f_{2,3})_\blacktriangle\circ (f_{1,2})_\blacktriangle \to 
(f_{2,3}\circ f_{1,2})_\blacktriangle$$
is an isomorphism. 

\sssec{} \label{sss:renorm cochains}

For $\CY_1=\CY$ and $\CY_2=\on{pt}$, we obtain the renormalized functor 
of sheaf cochains, denoted
$$\on{C}^\cdot_\blacktriangle(\CY,-):\Shv(\CY)\to \Vect.$$

\medskip

By \secref{sss:renorm dir im}, for a map $f:\CY_1\to \CY_2$, we have
$$\on{C}^\cdot_\blacktriangle(\CY_1,-)\simeq \on{C}^\cdot_\blacktriangle(\CY_2,-)\circ f_\blacktriangle.$$

\sssec{}

For future reference, we claim:

\begin{lem} \label{l:blacktriangle bdd}
The functor $f_\blacktriangle:\Shv(\CY_1)\to \Shv(\CY_2)$ has a cohomological 
amplitude bounded on the right.
\end{lem}

\begin{proof}

Covering $\CY_2$ by a scheme (and using base change, see \secref{sss:base change triangle} below), 
we can assume that $\CY_2$ is a scheme $Y_2$. 

\medskip

By the assumption in \secref{sss:stacks}, we can assume that
$\CY_1$ is of the form $Y_1/H$, where $Y_1$ is a (quasi-compact) scheme and $H$
is an algebraic group. We can factor the map $f$ as
$$Y_1/H\to Y_2/H=Y_2\times BH\to Y_2.$$

The arrow $Y_1/H\to Y_2/H$ is schematic, so the $-_\blacktriangle$ functor is the same as  
$-_*$, and so has a finite cohomological amplitude. Hence, it remains to consider the morphism
of the form 
$$p_Y:Y\times BH\to Y.$$

Any object $\CF\in \Shv(Y\times BH)$ admits a finite canonical filtration with subquotients of the form
$$p_Y^*\circ q_Y^*(\CF)[n],$$ where $q_Y$ is the map $Y\to Y\times BH$. Hence, it remains to show
that 
$$\on{C}^\cdot_\blacktriangle(BH,\ul\sfe_{BH})\in \Vect$$
is bounded above.  

\medskip

However, this is a computation performed in \cite[Example 9.1.6]{DrGa0}.

\end{proof}

\sssec{} \label{sss:constr almost preserved} 

in general, the functors $f_*$ and $f_\blacktriangle$ do not preserve constructibility. However,
the argument in \lemref{l:blacktriangle bdd} shows that both these functors send constructible
objects to objects with constructible perverse cohomologies. 

\medskip

Furthermore, we have the following assertion:

\begin{prop} \label{p:triang and *}
For a constructible object $\CF\in \Shv(\CY_1)$, the following conditions are equivalent:

\smallskip

\noindent{\em(i)} The map $f_\blacktriangle(\CF)\to f_*(\CF)$
is an isomorphism;

\smallskip

\noindent{\em(ii)} The object $f_\blacktriangle(\CF)\in \Shv(\CY_2)$ is cohomologically bounded on the left;

\smallskip

\noindent{\em(ii)} The object $f_*(\CF)\in \Shv(\CY_2)$ is cohomologically bounded on the right.

\end{prop}

\begin{proof}

Since both objects $f_\blacktriangle(\CF)$ and $f_*(\CF)$ have constructible cohomologies, it is enough to 
show that for every $k$-point 
$$\on{pt} \overset{\bi_{y_2}}\longrightarrow \CY_2,$$
the following conditions are equivalent: 

\medskip

\noindent{(i)} The map $\bi_{y_2}^!\circ f_\blacktriangle(\CF)\to \bi_{y_2}^!\circ f_*(\CF)$
is an isomorphism;

\smallskip

\noindent{(ii)} The object $\bi_{y_2}^!\circ f_\blacktriangle(\CF)\in \Vect$ is cohomologically bounded on the left;

\smallskip

\noindent{(ii)} The object $\bi_{y_2}^!\circ f_*(\CF)\in \Vect$ is cohomologically bounded on the right.

\medskip

Note that both functors $f_\blacktriangle$ and $f_*$ satisfy base change against !-pullbacks: 
this is obvious for $f_*$, e.g., by passing to left adjoints, and for $f_\blacktriangle$
this is established in \secref{sss:base change triangle} below.

\medskip

This reduces the proposition to the case when $\CY_2=\on{pt}$. Write $\CY_1=:\CY=Z/H$, where $Z$ is
a scheme, and let is factor the projection $\CY\to \on{pt}$ as
$$\CY\overset{g}\to \on{pt}/H \to \on{pt}.$$

The map $g$ schematic, and hence $g_\blacktriangle\to g_*$ is an isomorphism. This allows us to replace
$\CY$ by $\on{pt}/H$. However, in the latter case, the required assertion was established in \cite[Proposition 10.4.7]{DrGa0}.

\end{proof}

\ssec{Base change maps}

Let
$$
\CD
\CY'_1 @>{f'}>> \CY'_2 \\
@V{g_1}VV @VV{g_2}V \\
\CY_1 @>{f}>> \CY_2
\endCD
$$
be a Cartesian diagram of quasi-compact algebraic stacks. 

\medskip

In this subsection we will construct certain natural transformations that have to
do with applying various direct/inverse image functors along the arrows in the
above diagram. 

\sssec{}  \label{sss:base change triangle}

First, we claim that there is a canonical isomorphism
\begin{equation} \label{e:base change triangle}
g_2^! \circ f_\blacktriangle \simeq f'_\blacktriangle \circ g_1^!.
\end{equation}

First, we consider the case when $g_2$ is schematic (and hence so is $g_1$).

\medskip

Note that for a schematic map $g:\CY'\to \CY$ the functor $g^*$ preserves compactness. Indeed,
its right adjoint $g_*$ is continuous. Hence, $g^!$ also preserves compactness, by Verdier-compatibility. 

\medskip

Hence, \eqref{e:base change triangle} is obtained by ind-extending the restriction to compact objects 
the usual base change isomorphism
$$g_2^! \circ f_*\simeq f'_*\circ g_1^!.$$

\sssec{} \label{sss:base change triangle aux}

For a general $g_2$ we proceed as follows. 

\medskip

It suffices to construct a compatible system of isomorphisms 
\begin{equation} \label{e:base change triangle aux}
h_2^! \circ g_2^! \circ f_\blacktriangle \simeq h_2^! \circ  f'_\blacktriangle \circ g_1^!
\end{equation}
in the Cartesian diagrams 
\begin{equation} \label{e:com diag aux}
\CD
\CY''_1 @>{f''}>> \CY''_2 \\
@V{h_1}VV @VV{h_2}V \\
\CY'_1 @>{f'}>> \CY'_2 \\
@V{g_1}VV @VV{g_2}V \\
\CY_1 @>{f}>> \CY_2,
\endCD
\end{equation}
where $h_2:\CY''_2\to \CY'_2$ is a smooth map from an affine scheme. 

\medskip

We have
$$
h_2^! \circ g_2^! \circ f_\blacktriangle \simeq
(g_2\circ h_2)^!  \circ f_\blacktriangle  \overset{g_2\circ h_2\text{ is schematic}}\simeq 
f''_\blacktriangle \circ (g_1\circ h_1)^! \simeq 
f''_\blacktriangle \circ h_1^! \circ g_1^! \overset{h_2\text{ is schematic}}\simeq 
h_2^! \circ  f'_\blacktriangle \circ g_1^!.$$

\sssec{} \label{sss:strange base change 1}

Next, we claim that there is natural transformation
$$(g_2)_! \circ f'_\blacktriangle \to f_\blacktriangle \circ  (g_1)_!.$$

Indeed, it is obtained by ind-extending the restriction to compact objects 
the natural transformation
$$(g_2)_! \circ f'_* \to f_* \circ  (g_1)_!$$
arising by adjunction from the base change isomorphism $f'_*\circ g_1^!\simeq g_2^!\circ f_*$.

\sssec{} \label{sss:strange base change 2}

Finally, we claim that there is a natural transformation
\begin{equation} \label{e:strange base change}
g_2^* \circ f_\blacktriangle \to f'_\blacktriangle \circ (g_1)^*.
\end{equation}

First, the if $g_2$ is schematic (and hence so is $g_1$, and so the functors $g_2^*$ and $g_1^*$ preserve compactness),
the natural transformation \eqref{e:strange base change} is obtained by ind-extending the restriction to compact objects 
the natural transformation
$$g_2^* \circ f_* \to f'_* \circ g_1^*$$
arising by adjunction from the isomorphism $f_*\circ (g_1)_*\simeq (g_2)_*\circ f'_*$.

\medskip

Note that \eqref{e:strange base change} is an isomorphism if $g_2$ is smooth. 

\sssec{}
 
For a general $g_2$ we proceed as follows:

\medskip

It suffices to construct a compatible system of natural transformations
\begin{equation} \label{e:strange base change aux}
h_2^*\circ g_2^* \circ f_\blacktriangle \to h_2^*\circ f'_\blacktriangle \circ g_1^*
\end{equation}
for $h_2$ as in \secref{e:com diag aux}. 

\medskip

We rewrite the left-hand side in \eqref{e:strange base change aux} as
$(g_2\circ h_2)^* \circ f_\blacktriangle$, and it admits a natural transformation to
$f''_\blacktriangle \circ (g_1\circ h_1)^*$, since the map $g_2\circ h_2$ is schematic.

\medskip

Finally, we rewrite $f''_\blacktriangle \circ (g_1\circ h_1)^*$ as
$$f''_\blacktriangle \circ h_1^* \circ g_1^* \overset{h_2\text{ is smooth}}\simeq 
h_2^* \circ f'_\blacktriangle \circ g_1^*,$$
which is the right-hand side in \eqref{e:strange base change aux}.

\ssec{Verdier self-duality} \label{ss:Verdier}

In this subsection we continue to assume that our substacks are quasi-compact 
(and Verdier-compatible).

\medskip

We will show how the Verdier (anti)-involution $\Shv(\CY)^c$
gives rise to a self-duality of $\Shv(\CY)$ as a DG category. 

\sssec{} \label{sss:Verdier}

By assumption, the functor \eqref{e:Verdier inv} induces an equivalence 
\begin{equation} \label{e:Verdier functor}
\BD^{\on{Verdier}}:(\Shv(\CY)^c)^{\on{op}}\to \Shv(\CY)^c.
\end{equation}

Hence, we obtain an identification
\begin{equation} \label{e:Verdier}
\Shv(\CY)^\vee\simeq \Shv(\CY),
\end{equation}
see \secref{sss:comp}. 

\sssec{}

We claim that the counit of the duality \eqref{e:Verdier} is given by the functor $\on{ev} _\CY$
\begin{equation} \label{e:pairing sheaves Y}
\Shv(\CY)\otimes \Shv(\CY) \overset{\sotimes}\to \Shv(\CY) \overset{\on{C}^\cdot_\blacktriangle(\CY,-)}\longrightarrow \Vect.
\end{equation}

Indeed, this follows from \eqref{e:Verdier duality ten}. 

\sssec{} \label{sss:unit Verdier}

The unit of the duality \eqref{e:Verdier} is an object that we denote
$$\on{u}_{\Shv(\CY)}\in \Shv(\CY)\otimes \Shv(\CY).$$

This object should not be confused with the object
\begin{equation} \label{u:notation}
\on{u}_\CY=(\Delta_\CY)_*(\omega_\CY)\in \Shv(\CY\times \CY),
\end{equation}
see \secref{sss:u shv}. 

\sssec{}

It follows from the projection formula that for a map $f:\CY_1\to \CY_2$ between algebraic stacks, we have
$$(f^!)^\vee\simeq f_\blacktriangle,$$
where we identify
$$\Shv(\CY_1)^\vee\simeq \Shv(\CY_1) \text{ and } \Shv(\CY_2)^\vee\simeq \Shv(\CY_2)$$
by means of \eqref{e:Verdier}.  

\ssec{Specifying singular support} 

In this subsection we will see how the notions reviewed in the previous subsection interact with
the condition imposed by singular support. 

\sssec{} \label{sss:shv N}

Let $\CN$ be a conical Zariski-closed subset of $T^*(\CY)$ (see \cite[Sects. F.6.1-F.6.2]{AGKRRV}). 
In this case, following \cite[Sect. F.6.3]{AGKRRV},
we define a full subcategory
$$\Shv_\CN(\CY)\subset \Shv(\CY).$$

\medskip

For $\CN$ being the zero-section $\{0\}$ and $\CY$ smooth, we will use the notation
$$\qLisse(\CY):=\Shv_{\{0\}}(\CY).$$

\sssec{} \label{sss:constraccess}

Following \cite[Sect. F.7]{AGKRRV}, we shall say that the pair $(\CY,\CN)$ is 
\emph{constraccessible} if $\Shv_\CN(\CY)$ is generated by objects that are compact in $\Shv(\CY)$.

\medskip

In general, it may happen that $\Shv_\CN(\CY)$ is compactly generated as an abstract DG category,
but its compact generators are \emph{not} compact as objects of $\Shv(\CY)$. For example, this happens for $\CY=\BP^1$,
$\CN=\{0\}$, see \cite[Sect. E.2.6]{AGKRRV}.

\begin{rem}
The above definition of constraccessibility is slightly more restrictive than the one given in 
\cite[Definition F.7.5]{AGKRRV}. However, the two definitions agree under our assumptions
on algebraic stacks (see \secref{sss:stacks}) by \cite[Proposition F.7.9]{AGKRRV}. 

\medskip

If $\CY$ is non-quasi-compact, but is $\CN$-truncatable (see \secref{sss:N-trunc} below), 
then the two definitions agree by \cite[Corollary F.8.11]{AGKRRV}.

\end{rem}

\sssec{} \label{sss:duality adapted}

Assume for a moment that $\CY$ is quasi-compact. 

\medskip

We shall say that the pair $(\CY,\CN)$ is \emph{duality-adapted} if the restriction of \eqref{e:pairing sheaves Y} along
$$\Shv_\CN(\CY)\otimes \Shv_\CN(\CY)\to \Shv(\CY)\otimes \Shv(\CY)$$
defines the counit of a self-duality. 

\medskip

This is always the case when $(\CY,\CN)$ is constraccessible: indeed, 
the resulting contravariant self-equivalence
$$(\Shv_\CN(\CY)^c)^{\on{op}}\to \Shv_\CN(\CY)^c$$
is induced by restricting the Verdier duality functor $\BD^{\on{Verdier}}$ to
$$\Shv_\CN(\CY)^c\simeq \Shv_\CN(\CY)\cap \Shv(\CY)^c\subset \Shv(\CY)^{\on{constr}}.$$

\medskip

However, the pair $(\CY,\CN)$ may be duality-adapted \emph{without being constraccessible}, see again 
\cite[Sect. E.2.6 and Corollary E.4.7]{AGKRRV}.

\sssec{} \label{sss:duality adapted curves}

According to \cite[Corollary E.4.7]{AGKRRV}, if $X$ is a smooth algebraic curve, then the pair $(X,\{0\})$ is
duality-adapted.

\ssec{Categorical K\"unneth formulas}

Let $\CY_1$ and $\CY_2$ be a pair of algebraic stacks. External tensor product defines a functor
\begin{equation} \label{e:Kunneth 0}
\Shv(\CY_1)\otimes  \Shv(\CY_2)\overset{\boxtimes}\to \Shv(\CY_1\times \CY_2),
\end{equation}
which is easily seen to be fully faithful. However, \eqref{e:Kunneth 0} is rarely an equivalence.

\medskip

Categorical K\"unneth formulas are assertions to the effect that \eqref{e:Kunneth 0} becomes
an isomorphism after adjusting both sides. One such assertion is formulated here as 
\thmref{t:Kunneth}. 

\sssec{}

Let $\CN_i\subset T^*(\CY_i)$ be conical Zariski-closed subsets. Then the functor \eqref{e:Kunneth 0}
gives rise to a functor
\begin{equation} \label{e:Kunneth 1}
\Shv_{\CN_1}(\CY_1)\otimes  \Shv_{\CN_2}(\CY_2)\to \Shv_{\CN_1\times \CN_2}(\CY_1\times \CY_2).
\end{equation}

If one of the categories $\Shv_{\CN_1}(\CY_1)$ or $\Shv_{\CN_2}(\CY_2)$ is dualizable, then
\eqref{e:Kunneth 1} is also fully faithful.

\sssec{} \label{sss:times 1/2-dim}

Throughout the paper, we use the following notation. Let $\CN_1\subset T^*(\CY_1)$ be a closed conical subset.

\medskip

Let 
$$\Shv_{\CN_1\times \frac{1}{2}\on{-dim}}(\CY_1\times \CY_2)\subset \Shv(\CY_1\times \CY_2)$$
be the full subcategory consisting of objects $\CF$ with the property that for every $m$ and every constructible sub-object
$\CF'$ of $H^m(\CF')$, the singular support of $\CF'$ is contained in a subset of the form
$$\CN_1\times \CN_2\subset T^*(\CY_1\times \CY_2),$$ where $\CN_2\subset T^*(\CY_2)$ is \emph{half-dimensional}. 

\sssec{}

We have the obvious inclusion
$$\Shv_{\CN_1\times \frac{1}{2}\on{-dim}}(\CY_1\times \CY_2)\subset 
\Shv_{\CN_1\times T^*(\CY_2)}(\CY_1\times \CY_2).$$

If $\CN_1$ is Lagrangian \emph{and} $\on{char}(k)=0$, then the above inclusion is an equality: this follows from the Lagrangian
property of the singular support. But if $\on{char}(k)\neq 0$, this may be a proper inclusion.

\medskip

Note that the image of the (fully faithful) functor
$$\Shv_{\CN_1}(\CY_1)\otimes  \Shv(\CY_2)\to \Shv(\CY_1\times \CY_2)$$
is contained in $\Shv_{\CN_1\times \frac{1}{2}\on{-dim}}(\CY_1\times \CY_2)$. 

\sssec{}

Assume now that $\CY_1=Y_1$ is a smooth and proper scheme. In this case, we have the following assertion
(see \cite[Theorems E.9.5 and F.9.7]{AGKRRV}):

\begin{thm} \hfill \label{t:Kunneth}

\smallskip

\noindent{\em(a)}
Assume that $\qLisse(Y_1)$ is dualizable as a DG category. Then the functor 
$$\qLisse(Y_1)\otimes \Shv(\CY_2)\to \Shv_{\{0\}\times \frac{1}{2}\on{-dim}}(Y_1\times \CY_2)$$
is an equivalence. 

\smallskip

\noindent{\em(b)}
Assume that the pair $(Y_1,\{0\})$ is duality-adapted. Then for a half-dimensional closed conical 
subset $\CN_2\subset T^*(\CY_2)$, the functor
$$\qLisse(Y_1)\otimes \Shv_{\CN_2}(\CY_2)\to \Shv_{\{0\}\times \CN_2}(Y_1\times \CY_2)$$
is an equivalence. 

\end{thm}

\section{Functors defined by kernels and miraculous duality} \label{s:ker}

Some of the key methods employed in the bulk of this paper have to do
with the formalism of \emph{functors defined by kernels}. In this subsection
we review the relevant theory.

\medskip

We should also mentioned that for the main results of this paper, all we need is
the contents of \secref{ss:ker}. Other subsections of this section are needed for
the material in \secref{s:adj}. 

\ssec{Functors defined by kernels} \label{ss:ker}

In this subsection we introduce what we mean by functors defined by kernels. 

\sssec{} \label{sss:ker}

Let $\CY_1$ and $\CY_2$ be quasi-compact algebraic stacks. The 
\emph{category of functors $\Shv(\CY_1)\to \Shv(\CY_2)$ defined by kernels}
is by definition
$$\Shv(\CY_1\times \CY_2).$$

\medskip

An object $\CQ\in \Shv(\CY_1\times \CY_2)$ defines an actual functor
\begin{equation} \label{e:functor defined by kernel}
\sQ:\Shv(\CY_1)\to \Shv(\CY_2), \quad \CF\mapsto (p_2)_\blacktriangle(p_1^!(\CF)\sotimes \CQ).
\end{equation}

\medskip

We shall say that a given functor $\Shv(\CY_1)\to \Shv(\CY_2)$ is defined by a kernel if it comes
from an object in $\Shv(\CY_1\times \CY_2)$ by the above procedure. 

\sssec{} \label{sss:sigma}

For an object $\CQ\in \Shv(\CY_1\times \CY_2)$ we will denote by $\CQ^\sigma$ the object of 
$\Shv(\CY_2\times \CY_1)$ obtained by swapping the two factors.

\medskip

The corresponding functor $\sQ^\sigma$ is the dual functor of $\sQ$ with respect to the Verdier 
self-duality of $\Shv(\CY_i)$:
$$\on{ev} _{\CY_1}(\CF_1\otimes \sQ^\sigma(\CF_2))\simeq \on{ev} _{\CY_2}(\sQ(\CF_1)\otimes \CF_2)$$
(see \eqref{e:pairing sheaves Y} for the notation $\on{ev} _\CY$). 

\sssec{} \label{sss:ker Z}

More generally, for any algebraic stack $\CZ$ we have a well-defined 
functor
\begin{equation} \label{e:functors ten Id} 
\on{Id}_\CZ\boxtimes \sQ:\Shv(\CZ\times \CY_1)\to \Shv(\CZ\times \CY_2), \quad 
\CF\mapsto (p_{\CZ,\CY_2})_\blacktriangle(p_{\CZ,\CY_1}^!(\CF)\sotimes p_{\CY_1,\CY_2}^!(\CQ)),
\end{equation} 
where $p_{\CZ,\CY_1}$, $p_{\CZ,\CY_2}$ and $p_{\CY_1,\CY_2}$ are the maps
$$\CZ\times \CY_1\times \CY_2\to \CY_1\times \CY_2,\,\, 
\CZ\times \CY_1\times \CY_2\to \CZ\times \CY_1 \text{ and } \CZ\times \CY_1\times \CY_2\to \CZ\times \CY_2,$$
respectively.  

\medskip

Note that we have a commutative diagram
\begin{equation} \label{e:boxtimes vs otimes}
\CD
\Shv(\CZ)\otimes \Shv(\CY_1) @>>> \Shv(\CZ\times \CY_1) \\
@V{\on{Id}_{\Shv(\CZ)}\otimes \sQ}VV @VV{\on{Id}_\CZ\boxtimes \sQ}V \\
\Shv(\CZ)\otimes \Shv(\CY_2) @>>> \Shv(\CZ\times \CY_2).
\endCD
\end{equation} 

\sssec{}

It follows from the description of compact generators in \secref{sss:Verdier compatible bis} that if $\CQ$ is compact, then 
the functors $\on{Id}_\CZ\boxtimes \sQ$ send compact objects to compact objects. 

\sssec{} \label{sss:compat kernels new}

Using \secref{sss:base change triangle}, we obtain that 
the functors $\on{Id}_\CZ\boxtimes \sQ$ commute with the following operations:

\begin{itemize} 

\item For a map $f:\CZ'\to \CZ$, the diagram
$$
\CD
\Shv(\CZ'\times \CY_1)  @>{\on{Id}_{\CZ'}\boxtimes \sQ}>> \Shv(\CZ'\times \CY_2) \\
@V{(f\times \on{id})_\blacktriangle}VV @VV{(f\times \on{id})_\blacktriangle}V \\
\Shv(\CZ\times \CY_1)  @>{\on{Id}_\CZ\boxtimes \sQ}>> \Shv(\CZ\times \CY_2) 
\endCD
$$
commutes.

\smallskip

\item For a map $f:\CZ'\to \CZ$, the diagram
$$
\CD
\Shv(\CZ'\times \CY_1)  @>{\on{Id}_{\CZ'}\boxtimes \sQ}>> \Shv(\CZ'\times \CY_2) \\
@A{(f\times \on{id})^!}AA @AA{(f\times \on{id})^!}A \\
\Shv(\CZ\times \CY_1)  @>{\on{Id}_\CZ\boxtimes \sQ}>> \Shv(\CZ\times \CY_2) 
\endCD
$$
commutes.

\smallskip

\item For $\CF\in \Shv(\CZ)$, the diagram
$$
\CD
\Shv(\CZ\times \CZ'\times \CY_1)  @>{\on{Id}_{\CZ\times \CZ'}\boxtimes \sQ}>> \Shv(\CZ\times \CZ'\times \CY_2) \\
@A{\CF\boxtimes-}AA @AA{\CF\boxtimes-}A \\
\Shv(\CZ'\times \CY_1)  @>{\on{Id}_{\CZ'}\boxtimes \sQ}>> \Shv(\CZ\times \CY_2) 
\endCD
$$
commutes.

\end{itemize} 

Vice versa, a system of functors $\on{Id}_\CZ\boxtimes \sQ$ as in \eqref{e:functors ten Id} that satisfies the above three properties,
satisfying an appropriate system of compatibilities for diagrams of morphisms between the $\CZ$'s, 
comes from a uniquely defined object $\CQ\in \Shv(\CY_1\times \CY_2)$. Namely, $\CQ$ is recovered as
$$(\on{Id}_{\CY_1}\boxtimes \sQ)(\on{u}_{\CY_1}),$$
where $\on{u}_{\CY_1}$ is as in \eqref{u:notation}.

\sssec{}

Let $\CQ\in \Shv(\CY_1\times \CY_2)$ and $\CQ'\in \Shv(\CY'_1\times \CY'_2)$ be a pair of objects.
We can consider
$$(\CQ\boxtimes \CQ')^{\sigma_{2,3}}\in \Shv((\CY_1\times \CY'_1)\times (\CY_2\times \CY'_2)).$$

We will denote the corresponding functor $\Shv(\CY_1\times \CY'_1)\to \Shv(\CY_2\times \CY'_2)$ by $\sQ\boxtimes \sQ'$.
We have
\begin{equation} \label{e:order does not matter}
(\sQ\boxtimes \on{Id}_{\CY'_2})\circ (\on{Id}_{\CY_1}\boxtimes \sQ') \simeq \sQ\boxtimes \sQ'\simeq 
(\on{Id}_{\CY_2}\boxtimes \sQ') \circ (\sQ\boxtimes \on{Id}_{\CY'_1}).
\end{equation}

\sssec{} \label{sss:comp corr}

For $\CQ_{1,2}\in \Shv(\CY_1\times \CY_2)$ and $\CQ_{2,3}\in \Shv(\CY_2\times \CY_3)$ we have a naturally defined composition
$$\CQ_{1,3}:=\CQ_{2,3}\star \CQ_{1,2}:=(p_{1,3})_\blacktriangle (p_{2,3}^!(\CQ_{2,3})\sotimes p_{1,2}^!(\CQ_{1,2}))\in \Shv(\CY_1\times \CY_3).$$

From \secref{sss:base change triangle}, we obtain that for any $\CZ$, we have
$$\on{Id}_\CZ\boxtimes \sQ_{1,3}\simeq (\on{Id}_\CZ\boxtimes \sQ_{2,3})\circ (\on{Id}_\CZ\boxtimes \sQ_{1,2}).$$

Note that we have
\begin{equation} \label{e:green formula}
(\on{Id}_{\CY_1}\boxtimes \sQ_{2,3})(\CQ_{1,2})
\simeq \CQ_{1,3} \simeq (\sQ_{1,2} \boxtimes \on{Id}_{\CY_3})(\CQ_{2,3}).
\end{equation}

\sssec{} \label{sss:2-categ}

We can form an $(\infty,2)$-category whose objects are algebraic stacks, and whose categories of 1-morphisms
are $\Shv(\CY_1\times \CY_2)$, with composition functors defined as above.  (For the purposes of this paper,
we only need the \emph{homotopy category} of this $(\infty,2)$-category, i.e., we can replace the \emph{spaces}
of 2-morphisms by their sets of connected components.) 

\medskip

For a given $\CY$, the unit 1-morphism in the above 2-category is given by the object
\begin{equation} \label{u:notation again}
\on{u}_\CY=(\Delta_\CY)_*(\omega_\CY)\in \Shv(\CY\times \CY)
\end{equation}
of \eqref{u:notation}. 

\ssec{A discontinuous version}

\sssec{} \label{sss:ker disc}

Note that in addition to the functor \eqref{e:functor defined by kernel}, one can consider the functor
\begin{equation} \label{e:functor defined by kernel disc}
\sQ_{\on{disc}}:\Shv(\CY_1)\to \Shv(\CY_2), \quad \CF\mapsto (p_2)_*(p_1^!(\CF)\sotimes \CQ).
\end{equation}

We have a tautologically defined map
\begin{equation} \label{e:ker to disc}
\sQ \to \sQ_{\on{disc}}. 
\end{equation}

It follows from the description of compact generators of $\Shv(-)$ in \secref{sss:Verdier compatible bis} and
the projection formula that 
the natural transformation \eqref{e:ker to disc} is an isomorphism when evaluated
on compact objects. 

\medskip

In particular, both sides of \eqref{e:ker to disc} send compact objects to constructible ones.  

\medskip

Similarly, \eqref{e:ker to disc} is an isomorphism if $\CQ$ is compact as an object of $\Shv(\CY_1\times \CY_2)$. 

\medskip

If $\CQ$ is just constructible, in general, neither side in \eqref{e:ker to disc} preserves constructibility. However, the functor $\sQ$
has a cohomological amplitude bounded on the right, and the functor $\sQ_{\on{disc}}$
has a cohomological amplitude bounded on the left, and both functors send constructible objects to
objects with constructible perverse cohomologies, see \secref{sss:constr almost preserved}. 

\sssec{} \label{sss:boxtimes disc}

We can similarly consider the (a priori, discontinuous) functors
$$\on{Id}_\CZ\boxtimes \sQ_{\on{disc}}:\Shv(\CZ\times \CY_1)\to \Shv(\CZ\times \CY_2),$$
namely
$$(\on{Id}_\CZ\boxtimes \sQ_{\on{disc}})(\CF):=
(p_{\CZ,\CY_2})_*(p_{\CZ,\CY_1}^!(\CF)\sotimes p_{\CY_1,\CY_2}^!(\CQ)).$$

The functors $\on{Id}_\CZ\boxtimes \sQ_{\on{disc}}$ make the diagrams similar to ones in \secref{sss:compat kernels new} commute, where 
in the first diagram we use $-_*$ instead of $-_\blacktriangle$. 

\medskip

We have the natural transformations
\begin{equation} \label{e:disc}
\on{Id}_\CZ\boxtimes \sQ\to \on{Id}_\CZ\boxtimes \sQ_{\on{disc}}.
\end{equation}

\sssec{} \label{sss:safe}

We shall say that $\sQ$ is \emph{safe} (as a functor defined by a kernel) if the natural transformations
\eqref{e:disc} are isomorphisms. 

\medskip

By definition, $\sQ$ is safe if and only if the functors $\on{Id}_\CZ\boxtimes \sQ_{\on{disc}}$ are 
actually continuous. 

\medskip

For example, if $\CQ$ is compact as an object of $\Shv(\CY_1\times \CY_2)$, then $\sQ$ is safe. 

\medskip

In addition, $\sQ$ is automatically safe if $\CY_1$ is a scheme. 

\sssec{} \label{sss:boxtimes disc bis}

It follows from \propref{p:triang and *} that if $\CQ$ and $\CF$ are both constructible, 
the following conditions are equivalent:

\smallskip

\noindent{(i)} The natural transformation \eqref{e:disc} becomes an isomorphism when evaluated on $\CF$;

\smallskip

\noindent{(ii)} The object $(\on{Id}_\CZ\boxtimes \sQ)(\CF)$ is cohomologically bounded on the left;

\smallskip

\noindent{(iii)} The object $(\on{Id}_\CZ\boxtimes \sQ_{\on{disc}})(\CF)$ is cohomologically bounded on the right.

\medskip

In particular, we obtain that if $\CQ$ is constructible and $\sQ$ safe, then the functors in \eqref{e:disc} preserve constructibility
and are of finite cohomological dimension. 

\sssec{} \label{sss:conv disc}

Given two objects $\CQ_{1,2}\in \Shv(\CY_1\times \CY_2)$ and $\CQ_{2,3}\in \Shv(\CY_2\times \CY_3)$ we can form
$$\CQ_{1,3}:=\CQ_{2,3}\star_{\on{disc}} \CQ_{1,2}:=(p_{1,3})_* (p_{2,3}^!(\CQ_{2,3})\sotimes p_{1,2}^!(\CQ_{1,2}))]\in \Shv(\CY_1\times \CY_3).$$

We have 
$$\on{Id}_\CZ\boxtimes (\sQ_{1,3})_{\on{disc}}\simeq (\on{Id}_\CZ\boxtimes (\sQ_{2,3})_{\on{disc}})\circ (\on{Id}_\CZ\boxtimes (\sQ_{1,2})_{\on{disc}}).$$

The natural transformation \eqref{e:disc} gives rise to a map
$$\CQ_{2,3}\star \CQ_{1,2}\to \CQ_{2,3}\star_{\on{disc}} \CQ_{1,2},$$
which is an isomorphism if either $\CQ_{2,3}$ or $\CQ_{1,2}^\sigma$ is safe. 

\ssec{Functors \emph{co}-defined by kernels}

In this subsection we explore the Verdier-dual notion: that of functors \emph{codefined} by kernels. 

\medskip

A special family of interest is formed by functors that are both \emph{defined and codefined} by kernels. 

\sssec{} \label{sss:codefined}

Starting from an object $\CP\in \Shv(\CY_1\times \CY_2)$ we can produce another family of functors, denoted
\begin{equation} \label{e:functors ten Id co} 
\on{Id}_\CZ\boxtimes \sP^l:\Shv(\CZ\times \CY_1)\to \Shv(\CZ\times \CY_2), \quad 
\CF\mapsto (p_{\CZ,\CY_2})_!(p_{\CZ,\CY_1}^*(\CF)\overset{*}\otimes p_{\CY_1,\CY_2}^*(\CP)).
\end{equation} 

We will refer to this family as functors \emph{codefined} by $\CP$. They satisfy compatibility
conditions parallel to ones in \secref{sss:compat kernels new} with $-^!$ replaced by $-^*$ and $-_\blacktriangle$ 
replaced by $-_!$. 

\sssec{} \label{sss:left convolution}

Let $\CP_{1,2}\in \Shv(\CY_1\times \CY_2)$ and $\CP_{2,3}\in \Shv(\CY_2\times \CY_3)$ be a pair of objects.
Define
$$\CP_{1,3}:=\CP_{2,3}\overset{l}\star \CP_{1,2}:=(p_{1,3})_! (p_{2,3}^*(\CQ_{2,3})\overset{*}\otimes p_{1,2}^*(\CQ_{1,2}))\in \Shv(\CY_1\times \CY_3).$$

We have
$$\on{Id}_\CZ\boxtimes \sP^l_{1,3}\simeq (\on{Id}_\CZ\boxtimes \sP^l_{2,3})\circ (\on{Id}_\CZ\boxtimes \sP^l_{1,2}).$$

Furthermore,
\begin{equation} \label{e:green formula left}
(\on{Id}_{\CY_1}\boxtimes \sP^l_{2,3})(\CP_{1,2})
\simeq \CP_{1,3} \simeq (\sP^l_{1,2} \boxtimes \on{Id}_{\CY_3})(\CP_{2,3}).
\end{equation}

\sssec{} \label{sss:left to right Verdier}

Let $\CQ\in \Shv(\CY_1\times \CY_2)$ be constructible, and set $\CP:=\BD^{\on{Verdier}}(\CQ)$. Let $\CF\in \Shv(\CZ\times \CY_1)^{\on{constr}}$
be an object satisfying the equivalent conditions of \secref{sss:boxtimes disc bis}. 

\medskip

We obtain that in this case
$$(\on{Id}_\CZ\boxtimes \sP^l)(\BD^{\on{Verdier}}(\CF))\simeq
\BD^{\on{Verdier}}\left((\on{Id}_\CZ\boxtimes \sQ_{\on{disc}})(\CF)\right) \simeq 
\BD^{\on{Verdier}}\left((\on{Id}_\CZ\boxtimes \sQ)(\CF)\right).$$

\sssec{} \label{sss:left to right}

Let us be given two objects
$$\CQ\in \Shv(\CY_1\times \CY_2) \text{ and } \wt\CP\in \Shv(\wt\CY_1\times \wt\CY_2).$$

From Sects. \ref{sss:strange base change 1}-\ref{sss:strange base change 2}, we obtain a natural transformation 
\begin{equation} \label{e:left to right}
(\on{Id}_{\CY_2}\boxtimes \wt\sP^l)\circ (\sQ\boxtimes \on{Id}_{\wt\CY_1}) \to 
(\sQ\boxtimes \on{Id}_{\wt\CY_2}) \circ (\on{Id}_{\CY_1}\boxtimes \wt\sP^l),
\quad \Shv(\CY_1\times \wt\CY_1)\rightrightarrows \Shv(\CY_2\times \wt\CY_2),
\end{equation} 
i.e., a 2-morphism in the diagram
\begin{equation} \label{e:left to right xy}
\xy
(0,0)*+{\Shv(\CY_1\times \wt\CY_1)}="A";
(50,0)*+{\Shv(\CY_2\times \wt\CY_1)}="B";
(0,-30)*+{\Shv(\CY_1\times \wt\CY_2)}="C";
(50,-30)*+{\Shv(\CY_2\times \wt\CY_2)}="D";
{\ar@{->}^{\sQ\boxtimes \on{Id}_{\wt\CY_1}} "A";"B"};
{\ar@{->}_{\sQ\boxtimes \on{Id}_{\wt\CY_2}} "C";"D"};
{\ar@{->}_{\on{Id}_{\CY_1}\boxtimes \wt\sP^l} "A";"C"};
{\ar@{->}^{\on{Id}_{\CY_2}\boxtimes \wt\sP^l} "B";"D"};
{\ar@{=>} "B";"C"};
\endxy
\end{equation} 

\sssec{}

Similarly (but more tautologically), we obtain a natural transformation 
\begin{equation} \label{e:left to right disc}
(\on{Id}_{\CY_2}\boxtimes \wt\sP^l)\circ (\sQ_{\on{disc}}\boxtimes \on{Id}_{\wt\CY_1}) \to 
(\sQ_{\on{disc}}\boxtimes \on{Id}_{\wt\CY_2}) \circ (\on{Id}_{\CY_1}\boxtimes \wt\sP^l).
\end{equation} 

\sssec{} \label{sss:defined and codefined}

In particular, given a functor $\sQ$ defined by a kernel, the diagrams in \secref{sss:compat kernels new} 
with $-^!$ replaced by $-^*$ and $-_\blacktriangle$ replaced by $-_!$, \emph{commute up to natural transformations}.

\medskip

We shall say that $\sQ$ is \emph{defined and codefined by kernels} if these natural transformations
are isomorphisms. In particular, a functor defined and codefined by a kernel gives rise to a functor 
codefined by a kernel, i.e., it corresponds to an object $\CP\in \Shv(\CY_1\times \CY_2)$, so that
\begin{equation} \label{e:left and right}
\on{Id}_\CZ\boxtimes \sQ \simeq \on{Id}_\CZ\boxtimes \sP^l
\end{equation} 
for all $\CZ$, in a way compatible with $!-$ and $*-$ pullbacks and $!-$ and $\blacktriangle-$ 
pushforwards. 

\medskip

Note that in this case, the natural transformations \eqref{e:left to right} are the \emph{isomorphisms}
$$(\on{Id}_{\CY_2}\boxtimes \wt\sP^l)\circ (\sP^l\boxtimes \on{Id}_{\wt\CY_1}) \simeq \sP^l\boxtimes \wt\sP^l\simeq 
(\sP^l\boxtimes \on{Id}_{\wt\CY_2}) \circ (\on{Id}_{\CY_1}\boxtimes \wt\sP^l).$$

\medskip

In what follows we will give an explicit description of $\CP$ in terms of $\CQ$, see \propref{p:left via right}. 

\sssec{}

Similarly, given a functor $\sP^l$ codefined by a kernel $\CP\in \Shv(\CY_1\times \CY_2)$, it may have the property that
it is \emph{defined and codefined by kernels}, and thus corresponds to an object $\CQ\in \Shv(\CY_1\times \CY_2)$. 

\ssec{The miraculous functor} \label{ss:Mir}

In this subsection we recall the construction of the \emph{miraculous} endofunctor 
of $\Shv(\CY)$, defined for any algebraic stack. 

\medskip

Apart from its other applications in the bulk of the paper, the miraculous functor allows
us to relate functors defined and codefined by kernels. 

\sssec{} \label{sss:Mir qc}

Let $\CY$ be a quasi-compact algebraic stack. The miraculous functor on $\CY$, denoted $\Mir_\CY$,
is the functor defined by the kernel
\begin{equation} \label{e:ps u first}
\on{ps-u}_\CY:=(\Delta_\CY)_!(\sfe_\CY).
\end{equation}

\sssec{} \label{sss:left-to-right trans}

For $\CQ\in \Shv(\CY_1\times \CY_2)$, set 
$$\CP\simeq (\Mir_{\CY_1}\boxtimes \on{Id}_{\CY_2})(\CQ)\in \Shv(\CY_1\times \CY_2).$$

We claim that there is a canonically defined natural transformation
\begin{equation} \label{e:left-to-right trans}
\on{Id}_\CZ \boxtimes \sP^l  \to \on{Id}_\CZ \boxtimes \sQ .
\end{equation}

Indeed, the value of \eqref{e:left-to-right trans} on a given $\CF\in \Shv(\CZ\times \CY_1)$ is the value on
$\on{ps-u}_{\CY_1}$ of 
the natural transformation \eqref{e:left to right xy} in the diagram
$$
\xy
(0,0)*+{\Shv(\CY_1\times \CY_1)}="A";
(50,0)*+{\Shv(\CY_1\times \CY_2)}="B";
(0,-30)*+{\Shv(\CZ \times \CY_1)}="C";
(50,-30)*+{\Shv(\CZ \times \CY_2)}="D";
{\ar@{->}^{\on{Id}_{\CY_1}\boxtimes \sQ} "A";"B"};
{\ar@{->}_{\on{Id}_{\CZ}\boxtimes \sQ} "C";"D"};
{\ar@{->}_{(\sF^\sigma)^l\boxtimes \on{Id}_{\CY_1}} "A";"C"};
{\ar@{->}^{(\sF^\sigma)^l\boxtimes \on{Id}_{\CY_2}} "B";"D"};
{\ar@{=>} "B";"C"};
\endxy
$$
where $(\sF^\sigma)^l$ is the functor codefined by $\CF^\sigma\in \Shv(\CY_1\times \CZ)$. 

\sssec{}

The following is immediate from the above diagram:

\begin{prop} \label{p:left via right} \hfill

\smallskip

\noindent{\em(a)} A functor defined by $\CQ$ is defined and codefined by a kernel if and only if
the natural transformations \eqref{e:left-to-right trans} are isomorphisms (for all $\CZ$). 

\smallskip

\noindent{\em(b)} In the situation of point \em{(a)}, the co-defining object is given by 
$$\CP\simeq (\Mir_{\CY_1}\boxtimes \on{Id}_{\CY_2})(\CQ).$$
\end{prop}

%\begin{proof}
%
%Apply
%$$\on{Id}_{\CY_1}\boxtimes \sP^l\simeq \on{Id}_{\CY_1}\boxtimes \sQ$$
%to $\on{ps-u}_{\CY_1}$
%and use the fact that
%$$(\on{Id}_{\CY_1}\boxtimes \sQ)(\on{ps-u}_{\CY_1}) \simeq  (\Mir_{\CY_1}\boxtimes \on{Id}_{\CY_2})(\CQ),$$
%see \eqref{e:green formula}. 
%
%\end{proof}

%\begin{rem}
%
%Let $\CQ\in \Shv(\CY_1\times \CY_2)$ be arbitrary, and define $\CP$ as $(\Mir_{\CY_1}\boxtimes \on{Id}_{\CY_2})(\CQ)$.
%Then, as in Remark \ref{r:* Mir to !}, one constructs a natural transformation
%\begin{equation} \label{e:left-to-right trans}
%\on{Id}_\CZ\boxtimes \sP^l\to \on{Id}_\CZ\boxtimes \sQ.
%\end{equation}
%
%We can reformulate \propref{p:left via right} as saying that if $\sQ$ is defined and codefined by a kernel,
%then the natural transformation \eqref{e:left-to-right trans} is an isomorphism.
%
%\end{rem}

\sssec{}  \label{sss:Mir and ev}

Note that in the particular case of 
$$\CY_1=\CY, \quad \CY_2=\CZ=\on{pt}, \quad \CQ=\CF\in \Shv(\CY),$$
the natural transformation \eqref{e:left-to-right trans} gives rise to a map
\begin{equation} \label{e:Mir and ev}
\on{ev}_\CY^l(\CF',\Mir_\CY(\CF))\to \on{ev}_\CY(\CF',\CF), \quad \CF'\in \Shv(\CY),
\end{equation}
where we recall that $\on{ev}^l_\CY(-,-)$ and $\on{ev}_\CY(-,-)$ denote the pairings
$$\on{C}^\cdot_c(\CY,-\overset{*}\otimes -) \text{ and } \on{C}^\cdot_\blacktriangle(\CY,-\sotimes -),$$
respectively. 

\medskip

Let us write down the map \eqref{e:Mir and ev} explicitly. Namely, it is obtained by applying the natural
transformation
$$((\on{C}^\cdot_c(\CY,-)\circ \Delta_\CY^*)\boxtimes \on{Id}_{\on{pt}})\circ (\on{Id}_{\CY\times \CY}\boxtimes (\on{C}^\cdot_\blacktriangle(\CY,-)\circ \Delta_\CY^!))\to
(\on{Id}_{\on{pt}}\boxtimes (\on{C}^\cdot_\blacktriangle(\CY,-)\circ \Delta_\CY^!))\circ ((\on{C}^\cdot_c(\CY,-)\circ \Delta_\CY^*)\boxtimes \on{Id}_{\CY\times \CY})$$
of \eqref{e:left to right} to the object
$$\CF\boxtimes \on{ps-u}_\CY \boxtimes \CF'\in \Shv(\CY\times \CY\times \CY\times \CY).$$

\sssec{}  \label{sss:left vs right conv}

Note also that in the situation of Sects. \ref{sss:comp corr} and \ref{sss:left convolution}, for
$$\CQ_{1,2}\in \Shv(\CY_1\times \CY_2) \text{ and } \CQ_{2,3}\in \Shv(\CY_2\times \CY_3)$$
we have a map
\begin{equation} \label{e:left vs right conv}
\CQ_{2,3}\lstar ((\on{Id}_{\CY_1}\boxtimes \Mir_{\CY_2})(\CQ_{1,2})) \to \CQ_{2,3}\star \CQ_{1,2}.
\end{equation}

It follows from \propref{p:left via right} that if $\sQ_{1,2}$ is defined and codefined by a kernel,
the map \eqref{e:left vs right conv} is an isomorphism. 

\sssec{} \label{sss:Id l}

Consider the endofunctor \emph{codefined} by $\on{u}_\CY$; denote it by $\on{Id}^l_\CY$. 

\medskip

Let $\CP$ be an object of $\Shv(\CY_1\times \CY_2)$ and set 
$$\CQ:=(\on{Id}^l_{\CY_1}\boxtimes \on{Id}_{\CY_2})(\CP).$$

As in \secref{sss:left-to-right trans} we produce a natural transformation
\begin{equation} \label{e:left-to-right trans bis}
\on{Id}_\CZ\boxtimes \sP^l\to \on{Id}_\CZ\boxtimes \sQ.
\end{equation} 

As in \propref{p:left via right} we have:

\begin{prop} \label{p:right via left} \hfill

\smallskip

\noindent{\em(a)} A functor codefined by $\CP$ is defined and codefined by a kernel if and only if
the natural transformations \eqref{e:left-to-right trans bis} are isomorphisms (for all $\CZ$). 

\smallskip

\noindent{\em(b)} In the situation of point \em{(a)}, the defining object is given by 
$$\CQ\simeq (\on{Id}^l\boxtimes \on{Id}_{\CY_2})(\CP).$$
\end{prop}

\sssec{} \label{sss:two versions left-to-right trans}

The natural transformations \eqref{e:left-to-right trans} and \eqref{e:left-to-right trans bis} are compatible as follows. 
Start with an object $\CQ\in \Shv(\CY_1\times \CY_2)$, and set
$$\CP:=(\Mir_{\CY_1}\boxtimes \on{Id}_{\CY_2})(\CQ) \text{ and }
\CQ':=(\on{Id}^l_{\CY_1}\boxtimes \on{Id}_{\CY_2})(\CP).$$

Using adjunction (see \secref{sss:codef left} below) and the natural transformation \eqref{e:disc}, we obtain a map
\begin{equation} \label{e:back to Q}
\CQ'\to \CQ.
\end{equation}

It is straightforward to check that the composition
$$\on{Id}_\CZ\boxtimes \sP^l \overset{\text{\eqref{e:left-to-right trans bis}}}\longrightarrow 
\on{Id}_\CZ\boxtimes \sQ' \overset{\text{\eqref{e:back to Q}}}\longrightarrow \on{Id}_\CZ\boxtimes \sQ$$
identifies with the map \eqref{e:left-to-right trans}. 

\medskip

In particular, if $\sQ$ is defined and codefined by a kernel, all of the above maps are isomorphisms,
including the map \eqref{e:back to Q}. 

\sssec{} \label{sss:right vs left conv}

As in \secref{sss:left vs right conv}, for
$$\CP_{1,2}\in \Shv(\CY_1\times \CY_2) \text{ and } \CP_{2,3}\in \Shv(\CY_2\times \CY_3),$$
we have a map
\begin{equation} \label{e:right vs left conv}
\CP_{2,3}\lstar \CP_{1,2}\to \CP_{2,3}\star ((\on{Id}_{\CY_1}\boxtimes \on{Id}^l_{\CY_2})(\CP_{1,2})).
\end{equation} 

If $\sP^l_{1,2}$ is defined and codefined by a kernel,
the map \eqref{e:right vs left conv} is an isomorphism.

%\begin{cor} 
%If $\on{Id}_\CZ\boxtimes \sQ\simeq \on{Id}_\CZ\boxtimes \sP^l$ is defined and codefined by $\CQ$ and $\CP$, respectively,
%then $\CQ$ is constructible if and only $\CP$ is.
%\end{cor}

%\begin{rem}
%
%Assume that in the situation of \propref{p:left via right}, the object $\CQ$ is constructible
%and safe; then $\CP$ is also constructible by \propref{p:left via right}.  
%Suppose also that $\sQ$ is safe (see \secref{sss:boxtimes disc} for what this means). 
%
%\medskip
%
%Combining \eqref{e:left to right Verdier} and \eqref{e:left and right} we obtain
%\begin{equation} \label{e:left as right}
%\on{Id}_\CZ\boxtimes \sS^l \simeq \on{Id}_\CZ\boxtimes \sT,
%\end{equation} 
%where 
%$$\CS:=\BD^{\on{Verdier}}(\CQ) \text{ and } \CT:=\BD^{\on{Verdier}}(\CP).$$
%
%Moreover, one can show that in this case, $\sT$ is also safe. 
%
%\end{rem}

\ssec{Adjunctions for functors defined by kernels}

In this subsection we focus on the phenomenon of adjunction 
of functors defined by kernels. 

\sssec{}

Consider the notion of adjunction of 1-morphisms in the 2-category
introduced in \secref{sss:2-categ}. 

\medskip

Explicitly, given a pair of objects
$$\CQ_{1,2}\in \Shv(\CY_1\times \CY_2) \text{ and } \CQ_{2,1}\in \Shv(\CY_2\times \CY_1),$$
a datum that makes $(\CQ_{1,2},\CQ_{2,1})$ into an adjoint pair is a pair of morphisms
$$\on{u}_{\CY_1}\to \CQ_{2,1}\star \CQ_{1,2} \text{ and } \CQ_{1,2}\star \CQ_{2,1} \to \on{u}_{\CY_2}$$
that satisfy the usual axioms.

\sssec{} \label{sss:adj naive}

Note that in this case, for any $\CZ$, the corresponding functors
\begin{equation} \label{e:actual functors}
(\on{Id}_\CZ\boxtimes \sQ_{1,2},\on{Id}_\CZ\boxtimes \sQ_{2,1})
\end{equation}
form an adjoint pair.

\medskip

Vice versa, the datum of adjunction $(\CQ_{1,2},\CQ_{2,1})$ is equivalent to the datum of adjunction
of the functors \eqref{e:actual functors} compatible with the isomorphisms from \secref{sss:compat kernels new}.

\begin{rem}
In what follows we will sometimes abuse the terminology as follows: for a functor 
$\sQ:\Shv(\CY_1)\to \Shv(\CY_2)$ defined by a kernel we will say that it \emph{admits an adjoint
as a functor defined by the kernel} if the corresponding object $\CQ\in \Shv(\CY_1\times \CY_2)$
has this property.
\end{rem}

\sssec{} \label{sss:create left adjoint}

%Let $\sigma$ denote the transposition acting on $\CY_1\times \CY_2$. 
Note that if $(\CQ_{1,2},\CQ_{2,1})$
is an adjoint pair, then so is 
$$(\CQ^\sigma_{2,1},\CQ^\sigma_{1,2}).$$

\sssec{}

We claim:

\begin{prop} \label{p:adj constr}
Suppose that $\CQ\in \Shv(\CY_1\times \CY_2)$ admits a right adjoint, to be denoted $\CQ^R\in \Shv(\CY_2\times \CY_1)$.
Then both $\CQ$ and $\CQ^R$ are constructible.
\end{prop}

\begin{proof}

First, observe that the statement for $\CQ$ implies that for $\CQ^R$ by \secref{sss:create left adjoint}. 

\medskip

Choose a smooth covering by an affine scheme $f:S_1\to \CY_1$. In order to show that $\CQ$ is constructible, it suffices to show
that $(f\times \on{id})^!(\CQ)$ is compact.

\medskip

Since the diagonal morphism
$$\Gamma_f:S_1\to S_1\times \CY_1$$
is schematic, the object
$$(\Gamma_f)_*(\omega_{S_1})\in \Shv(S_1\times \CY_1)$$
is compact. Note that 
$$(f\times \on{id})^!(\CQ)\simeq (\on{Id}_{S_1}\boxtimes \sQ)((\Gamma_f)_*(\omega_{S_1})).$$

Now, since the functor $\on{Id}_{S_1}\boxtimes \sQ$ admits a continuous right adjoint, we obtain that 
$$(\on{Id}_{S_1}\boxtimes \sQ)((\Gamma_f)_*(\omega_{S_1}))\in  \Shv(S_1\times \CY_2)$$
is compact, as required.

\end{proof}

\ssec{Adjunctions and Verdier duality}

In this subsection we relate the phenomenon of adjunction of functors to
Verdier duality of kernels that define them. 

\sssec{} \label{sss:codef left}

Let $\CQ\in \Shv(\CY_1\times \CY_2)$ be constructible; in particular $\BD^{\on{Verdier}}(\CQ)$
is well-defined. % and is also compact (under our assumptions on algebraic stacks, see \secref{sss:Verdier}).
Denote $$\CP:=\BD^{\on{Verdier}}(\CQ)^\sigma\in \Shv(\CY_2\times \CY_1).$$
%, where $\sigma$ denotes the swap of the two factors in $\CY_1\times \CY_2$. 

\medskip

Note that for an individual $\CZ$, the functor 
$$\on{Id}_\CZ\boxtimes \sP^l:\Shv(\CZ\times \CY_2)\to \Shv(\CZ\times \CY_1)$$
is the left adjoint of the functor
$$\on{Id}_\CZ\boxtimes \sQ_{\on{disc}}:\Shv(\CZ\times \CY_1)\to \Shv(\CZ\times \CY_2),$$
see \secref{sss:boxtimes disc} for the notation.

%
%In the situation of \secref{sss:left to right}, the natural transformation
%$$(\on{Id}_{\CY'_2}\boxtimes^l \sP)\circ (\sQ'\boxtimes \on{Id}_{\CY_2}) \to (\sQ'\boxtimes \on{Id}_{\CY_1}) \circ (\on{Id}_{\CY'_1}\boxtimes^l \sP)$$
%of \eqref{e:left to right} is obtained by adjunction
%from the isomorphism
%$$(\sQ'\boxtimes \on{Id}_{\CY_2})\circ (\on{Id}_{\CY'_1}\boxtimes \sQ)\simeq 
%(\on{Id}_{\CY'_2}\boxtimes \sQ) \circ (\sQ'\boxtimes \on{Id}_{\CY_1})$$
%of \eqref{e:order does not matter}. 

\sssec{}

We claim:

\begin{thm} \label{t:right adj main}
Suppose that $\sQ$ admits a right adjoint $\sQ^R$ as a functor defined by a kernel with the 
corresponding object of $\Shv(\CY_2\times \CY_1)$ denoted by $\CQ^R$. Denote
$$\CP:=\BD^{\on{Verdier}}(\CQ^R)^\sigma.$$
Then:

\smallskip

\noindent{\em(a)} The functor $\sQ$ is defined and codefined by a kernel, with codefining object $\CP$. 

\smallskip

\noindent{\em(b)} The functor $\sQ^R$ is safe.\footnote{See \secref{sss:safe} for what this means.}

\end{thm}

Let us emphasize that point (a) of the theorem says that
$$\on{Id}_\CZ \boxtimes\sQ\simeq \on{Id}_\CZ \boxtimes \sP^l$$ for $\sP$ codefined by the object $\CP$ above.

\begin{proof}

It suffices to prove point (a):  indeed, point (b) would follow from Sects. \ref{sss:adj naive} and \ref{sss:codef left}, 
since the functors $\on{Id}_\CZ \boxtimes \sQ^R$ and
$\on{Id}_\CZ \boxtimes \sQ^R_{\on{disc}}$ are both right adjoint to $\on{Id}_\CZ \boxtimes\sQ\simeq \on{Id}_\CZ\boxtimes \sP^l$.

\medskip

Let $f:\wt\CY_2\to \CY_2$ be a smooth cover by an affine scheme. In order to prove point (a), it suffices to construct isomorphisms
\begin{equation} \label{e:left right to check}
(\on{id}\times f)^*\circ (\on{Id}_\CZ \boxtimes\sQ) \simeq (\on{id}\times f)^*\circ (\on{Id}_\CZ \boxtimes \sP^l)
\end{equation}
as functors
$$\Shv(\CZ\times \CY_1)\to \Shv(\CZ\times \wt\CY_2).$$

\medskip

Note that since $f$ is smooth, the functor $f^*$ is both defined and codefined by a kernel so that 
$$(\on{id}\times f)^* \simeq (\on{Id}_\CZ \boxtimes f^*).$$ 

Denote $\wt\CQ=(\on{id}\times f)^*(\CQ)$,
so that
$$\on{Id}_\CZ \boxtimes \wt\sQ\simeq 
(\on{Id}_\CZ \boxtimes f^*) \circ  (\on{Id}_\CZ \boxtimes \sQ).$$

Since the $f^*$ admits a right adjoint as a functor defined by a kernel (namely, $f_*)$, so does the functor
$\wt\sQ$. Note that the corresponding functor $\wt\sP^l$ is also given by $f^*\circ \sP^l$, as a functor
\emph{codefined} by a kernel. 

\medskip

Hence, the sought-for isomorphism \eqref{e:left right to check} becomes
$$\on{Id}_\CZ \boxtimes \wt\sQ \simeq \on{Id}_\CZ \boxtimes \wt\sP^l.$$

In other words, we have reduced the verification of point (a) to the case when $\CY_2$ is an affine scheme,
which we will from now on assume. 

\medskip

By \secref{sss:codef left}, the functors $(\on{Id}_\CZ\boxtimes \sP^l,\on{Id}_\CZ \boxtimes \sQ^R_{\on{disc}})$
form an adjoint pair. Note, however, that since $\CY_2$ is an affine scheme, the functor 
$\on{Id}_\CZ\boxtimes \sP^l$ preserves compactness (this is true for any object codefined by a constructible 
kernel with target a scheme, see \secref{sss:stacky}). Hence, its right adjoint is continuous. This implies
that $\on{Id}_\CZ \boxtimes \sQ^R_{\on{disc}}$ is continuous, so $\sQ^R$ is safe. 
(Note that this proves point (b) in the case when $\CY_2$ is an affine scheme.) 

\medskip

In particular, we obtain that $(\on{Id}_\CZ\boxtimes \sP^l,\on{Id}_\CZ \boxtimes \sQ^R)$
form an adjoint pair. However, by \secref{sss:adj naive}, the functors $(\on{Id}_\CZ\boxtimes \sQ,\on{Id}_\CZ \boxtimes \sQ^R)$
also form an adjoint pair. This implies 
$$\on{Id}_\CZ \boxtimes\sQ\simeq \on{Id}_\CZ \boxtimes \sP^l,$$
i.e., the isomorphism of point (a).

\end{proof}

\sssec{}

We also have the following partial converse of \thmref{t:right adj main}:

\begin{prop} \label{p:conv righ adj main}
Let $\CQ$ be constructible, and suppose that the functor $\sQ$ is defined and 
codefined by a kernel with codefining object $\CP$, i.e.,
$\on{Id}_\CZ \boxtimes\sQ\simeq \on{Id}_\CZ \boxtimes \sP^l$. Suppose also that the functors
$\on{Id}_\CZ \boxtimes\sQ$ preserve compactness. Then:

\smallskip

\noindent{\em(a)} The object $\CP$ is constructible. 

\smallskip

\noindent{\em(b)} The functor $\sQ$ admits a right adjoint as a functor defined by a kernel with the 
corresponding object of $\Shv(\CY_2\times \CY_1)$ being $\CQ^R:=\BD^{\on{Verdier}}(\CP)^\sigma$.

\smallskip

\noindent{\em(c)} The functor $\sQ^R$ is safe. 
\end{prop}

\begin{proof}

The fact that $\CP$ is constructible follows as in \propref{p:adj constr}. 

\medskip

By \secref{sss:codef left}, the functor $\on{Id}_\CZ \boxtimes \sQ^R_{\on{disc}}$ is the right adjoint
of $\on{Id}_\CZ \boxtimes \sP^l$, and since the latter preserves compactness, we obtain that 
$\on{Id}_\CZ \boxtimes \sQ^R_{\on{disc}}$ is continuous. Hence, $\sQ^R$ is safe; this proves point (c). 

\medskip

Thus, the  
natural transformation $\on{Id}_\CZ \boxtimes \sQ^R\to \on{Id}_\CZ \boxtimes \sQ^R_{\on{disc}}$
is an isomorphism, we obtain that $\on{Id}_\CZ \boxtimes \sQ^R$ is the right adjoint of
$\on{Id}_\CZ \boxtimes\sQ$. This proves point (b).

\end{proof}

\begin{rem}
Note that the assertion of \propref{p:conv righ adj main} would be false without the assumption that
the functors $\on{Id}_\CZ \boxtimes\sQ$ preserve compactness. 

\medskip

For example, take $\CY_1=\on{pt}$, $\CY_2=B\BG_m$ and let $\CQ$ be the constant sheaf.

\end{rem}

\sssec{}

Combining \thmref{t:right adj main} and \propref{p:left via right}, we obtain: 

\begin{cor} \label{c:correction}
Assume that a functor $\sQ$ defined by the kernel $\CQ\in \Shv(\CY_1\times \CY_2)$ admits a right adjoint
as a functor defined by a kernel. Denote the corresponding 
object by $\CQ^R\in \Shv(\CY_2\times \CY_1)$. Then 
$$\BD^{\on{Verdier}}(\CQ^R)\simeq \left((\Mir_{\CY_1}\boxtimes \on{Id}_{\CY_2})(\CQ)\right)^\sigma$$
as objects of $\Shv(\CY_1\times \CY_2)$. In particular, the object
$$\CP:=(\Mir_{\CY_1}\boxtimes \on{Id}_{\CY_2})(\CQ)$$
is constructible. 
\end{cor}

The next assertion reproduces \cite[Theorem 6.3.2]{Ga2}:

\begin{cor} \label{c:correction bis}
In the situation of \corref{c:correction}, we have
$$\BD^{\on{Verdier}}(\CQ)\simeq (\Mir_{\CY_1}\boxtimes \on{Id}_{\CY_2})((\CQ^R)^\sigma).$$
\end{cor}

\begin{proof}

Follows by combining \secref{sss:create left adjoint} and \corref{c:correction}.

\end{proof}

\begin{rem} \label{r:right adj main non qc}

At one point on the main body of the paper (namely, \thmref{t:char of Nilp adj}), 
we use an extension of \thmref{t:right adj main} to the case when the stack $\CY_1$ is not 
necessarily quasi-compact, but is truncatable (see \secref{sss:truncatable} for what this means).

\medskip

In this case, $\CQ$ is an object of the category $\Shv(\CY_1\times \CY_2)_{\on{co}}$
(see \secref{sss:co} for the notation), such that for any cotruncative
quasi-compact open substack $\CU_1\overset{j}\hookrightarrow \CY_1$, the object
$$(j^?\boxtimes \on{id})(\CQ)\in \Shv(\CU_1\times \CY_2)$$
is constructible.

\medskip

The right adjoint $\CQ^R$ is characterized by the property that
$(j\times \on{id})^*(\CQ_R)$ is the right adjoint of $(j^?\boxtimes \on{id})(\CQ)$.
The codefining object 
$$\CP\in \Shv(\CY_1\times \CY_2)$$
is related to $\CQ^R$ by the same formula as in \thmref{t:right adj main}.

\medskip

The statement of \corref{c:correction} remains unchanged, where we now understand 
$\Mir_{\CY_1}\boxtimes \on{Id}_{\CY_2}$ as a functor
$$\Shv(\CY_1\times \CY_2)_{\on{co}}\to \Shv(\CY_1\times \CY_2).$$

\end{rem}

\begin{rem} \label{r:right adj main non qc bis}

Let us continue to be in the setting of Remark \ref{r:right adj main non qc}. Let 
$\CQ$ be a constructible object of $\Shv(\CY_1\times \CY_2)$, and let 
us view $\sQ$ is a functor $\Shv(\CY_1)_{\on{co}}\to \Shv(\CY_2)$, defined by a kernel
(see \secref{ss:functors by ker non qc}). 

\medskip

Suppose that $\CQ$ admits a right adjoint\footnote{In the sense of 2-category of \secref{sss:2-categ}.}, which is now an object 
$$\CQ^R\in \Shv(\CY_2\times \CY_1)_{\on{co}}.$$

\medskip

We have 
$$(j^?\boxtimes \on{id})(\CQ^R)\simeq \left((j\times \on{id})^*(\CQ)\right)^R,$$
for $\CU_1\overset{j}\hookrightarrow \CY_1$ as in Remark \ref{r:right adj main non qc}.  

\medskip

In particular, \corref{c:correction bis} is applicable. Explicitly, we have
$$\CQ^R\simeq (\on{Id}^l_{\CY_1}\boxtimes \on{Id}_{\CY_2})(\BD^{\on{Verdier}}(\CQ))^\sigma,$$
where we view $\on{Id}^l_{\CY_1}$ as a functor
$\Shv(\CY_1)\to \Shv(\CY_1)_{\on{co}}$, codefined by a kernel, see \secref{sss:codef co}. 

\medskip

Note, however, that in this case, there is no sense in which we can talk about $\sQ$ being
codefined by a kernel. 

\end{rem}

\ssec{Miraculous stacks} \label{ss:Mir stacks qc}

In this subsection we review what it means for a quasi-compact algebraic stack to 
be miraculous. 

\sssec{} \label{sss:Mir stacks qc}

Let $\CY$ be a quasi-compact stack. We shall say that $\CY$ is \emph{miraculous} 
if the functor $\Mir_\CY$ is \emph{invertible}, as a functor defined by a kernel.

\medskip

Using \secref{sss:compat kernels new}, this is equivalent to the requirement that 
$$\on{Id}_\CZ\boxtimes \Mir_\CY: \Shv(\CZ\times \CY)\to \Shv(\CZ\times \CY)$$
is an equivalence for every $\CZ$. 

\sssec{}  \label{sss:D Mir}

Assume that $\CY$ is miraculous. In particular, $\Mir_\CY$, viewed as a plain endofunctor 
of $\Shv(\CY)$ is a self-equivalence.

\medskip

In this case, we can define a \emph{new} self-duality on $\Shv(\CY)$ with counit given by
$$\Shv(\CY)\otimes \Shv(\CY) \overset{\Mir_\CY^{-1}\otimes \on{Id}}\longrightarrow
\Shv(\CY)\otimes \Shv(\CY) \overset{\on{ev} _\CY}\longrightarrow \Vect;$$
we denote it by $\on{ev}^{\Mir}_\CY$.

\medskip

We will refer to it as the \emph{miraculous self-duality} of $\Shv(\CY)$. The corresponding
contravariant self-equivalence, denoted 
$$\BD^{\Mir}:(\Shv(\CY)^c)^{\on{op}}\to \Shv(\CY)^c,$$
is given by 
$$\BD^{\Mir}\simeq \BD^{\on{Verdier}}\circ \Mir_\CY^{-1}.$$

\sssec{}  \label{sss:D Mir sym}

Since the object $\on{ps-u}_\CY$ is swap-equivariant, we have
$$(\Mir_\CY)^\vee\simeq \Mir_\CY,$$
and hence
$$(\Mir_\CY^{-1})^\vee\simeq \Mir_\CY^{-1}.$$

From here we obtain that the pairing $\on{ev}^{\Mir}_\CY$ is swap-equivariant. Hence, the
functor $\BD^{\Mir}$ is involutive.

\sssec{} \label{sss:Mir and Id l}

Let $\CY$ be miraculous. In particular, we obtain that the functor $\Mir_\CY$ 
admits a left adjoint as a functor defined by a kernel.

\medskip

Applying \thmref{t:right adj main}(b) we obtain that the functor $\Mir_\CY$ is safe. In particular,
it preserves constructibility.

\medskip

Furthermore, from \thmref{t:right adj main}(a) we obtain that the inverse of $\Mir_\CY$, viewed
as a functor codefined by a kernel is given by $\on{Id}^l_\CY$, see \secref{sss:Id l} for the notation. 

\ssec{A criterion for admitting an adjoint} \label{ss:unit adj}

%\sssec{} \label{sss:Mir qc}
%
%Let $\CY$ be an algebraic stack. Recall that the miraculous functor on $\CY$, denoted $\Mir_\CY$,
%is the functor defined by the kernel
%$$\on{ps-u}_\CY:=(\Delta_\CY)_!(\sfe_\CY).$$ 
%
%\sssec{}
%
%Our current goal is to prove the following assertion:
%
%\begin{prop} \label{p:correction}
%Assume that $\CQ\in \Shv(\CY_1\times \CY_2)$ admits a right adjoint; denote the corresponding 
%object by $\CQ^R\in \Shv(\CY_2\times \CY_1)$. Then 
%$$(\CQ^R)^\sigma\simeq \BD^{\on{Verdier}}\circ (\Mir_{\CY_1}\boxtimes \on{Id}_{\CY_2})(\CQ)$$
%as objects of $\Shv(\CY_1\times \CY_2)$.
%\end{prop}
%
%\begin{proof}
%
%We have to show that 
%$$\CP:=\BD^{\on{Verdier}}(\CQ^R)^\sigma\simeq (\Mir_{\CY_1}\boxtimes \on{Id}_{\CY_2})(\CQ).$$
%
%Evaluating both sides of \eqref{e:right as left} on $\on{ps-u}_{\CY_1}$, and using the fact that
%$$(\on{Id}_{\CY_1}\boxtimes^l \sP)(\on{ps-u}_{\CY_1})\simeq \CP,$$
%we obtain that
%$$\BD^{\on{Verdier}}(\CQ^R)^\sigma \simeq (\on{Id}_{\CY_1}\boxtimes \sQ)(\on{ps-u}_{\CY_1}).$$
%
%However, the latter tautologically identifies with
%$$(\Mir_{\CY_1}\boxtimes \on{Id}_{\CY_2})(\CQ).$$
%
%\end{proof} 
%
%Combining with \eqref{e:right as left}, we obtain: 

%\begin{cor} \label{c:correction}
%If $\sQ\simeq \sP$ is defined and codefined by $\CQ$ and $\CP$, respectively, we have
%$$\CP\simeq (\Mir_{\CY_1}\boxtimes \on{Id}_{\CY_2})(\CQ).$$
%\end{cor}

\sssec{} 

Let $\CQ\in \Shv(\CY_1\times \CY_2)$ be a constructible object. Denote
\begin{equation} \label{e:right adj cand bis}
'\!\CQ^R:= (\on{Id}_{\CY_2} \boxtimes \on{Id}^l_{\CY_1})(\BD^{\on{Verdier}}(\CQ)^\sigma).
\end{equation}

\begin{rem} \label{r:right adj when exists}
Assume for a moment that $\CQ$ admits a right adjoint\footnote{In the sense of 2-category of \secref{sss:2-categ}.}, 
so that
$$\BD^{\on{Verdier}}(\CQ^R)\simeq \left((\Mir_{\CY_1}\boxtimes \on{Id}_{\CY_2})(\CQ)\right)^\sigma,$$
by \corref{c:correction}. 

\medskip

Since $\BD^{\on{Verdier}}(\CQ^R)$ is constructible, it follows from \secref{sss:left to right Verdier} that in this case we have
$$'\!\CQ^R \simeq \CQ^R.$$

\end{rem}

%
%natural transformations
%\begin{equation} \label{e:Z disc}
%\on{Id}_\CZ \boxtimes \sQ^\sigma \to \on{Id}_\CZ \boxtimes \sQ^\sigma _{\on{disc}}
%\end{equation}
%or the 
%natural transformations
%\begin{equation} \label{e:Z disc bis}
%\on{Id}_\CZ \boxtimes\, {}'\!\sQ^R \to \on{Id}_\CZ \boxtimes \,{}'\!\sQ^R_{\on{disc}}
%\end{equation}
%are isomorphisms (for example, this is the case if either $\CQ$ or $'\CQ^R$ is compact).

\sssec{} \label{sss:safety for adj}

Assume that %$\sQ^\sigma$ or 
$'\!\sQ^R$ is safe. 

\medskip

Note that this assumption holds if $\CQ$ admits a right adjoint. 
Indeed, this follows from \thmref{t:right adj main} and Remark
\ref{r:right adj when exists} above. 

\medskip

Note also that the safety condition is automatic whenever%$\CQ$ or 
$'\!\CQ^R$ is compact or $\CY_2$ is a scheme, see \secref{sss:safe}. 

\sssec{} \label{sss:unit adj}

We claim that there exists a canonically defined map 
\begin{equation} \label{e:ker adj unit}
\on{u}_{\CY_1}\to {}'\!\CQ^R\star \CQ.
\end{equation} 

We start with the map
\begin{equation} \label{e:ker adj unit prel prel}
\on{ps-u}_{\CY_1}\to \BD^{\on{Verdier}}(\CQ)^\sigma \star_{\on{disc}} \CQ,
\end{equation} 
arising by adjunction, see \secref{sss:conv disc} for the notation $\star_{\on{disc}}$.

\medskip

We apply to both sides the functor
$$\on{Id}_{\CY_1}\boxtimes \on{Id}^l_{\CY_1},$$
and we obtain a map
\begin{equation} \label{e:ker adj unit prel prel bis}
\on{u}_{\CY_1}\to (\on{Id}_{\CY_1}\boxtimes \on{Id}^l_{\CY_1})(\BD^{\on{Verdier}}(\CQ)^\sigma \star_{\on{disc}} \CQ),
\end{equation} 
and we follow it by the map
$$(\on{Id}_{\CY_1}\boxtimes \on{Id}^l_{\CY_1})(\BD^{\on{Verdier}}(\CQ)^\sigma \star_{\on{disc}} \CQ) \overset{\text{\eqref{e:left to right disc}}}\longrightarrow
((\on{Id}_{\CY_2}\boxtimes \on{Id}^l_{\CY_1})(\BD^{\on{Verdier}}(\CQ)^\sigma))\star_{\on{disc}} \CQ ={}'\!\CQ^R\star_{\on{disc}} \CQ.$$

Finally, using the assumption that $'\!\CQ^R$ is safe, we obtain that the map
$$'\!\CQ^R\star \CQ\to '\!\CQ^R\star_{\on{disc}} \CQ$$
is an isomorphism. 

%$$\sigma_{1,23}(\on{ps-u}_{\CY_1}\boxtimes  \on{ps-u}_{\CY_2})\to \CQ \boxtimes \BD(\CQ)^\sigma,$$ 
%where $\sigma_{1,23}=\sigma_{2,3}\circ \sigma_{1,2}$. 

%\medskip
%
%Consider the map
%$$\sigma_{1,23}(\on{ps-u}_{\CY_1}\boxtimes  \on{ps-u}_{\CY_2}) \boxtimes \on{u}_{\CY_1}\to 
%\CQ \boxtimes \BD(\CQ)^\sigma \boxtimes \on{u}_{\CY_1},$$
%and apply
%$$(p_{1,2,3,5})_! \circ (\on{id}_{\CY_1\times \CY_2\times \CY_2\times \CY_1}\times \Delta_{\CY_1}\times \on{id}_{\CY_1})^*$$
%to both sides. We obtain a map
%$$\sigma_{1,23}(\on{u}_{\CY_1}\boxtimes  \on{ps-u}_{\CY_2}) \to \CQ \boxtimes {}'\!\CQ^R.$$
%
%By adjunction, the latter map give rise to the sought-for map \eqref{e:ker adj unit}. 

\sssec{}

Note also that by adjunction, we have a canonically defined map
\begin{equation} \label{e:ker adj counit prel}
\CQ \lstar \BD^{\on{Verdier}}(\CQ)^\sigma\to \on{u}_{\CY_2}.
\end{equation}

%Indeed, the map \eqref{e:ker adj counit prel} is obtained by adjunction from the canonical map
%$$\CQ \boxtimes \BD(\CQ)^\sigma \to \sigma_{12,3}(\on{u}_{\CY_2}\boxtimes \on{u}_{\CY_1}),$$
%where $\sigma_{12,3}=\sigma_{1,2}\circ \sigma_{2,3}$. 

\sssec{}

Recall now that we have a canonically defined map
\begin{equation} \label{e:compare counits}
\CQ \lstar \BD^{\on{Verdier}}(\CQ)^\sigma\to \CQ \star {}'\!\CQ^R,
\end{equation}
see \eqref{e:right vs left conv}. 

\sssec{}

We claim:

\begin{thm} \label{t:right adj crit} For a constructible object $\CQ\in \Shv(\CY_1\times \CY_2)$, 
the following conditions are equivalent\footnote{In the statements below we refer to adjunctions in the 2-category of \secref{sss:2-categ}.}:

\smallskip

\noindent{\em(i)} The object $\CQ$ admits a right adjoint;

\smallskip

\noindent{\em(ii)} The map \eqref{e:ker adj unit} is the unit of an adjunction;

\smallskip

\noindent{\em(iii)} The map \eqref{e:compare counits} is an isomorphism;

\smallskip

\noindent{\em(iv)} The map \eqref{e:compare counits} is an isomorphism, and the map 
\eqref{e:ker adj counit prel} precomposed with the inverse of the isomorphism \eqref{e:compare counits}
is the counit of an adjunction;

\smallskip

\noindent{\em(v)} The map \eqref{e:compare counits} is an isomorphism, and the map 
\eqref{e:ker adj counit prel} precomposed with the inverse of the isomorphism \eqref{e:compare counits}
is the counit of an adjunction, with the unit being \eqref{e:ker adj unit}. 

\end{thm}

\begin{rem} \label{r:right adj crit non-qc}

At one place in the main body of the paper (namely, \corref{c:equiv BunG}), we use the extension of \thmref{t:right adj crit} 
to the case when $\CY_1$ is not necessarily quasi-compact, but truncatable. 

\medskip 

Here $\CQ$ is an object of $\Shv(\CY_1\times \CY_2)$, and we view $\sQ$
as a functor $\Shv(\CY_1)_{\on{co}}\to \Shv(\CY_2)$, defined by a kernel (see \secref{ss:functors by ker non qc}).

\medskip

We view $\on{Id}^l_{\CY_1}$ as a functor $\Shv(\CY_1)\to \Shv(\CY_1)_{\on{co}}$,
codefined by a kernel, so $'\!\CQ^R$ is an object of 
$\Shv(\CY_1\times \CY_2)_{\on{co}}$, see \secref{sss:codef co}. 

\medskip

The proof of \thmref{t:right adj crit} extends to this case in a straightforward way
(instead of \thmref{t:right adj main} we use its variant described in 
Remark \ref{r:right adj main non qc bis}). 

\end{rem}

\ssec{Proof of \thmref{t:right adj crit}}

\sssec{}

We have the tautological implications (v) $\Rightarrow$ (ii) $\Rightarrow$ (i) and (v) $\Rightarrow$ (iv) $\Rightarrow$ (iii). 

\sssec{}

Assume that $\CQ$ admits a right adjoint. Then the map \eqref{e:compare counits} is an isomorphism
by \thmref{t:right adj main} and \secref{sss:right vs left conv} (i.e., we have (i) $\Rightarrow$ (iii)). 

\medskip

Furthermore, by Remark \ref{r:right adj when exists}, in this case $'\!\CQ^R\simeq \CQ^R$. 
Unwinding the identification of \thmref{t:right adj main}(a), 
we obtain that the unit of the $(\CQ,\CQ^R)$-adjunction is given by \eqref{e:ker adj unit} and the counit is given by 
the map \eqref{e:ker adj counit prel} precomposed with the isomorphism \eqref{e:compare counits}. 

\medskip

Thus, we obtain that (i) $\Rightarrow$ (v). 

\sssec{}

It remains to show that (iii) $\Rightarrow$ (i) or equivalently  (iii) $\Rightarrow$ (v). 
As a first step, we will reduce the assertion to the case when 
$\CY_2$ is a scheme.

\medskip

Choose a smooth cover $f:\wt\CY_2\to \CY_2$, where $\wt\CY_2$ is a scheme. Denote
$$\wt\CQ:=(\on{id}\times f)^*(\CQ)\in \Shv(\CY_1\times \wt\CY_2).$$

Note that we have a canonical identification 
$$'\!\wt\sQ^R\simeq {}'\!\sQ^R\circ f_*,$$ as functors defined by kernels. 

\medskip
 
Furthermore, it is easy to see that condition (iii) for $\CQ$ implies the corresponding condition for $\wt\CQ$. 

\medskip

Assume that the implication (iii) $\Rightarrow$ (v) holds for $\wt\CQ$. Let us show that this implies
condition (i) for the original $\CQ$. 

\medskip

By assumption, the map
$$\CHom_{\Shv(\CZ\times \wt\CY_2)}((\on{Id}_\CZ\boxtimes \wt\sQ)(\CF),\wt\CF') \to 
\CHom_{\Shv(\CZ\times \CY_1)}(\CF,(\on{Id}_\CZ\boxtimes {}'\!\wt\sQ^R)(\wt\CF')), \quad \wt\CF'\in   \Shv(\CZ\times \wt\CY_2),$$
induced by \eqref{e:ker adj unit}, is an isomorphism. 

\medskip

We have a commutative diagram
$$
\CD
\CHom_{\Shv(\CZ\times \wt\CY_2)}((\on{Id}_\CZ\boxtimes \wt\sQ)(\CF),\wt\CF') @>>> 
\CHom_{\Shv(\CZ\times \CY_1)}(\CF,(\on{Id}_\CZ\boxtimes {}'\!\wt\sQ^R)(\wt\CF')) \\
@V{\sim}VV @VV{\sim}V \\
\CHom_{\Shv(\CZ\times \CY_2)}\left((\on{Id}_\CZ\boxtimes \sQ)(\CF),(\on{id}\times f)_*(\wt\CF')\right) @>>>
\CHom_{\Shv(\CZ\times \CY_1)}\left(\CF,(\on{Id}_\CZ\boxtimes {}'\!\sQ^R)((\on{id}\times f)_*(\wt\CF'))\right).
\endCD
$$

Hence, we obtain that the map   
\begin{equation} \label{e:adj to check}
\CHom_{\Shv(\CZ\times \CY_2)}((\on{Id}_\CZ\boxtimes \sQ)(\CF),\CF') \to 
\CHom_{\Shv(\CZ\times \CY_1)}(\CF,(\on{Id}_\CZ\boxtimes {}'\!\sQ^R)(\CF')),
\end{equation}
induced by \eqref{e:ker adj unit}, is an isomorphism, whenever $\CF'\in \Shv(\CZ\times \CY_2)$ lies in the essential
image of the functor 
$$(\on{id}\times f)_*:\Shv(\CZ\times \wt\CY_2)\to \Shv(\CZ\times \CY_2).$$

\medskip

We claim that this implies that \eqref{e:adj to check} is an isomorphism for all $\CF'\in \Shv(\CZ\times \CY_2)$,
which would mean that (i) holds. 

\medskip

Indeed, by the assumption in \secref{sss:safety for adj}, the functor
$\on{Id}_\CZ\boxtimes {}'\!\sQ^R$ maps isomorphically to $\on{Id}_\CZ\boxtimes {}'\!\sQ^R_{\on{disc}}$,
and the latter functor commutes with \emph{limits}. Now, every object in $\CF'\in \Shv(\CZ\times \CY_2)$ can
be written as a totalization of a cosimplicial object, whose terms belong to the essential
image of $(\on{id}\times f)_*$. 

\sssec{}

Thus, we can assume that $\CY_2$ is a scheme. (Note that in this case, the assumption in
\secref{sss:safety for adj} holds automatically.)

\medskip

We will show that (iii) implies (v). Namely, we will show that if (iii) holds, then the morphisms 
\eqref{e:ker adj unit} and \eqref{e:ker adj counit prel} satisfy the adjunction axioms.

\medskip

We will show that the composition
\begin{equation} \label{e:compos axiom}
\CQ \simeq \CQ \star \on{u}_{\CY_1}\overset{\text{\eqref{e:ker adj unit}}}\longrightarrow \CQ \star ({}'\!\CQ^R\star \CQ) \simeq
(\CQ \star {}'\!\CQ^R) \star \CQ \overset{\text{\eqref{e:compare counits}}}\simeq 
(\CQ \lstar \BD^{\on{Verdier}}(\CQ)^\sigma) \star \CQ \overset{\text{\eqref{e:ker adj counit prel}}}\longrightarrow 
\on{u}_{\CY_2}\star \CQ\simeq \CQ
\end{equation}
is the identity map. The other adjunction axiom is checked similarly.

\sssec{}

Since $\CY_2$ is a scheme, the map
$$\BD^{\on{Verdier}}(\CQ)^\sigma \star\CQ\to \BD^{\on{Verdier}}(\CQ)^\sigma \star_{\on{disc}} \CQ$$
is an isomorphism. 

\medskip

Hence, we can view the map \eqref{e:ker adj unit prel prel} as a map
\begin{equation} \label{e:ker adj unit prel}
\on{ps-u}_{\CY_1}\to \BD^{\on{Verdier}}(\CQ)^\sigma \star \CQ. 
\end{equation} 

Similarly, we can view the map \eqref{e:ker adj unit prel prel bis} as a map
\begin{equation} \label{e:ker adj unit prel bis}
\on{u}_{\CY_1}\to (\on{Id}_{\CY_1}\boxtimes \on{Id}^l_{\CY_1})(\BD^{\on{Verdier}}(\CQ)^\sigma \star\CQ),
\end{equation} 

\sssec{}

For any $\CQ\in \Shv(\CY_1\times \CY_2)^{\on{constr}}$, we have a commutative diagram
$$
\CD
\CQ @>{\on{id}}>> \CQ \\
@V{\sim}VV @VV{\sim}V \\
\CQ \lstar \on{ps-u}_{\CY_1} @>{\text{\eqref{e:right vs left conv}}}>> \CQ \star \on{u}_{\CY_1} \\
@V{\text{\eqref{e:ker adj unit prel}}}VV @VV{\text{\eqref{e:ker adj unit prel bis}}}V \\
\CQ \lstar (\BD^{\on{Verdier}}(\CQ)^\sigma \star \CQ)  
@>{\text{\eqref{e:right vs left conv}}}>> \CQ \star ((\on{Id}_{\CY_1}\boxtimes \on{Id}^l_{\CY_1})(\BD^{\on{Verdier}}(\CQ)^\sigma \star \CQ)) \\
@V{\text{\eqref{e:left to right}}}VV @VV{\text{\eqref{e:left to right}}}V \\
(\CQ \lstar \BD^{\on{Verdier}}(\CQ)^\sigma) \star \CQ  & & \CQ \star ({}'\!\CQ^R\star \CQ) \\
@V{\on{id}}VV @VV{\sim}V \\
(\CQ \lstar \BD^{\on{Verdier}}(\CQ)^\sigma) \star \CQ @>{\text{\eqref{e:right vs left conv}}=\text{\eqref{e:compare counits}}}>>  
(\CQ \star {}'\!\CQ^R)\star \CQ \\
@V{\text{\eqref{e:ker adj counit prel}}}VV \\
\on{u}_{\CY_2}\star \CQ \\
@V{\sim}VV \\
\CQ.
\endCD
$$

If (iii) holds, then the bottom horizontal arrow in the above diagram is an isomorphism. In this 
case, by construction,
the composition \eqref{e:compos axiom} identifies with the composite left vertical
arrow in this diagram. 

\medskip

However, it is a straightforward verification that this composite
map is the identity.

\ssec{An example: the ULA property} \label{ss:ULA}

In this subsection we will recast the property of a sheaf to be ULA (with respect to a given map)
in terms of adjunction of functors given by kernels. 

\sssec{}  \label{sss:ULA}

Let $f:\CY\to S$ be a map between algebraic stacks, where $S$ is a separated scheme.
Let $\CF$ be an object of $\Shv(\CY)$. 

\begin{defn}
We shall say that $\CF$ is ULA with respect to $f$ if  the functor 
$$\sF:\Shv(S)\to \Shv(\CY), \quad \CG\mapsto \CF\sotimes f^!(\CG),$$
viewed as a functor defined by the kernel 
$$(\on{Graph}_f)_*(\CF)\in \Shv(S\times \CY),$$
is defined and codefined by a kernel.
\end{defn}

\begin{rem} \label{r:ULA non-constr new}
Usually, one defines the notion of ULA for objects that are constructible. The above definition is
justified by the following fact (see \cite{Ga3}):

\medskip  

An object $\CF\in \Shv(\CY)$ is ULA with respect to $f:\CY\to S$ if and only if for every $k$, 
every constructible sub-object $\CF'\subset H^k(\CF)$ has this property.

\medskip

For constructible objects, the above notion of ULA coincides with the classical one, 
see \cite[Appendix B]{BG}. 

\end{rem}

\begin{rem} \label{r:ULA comp}

Assume for a moment that $\CF$ is compact. From \propref{p:conv righ adj main}, we obtain that $\CF$ is ULA
if and only if the above functor $\sF$ admits a right adjoint, as a functor defined by a kernel.

\end{rem}

\begin{rem}

We claim that in order to check that a given object $\CF\in \Shv(\CY)^{\on{constr}}$ is ULA with respect to $f$, it is enough to
check that one particular morphism taking place in $\Shv(\CY\underset{S}\times \CY)$ is an isomorphism
(cf. \cite[Proposition 3.3]{HS}).

\medskip

This map in question is
$$'\!\Delta_S^*(\CF\boxtimes \BD^{\on{Verdier}}(\CF))\to 
{}'\!\Delta_S^!(\CF\boxtimes (f^*(\ul\sfe_S)\overset{*}\otimes \BD^{\on{Verdier}}(\CF))),$$
where $'\!\Delta_S$ is the map 
$$\CY\underset{S}\times \CY\to \CY\times \CY.$$

\medskip

Indeed, since the ULA condition is local on $\CY$, we can assume that $\CY$ is a scheme, and hence
$\CF$ is compact. Then our assertion follows from Remark \ref{r:ULA comp} and 
the equivalence (i) $\Leftrightarrow$ (iii) in \thmref{t:right adj crit}.

\end{rem}

\sssec{} \label{sss:preserve ULA}

We will now prove a statement that the property of being ULA is preserved by certain 
functors. 

\medskip

Let $S$ be as in \secref{sss:ULA}, and let $f_i:\CY_i\to S$, $i=1,2$ be a pair of algebraic
stacks over $S$. Let $\CQ$ be an object of $\Shv(\CY_1\times \CY_2)$, such that the corresponding
functor $\sQ$ is defined and codefined by a kernel. Assume
that $\CQ$ is supported on 
$$\CY_1\underset{S}\times \CY_2\subset \CY_1\times \CY_2.$$

\medskip

We claim:

\begin{cor} \label{c:preserve ULA}
Under the above circumstances, the functor $\sQ:\Shv(\CY_1)\to \Shv(\CY_2)$
maps objects that are ULA with respect to $f_1$ to objects that are ULA with respect to $f_2$.
\end{cor}

\begin{proof}

Let $\CF$ be an object of $\Shv(\CY_1)$ that is ULA over $S$. Consider the object
$$(\on{Graph}_{f_1})_*(\CF)\in \Shv(S\times \CY_1)$$
and the corresponding functor $\sF:\Shv(S)\to \Shv(\CY_1)$. 

\medskip

By assumption, the functor $\sF$ is defined and codefined by a kernel. By assumption, the same is true
for the composition $\sQ\circ \sF$. 

\medskip

The functor $\sQ\circ \sF$ is defined by the kernel 
$$(\on{Id}_S\boxtimes \sQ)((\on{Graph}_{f_1})_*(\CF))\in \Shv(S\times \CY_2).$$ 

Now, the fact that $\CQ$ is supported on $\CY_1\underset{S}\times \CY_2$ implies that
$$(\on{Id}_S\boxtimes \sQ)((\on{Graph}_{f_1})_*(\CF))\simeq
(\on{Graph}_{f_2})_*(\sQ(\CF)).$$

Hence, $\sQ(\CF)$ is ULA with respect to
$f_2$, as desired.

\end{proof} 

\sssec{} \label{sss:preserve ULA prod}

An example of the situation described in \secref{sss:preserve ULA} is when
$$\CY_i\simeq S\times \CY'_i,$$
and $\CQ$ comes from an object $\CQ'\in \Shv(\CY'_1\times \CY'_2)$ such that the
corresponding functor $\sQ'$ is defined and codefined by a kernel.
 %and the functors $\on{Id}_\CZ\boxtimes \sQ$ preserve constructibility. 

\sssec{}

In the main body of the paper, we use the following particular case of \corref{c:preserve ULA},
in the setting of \secref{sss:preserve ULA prod}.

\medskip

Let $\CZ$ be an algebraic stack, and let $\CN$ be a conical Zariski-closed subset
of $T^*(\CZ)$. Let $\CY$ be another algebraic stack, and let $\CF\in \Shv(\CY)$
be an object such that the corresponding functor 
$$\sF:\Shv(\CY)\to \Vect$$
is defined and codefined by a kernel.
%, and the functors $\on{Id}_\CZ\boxtimes \sF$ preserve constructibility. 

\medskip

We claim:

\begin{cor} \label{c:preserve sing supp}
Under the above circumstances, the functor
$$\on{Id}_\CZ\boxtimes \sF:\Shv(\CZ\times \CY)\to \Shv(\CZ)$$
sends the subcategory 
$$\Shv_{\CN\times T^*(\CY)}(\CZ\times \CY)\subset \Shv(\CZ\times \CY)$$
to 
$$\Shv_\CN(\CZ) \subset \Shv(\CZ).$$
\end{cor}

\begin{proof}
%
%The fact that $\on{Id}_\CZ\boxtimes \sF$
%sends $\Shv(\CZ\times \CY)^c$ to $\Shv(\CZ)^c$ follows from the fact that it admits a continuous
%right adjoint.
%
%\medskip

%By assumption, $\CF$ is constructible, and by \secref{sss:boxtimes disc bis}, the functor $\on{Id}_\CZ\boxtimes \sF$
%is of finite cohomological dimension. Hence, it suffices to show that the functor $\on{Id}_\CZ\boxtimes \sF$
%sends 
%$$\Shv_{\CN\times T^*(\CY)}(\CZ\times \CY)^{\on{constr}}\subset \Shv(\CZ\times \CY)$$
%to 
%$$\Shv_\CN(\CZ)^{\on{constr}} \subset \Shv(\CZ).$$
%
%\medskip

By the definition of singular support in \cite{Bei}\footnote{This definition is applicable for objects of $\Shv(-)$ rather
than $\Shv(-)^{\on{constr}}$ using Remark \ref{r:ULA non-constr new}.}
, and up to base changing $\CZ$ to a scheme, we have to show that for a map of schemes
$f:S\to S'$, if an object $\CG\in \Shv_{\CN\times T^*(S)}(S\times \CY)^{\on{constr}}$
is ULA with respect to the composition 
$$S\times \CY\to S\overset{f}\to S',$$
then $(\on{Id}_S\boxtimes \sF)(\CG)\in \Shv(S)$ is ULA with respect to $f$. 

\medskip

However, this follows from \corref{c:preserve ULA}.

\end{proof}
%
%\begin{rem} \label{r:ULA non-constr curves}
%As in Remark \ref{r:ULA non-constr}, one can show that the assertion of \corref{c:preserve sing supp} holds
%without the extra assumption that $\sF$ preserve compactness. 
%
%\medskip
%
%Let us prove this when $\CZ=X$ is a smooth curve and $\CN=\{0\}$:
%
%\medskip
%
%The above argument shows that the functor $\on{Id}_X\boxtimes \sF$ sends objects from 
%$\Shv_{\{0\}\times T^*(\CY)}(X\times \CY)$ to objects $\CG\in \Shv(X)$ with the property
%that the functor
%$$\CG\sotimes -: \Shv(X)\to \Shv(X)$$
%is defined and codefined by a kernel. Such objects $\CG$ have the property that for any $x\in X$,
%the total vanishing cycles functor $\Phi_x$, applied to $\CG$, is zero. However, an object $\CG$ with this
%property belongs to $\qLisse(X)$. 
%
%\end{rem}

\section{Sheaves on non quasi-compact algebraic stacks} \label{s:non qc}

The theory reviewed in Sects. \ref{s:sheaves}-\ref{s:ker} mostly pertains to quasi-compact algebraic stacks. 
In this section we explore the adjustments once needs to make in order to extend the theory to the non
quasi-compact case. 

\ssec{Cotruncative substacks} 

In this subsection, we review, mostly following \cite{DrGa1}, the notion of \emph{cotruncativeness}
of an open embedding, as well associated notions in the presence of a singular support condition. 

\sssec{} \label{sss:truncative}

Let 
$$j:\CU_1\hookrightarrow \CU_2$$
be an open embedding of (quasi-compact) algebraic stacks.

\medskip

We shall say that $\CU_1$ is a \emph{cotruncative} substack of $\CU_2$ if the functor $j_*$, \emph{viewed as a functor defined by a kernel},
admits a right adjoint, to be denoted $j^?$. Let $\CQ\in \Shv(\CU_2\times \CU_1)$ be the kernel of this right adjoint. 

\medskip

It follows from \secref{sss:create left adjoint} that the object $\CQ^\sigma\in \Shv(\CU_1\times \CU_2)$
defines the \emph{left} adjoint to $j^*$, as a functor defined by a kernel. In particular, the system of functors
$$(\on{id}\times j)_!:\Shv(\CZ\times \CU_1)\to \Shv(\CZ\times \CU_2)$$
satisfies the compatibilities of \secref{sss:compat kernels new}. 

\medskip

As in \cite[Lemma 3.7.1]{DrGa1}, one shows:

\begin{lem} \label{l:union cotruncative}
If $\CU'_1\subset \CU_2$ and $\CU''_1\subset \CU_2$
are cotruncative, then so is their union.
\end{lem}  

\sssec{} 

Let 
$$j:\CU\hookrightarrow \CY$$
be an open embedding of not necessarily quasi-compact algebraic stacks. We say that $\CU$ is 
a \emph{cotruncative} substack of $\CY$ if for every quasi-compact open $\CU'\subset \CY$,
the intersection $\CU\cap \CU'$ is cotruncative as an open substack of $\CU'$.

\sssec{} \label{sss:truncatable}

We shall say that a (not necessarily quasi-compact) algebraic stack $\CY$ is \emph{truncatable}
if it can be written as a union of quasi-compact cotruncative open substacks. 

\medskip

If $\CY$ is truncatable, by \lemref{l:union cotruncative} the poset of its quasi-compact cotruncative 
open substacks is filtered and cofinal in the poset of all quasi-compact open substacks.

\sssec{} \label{sss:N-cotrunc}

Let $j:\CU\hookrightarrow \CY$ be as above. Let $\CN$ be a conical Zariski closed subset 
of $T^*(\CY)$. By a slight abuse of notation we will denote by the same symbol $\CN$
its restriction to $\CU$. 

\medskip

We shall say that $\CU$ is $\CN$-\emph{cotruncative} if it is cotruncative \emph{and} the extension 
functor $j_!$ (equivalently, $j_*$) sends 
$$\Shv_{\CN}(\CU)\to \Shv_{\CN}(\CY).$$

\medskip

We shall say that $\CU$ is \emph{universally} $\CN$-\emph{cotruncative} if it is cotruncative \emph{and} 
for any algebraic stack $\CZ$ and a closed conical subset $\CN_\CZ \subset T^*(\CZ)$, the extension 
functor $(\on{id}\times j)_!$ (equivalently, $(\on{id}\times j)_*$)
sends 
$$\Shv_{\CN_\CZ\times \CN}(\CZ\times \CU)\to \Shv_{\CN_\CZ \times \CN}(\CZ\times \CY).$$

\sssec{} \label{sss:N-trunc}

We shall say that $\CY$ is $\CN$-truncatable (resp., universally $\CN$-truncatable) 
if it can be written as a \emph{filtered} union of its $\CN$-cotruncative (resp., universally $\CN$-cotruncative) 
quasi-compact open substacks. 

\ssec{The non-quasi-compact case: the ``co"-category} \label{ss:co}

If $\CY$ is a non quasi-compact algebraic stack, Verdier duality is no longer a self-duality on $\Shv(\CY)$,
but rather a duality between $\Shv(\CY)$ and another category, denoted $\Shv(\CY)_{\on{co}}$.

\medskip

In this subsection we introduce the category $\Shv(\CY)_{\on{co}}$ and study its basic properties. 

\sssec{} \label{sss:Shv on non qc}

Let $\CY$ be a (not necessarily quasi-compact) algebraic stack. By definition, $\Shv(\CY)$ is
$$\underset{\CU}{\on{lim}}^*\, \Shv(\CU),$$
where the index category is the poset of quasi-compact open substacks $\CU\subset \CY$,
and for $\CU_1\overset{j}\hookrightarrow \CU_2$ the transition functor
$\Shv(\CU_2)\to \Shv(\CU_1)$ is $j^*$.  

\medskip

By \cite[Proposition 1.7.5]{DrGa1}, we can rewrite the above limit as a colimit
$$\underset{\CU}{\on{colim}_!}\, \Shv(\CU),$$
where the transition functors are 
$\Shv(\CU_1)\to \Shv(\CU_2)$ are $j_!$.  

\sssec{} \label{sss:co}

We define the category $\Shv(\CY)_{\on{co}}$ as
$$\underset{\CU}{\on{colim}_*}\, \Shv(\CU),$$
where the transition functors are 
$\Shv(\CU_1)\to \Shv(\CU_2)$ are $j_*$.  

\medskip

For a given quasi-compact
$$U\overset{j}\hookrightarrow \CY,$$
we will denote by $j_{*,\on{co}}$ the tautologically defined functor
$$\Shv(\CU)\to \Shv(\CY)_{\on{co}}.$$

\medskip

Note that when $\CY$ is truncatable, we can replace the poset of all $\CU$ by the cotruncative
ones. In this case, we can write $\Shv(\CY)_{\on{co}}$ also as 
$$\underset{\CU}{\on{lim}^?}\, \Shv(\CU),$$
where the transition functors are 
$\Shv(\CU_2)\to \Shv(\CU_1)$ are $j^?$. 

\sssec{} \label{sss:Id naive}

We have a tautologically defined functor
$$\on{Id}^{\on{naive}}_\CY:\Shv(\CY)_{\on{co}}\to \Shv(\CY).$$
It corresponds to the compatible family of functors
$$(\CU\overset{j}\hookrightarrow \CY) \rightsquigarrow (j_*:\Shv(\CU)\to \Shv(\CY)).$$

\medskip

For a quasi-compact $U\overset{j}\hookrightarrow \CY$ we have 
$$\on{Id}^{\on{naive}}_\CY\circ j_{*,\on{co}}\simeq j_*.$$

\medskip

In addition, we have a naturally defined monoidal action of $\Shv(\CY)$ on $\Shv(\CY)_{\on{co}}$
given by $\sotimes$.

\sssec{} \label{sss:co N}

Let $\CN$ be a closed conical subset in $T^*(\CY)$, and assume that $\CY$ is $\CN$-truncatable. 

\medskip

In this case we define the category
$$\Shv_\CN(\CY)_{\on{co}}$$ 
as 
$$\underset{\CU}{\on{colim}_*}\, \Shv_\CN(\CU),$$
where the colimit is taken over the poset of $\CN$-cotruncative quasi-compact open substacks of
$\Bun_G$. 

\medskip

The fact that the above poset is filtered implies that the tautologically defined functor
$$\Shv_\CN(\CY)_{\on{co}}\to \Shv(\CY)_{\on{co}}$$
is fully faithful. 

\medskip

Furthermore, the functor $\on{Id}^{\on{naive}}_\CY$ sends 
$$\Shv_\CN(\CY)_{\on{co}}\to \Shv_\CN(\CY).$$

\ssec{Verdier duality in the non-quasi-compact case} \label{ss:Verdier non qc}

In this subsection we establish a duality between $\Shv(\CY)$ and $\Shv(\CY)_{\on{co}}$. 

\sssec{} \label{sss:dual of colimit}

Recall the following paradigm:

\medskip

Let $$i\mapsto \bC_i, \quad i\in I$$
be a diagram of DG categories. Denote
$$\bC:=\underset{i\in I}{\on{colim}}\, \bC_i.$$

Suppose that for each arrow $i_1\to i_2$, the transition functor
$$\phi_{i_1,i_2}:\bC_{i_1}\to \bC_{i_2}$$
admits a continuous right adjoint. Suppose also that each $\bC_i$ is dualizable.

\medskip

We can form a new family
\begin{equation} \label{e:dual family}
i\mapsto \bC^\vee_i, \quad i\in I,
\end{equation}
where the transition functor $\bC^\vee_{i_1}\to \bC^\vee_{i_2}$
is $(\phi_{i_1,i_2}^R)^\vee$. 

\medskip

Denote 
$$\wt\bC:=\underset{i\in I}{\on{colim}}\, \bC^\vee_i.$$

Supposed for simplicity that the index category $I$ is sifted, and that the transition functors 
$\phi_{i_1,i_2}$ are fully faithful.
 
 \medskip
 
Then we have a naturally defined pairing
\begin{equation} \label{e:colimit dual}
\bC\otimes \wt\bC\to \Vect.
\end{equation}

Namely, using the siftedness assumption on $I$, we write 
$$\bC\otimes \wt\bC \simeq \underset{i\in I}{\on{colim}}\, \bC_i\otimes \bC^\vee_i,$$
and the sought-for pairing is given by the tautological pairings $\bC_i\otimes \bC^\vee_i\to \Vect$. 

\medskip

The following is a particular case of \cite[Proposition 1.8.3]{DrGa1}:

\begin{prop} \label{p:colimit dual}
The functor \eqref{e:colimit dual} identifies $\wt\bC$ with the dual of $\bC$.
\end{prop}

\sssec{} \label{sss:Verdier non-qc}

We apply the paradigm of \secref{sss:dual of colimit} as follows. Let $\CY$ be a (not necessarily quasi-compact) algebraic stack.
Consider the index set consisting of its quasi-compact open substacks, and consider the functor
$$\CU \rightsquigarrow \Shv(\CU), \,\, (\CU_1\overset{j_{1,2}}\hookrightarrow \CU_2)\rightsquigarrow (j_{1,2})_!.$$

By \secref{sss:Shv on non qc},
$$\underset{\CU}{\on{colim}}\, \Shv(\CU)=:\underset{\CU}{\on{colim}}{}_!\, \Shv(\CU)\simeq \Shv(\CY).$$

Recall (see \secref{sss:Verdier}) that for a quasi-compact stack the category of sheaves is naturally
self-dual via Verdier duality. 

\medskip

The corresponding dual family \eqref{e:dual family} identifies with 
$$\CU \rightsquigarrow \Shv(\CU), \,\, (\CU_1\overset{j_{1,2}}\hookrightarrow \CU_2)\rightsquigarrow (j_{1,2})_*,$$
and its colimit
$$\underset{\CU}{\on{colim}}{}_*\, \Shv(\CU)$$
is by definition $\Shv(\CY)_{\on{co}}$.

\medskip

Applying \propref{p:colimit dual}, we obtain a canonical identification 
\begin{equation} \label{e:Verdier duality non qc}
\Shv(\CY)^\vee \simeq \Shv(\CY)_{\on{co}}.
\end{equation}

\sssec{} \label{sss:pairing non qc}

Unwinding the definitions, we obtain that the counit of the duality \eqref{e:Verdier duality non qc} is given by the functor $\on{ev} _\CY$, i.e., 
$$\Shv(\CY)\otimes \Shv(\CY)_{\on{co}}\overset{\sotimes}\to 
\Shv(\CY)_{\on{co}} \overset{\on{C}^\cdot_\blacktriangle(\CY,-)}\longrightarrow \Vect,$$
where 
$$\on{C}^\cdot_\blacktriangle(\CY,-):\Shv(\CY)_{\on{co}}\to \Vect$$
is the functor given by the compatible collection of functors
$$\on{C}^\cdot_\blacktriangle(\CU,-):\Shv(\CU)\to \Vect, \quad \CU \text{ is a quasi-compact open in } \CY.$$

\medskip

The corresponding contravariant self-equivalence
$$\BD^{\on{Verdier}}:(\Shv(\CY)^c)^{\on{op}}\to (\Shv(\CY)_{\on{co}})^c$$
sends
$$j_!(\CF_\CU)  \mapsto j_{*,\on{co}}(\BD^{\on{Verdier}}(\CF_\CU)), \quad \CF_\CU\in \Shv(\CU)^c$$
for a quasi-compact open $\CU\overset{j}\hookrightarrow \CY$. 

\sssec{} \label{sss:Verdier non-qc N}

Let now $\CN\subset T^*(\CY)$ be a conical Zariski-closed subset. Assume that $\CY$ is (universally) $\Nilp$-truncatable.

\medskip

Let us suppose now that for every (universally) $\CN$-cotruncative quasi-compact open substack $\CU\subset \CY$,
the pair $(\CU,\CN)$ is duality-adapted. 

\medskip

It then follows from \propref{p:colimit dual} that the restriction of $\on{ev} _\CY$ along
$$\Shv_\CN(\CY)\otimes \Shv_\CN(\CY)_{\on{co}} \to \Shv(\CY)\otimes \Shv(\CY)_{\on{co}}$$
defines an identification 
$$\Shv_\CN(\CY)^\vee \simeq \Shv_\CN(\CY)_{\on{co}}.$$

\ssec{Functors defined by kernels in the non-quasi-compact case} \label{ss:functors by ker non qc}

In this subsection we explain modifications to the formalism of functors defined by kernels, required in order
to consider the non quasi-compact case. 

\sssec{}

We first consider
the formalism of functors \emph{codefined} be kernels. This theory requires no modifications
from the quasi-compact case. Namely, an object
$$\CP\in \Shv(\CY_1\times \CY_2),$$
defines a family of functors
$$\on{Id}_\CZ\boxtimes \sP^l:\Shv(\CZ\times \CY_1)\to \Shv(\CZ\times \CY_2),$$
by the same formula
$$(\on{Id}_\CZ\boxtimes \sP^l)(\CF)=(p_2)_!(p_1^*(\CF)\overset{*}\otimes \CP).$$

The observation in \secref{sss:compat kernels new} applies with $-_\blacktriangle$ and $-^!$ replaced by
$-_!$ and $-^*$, respectively. 

%and assume that $\CY$ is truncatable.

\sssec{} \label{sss:mixed categories}

We now consider functors \emph{defined} by kernels. 

\medskip

Let $\CY$ and $\CZ$ be a pair of (not necessarily quasi-compact) algebraic stacks. 
In this case, in addition to
$$\Shv(\CZ\times \CY) \text{ and } \Shv(\CZ\times \CY)_{\on{co}},$$
one can define the mixed categories
$$\Shv(\CZ\times \CY)_{\on{co}_\CZ} \text{ and } \Shv(\CZ\times \CY)_{\on{co}_\CY}.$$

Namely,
$$\Shv(\CZ\times \CY)_{\on{co}_\CZ}:=\underset{\CU_\CY\subset \CY}{\on{lim}^*}\, \Shv(\CZ\times \CU_\CY)_{\on{co}}\simeq 
\underset{\CU_\CY\subset \CY}{\on{lim}^*}\, \underset{\CU_\CZ\subset \CZ}{\on{colim}_*}\, 
\Shv(\CU_\CZ\times \CU_\CY)$$
and
$$\Shv(\CZ\times \CY)_{\on{co}_\CY}:=\underset{\CU_\CZ\subset \CZ}{\on{lim}^*}\,  \Shv(\CU_\CZ\times \CY)_{\on{co}} \simeq 
\underset{\CU_\CZ\subset \CZ}{\on{lim}^*}\, \underset{\CU_\CY\subset \CY}{\on{colim}_*}\, 
\Shv(\CU_\CZ\times \CU_\CY),$$
where $\CU_\CY$ and $\CU_\CZ$ are quasi-compact open substacks of $\CY$ and $\CZ$, respectively. 

\medskip

We have naturally defined functors
$$\on{Id}_\CZ\boxtimes \on{Id}^{\on{naive}}_\CY:\Shv(\CZ\times \CY)_{\on{co}}\to \Shv(\CZ\times \CY)_{\on{co}_\CZ}$$
and
$$\on{Id}^{\on{naive}}_\CZ\boxtimes \on{Id}_\CY: \Shv(\CZ\times \CY)_{\on{co}_\CZ}\to \Shv(\CZ\times \CY),$$
whose composition is the functor
$$\on{Id}^{\on{naive}}_{\CZ\times \CY}:\Shv(\CZ\times \CY)_{\on{co}}\to \Shv(\CZ\times \CY).$$

\sssec{}

Assume now that $\CY$ is truncatable. Then we claim that we can rewrite $\Shv(\CZ\times \CY)_{\on{co}_\CZ}$
also we as 
$$\underset{\CU_\CZ\subset \CZ}{\on{colim}_*}\,  \underset{\CU_\CY\subset \CY}{\on{lim}^*}\, 
\Shv(\CU_\CZ\times \CU_\CY)\simeq \underset{\CU_\CZ\subset \CZ}{\on{colim}_*}\, \Shv(\CU_\CZ\times \CY).$$

Indeed, we note that for a cotruncative 
$$\CU_{1,\CY}\overset{j_\CY}\hookrightarrow \CU_{2,\CY} $$
and any $\CU_{1,\CZ}\overset{j_\CZ}\hookrightarrow \CU_{2,\CZ}$, the diagram
$$
\CD
\Shv(\CU_{1,\CU}\times \CU_{1,\CZ}) @>{(j_\CY\times \on{id})_!}>> \Shv(\CU_{2,\CU}\times \CU_{1,\CZ}) \\
@V{(\on{id}\times j_\CZ)_*}VV @VV{(\on{id}\times j_\CZ)_*}V \\
\Shv(\CU_{1,\CU}\times \CU_{2,\CZ}) @>{(j_\CY\times \on{id})_!}>> \Shv(\CU_{2,\CU}\times \CU_{2,\CZ})  
\endCD
$$
commutes.

\medskip

Hence, we can rewrite 
\begin{multline*} 
\underset{\CU_\CY\subset \CY}{\on{lim}^*}\, \underset{\CU_\CZ\subset \CZ}{\on{colim}_*}\, 
\Shv(\CU_\CZ\times \CU_\CY)\simeq 
\underset{\CU_\CY\subset \CY}{\on{colim}_!}\, \underset{\CU_\CZ\subset \CZ}{\on{colim}_*}\, \Shv(\CU_\CZ\times \CU_\CY)\simeq  \\
\simeq \underset{\CU_\CZ\subset \CZ}{\on{colim}_*}\, \underset{\CU_\CY\subset \CY}{\on{colim}_!}\, \Shv(\CU_\CZ\times \CU_\CY)
\simeq \underset{\CU_\CZ\subset \CZ}{\on{colim}_*}\,  \underset{\CU_\CY\subset \CY}{\on{lim}^*}\, 
\Shv(\CU_\CZ\times \CU_\CY).
\end{multline*} 

\sssec{} \label{sss:kernels non qc}

The material in Sect. \ref{ss:ker} goes through with the following modifications:

\medskip

Let $\CY_1$ and $\CY_2$ be a pair of truncatable stacks, and let $\CZ$ be another algebraic stack.  

\begin{itemize}

\item An object $\CQ\in \Shv(\CY_1\times \CY_2)$ gives rise to functors
$$\on{Id}_\CZ\boxtimes \sQ:\Shv(\CZ\times \CY_1)_{\on{co}}\to \Shv(\CZ\times \CY_2)_{\on{co}_\CZ} \text{ and } 
\on{Id}_\CZ\boxtimes \sQ:\Shv(\CZ\times \CY_1)_{\on{co}_{\CY_1}}\to \Shv(\CZ\times \CY_2).$$

\item An object $\CQ\in \Shv(\CY_1\times \CY_2)_{\on{co}}$ gives rise to functors
$$\on{Id}_\CZ\boxtimes \sQ:\Shv(\CZ\times \CY_1)\to \Shv(\CZ\times \CY_2)_{\on{co}_{\CY_2}} \text{ and }
\on{Id}_\CZ\boxtimes \sQ:\Shv(\CZ\times \CY_1)_{\on{co}_\CZ}\to \Shv(\CZ\times \CY_2)_{\on{co}}.$$

\item An object $\CQ\in \Shv(\CY_1\times \CY_2)_{\on{co}_{\CY_1}}$ gives rise to functors
$$\on{Id}_\CZ\boxtimes \sQ:\Shv(\CZ\times \CY_1)\to \Shv(\CZ\times \CY_2) \text{ and }
\on{Id}_\CZ\boxtimes \sQ:\Shv(\CZ\times \CY_1)_{\on{co}_\CZ}\to \Shv(\CZ\times \CY_2)_{\on{co}_\CZ}.$$ 

\item An object $\CQ\in \Shv(\CY_1\times \CY_2)_{\on{co}_{\CY_2}}$ gives rise to functors
$$\on{Id}\otimes \sQ:\Shv(\CZ\times \CY_1)_{\on{co}_{\CY_1}} \to \Shv(\CZ\times \CY_2)_{\on{co}_{\CY_2}} \text{ and }
\on{Id}\otimes \sQ:\Shv(\CZ\times \CY_1)_{\on{co}}\to \Shv(\CZ\times \CY_2)_{\on{co}}.$$ 

\end{itemize}

\sssec{} \label{sss:co1}

When $\CZ=\CY$, we will use the notations
$$\Shv(\CY\times \CY)_{\on{co}_1} \text{ and } \Shv(\CY\times \CY)_{\on{co}_2}$$
for the corresponding mixed categories.

\medskip

Let $\CF$ be an object of $\Shv(\CY)$. We have
the usual diagonal object
$$(\Delta_\CY)_*(\CF)\in \Shv(\CY\times \CY).$$

In addition, we can consider the objects
$$(\Delta_\CY)_*(\CF)_{\on{co}_1}\in \Shv(\CY\times \CY)_{\on{co}_1} \text{ and }
(\Delta_\CY)_*(\CF)_{\on{co}_2}\in \Shv(\CY\times \CY)_{\on{co}_2},$$
defined as follows:

\medskip

The restriction of $(\Delta_\CY)_*(\CF)_{\on{co}_1}$ to $\CY\times \CU$ for a quasi-compact open
$\CU\overset{j}\hookrightarrow \CY$ is
$$(j\times \on{id})_{*,\on{co}}\circ (\Delta_\CU)_*\circ j^*(\CF)\in 
\Shv(\CY\times \CU)_{\on{co}},$$
and similarly for $(\Delta_\CY)_*(\CF)_{\on{co}_2}$. 

\medskip

By \secref{sss:kernels non qc}, the object $(\Delta_\CY)_*(\CF)_{\on{co}_1}$ defines a functor
$$\sF:\Shv(\CY)\to \Shv(\CY),$$
and the object $(\Delta_\CY)_*(\CF)_{\on{co}_2}$ defines a functor
$$\sF:\Shv(\CY)_{\on{co}}\to \Shv(\CY)_{\on{co}}.$$

In both contexts, the functor $\sF$ is given by $\CF\sotimes -$. 
%In particular, if $\CF=\omega_\CY$,
%the resulting endofunctors of $\Shv(\CY)$ and $\Shv(\CY)_{\on{co}}$, respectively, are given by the identity. 

\medskip

Let now $\CF$ be an object of $\Shv(\CY)_{\on{co}}$. Then $(\Delta_\CY)_*(\CF)$ is naturally 
an object of $\Shv(\CY\times \CY)_{\on{co}}$. The resulting functor
$$\sF:\Shv(\CY)\to \Shv(\CY)_{\on{co}}$$
is given by $-\sotimes \CF$.

\sssec{} \label{sss:u co 1 2}

When $\CF=\omega_\CY\in \Shv(\CY)$, we will use the notation
$$\on{u}_\CY^{\on{naive}}:=(\Delta_\CY)_*(\omega_\CY),\,\,
\on{u}_{\CY,\on{co}_1}:=(\Delta_\CY)_*(\omega_\CY)_{\on{co}_1},\,\, 
\on{u}_{\CY,\on{co}_2}:=(\Delta_\CY)_*(\omega_\CY)_{\on{co}_2}.$$

\medskip

In terms of \secref{sss:kernels non qc}, the object $\on{u}_\CY^{\on{naive}}$ defines the functors
$$\on{Id}_\CZ\boxtimes \on{Id}^{\on{naive}}_\CY: \Shv(\CZ\times \CY)_{\on{co}}\to \Shv(\CZ\times \CY)_{\on{co}_\CZ}$$
and
$$\on{Id}_\CZ\boxtimes \on{Id}^{\on{naive}}_\CY: \Shv(\CZ\times \CY)_{\on{co}_\CY}\to \Shv(\CZ\times \CY).$$

The object $\on{u}_{\CY,\on{co}_1}$ defines \emph{the identity} functors 
$$\Shv(\CZ \times \CY)\to \Shv(\CZ\times \CY) \text{ and } \Shv(\CZ \times \CY)_{\on{co}_\CZ}\to \Shv(\CZ\times \CY)_{\on{co}_\CZ}$$
and the 
object $\on{u}_{\CY,\on{co}_2}$ defines \emph{the identity} functors 
$$\Shv(\CZ \times \CY)_{\on{co}}\to \Shv(\CZ \times \CY)_{\on{co}} \text{ and }
\Shv(\CZ \times \CY)_{\on{co}_\CY}\to \Shv(\CZ \times \CY)_{\on{co}_\CY}.$$

\sssec{} \label{sss:Mir non qc}

We now consider the object
$$\on{ps-u}_\CY:=(\Delta_\CY)_!(\sfe_\CY)\in \Shv(\CY\times \CY).$$

By \secref{sss:kernels non qc}, it defines a functor
$$\Shv(\CY)_{\on{co}}\to \Shv(\CY).$$

This is the \emph{miraculous functor}, to be denoted $\Mir_\CY$. 

\sssec{} \label{sss:Mir and j}

Let $\CY'\overset{j}\hookrightarrow \CY$ be a cotruncative open embedding. 
It follows formally that we have
\begin{equation} \label{e:Mir and j}
j_!\circ \Mir_{\CY'} \simeq \Mir_\CY \circ j_{*,\on{co}},
\end{equation}
as functors given by kernels. 

\sssec{} \label{sss:codef co}

We will now consider one more variant of the construction of functors attached to 
sheaves on the product of two stacks.

\medskip

Let $\CY_1$ and $\CY_2$ be truncatable, and let $\CP$ be an object of $\Shv(\CY_1\times \CY_2)_{\on{co}_{\CY_2}}$.
Then for any $\CZ$, we have a well-defined functor
$$\on{Id}_\CZ\boxtimes \sP^l:\Shv(\CZ\times \CY_1)\to \Shv(\CZ\times \CY_2)_{\on{co}_{\CY_2}}.$$

\ssec{Miraculous stacks--the non quasi compact case} \label{ss:Mir stacks non-qc}

In this subsection we define what it means for a not necessarily quasi-compact algebraic stack to 
be miraculous. 

\sssec{} \label{sss:Mir stacks non qc}

Let now $\CY$ by a not necessarily quasi-compact truncatable stack. 

\medskip

We shall say that $\CY$ is \emph{miraculous}
if $\Mir_\CY$ is \emph{invertible, as a functor defined by a kernel}. I.e., if there exists an object
$$\CQ\in \Shv(\CY\times \CY)_{\on{co}}$$
such that
$$\CQ\star \on{ps-u}_\CY\simeq \on{u}_{\CU,\on{co}_2}$$
as objects of $\Shv(\CY\times \CY)_{\on{co}_2}$ 
and
$$\on{ps-u}_\CY\star \CQ \simeq \on{u}_{\CU,\on{co}_1}$$
as objects of $\Shv(\CY\times \CY)_{\on{co}_1}$.  

\sssec{} \label{sss:Mir inherited}

It is easy to see that $\CY$ is miraculous if and only if every \emph{cotruncative} quasi-compact open substack $\CU\subset \CY$
is miraculous: 

\medskip

Indeed, the ``if" direction follows from \eqref{e:Mir and j}. For the ``only if" direction, if $\CY$ is miraculous and $\CU\subset \CY$
is cotruncative, formula \eqref{e:Mir and j} implies that $\on{Id}_\CZ \boxtimes \Mir_\CU$ is fully faithful. The essential surjectivity follows from the formula
obtained from \eqref{e:Mir and j} by duality: 
$$\Mir_\CU \circ j^? \simeq j^* \circ \Mir_\CY,$$
where $j^?$ is the right adjoint of $j_{*,\on{co}}$. 

\sssec{}

Still equivalently, $\CY$ is miraculous if and only if the functor
$$(\on{Id}_\CZ \boxtimes \Mir_\CY):\Shv(\CZ\times \CY)_{\on{co}}\to \Shv(\CZ\times \CY)$$
is an equivalence for every quasi-compact $\CZ$.

\medskip

If $\CY$ is miraculous, then the functor
$$(\on{Id}_\CZ\boxtimes \Mir_\CY):\Shv(\CZ\times \CY)_{\on{co}_\CY}\to \Shv(\CZ\times \CY)$$
and 
$$(\on{Id}_\CZ\boxtimes \Mir_\CY):\Shv(\CZ\times \CY)_{\on{co}}\to \Shv(\CZ\times \CY)_{\on{co}_\CZ}$$
are equivalences for an arbitrary $\CZ$.

\sssec{} \label{sss:Mir stacks non qc bis}

Let $\CY$ be miraculous. In particular, $\Mir_\CY$, viewed just as a functor
$$\Shv(\CY)_{\on{co}}\to \Shv(\CY),$$
is an equivalence. 

\medskip 

In this case we define the \emph{miraculous self-duality} of $\Shv(\CY)$ to have as counit the functor
$$\Shv(\CY)\otimes \Shv(\CY) \overset{\Mir_\CY^{-1}\otimes \on{Id}}\longrightarrow
\Shv(\CY)_{\on{co}}\otimes \Shv(\CY) \overset{\on{ev} _\CY}\longrightarrow \Vect,$$
which we denote $\on{ev}^{\Mir}_\CY$.

\medskip

The corresponding contravariant self-equivalence, denoted 
$$\BD^{\Mir}:(\Shv(\CY)^c)^{\on{op}}\to \Shv(\CY)^c,$$
is given by 
$$\BD^{\Mir}\simeq \BD^{\on{Verdier}}\circ \Mir_\CY^{-1}.$$

Thus, explicitly, for a cotruncative quasi-compact $\CU\overset{j}\hookrightarrow \Bun_G$ and $\CF_\CU\in \Shv(\CU)^c$, 
using \eqref{e:Mir and j}, we obtain
$$\BD^{\Mir}(j_!(\CF_\CU))\simeq j_!\left(\BD^{\on{Verdier}} \circ \Mir_\CU^{-1}(\CF_\CU)\right).$$

As in \secref{sss:D Mir sym}, one shows that the pairing $\on{ev}^{\Mir}_\CY$ is swap-equivariant, and hence the
functor $\BD^{\Mir}$ is involutive. 

\begin{rem}
Note that unlike the quasi-compact case, Verdier duality is \emph{not} a self-duality of 
$\Shv(\CY)$, but rather a duality between $\Shv(\CY)$ and $\Shv(\CY)_{\on{co}}$.
So, a priori, miraculous duality is our only option for a self-duality of $\Shv(\CY)$. 
\end{rem}

\end{document}